\documentclass[reqno]{amsart}
\usepackage{amssymb,latexsym,amsmath,amsthm,enumerate,amsbsy}
\usepackage[mathscr]{eucal}
\usepackage{framed,color,graphicx}
\usepackage{mathrsfs}
\usepackage[all]{xy}
\usepackage{tikz}
\usepackage{cite}

\usepackage{longtable}
\usepackage{ltxtable}
\usepackage{subcaption}
\usepackage{graphicx}
\usepackage{subfloat}
\usepackage{subfig}

\usepackage{makecell}

\usepackage{subcaption}

\makeatletter
\@namedef{subjclassname@2020}{\textup{2020} Mathematics Subject Classification}
\makeatother

\usetikzlibrary{positioning,decorations.pathreplacing,patterns,decorations.pathmorphing}
\tikzset{%
element/.style={draw, shape=circle, fill=white, inner sep=1.4pt}
}

\DeclareSymbolFont{bbold}{U}{bbold}{m}{n}
\DeclareSymbolFontAlphabet{\mathbbold}{bbold}

\theoremstyle{plain}
\newtheorem{thm}{Theorem}[section]
\newtheorem{lem}[thm]{Lemma}
\newtheorem{cor}[thm]{Corollary}
\newtheorem{pro}[thm]{Proposition}

\theoremstyle{definition}

\newtheorem{remark}[thm]{Remark}

\renewcommand{\ge}{\geqslant}
\renewcommand{\le}{\leqslant}

\newcommand{\bp}{\mathbf{p}}
\newcommand{\bq}{\mathbf{q}}

\newcommand{\bu}{\mathbf{u}}

\newcommand{\bw}{\mathbf{w}}

\begin{document}

\title[The finite basis problem for ai-semirings of order four]
{The finite basis problem for additively idempotent semirings of order four, IV}

\author[Miaomiao Ren et al.]{Miaomiao Ren}
\address{School of Mathematics, Northwest University, Xi'an, 710127, Shaanxi, P.R. China}
\email{miaomiaoren@yeah.net}

\author[]{Mengya Yue}
\address{School of Mathematics, Northwest University, Xi'an, 710127, Shaanxi, P.R. China}
\email{myayue@yeah.net}

\author[]{Mengyu Yuan}
\address{School of Mathematics, Northwest University, Xi'an, 710127, Shaanxi, P.R. China}
\email{yuanmy1005@163.com}

\author[]{Simin Lyu}
\address{School of Mathematics, Northwest University, Xi'an, 710127, Shaanxi, P.R. China}
\email{siminlyu@yeah.net}

\author[]{Chenyu Yang}
\address{School of Mathematics, Northwest University, Xi'an, 710127, Shaanxi, P.R. China}
\email{chenyuyang0421@yeah.net}

\author[]{Ting Yu}
\address{School of Mathematics, Northwest University, Xi'an, 710127, Shaanxi, P.R. China}
\email{yuting20020116@163.com}


\begin{abstract}
This paper is the fourth in a series devoted to the finite basis problem for $4$-element additively idempotent semirings.
These algebras are divided into five types according to their additive reducts;
we study the largest of these types, namely those whose additive reducts are chains.
Up to isomorphism, there are $386$ such algebras, denoted by $S_{(4, k)}$, $481 \leq k \leq 866$.
We show that all but three of them, namely $S_{(4,545)}$, $S_{(4,634)}$, and $S_{(4,710)}$, are finitely based.
This completes the classification, with respect to the finite basis property,
of all $4$-element additively idempotent semirings with chain additive reducts.
\end{abstract}

\maketitle

\section{Introduction}
An \emph{additively idempotent semiring} (ai-semiring for short)
is an algebra $(S, +, \cdot)$ with two binary operations $+$ and $\cdot$ satisfying the following axioms:
\begin{itemize}
\item The additive reduct $(S, +)$ is a commutative idempotent semigroup;

\item The multiplicative reduct $(S, \cdot)$ is a semigroup;

\item The distributive laws hold:
\[
x(y+z) \approx xy+xz, \quad (x+y)z \approx xz+yz.
\]
\end{itemize}

Let $S$ be an ai-semiring. The additive reduct of $S$ is a semilattice,
and the order $\leq$ on $S$ defined by
\[
a \leq b\;\Leftrightarrow\; a+b=b
\]
is compatible with both addition and multiplication.
Consequently, an ai-semiring is often called a \emph{semilattice-ordered semigroup}.
Whenever an order on an ai-semiring is mentioned, it refers to this order.

A \emph{variety} of ai-semirings is a class of ai-semirings closed under
taking subalgebras, homomorphic images, and arbitrary direct products.
Birkhoff's celebrated theorem tells us that a class of ai-semirings is a variety
if and only if it is an \emph{equational class};
that is, the class of all ai-semirings satisfying a certain set of identities.
Let $\mathcal{V}$ be a variety of ai-semirings.
If $\Sigma$ is a set of identities that defines $\mathcal{V}$, then $\Sigma$ is called an \emph{equational basis} of $\mathcal{V}$.
The variety $\mathcal{V}$ is \emph{finitely based} if it admits a finite equational basis;
otherwise, it is \emph{nonfinitely based}.

For an ai-semiring $S$, let $\mathsf{V}(S)$ denote the variety generated by $S$;
that is, the smallest variety containing $S$.
Equivalently, $\mathsf{V}(S)$ consists of all homomorphic images of subalgebras of direct products of copies of $S$
(or simply, the class of all ai-semirings that satisfy every identity that holds in $S$).
Then $S$ and $\mathsf{V}(S)$ satisfy precisely the same identities.
If $\Sigma$ is an equational basis of $\mathsf{V}(S)$, we also say that $\Sigma$ is an equational basis of $S$.
The ai-semiring $S$ is said to be finitely based (resp., nonfinitely based) if $\mathsf{V}(S)$ is finitely based (resp., nonfinitely based).

Two ai-semirings are \emph{equationally equivalent} if they generate the same variety;
that is, if they satisfy precisely the same identities.
Therefore, if two ai-semirings are equationally equivalent, then they have the same finite basis property.
We shall also use the following elementary fact: if two ai-semirings have the same additive reduct and dual multiplications,
then they also have the same finite basis property.

The \emph{finite basis problem} for a class of ai-semirings,
one of the central problems in universal algebra,
concerns the classification of its members according to whether they are finitely based.
Over the past two decades, this problem has been intensively studied and considerable progress has been made (see~\cite{jrz, rjzl}).
For more information on the finite basis problem for other related algebras such as semigroups and monoids,
we refer the reader to the monograph~\cite{lee}.

The present paper is the fourth in a series devoted to the finite basis problem for $4$-element ai-semirings,
following \cite{rlzc, yrzs, rlyc}.
We focus on those whose additive reducts are chains.
For further background and motivation, we refer the reader to \cite{rlzc, yrzs, rlyc}.

\begin{table}[htbp]
\caption{The 2-element ai-semirings}\label{2}
\begin{tabular}{cccccc}
\hline
Semiring&   $+$ &  $\cdot$ & Semiring &   $+$ &  $\cdot$\\
\Xhline{1.05pt}
$L_2$&
\begin{tabular}{cc}
                    0 & 1  \\
                    1 & 1  \\
\end{tabular}&
\begin{tabular}{cc}
                    0 & 0  \\
                    1 & 1  \\
\end{tabular}
&
$R_2$&   \begin{tabular}{cc}
                    0 & 1  \\
                    1 & 1  \\
\end{tabular}&
\begin{tabular}{cc}
                    0 & 1  \\
                    0 & 1  \\
\end{tabular}
\\
\hline

$M_2$&
\begin{tabular}{cc}
                    0 & 1  \\
                    1 & 1  \\
\end{tabular}&
\begin{tabular}{cc}
                    0 & 1  \\
                    1 & 1  \\
\end{tabular}
&
$D_2$&   \begin{tabular}{cc}
                    0 & 1  \\
                    1 & 1  \\
\end{tabular}&
\begin{tabular}{cc}
                    0 & 0  \\
                    0 & 1  \\
\end{tabular}
\\
\hline
$N_2$&
\begin{tabular}{cc}
                    0 & 1  \\
                    1 & 1  \\
\end{tabular}&
\begin{tabular}{cc}
                    0 & 0  \\
                    0 & 0  \\
\end{tabular}
&
$T_2$&   \begin{tabular}{cc}
                    0 & 1  \\
                    1 & 1  \\
\end{tabular}&
\begin{tabular}{cc}
                    1 & 1  \\
                    1 & 1  \\
\end{tabular}
\\
\hline
\end{tabular}
\end{table}

There are, up to isomorphism, precisely $6$ ai-semirings of order two,
denoted by $L_2$, $R_2$, $M_2$, $D_2$, $N_2$, and $T_2$.
We assume that the carrier set of each of these semirings is $\{0, 1\}$.
Their Cayley tables for addition and multiplication are listed in Table~\ref{2}.
The solution of the equational problem for these algebras will be provided in Lemma~\ref{nlemma1}.
For ai-semirings of order three, there are $61$ isomorphism types,
indexed as $S_i$ for $1 \leq i \leq 61$.
A comprehensive description of these algebras is available in \cite{zrc}.
Rather than listing all $61$ types here,
we will provide the relevant information for a $3$-element ai-semiring only when it is actually used in the sequel.

\setlength{\unitlength}{0.6cm}
\begin{figure}[ht]
\centering
\begin{subfigure}[b]{0.32\textwidth}
\centering
\begin{picture}(7,3)
\put(3.5,2.2){\line(0,-1){1}}
\put(3.5,2.2){\line(2,-1){2}}
\put(3.5,2.2){\line(-2,-1){2}}

\multiput(3.5,2.2)(0,-1){2}{\circle*{0.1}}
\multiput(3.5,2.2)(2,-1){2}{\circle*{0.1}}
\multiput(3.5,2.2)(-2,-1){2}{\circle*{0.1}}
\end{picture}
\caption{Type I} \label{Type I}
\end{subfigure}
\hfill
\begin{subfigure}[b]{0.32\textwidth}
\centering
\begin{picture}(7,3)
\put(3.5,2.2){\line(1,-1){1}}
\put(3.5,2.2){\line(-1,-1){1}}

\put(3.5,0.2){\line(1,1){1}}
\put(3.5,0.2){\line(-1,1){1}}

\multiput(3.5,2.2)(1,-1){2}{\circle*{0.1}}
\multiput(3.5,2.2)(-1,-1){2}{\circle*{0.1}}
\multiput(3.5,0.2)(1,1){2}{\circle*{0.1}}
\end{picture}
\caption{Type II} \label{Type II}
\end{subfigure}
\hfill
\begin{subfigure}[b]{0.32\textwidth}
\centering
\begin{picture}(7,3)
\put(3.5,2.2){\line(0,-1){1}}
\put(3.5,1.2){\line(1,-1){1}}
\put(3.5,1.2){\line(-1,-1){1}}

\multiput(3.5,2.2)(0,-1){2}{\circle*{0.1}}
\multiput(3.5,1.2)(1,-1){2}{\circle*{0.1}}
\multiput(3.5,1.2)(-1,-1){2}{\circle*{0.1}}
\end{picture}
\caption{Type III} \label{Type III}
\end{subfigure}

\begin{subfigure}[b]{0.495\textwidth}
\centering
\begin{picture}(6,3)
\put(3.3,2.2){\line(1,-1){1}}
\put(3.3,2.2){\line(-1,-1){1}}
\put(2.3,1.2){\line(0,-1){1}}

\multiput(3.3,2.2)(1,-1){2}{\circle*{0.1}}
\multiput(2.3,1.2)(0,-1){2}{\circle*{0.1}}
\end{picture}
\caption{Type IV} \label{Type IV}
\end{subfigure}
\hfill
\begin{subfigure}[b]{0.495\textwidth}
\centering
\begin{picture}(6,3)
\put(3.3,2.2){\line(0,-1){2}}
\multiput(3.3,2.2)(0,-1){1}{\circle*{0.1}}
\multiput(3.3,0.2)(0,-1){1}{\circle*{0.1}}
\multiput(3.3,1.54)(0,-1){1}{\circle*{0.1}}
\multiput(3.3,0.86)(0,-1){1}{\circle*{0.1}}
\end{picture}
\caption{Type V} \label{Type V}
\end{subfigure}
\caption{The additive orders of $4$-element ai-semirings}\label{figure1}
\end{figure}
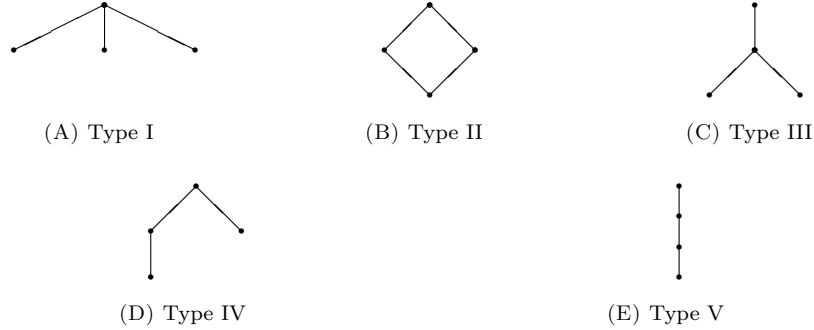

Up to isomorphism, there are exactly $866$ ai-semirings of order four,
denoted by $S_{(4, k)}$, $1 \leq k \leq 866$.
These algebras are classified into five distinct types based on their additive orders, as illustrated in Figure~\ref{figure1}.
The finite basis problem for algebras in the first three types,
namely $\{S_{(4, k)} \mid 1 \leq k \leq 58\}$, $\{S_{(4, k)} \mid 388 \leq k \leq 480\}$, and $\{S_{(4, k)} \mid 276 \leq k \leq 387\}$,
has been resolved by Gao et al.~\cite{gmrz}, Ren et al.~\cite{rlyc, rlzc}, Shaprynski\v{\i}~\cite{shap23},
Wu et al.~\cite{wrz}, and Yue et al.~\cite{yrzs}.

The present paper addresses the finite basis problem for the fifth and largest type,
namely $\{S_{(4,k)} \mid 481 \leq k \leq 866\}$,
which consists of $386$ algebras whose additive reducts are chains.
This type accounts for nearly half of all $4$-element ai-semirings.
Because of its size, this type was initially regarded as the most difficult one, and was expected to be the last to be settled;
it was originally intended to appear in a fifth paper of this series.
However, with the breakthroughs in \cite{yrg}, the difficulties that previously impeded progress have now been overcome,
allowing us to complete this type in the present work.

\setlength{\unitlength}{0.9cm}
\begin{figure}[htbp]
\begin{picture}(35, 3.25)
\put(6.5,2){\line(0,1){1}}
\put(6.5,1){\line(0,1){1}}
\put(6.5,0){\line(0,1){1}}

\put(6.5,3){\circle*{0.1}}
\put(6.5,2){\circle*{0.1}}
\put(6.5,1){\circle*{0.1}}
\put(6.5,0){\circle*{0.1}}

\put(7.0,3){\makebox(0,0){$1$}}
\put(7.0,2){\makebox(0,0){$2$}}
\put(7.0,1){\makebox(0,0){$3$}}
\put(7.0,0){\makebox(0,0){$4$}}
\end{picture}
\caption{The additive order of $S_{(4, k)}$, $481\leq k \leq 866$}\label{figure01}
\end{figure}
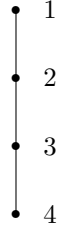

We assume that the carrier set of each of these semirings is $\{1, 2, 3, 4\}$.
Their Cayley tables for addition are determined by Figure~\ref{figure01},
and their Cayley tables for multiplication are listed in Table~\ref{tb1}.
Thus all of these algebras have the same additive reduct:
$1$ is the additive maximum element and $4$ is the additive minimum element.

The main result of this paper is the following, whose proof will be developed in the subsequent sections.
\begin{thm}\label{maintheorem}
The only nonfinitely based ai-semirings in $\{S_{(4, k)} \mid 481 \leq k \leq 866\}$ are
$S_{(4, 545)}$, $S_{(4, 634)}$, and $S_{(4, 710)}$.
\end{thm}

The proof of Theorem~\ref{maintheorem} is long and consists of a substantial number of auxiliary results.
To help the reader navigate the technical details, we outline the overall strategy here.
Section~\ref{section-prelim} collects the necessary preliminaries, including the basic definitions and basic facts that will be used throughout the proof.

In Section~\ref{section-known}, we apply results from \cite{gpz, pas05, rz19, sr, zrc}
to solve the finite basis problem for 236 algebras in $\{S_{(4, k)} \mid 481 \leq k \leq 866\}$.
For the remaining 150 algebras, we employ two different approaches.
For some of these algebras, we show that they are equationally equivalent to algebras already considered in \cite{rlzc, yrzs, rlyc};
their finite basis property then follows immediately.
For the rest, we proceed as follows:
we first select a representative subset of these algebras and exhibit an explicit finite equational basis for each representative;
we then show that each of the remaining algebras is either equationally equivalent to one of these representatives,
or has the dual multiplication of one of them.

The general strategy for establishing a finite equational basis is as follows:
for each representative algebra, we exhibit a finite set of identities that it satisfies,
and then prove, using standard equational logic, that every identity of this algebra can be derived from this finite set.
This constitutes the main body of the paper and is carried out in Sections~\ref{section-s44}--\ref{section-si}.

Our approach relies on a thorough understanding of the equational problem for ai-semirings of orders two and three,
which provides the necessary technical foundation and information required for the present work.
We thus obtain the complete classification stated in Theorem~\ref{maintheorem}.

\section{Preliminaries}\label{section-prelim}
In this section, we collect the notions, notation, and tools that will be used throughout the paper.
Most of the material is standard and can be found in \cite{yrg}.

Let $X$ be a countably infinite set of variables, and let $X^+$ denote the free semigroup over $X$.
An \textit{ai-semiring term} (or simply a \textit{term}) over $X$ is a finite nonempty set of words in $X^+$.
(In the sequel, terms are denoted by bold lowercase letters $\mathbf{u}, \mathbf{v}, \mathbf{w}, \dots$,
while ordinary lowercase letters $x, y, z, \dots$ represent variables.)
A term is represented as a formal sum of its elements.
Specifically, $\mathbf{u} = \mathbf{u}_1 + \mathbf{u}_2 + \cdots + \mathbf{u}_n$
indicates that $\mathbf{u} = \{ \mathbf{u}_1, \mathbf{u}_2, \dots, \mathbf{u}_n \}$.
The order of the summands in the formal sum is irrelevant,
and multiple occurrences of the same word are collapsed into a single occurrence.
Two terms are equal if and only if their underlying sets coincide.

Let $P_f(X^+)$ denote the collection of all terms over $X$.
This set forms an ai-semiring under the usual term addition and multiplication.
It follows from ~\cite[Theorem 2.5]{kp} that
$P_f(X^+)$ is the free ai-semiring over $X$.
An \textit{ai-semiring substitution} (or just a \textit{substitution})
is defined as a semiring homomorphism from $P_f(X^+)$ to itself.

An \emph{ai-semiring identity} (or simply an \emph{identity})
is a formal expression of the form $\mathbf{u} \approx \mathbf{v}$, where $\mathbf{u}$ and $\mathbf{v}$ are terms.
Let $S$ be an ai-semiring and $\mathbf{u} \approx \mathbf{v}$ an identity.
We say that $S$ \emph{satisfies $\mathbf{u} \approx \mathbf{v}$}, or that $\mathbf{u} \approx \mathbf{v}$ \emph{holds in $S$},
if $\varphi(\mathbf{u}) = \varphi(\mathbf{v})$ for every semiring homomorphism $\varphi: P_f(X^+) \to S$.

We denote by $\mathbf{u} \preceq \mathbf{v}$ (or equivalently $\mathbf{v} \succeq \mathbf{u}$)
the identity $\mathbf{v}\approx \mathbf{v} + \mathbf{u}$,
which we refer to as an \textit{ai-semiring inequality} (or simply an \textit{inequality}).
It is routine to verify that an ai-semiring $S$ satisfies the inequality $\mathbf{u} \preceq \mathbf{v}$
if and only if for every semiring homomorphism $\varphi: P_f(X^+) \to S$,
we have $\varphi(\mathbf{u}) \leq \varphi(\mathbf{v})$.
Consequently, $S$ satisfies an identity $\mathbf{u} \approx \mathbf{v}$ precisely when it satisfies
both inequalities $\mathbf{u} \preceq \mathbf{v}$ and $\mathbf{v} \preceq \mathbf{u}$.
Therefore, $\mathbf{u} \approx \mathbf{v}$ is equivalent to
the inequalities $\mathbf{u} \preceq \mathbf{v}$ and $\mathbf{v} \preceq \mathbf{u}$.
For this reason, when dealing with a set of identities,
we may always assume without loss of generality that it consists of inequalities.
The following discussion shows that it suffices to consider inequalities of a special form.

Now let $\Sigma$ be a set of identities, and let $\mathbf{u} \approx \mathbf{v}$ be an identity such that
\[
\mathbf{u} = \mathbf{u}_1 + \cdots + \mathbf{u}_k, \quad
\mathbf{v} = \mathbf{v}_1 + \cdots + \mathbf{v}_\ell,
\]
where $\mathbf{u}_i, \mathbf{v}_j \in X^+$ for $1 \leq i \leq k$ and $1 \leq j \leq \ell$.
One readily checks that the ai-semiring variety defined by $\mathbf{u} \approx \mathbf{v}$
coincides with the ai-semiring variety defined by the inequalities
\[
\mathbf{u}_i \preceq \mathbf{v}, \quad \mathbf{v}_j \preceq \mathbf{u} \quad (1 \leq i \leq k, \, 1 \leq j \leq \ell).
\]
Consequently, to prove that $\mathbf{u} \approx \mathbf{v}$ is derivable from $\Sigma$,
it suffices to show that for every $i$ and $j$,
the inequalities $\mathbf{u}_i \preceq \mathbf{v}, \, \mathbf{v}_j \preceq \mathbf{u}$ can be derived from $\Sigma$.
In view of this, we always restrict our attention to the inequalities of the form
$\mathbf{q} \preceq \mathbf{u}$, where $\mathbf{q}$ is a word and $\mathbf{u}$ is a term.

We now fix some notation that will be used repeatedly.
Let $\bw$ be a word in $X^+$ and $x$ a variable in $X$. Then
\begin{itemize}
\item $h(\bw)$ denotes the head of $\bw$; that is, the first variable occurring in $\bw$;

\item $t(\bw)$ denotes the tail of $\bw$; that is, the last variable occurring in $\bw$;

\item $c(\bw)$ denotes the content of $\bw$; that is, the set of variables occurring in $\bw$;

\item $\ell(\bw)$ denotes the length of $\bw$; that is, the number of variables occurring in $\bw$ counting multiplicities;

\item $m(x, \bw)$ denotes the number of occurrences of $x$ in $\bw$;

\item $S_2(\bw)$ denotes the set of all subwords of $\bw$ of length $2$;

\item $s(\bw)$ denotes the suffix of $\bw$ obtained by deleting its head; that is, $\bw = h(\bw)s(\bw)$;

\item $p(\bw)$ denotes the prefix of $\bw$ obtained by deleting its tail; that is, $\bw = p(\bw)t(\bw)$.
\end{itemize}
For example, if $\bw=x_1x_3^2x_2x_1x_4^5x_2x_6x_5$, then
\[
h(\bw)=x_1,\; t(\bw)=x_5,\; c(\bw)=\{x_i \mid 1 \leq i \leq 6\},\; \ell(\bw)=11,\; m(x_2, \bw)=2,
\]
\[
S_2(\bw)=\{x_1x_3,\; x_3^2,\; x_3x_2,\; x_2x_1,\; x_1x_4,\; x_4^2,\; x_4x_2,\; x_2x_6,\; x_6x_5\},
\]
\[
s(\bw)=x_3^2x_2x_1x_4^5x_2x_6x_5,\quad p(\bw)=x_1x_3^2x_2x_1x_4^5x_2x_6.
\]

Let $\bu$ be a term such that $\bu=\bu_1+\bu_2+\cdots+\bu_n$,
where $\bu_i \in X^+$, $1 \leq i \leq n$.
Let $\bq$ be a nonempty word, and let $k\geq 1$ be an integer. Then
\begin{itemize}

\item $c(\bu)$ denotes the set $\bigcup_{1\leq i \leq n}c(\bu_i)$;

\item $p(\bu)$ denotes the set $\{p(\bu_i) \mid 1 \leq i \leq n\}$;

\item $L_{\geq k}(\bu)$ denotes the set $\{\bu_i \in \bu \mid \ell(\bu_i)\geq k\}$;

\item $L_{\leq k}(\bu)$ denotes the set $\{\bu_i \in \bu \mid \ell(\bu_i)\leq k\}$;

\item $L_k(\bu)$ denotes the set $\{\bu_i \in \bu \mid \ell(\bu_i)= k\}$;

\item $H_{\bq}(\bu)$ denotes the set $\{\bu_i \in \bu \mid h(\bu_i)=h(\bq)\}$;

\item $T_{\bq}(\bu)$ denotes the set $\{\bu_i \in \bu \mid t(\bu_i)=t(\bq)\}$;

\item $D_{\bq}(\bu)$ denotes the set $\{\bu_i \in \bu \mid c(\bu_i)\subseteq c(\bq)\}$;

\item $S_2(\bu)$ denotes the set $\bigcup_{1\leq i \leq n}S_2(\bu_i)$;

\item $M_{k}(\bq)$ denotes the set $\{x \in c(\bq) \mid m(x, \bq)=k\}$.
\end{itemize}

The following result, which can be found in \cite[Lemma 1.1]{sr}, will be repeatedly used without explicit reference.
\begin{lem}\label{nlemma1}
Let $\bq\preceq \bu$ be a nontrivial inequality such that
$\bu=\bu_1+\cdots+\bu_n$, where $\bu_i, \bq\in X^+$, $1\leq i \leq n$. Then
\begin{itemize}
\item[$(1)$] $\bq\preceq \bu$ holds in $L_2$ if and only if $H_{\bq}(\bu)\neq \emptyset$.

\item[$(2)$] $\bq\preceq \bu$ holds in $R_2$ if and only if $T_{\bq}(\bu)\neq \emptyset$.

\item[$(3)$] $\bq\preceq \bu$ holds in $M_2$ if and only if $c(\bq) \subseteq c(\bu)$.

\item[$(4)$] $\bq\preceq \bu$ holds in $D_2$ if and only if $D_{\bq}(\bu)\neq \emptyset$.

\item[$(5)$] $\bq\preceq \bu$ holds in $N_2$ if and only if $\ell(\bq)\geq 2$.

\item[$(6)$] $\bq\preceq \bu$ holds in $T_2$ if and only if $L_{\geq 2}(\bu)\neq \emptyset$.
\end{itemize}
\end{lem}

Let $S$ be an ai-semiring.
One can construct an ai-semiring $S^0$ from $S$ by adjoining a new element $0$.
The operations on $S^0=S \cup \{0\}$ are defined by:
\[
(\forall a\in S\cup \{0\}) \quad a+0=0+a=a,\quad a0=0a=0,
\]
while preserving the original operations on $S$.
Then $S$ is a subsemiring of $S^0$, and
$0$ is both the additive minimum element and the multiplicative zero of $S^0$.
The following result, which is due to Wu et al. \cite[Proposition 1.5]{wrz},
explores the relationship between the equational theories of $S^0$ and $S$.

\begin{lem}\label{lem001}
Let $\bq\preceq \bu$ be an inequality such that
$\bu=\bu_1+\bu_2+\cdots+\bu_n$ with $\bq, \bu_i\in X^+$ for $1\leq i \leq n$.
Then $\bq\preceq \bu$ is satisfied by ${S}^0$ if and only if
$D_\bq(\bu)\not=\emptyset$ and $\bq\preceq D_\bq(\bu)$ is satisfied by $S$.
\end{lem}

\section{Applying known criteria to 4-element ai-semirings}\label{section-known}
In this section, we solve the finite basis problem for several $4$-element ai-semirings using known criteria.
For a class $\mathcal{K}$ of ai-semirings, we denote by $\mathsf{V}(\mathcal{K})$ the variety generated by $\mathcal{K}$.

\begin{pro}\label{proNFB}
The ai-semirings $S_{(4, 545)}$, $S_{(4, 634)}$, and $S_{(4, 710)}$ are all nonfinitely based.
\end{pro}
\begin{proof}
By the main result of \cite{yrg}, $S_{(4, 545)}$ and $S_{(4, 634)}$ are both nonfinitely based.
It remains to show that $S_{(4,710)}$ is also nonfinitely based.
One checks that $S_{(4, 710)}$ is isomorphic to a subdirect product of $S_{53}$ and $D_2$
via two congruences with nontrivial blocks $\{1, 2\}$ and $\{2, 3, 4\}$;
the Cayley tables of $S_{53}$ are given in Table~\ref{tb24011601}.
Consequently,
\[
\mathsf{V}(S_{(4, 710)}) = \mathsf{V}(S_{53}, D_2).
\]
From \cite[Proposition 3.4]{yrg}, we have that $\mathsf{V}(S_{53}, D_2)=\mathsf{V}(S_{(4, 545)})$,
and so
\[
\mathsf{V}(S_{(4, 710)})=\mathsf{V}(S_{(4, 545)}).
\]
Therefore, $S_{(4, 710)}$ is also nonfinitely based.
\end{proof}

\begin{table}[ht]
\caption{The Cayley tables of $S_{53}$} \label{tb24011601}
\begin{tabular}{c|ccc}
$+$      &$1$ &$2$ &$3$\\
\hline
$1$      &$1$ &$1$ &$3$\\
$2$      &$1$ &$2$ &$3$\\
$3$      &$3$ &$3$ &$3$\\
\end{tabular}\qquad
\begin{tabular}{c|ccc}
$\cdot$      &$1$ &$2$ &$3$\\
\hline
$1$      &$3$ &$1$ &$3$\\
$2$      &$1$ &$2$ &$3$\\
$3$      &$3$ &$3$ &$3$\\
\end{tabular}
\end{table}

\begin{pro}\label{pro411}
The following $4$-element ai-semirings are finitely based:\\
$S_{(4,555)}$, $S_{(4,556)}$, $S_{(4,557)}$, $S_{(4,558)}$, $S_{(4,559)}$, $S_{(4,562)}$, $S_{(4,563)}$,
$S_{(4,568)}$, $S_{(4,569)}$, $S_{(4,570)}$, $S_{(4,571)}$, $S_{(4,572)}$, $S_{(4,574)}$, $S_{(4,576)}$,
$S_{(4,577)}$, $S_{(4,579)}$, $S_{(4,580)}$, $S_{(4,581)}$, $S_{(4,582)}$, $S_{(4,583)}$, $S_{(4,584)}$,
$S_{(4,585)}$, $S_{(4,586)}$, $S_{(4,610)}$, $S_{(4,611)}$, $S_{(4,612)}$, $S_{(4,613)}$, $S_{(4,614)}$,
$S_{(4,618)}$, $S_{(4,619)}$, $S_{(4,637)}$, $S_{(4,638)}$, $S_{(4,639)}$, $S_{(4,640)}$, $S_{(4,641)}$,
$S_{(4,642)}$, $S_{(4,643)}$, $S_{(4,644)}$, $S_{(4,645)}$, $S_{(4,649)}$, $S_{(4,650)}$, $S_{(4,659)}$,
$S_{(4,661)}$, $S_{(4,662)}$, $S_{(4,664)}$, $S_{(4,665)}$, $S_{(4,673)}$, $S_{(4,674)}$, $S_{(4,675)}$,
$S_{(4,676)}$, $S_{(4,679)}$, $S_{(4,680)}$, $S_{(4,699)}$, $S_{(4,700)}$, $S_{(4,701)}$, $S_{(4,702)}$,
$S_{(4,712)}$, $S_{(4,713)}$, $S_{(4,714)}$, $S_{(4,715)}$, $S_{(4,716)}$, $S_{(4,719)}$, $S_{(4,720)}$,
$S_{(4,725)}$, $S_{(4,726)}$, $S_{(4,730)}$, $S_{(4,731)}$, $S_{(4,732)}$, $S_{(4,737)}$, $S_{(4,739)}$,
$S_{(4,740)}$, $S_{(4,742)}$, $S_{(4,744)}$, $S_{(4,745)}$, $S_{(4,752)}$, $S_{(4,753)}$, $S_{(4,754)}$,
$S_{(4,755)}$, $S_{(4,758)}$, $S_{(4,759)}$, $S_{(4,760)}$, $S_{(4,761)}$.
\end{pro}
\begin{proof}
It is easy to check that each algebra in Proposition \ref{pro411} satisfies the identity $x^2\approx x$.
By the main results of~\cite{gpz, pas05}, these algebras are all finitely based.
\end{proof}

\begin{pro}\label{412}
The following $4$-element ai-semirings are finitely based:\\
$S_{(4,483)}$, $S_{(4,497)}$, $S_{(4,484)}$, $S_{(4,498)}$, $S_{(4,499)}$, $S_{(4,549)}$, $S_{(4,488)}$, $S_{(4,506)}$, $S_{(4,507)}$, $S_{(4,550)}$, $S_{(4,490)}$, $S_{(4,513)}$, $S_{(4,515)}$, $S_{(4,606)}$, $S_{(4,520)}$, $S_{(4,534)}$, $S_{(4,535)}$, $S_{(4,552)}$, $S_{(4,523)}$, $S_{(4,541)}$, $S_{(4,542)}$, $S_{(4,553)}$, $S_{(4,764)}$, $S_{(4,768)}$, $S_{(4,776)}$, $S_{(4,778)}$, $S_{(4,788)}$, $S_{(4,805)}$, $S_{(4,845)}$, $S_{(4,625)}$, $S_{(4,652)}$, $S_{(4,666)}$, $S_{(4,723)}$, $S_{(4,626)}$, $S_{(4,655)}$, $S_{(4,670)}$, $S_{(4,729)}$, $S_{(4,762)}$, $S_{(4,767)}$, $S_{(4,770)}$, $S_{(4,773)}$, $S_{(4,787)}$, $S_{(4,804)}$, $S_{(4,844)}$, $S_{(4,783)}$, $S_{(4,784)}$, $S_{(4,800)}$, $S_{(4,808)}$, $S_{(4,809)}$, $S_{(4,810)}$, $S_{(4,847)}$, $S_{(4,855)}$, $S_{(4,861)}$, $S_{(4,865)}$, $S_{(4,482)}$.
\end{pro}
\begin{proof}
It is easy to check that $S_{59}$ embeds into $S_{(4, 482)}$,
and so $\mathsf{V}(S_{59})$ is a subvariety of $\mathsf{V}(S_{(4, 482)})$.
Conversely, $S_{(4, 482)}$ satisfies the identities
\[
x_1x_2x_3 \approx y_1y_2y_3, \quad y \preceq x^3, \quad x^2+y^2 \approx xy,\quad x\preceq x^2,
\]
which form an equational basis of $S_{59}$ (see \cite[Proposition 11]{zrc}).
Hence $\mathsf{V}(S_{(4,482)})$ is a subvariety of $\mathsf{V}(S_{59})$. Therefore,
\[
\mathsf{V}(S_{(4,482)})=\mathsf{V}(S_{59}).
\]

By the same approach, one can prove that each of the remaining algebras in Proposition~\ref{412}
is equationally equivalent to some $3$-element ai-semiring:
\[
\begin{aligned}
&\mathsf{V}(S_{(4,483)})=\mathsf{V}(S_{(4,497)})=\mathsf{V}(S_{59}),\\
&\mathsf{V}(S_{(4,484)})=\mathsf{V}(S_{(4,498)})=\mathsf{V}(S_{(4,499)})=\mathsf{V}(S_{(4,549)})=\mathsf{V}(S_{60}),\\
&\mathsf{V}(S_{(4,488)})=\mathsf{V}(S_{(4,506)})=\mathsf{V}(S_{(4,507)})=\mathsf{V}(S_{(4,550)})=\mathsf{V}(S_{54}),\\
&\mathsf{V}(S_{(4,490)})=\mathsf{V}(S_{(4,513)})=\mathsf{V}(S_{(4,515)})=\mathsf{V}(S_{(4,606)})=\mathsf{V}(S_{56}),\\
&\mathsf{V}(S_{(4,520)})=\mathsf{V}(S_{(4,534)})=\mathsf{V}(S_{(4,535)})=\mathsf{V}(S_{(4,552)})=\mathsf{V}(S_{57}),\\
&\mathsf{V}(S_{(4,523)})=\mathsf{V}(S_{(4,541)})=\mathsf{V}(S_{(4,542)})=\mathsf{V}(S_{(4,553)})=\mathsf{V}(S_{53}),\\
&\mathsf{V}(S_{(4,764)})=\mathsf{V}(S_{(4,768)})=\mathsf{V}(S_{(4,776)})=\mathsf{V}(S_{(4,778)})=\mathsf{V}(S_{45}),\\
&\mathsf{V}(S_{(4,788)})=\mathsf{V}(S_{(4,805)})=\mathsf{V}(S_{(4,845)})=\mathsf{V}(S_{45}),\\
&\mathsf{V}(S_{(4,625)})=\mathsf{V}(S_{(4,652)})=\mathsf{V}(S_{(4,666)})=\mathsf{V}(S_{(4,723)})=\mathsf{V}(S_{58}),\\
&\mathsf{V}(S_{(4,626)})=\mathsf{V}(S_{(4,655)})=\mathsf{V}(S_{(4,670)})=\mathsf{V}(S_{(4,729)})=\mathsf{V}(S_{55}),\\
&\mathsf{V}(S_{(4,762)})=\mathsf{V}(S_{(4,767)})=\mathsf{V}(S_{(4,770)})=\mathsf{V}(S_{(4,773)})=\mathsf{V}(S_{44}),\\
&\mathsf{V}(S_{(4,787)})=\mathsf{V}(S_{(4,804)})=\mathsf{V}(S_{(4,844)})=\mathsf{V}(S_{44}),\\
&\mathsf{V}(S_{(4,783)})=\mathsf{V}(S_{(4,784)})=\mathsf{V}(S_{(4,800)})=\mathsf{V}(S_{(4,808)})=\mathsf{V}(S_{46}),\\
&\mathsf{V}(S_{(4,809)})=\mathsf{V}(S_{(4,810)})=\mathsf{V}(S_{(4,847)})=\mathsf{V}(S_{46}),\\
&\mathsf{V}(S_{(4,855)})=\mathsf{V}(S_{(4,861)})=\mathsf{V}(S_{(4,865)})=\mathsf{V}(S_{47}).
\end{aligned}
\]

By the above observations,
it folows from the main result of~\cite{zrc} that every algebra in Proposition~\ref{412} is finitely based.
\end{proof}

\begin{pro}\label{pro413}
The following $4$-element ai-semirings are finitely based:\\
$S_{(4,481)}$, $S_{(4,547)}$, $S_{(4,554)}$, $S_{(4,560)}$, $S_{(4,595)}$, $S_{(4,601)}$, $S_{(4,820)}$, $S_{(4,825)}$, $S_{(4,857)}, $
$S_{(4,561)}$, $S_{(4,564)}$, $S_{(4,596)}$, $S_{(4,615)}$, $S_{(4,622)}$, $S_{(4,821)}$, $S_{(4,827)}$, $S_{(4,573)}$, $S_{(4,698)}, $
$S_{(4,590)}$, $S_{(4,597)}$, $S_{(4,703)}$, $S_{(4,790)}$, $S_{(4,822)}$, $S_{(4,836)}$, $S_{(4,592)}$, $S_{(4,598)}$, $S_{(4,717)}, $
$S_{(4,795)}$, $S_{(4,823)}$, $S_{(4,838)}$, $S_{(4,594)}$, $S_{(4,819)}$, $S_{(4,602)}$, $S_{(4,609)}$, $S_{(4,617)}$, $S_{(4,620)}, $
$S_{(4,623)}$, $S_{(4,624)}$, $S_{(4,828)}$, $S_{(4,829)}$, $S_{(4,858)}$, $S_{(4,621)}$, $S_{(4,826)}$, $S_{(4,704)}$, $S_{(4,743)}, $
$S_{(4,575)}$, $S_{(4,711)}$, $S_{(4,718)}$, $S_{(4,722)}$, $S_{(4,796)}$, $S_{(4,839)}$, $S_{(4,692)}$, $S_{(4,736)}$, $S_{(4,741)}, $
$S_{(4,779)}$, $S_{(4,797)}$, $S_{(4,806)}$, $S_{(4,840)}$, $S_{(4,846)}$, $S_{(4,859)}$, $S_{(4,782)}$, $S_{(4,802)}$, $S_{(4,842)}, $
$S_{(4,785)}$, $S_{(4,803)}$, $S_{(4,811)}$, $S_{(4,843)}$, $S_{(4,848)}$, $S_{(4,860)}$, $S_{(4,789)}$, $S_{(4,835)}$, $S_{(4,794)}, $
$S_{(4,837)}$, $S_{(4,812)}$, $S_{(4,856)}$, $S_{(4,866)}$.
\end{pro}

\begin{proof}
It is easy to check that each semiring in Proposition~\ref{pro413} satisfies the identities
\[
x_1x_2x_3x_4 \approx x_1x_3x_2x_4,\quad (xy)^2 \approx xy, \quad xz+yx \preceq x+yz,
\]
which form an an equational basis of the variety generated by all ai-semirings of order two (see \cite[Corollary 2.3]{sr}).
By the main result of \cite{sr}, all these algebras are finitely based.
\end{proof}

\begin{pro}\label{pro414}
The following $4$-element ai-semirings are finitely based:
$S_{(4,578)}$, $S_{(4,593)}$, $S_{(4,599)}$, $S_{(4,646)}$, $S_{(4,647)}$, $S_{(4,651)}$, $S_{(4,658)}$, $S_{(4,660)}$, $S_{(4,663)}$, $S_{(4,687)}$, $S_{(4,689)}$, $S_{(4,691)}$, $S_{(4,738)}$, $S_{(4,781)}$, $S_{(4,798)}$, $S_{(4,824)}$, $S_{(4,832)}$, $S_{(4,833)}$, $S_{(4,834)}$, $S_{(4,841)}$.
\end{pro}
\begin{proof}
It is easy to verify that each semiring in Proposition~\ref{pro414} satisfies the identities
\[
x^2y \approx xy,\quad xy^2 \approx xy,\quad (xy)^2 \approx xy,\quad xyzx \approx xyxzx,
\]
\[
z^2 \preceq xy + z,\quad xyz \preceq xy + z,\quad zxy \preceq xy + z,\quad xzy \preceq xy + z.
\]
By \cite[Theorem 2.1, Corollary 3.7]{rz19}, all these algebras are finitely based.
\end{proof}

\section{Equational bases for 4-element ai-semirings related to $S_{44}$}\label{section-s44}
In this section, we study the finite basis problem for several 4-element ai-semirings related to
the $3$-element ai-semiring $S_{44}$,
whose Cayley tables are given in Table~\ref{tb26070302}.
\begin{table}[ht]
\caption{The Cayley tables of $S_{44}$} \label{tb26070302}
\begin{tabular}{c|ccc}
$+$      &$1$&$2$&$3$\\
\hline
$1$      &$1$&$1$&$3$\\
$2$      &$1$&$2$&$3$\\
$3$      &$3$&$3$&$3$\\
\end{tabular}\qquad
\begin{tabular}{c|ccc}
$\cdot$  &$1$&$2$&$3$\\
\hline
$1$      &$2$&$2$&$1$\\
$2$      &$2$&$2$&$2$\\
$3$      &$1$&$2$&$3$\\
\end{tabular}
\end{table}

The following result, due to Yue et al.~\cite[Lemma 7.12]{yrzs},
provides some information about the inequalities satisfied by $S_{44}$.
\begin{lem}\label{lem4401}
Let $\mathbf{q} \preceq \mathbf{u}$ be a nontrivial inequality such that $\mathbf{u} = \mathbf{u}_1 + \mathbf{u}_2 + \cdots + \mathbf{u}_n$ with $\mathbf{u}_i, \mathbf{q} \in X^+$ for $1 \leq i \leq n$.
Suppose that $\mathbf{q} \preceq \mathbf{u}$ holds in $S_{44}$.
Then $\ell(\mathbf{q}) \geq 2$ and $D_{\mathbf{q}}(\mathbf{u}) \neq \emptyset$.
Moreover, if $M_1(\mathbf{q}) \neq \emptyset$, then for every $x \in M_1(\mathbf{q})$,
there exists $\mathbf{u}_i \in D_{\mathbf{q}}(\mathbf{u})$ such that $m(x, \mathbf{u}_i) \leq 1$.
\end{lem}

\begin{pro}\label{pro74701}
The ai-semiring variety
$\mathsf{V}(S_{(4, 747)})$ is defined by the identities
\begin{align}
&xy \approx yx; \label{74701}\\
&xy \preceq x+y;    \label{74703}\\
&xy \preceq x+yz;    \label{74704}\\
&xy \approx x^2y+xy^2;    \label{74705}\\
&xy \preceq x+xyz.    \label{74706}
\end{align}
\end{pro}
\begin{proof}
It is easy to check that $S_{(4, 747)}$ satisfies the identities \eqref{74701}--\eqref{74706}.
By identifying $y$ and $x$ in \eqref{74705}, one obtains the identity
\begin{align}
x^3& \approx x^2. \label{74702}
\end{align}
It remains to prove that every inequality of $S_{(4, 747)}$
is derivable from \eqref{74701}--\eqref{74702}.
Let $\bq\preceq \bu$ be such a nontrivial inequality, where
$\bu=\bu_1+\bu_2+\cdots+\bu_n$ with $\bu_i, \bq \in X^+$ for $1 \leq i \leq n$.

Observe that $S_{(4,747)}$ is isomorphic to a subdirect product of $M_2$ and $S_{44}$
via two congruences with nontrivial blocks $\{1, 2, 3\}$ and $\{3, 4\}$.
Hence both $M_2$ and $S_{44}$ satisfy $\bq \preceq \bu$,
and so $c(\bq)\subseteq  c(\bu)$; by Lemma~\ref{lem4401}, $\ell(\bq)\geq 2$ and $D_\bq(\bu)\neq\emptyset$.

\textbf{Case 1.}
$M_1(\bq)$ is empty. Then $m(x, \bq)\geq 2$ for every $x\in c(\bq)$.
Choose a word $\bu_i$ in $D_{\bq}(\bu)$, so $c(\mathbf{u}_i) \subseteq c(\mathbf{q})$.
Then
\[
\bu =  \bu_1+\bu_2+\cdots+\bu_n \stackrel{\eqref{74703}}
\succeq \bu_i+\bu_1^2 \bu_2^2 \cdots \bu_n^2 \stackrel{\eqref{74701}, \eqref{74702}}
\approx \bu_i + \bq \bu_1^2 \bu_2^2 \cdots \bu_n^2 \stackrel{\eqref{74704}} \succeq \bu_i \bq \stackrel{\eqref{74701}, \eqref{74702}} \approx \bq.
\]
Here the third step follows from $c(\mathbf{q}) \subseteq c(\mathbf{u})$,
while the final step is justified by $c(\mathbf{u}_i) \subseteq c(\mathbf{q})$
and $m(x, \bq)\geq 2$ for every $x\in c(\bq)$.
This derives the inequality $\bu \succeq \bq$.

\textbf{Case 2.}
$M_1(\bq)$ is nonempty. By Lemma \ref{lem4401},
for every $x \in M_{1}(\bq)$ there exists $\bu_i \in D_\bq(\bu)$ such that $m(x, \bu_i)\leq 1$.
Assume that
\[
c(\bq)=\{x_1, \ldots, x_s, y_1, \ldots, y_t\},
\]
where $m(x_i, \bq )\geq 2$, $m(y_j, \bq )=1$, $0 \leq i \leq s$, $1 \leq j \leq t$.
Then
\begin{align*}
\bq
&\approx x_1^2x_2^2\cdots x_s^2y_1y_2\cdots y_t &&(\text{by}~ \eqref{74701}, \eqref{74702})\\
&\approx\sum\limits_{1\leq j\leq t}x_1^2x_2^2\cdots x_s^2y_1^2\cdots y_{j-1}^2y_{j+1}^2\cdots y_t^2y_j.
&&(\text{by}~\eqref{74701}, \eqref{74705}, \eqref{74702})
\end{align*}
So we only need to consider the case that $\bq=x_1^2x_2^2\cdots x_k^2y$.
By Lemma \ref{lem4401}, there exists $\bu_{\ell} \in D_\bq(\bu)$ such that $m(y, \bu_{\ell})\leq 1$;
in particular, $c(\mathbf{u}_\ell) \subseteq c(\mathbf{q})$.
Then
\begin{align*}
\bu
&= \bu_1+\bu_2+\cdots+\bu_n\\
&\succeq \bu_{\ell}+\bu_1^2\bu_2^2\cdots\bu_n^2 &&(\text{by}~\eqref{74703})\\
&\approx \bu_{\ell}+\bu_{\ell}\bq_1\bu_1^2\bu_2^2\cdots\bu_n^2 &&(\text{by}~\eqref{74701}, \eqref{74702})\\
&\succeq \bu_{\ell}\bq_1 &&(\text{by}~\eqref{74706})\\
&\approx \bq, &&(\text{by}~\eqref{74701}, \eqref{74702})
\end{align*}
where $\bq_1=\bq$ if $m(y, \bu_{\ell})=0$, and $\bq_1=x_1^2x_2^2\cdots x_k^2$ if $m(y, \bu_{\ell})=1$.
Here the last step follows from $c(\bu_{\ell}) \subseteq c(\bq)$.
This derives the inequality $\bu \succeq \bq$.
\end{proof}

\begin{remark}\label{remark26070250}
We note that $\mathsf{V}(S_{(4,747)})=\mathsf{V}(M_2, S_{44})$.
This follows from the fact that $S_{(4,747)}$ is isomorphic to a subdirect product of $M_2$ and $S_{44}$.
\end{remark}

\begin{cor}
The ai-semiring $S_{(4, 587)}$ is finitely based.
\end{cor}
\begin{proof}
It is straightforward to verify that $S_{(4, 587)}$ satisfies the identities \eqref{74701}--\eqref{74702}.
By Proposition~\ref{pro74701}, $\mathsf{V}(S_{(4, 587)})$ is a subvariety of $\mathsf{V}(S_{(4, 747)})$.
Conversely, since both $M_2$ and $S_{44}$ embed into $S_{(4, 587)}$
($M_2$ is isomorphic to the subalgebra $\{1,4\}$, and $S_{44}$ is isomorphic to the subalgebra $\{2,3,4\}$),
Remark~\ref{remark26070250} implies that $\mathsf{V}(S_{(4, 747)})$ is a subvariety of $\mathsf{V}(S_{(4, 587)})$.
Therefore, $\mathsf{V}(S_{(4, 587)})=\mathsf{V}(S_{(4, 747)})$.
By Proposition~\ref{pro74701}, $S_{(4, 587)}$ is finitely based.
\end{proof}

\begin{pro}\label{pro76901}
The ai-semiring variety
$\mathsf{V}(S_{(4, 769)})$ is defined by the identities
\begin{align}
&xyz \approx xzy; \label{76902}\\
&xy \preceq x;    \label{76903}\\
&xy \approx x^2y+xy^2;    \label{76904}\\
&yx \preceq x+yz;    \label{76905}\\
&xz \preceq xy+zx,    \label{76906}
\end{align}
where $z$ may be empty in \eqref{76905}, and $y$ may be empty in \eqref{76906}.
\end{pro}
\begin{proof}
It is routine to verify that $S_{(4, 769)}$ satisfies the identities \eqref{76902}--\eqref{76906}.
By identifying $y$ with $x$ in \eqref{76904}, we obtain the identity
\begin{align}
x^3 & \approx x^2. \label{76901}
\end{align}
It remains to show that every inequality of $S_{(4, 769)}$
is derivable from \eqref{76902}--\eqref{76901}.
Let $\bq \preceq \bu$ be such a nontrivial inequality, where
$\bu = \bu_1 + \cdots + \bu_n$ with $\bu_i, \bq \in X^+$ for $1 \leq i \leq n$.

It is easy to see that $S_{(4,769)}$ is isomorphic to a subdirect product of $L_2$ and $S_{44}$
via two congruences with nontrivial blocks $\{1, 2, 3\}$ and $\{3, 4\}$.
Thus both $L_2$ and $S_{44}$ satisfy $\bq \preceq \bu$,
and so $h(\bu_j) = h(\bq)$ for some $\bu_j\in \bu$;
by Lemma~\ref{lem4401}, $\ell(\mathbf{q}) \geq 2$ and $D_{\mathbf{q}}(\mathbf{u}) \neq \emptyset$.
We write $\bq=x_1x_2\cdots x_s$, where the $x_i$'s are not necessarily distinct.
By \eqref{76902} and \eqref{76904}, we have
\[
\bq \approx \sum_{i=1}^s x_1^2 \cdots x_{i-1}^2x_ix_{i+1}^2\cdots x_s^2.
\]
So we only need to consider the case that $\bq=x_1^2 \cdots x_{i-1}^2x_ix_{i+1}^2\cdots x_s^2$ for some $1\leq i \leq s$.
Using \eqref{76902} and \eqref{76901}, we may assume that each variable in $\bq$ occurs with multiplicity $1$ or $2$,
and that at most one variable occurs with multiplicity $1$;
furthermore, if such a variable exists, it must appear either at the head or at the tail of $\bq$.
Hence $\bq$ reduces to one of the following forms:
\[
x_1^2 \cdots x_s^2,\quad x_1x_2^2 \cdots x_s^2,\quad x_1^2 \cdots x_{s-1}^2x_s,
\]
where $x_1,\dots,x_s$ are distinct variables.

\textbf{Case 1.} $\bq=x_1^2 \cdots x_s^2$. Then $m(x, \bq)=2$ for all $x\in c(\bq)$.
Choose a word $\bu_i$ in $D_{\bq}(\bu)$, so $c(\bu_i)\subseteq c(\bq)$.
Then
\[
\begin{aligned}
\bu
&\succeq \bu_i+\bu_j
\stackrel{\eqref{76903}}{\succeq} \bu_i+\bu_js(\bq)
= \bu_i+ h(\bu_j)s(\bu_j)s(\bq) \\
&\stackrel{\eqref{76902}}{\approx} \bu_i+ h(\bq)s(\bq)s(\bu_j)
\approx \bu_i+ \bq s(\bu_j)
\stackrel{\eqref{76905}}{\succeq} \bq \bu_i
\stackrel{\eqref{76902}, \eqref{76901}}{\approx} \bq.
\end{aligned}
\]
The fourth step uses $h(\bu_j) = h(\bq)$,
while the last step follows from $c(\bu_i)\subseteq c(\bq)$ and $m(x, \bq)=2$ for all $x\in c(\bq)$.
This derives the identity $\bu \succeq \bq$.

\textbf{Case 2.} $\bq=x_1^2 \cdots x_{s-1}^2x_s$.
By Lemma \ref{lem4401}, there exists $\bu_{\ell} \in D_\bq(\bu)$ such that $m(x_s, \bu_{\ell})\leq 1$;
in particular, $c(\bu_{\ell}) \subseteq c(\bq)$.
Now we have
\[
\begin{aligned}
\bu
&\succeq \bu_{\ell}+\bu_j
 \stackrel{\eqref{76903}}{\succeq} \bu_{\ell} \bq_1 + \bu_j
 = \bu_{\ell} \bq_1 + h(\bu_j)s(\bu_j) \\
&\stackrel{\eqref{76905}}{\succeq} h(\bu_j) \bu_{\ell} \bq_1
 \approx h(\bq)\bu_{\ell} \bq_1
 \stackrel{\eqref{76902}, \eqref{76901}}{\approx} \bq,
\end{aligned}
\]
where $\bq_1=\bq$ if $m(x_s, \bu_{\ell})=0$, and $\bq_1=x_1^2 \cdots x_{s-1}^2$ if $m(x_s, \bu_{\ell})=1$.
The fifth step follows from $h(\bu_j) = h(\bq)$, while the last step uses $c(\bu_{\ell}) \subseteq c(\bq)$.
This derives the identity $\bu \succeq \bq$.

\textbf{Case 3.} $\bq=x_1x_2^2 \cdots x_s^2$.
By Lemma \ref{lem4401}, there exists $\bu_k \in D_\bq(\bu)$ such that $m(h(\bq), \bu_k)\leq 1$;
in particular, $c(\bu_k)\subseteq c(\bq)$.
Consider the following two subcases.

\textbf{Subase 3.1.} $m(h(\bq), \bu_k)= 0$. Then
\[
\bu \succeq \bu_k+\bu_j = \bu_k + h(\bq)s(\bu_j) \stackrel{\eqref{76905}}\succeq h(\bq)\bu_k
\stackrel{\eqref{76903}}\succeq h(\bq)\bu_ks(\bq)
\stackrel{\eqref{76902}, \eqref{76901}}\approx \bq.
\]
The second step follows from $h(\bu_j)=h(\bq)$,
and the last step uses $c(\bu_{k}) \subseteq c(\bq)$, $m(h(\bq), \bu_k)= 0$, and $h(\bq) \in M_{1}(\bq)$.
This derives the identity $\bu \succeq \bq$.

\textbf{Subase 3.2.} $m(h(\bq), \bu_k)= 1$.
If $h(\bu_k)=h(\bq)$, then
\[
\bu \succeq \bu_k \stackrel{\eqref{76903}}\succeq \bu_ks(\bq)
\stackrel{\eqref{76902}, \eqref{76901}}\approx \bq.
\]
The last step uses $c(\bu_{k}) \subseteq c(\bq)$, $m(h(\bq), \bu_k)= 1$, and $h(\bu_k)=h(\bq)$.
This derives the identity $\bu \succeq \bq$.
If $h(\bu_k)\neq h(\bq)$, then $h(\bq)\in c(s(\bu_k))$.
By \eqref{76902} one derives
\begin{equation}\label{id26071310}
\bu_k \approx h(\bu_k) \bu_k'' h(\bq),
\end{equation}
where $h(\bq)\notin c(\bu_k'')$. Furthermore, we have
\[
\begin{aligned}
\bu
&\succeq \bu_j+\bu_k
 = h(\bu_j)s(\bu_j)+\bu_k =h(\bq)s(\bu_j)+\bu_k\\
&\stackrel{\eqref{id26071310}}{\approx} h(\bq)s(\bu_j) + h(\bu_k) \bu_k'' h(\bq)
 \stackrel{\eqref{76906}}{\succeq} h(\bq)h(\bu_k) \bu_k'' \\
&\stackrel{\eqref{76903}}{\succeq} h(\bq)h(\bu_k) \bu_k'' s(\bq)
 \stackrel{\eqref{76902}, \eqref{76901}}{\approx} \bq.
\end{aligned}
\]
The third step follows $h(\bu_j)=h(\bq)$,
while the last step uses $h(\bu_k)\neq h(\bq)$, $h(\bq)\notin c(\bu_k'')$, and $c(\bu_k)\subseteq c(\bq)$.
This derives the inequality $\bu \succeq \bq$.\qedhere
\end{proof}

\begin{remark}
We note that $\mathsf{V}(S_{(4, 769)})=\mathsf{V}(L_2, S_{44})$,
since $S_{(4,769)}$ is isomorphic to a subdirect product of $L_2$ and $S_{44}$.
\end{remark}

\begin{cor}
The ai-semiring $S_{(4, 748)}$ is finitely based.
\end{cor}
\begin{proof}
Since $S_{(4, 748)}$ and $S_{(4, 769)}$ have dual multiplications, it follows from
Proposition {\ref{pro76901}} that $S_{(4, 748)}$ is also finitely based.
\end{proof}

\begin{table}[ht]
\caption{The Cayley tables of $S_{(4,379)}$} \label{tb437901}
\begin{tabular}{c|cccc}
$+$      &$1$ &$2$ &$3$ &$4$\\
\hline
$1$      &$1$ &$2$ &$1$ &$1$\\
$2$      &$2$ &$2$ &$2$ &$2$\\
$3$      &$1$ &$2$ &$3$ &$1$\\
$4$      &$1$ &$2$ &$1$ &$4$\\
\end{tabular}\qquad
\begin{tabular}{c|cccc}
$\cdot$      &$1$ &$2$ &$3$ &$4$\\
\hline
$1$      &$3$ &$1$ &$3$ &$3$\\
$2$      &$1$ &$2$ &$3$ &$1$\\
$3$      &$3$ &$3$ &$3$ &$3$\\
$4$      &$3$ &$1$ &$3$ &$3$\\
\end{tabular}
\end{table}

\begin{pro}
The ai-semiring $S_{(4, 746)}$ is finitely based.
\end{pro}
\begin{proof}
It is easy to check that
$S_{(4,746)}$ is isomorphic to a subdirect product of $S_{44}$ and $T_2$
via two congruences with nontrivial blocks $\{3, 4\}$ and $\{1, 2, 3\}$,
and that $S_{(4,379)}$ is isomorphic to a subdirect product of $S_{44}$ and $T_2$
via two congruences with nontrivial blocks $\{1, 4\}$ and $\{1, 2, 3\}$,
where the Cayley tables of $S_{(4,379)}$ are given in Table~\ref{tb437901}.
Thus $\mathsf{V}(S_{(4, 379)})=\mathsf{V}(S_{(4,746)})$, and so
$S_{(4, 379)}$ and $S_{(4,746)}$ have the same finite basis property.
By \cite[Proposition 5.4]{rlyc}, $S_{(4, 379)}$ is finitely based.
We therefore obtain that $S_{(4,746)}$ is also finitely based.
\end{proof}

\section{Equational bases for 4-element ai-semirings related to $S_{46}$}
In this section, we study the finite basis problem for several 4-element ai-semirings that relate to $S_{46}$,
whose Cayley tables are given in Table~\ref{tb4601}.

\begin{table}[ht]
\caption{The Cayley tables of $S_{46}$} \label{tb4601}
\begin{tabular}{c|ccc}
$+$      &$1$ &$2$ &$3$\\
\hline
$1$      &$1$ &$1$ &$3$\\
$2$      &$1$ &$2$ &$3$\\
$3$      &$3$ &$3$ &$3$\\
\end{tabular}\qquad
\begin{tabular}{c|ccc}
$\cdot$      &$1$ &$2$ &$3$\\
\hline
$1$      &$2$ &$2$ &$2$\\
$2$      &$2$ &$2$ &$2$\\
$3$      &$1$ &$2$ &$3$\\
\end{tabular}
\end{table}

The following result, which is due to Yue et al. \cite[Lemma 7.15]{yrzs},
provides some information about the inequalities satisfied by $S_{46}$.

\begin{lem}\label{lem4601}
Let $\bq \preceq \bu$ be a nontrivial inequality such that
$\bu=\bu_1+\bu_2+\cdots+\bu_n$ with $\bu_i, \bq \in X^+$ for $1\leq i \leq n$.
Suppose that $\bq \preceq \bu$ is satisfied by $S_{46}$.
Then $\ell(\bq)\geq 2$.
Moreover, if $m(t(\bq), \bq)=1$,
then there exists $\bu_i \in D_\bq(\bu)$ such that
$t(\bq)\notin c(p(\bu_i))$.
\end{lem}

\begin{pro}\label{pro79201}
The ai-semiring variety $\mathsf{V}(S_{(4, 792)})$ is defined by the identities
\begin{align}
&x^2y \approx xy; \label{79201}\\
&xyz \approx yxz; \label{79202}\\
&xy \preceq x+y; \label{79203}\\
&xy \preceq x+yz; \label{79204}\\
&yx \preceq x+yz.    \label{79205}
\end{align}
\end{pro}
\begin{proof}
It is easy to check that $S_{(4, 792)}$ satisfies the identities \eqref{79201}--\eqref{79205}.
In the remainder we need only prove that every inequality of $S_{(4, 792)}$
is derivable from \eqref{79201}--\eqref{79205}.
Let $\bq\preceq \bu$ be such a nontrivial inequality, where
$\bu=\bu_1+\bu_2+\cdots+\bu_n$ with $\bu_i, \bq \in X^+$ for $1 \leq i \leq n$.

Since $D_2$ is isomorphic to the subalgebra $\{1,3\}$ of $S_{(4,792)}$,
it follows that $D_2$ satisfies $\bq\preceq \bu$, and so
there exists $\bu_i \in \bu$ such that $c(\bu_i)\subseteq c(\bq)$.
Observe that $S_{(4,792)}$ is isomorphic to a subdirect product of $M_2$ and $S_{46}$
via two congruences with nontrivial blocks $\{1, 2, 3\}$ and $\{3, 4\}$.
Hence both $M_2$ and $S_{46}$ satisfy $\bq\preceq \bu$, and so $c(\bq)\subseteq c(\bu)$,
and by Lemma~\ref{lem4601}, $\ell(\bq)\geq 2$.
Moreover, if $m(t(\bq), \bq)=1$,
then by Lemma~\ref{lem4601} there exists $\bu_j \in D_\bq(\bu)$ such that $t(\bq)\notin c(p(\bu_j))$;
in particular, $c(\bu_j)\subseteq c(\bq)$.
Thus $t(\bq)$ occurs in $\bu_j$ at most once, and if it occurs, it must be the tail of $\bu_j$.

\textbf{Case 1.} $m(t(\bq),\bq)\geq 2$. Then
\[
\bu =  \bu_1+\bu_2+\cdots +\bu_n \stackrel{\eqref{79203}} \succeq
\bu_i+\bu_1^2 \bu_2^2 \cdots \bu_n^2 \stackrel{\eqref{79201}, \eqref{79202}}
\approx \bu_i+\bq\bp\stackrel{\eqref{79204}} \succeq \bu_i \bq\stackrel{\eqref{79201}, \eqref{79202}} \approx \bq.
\]
The third step uses $c(\bq)\subseteq c(\bu)$, and the last step follows from
$c(\bu_i)\subseteq c(\bq)$ and $m(t(\bq),\bq)\geq 2$.
This derives the inequality $\bu\succeq\bq$.

\textbf{Case 2.} $m(t(\bq),\bq)=1$. Then $t(\bq)\notin c(p(\bq))$.

If $t(\bq)\neq t(\bu_j)$,
then $t(\bq)\notin c(\bu_j)$, since $t(\bq)\notin c(p(\bu_j))$.
So we have
\[
\bu =  \bu_1+\bu_2+\cdots +\bu_n \stackrel{\eqref{79203}} \succeq
\bu_j+\bu_1^2 \bu_2^2 \cdots \bu_n^2 \stackrel{\eqref{79201}, \eqref{79202}}
\approx \bu_j+\bq\bp\stackrel{\eqref{79204}} \succeq \bu_j \bq\stackrel{\eqref{79201}, \eqref{79202}} \approx \bq.
\]
The last step follows from $m(t(\bq),\bq)=1$, $t(\bq)\notin c(\bu_j)$, and $c(\bu_j)\subseteq c(\bq)$.
This derives the inequality $\bu\succeq\bq$.

If $t(\bq)=t(\bu_j)$, then
\[
\bu=\bu_1+\bu_2+\cdots + \bu_n \stackrel{\eqref{79203}} \succeq \bu_j+\bu_1^2 \bu_2^2 \cdots \bu_n^2 \stackrel{\eqref{79201}, \eqref{79202}} \approx \bu_j+p(\bq) \bp\stackrel{\eqref{79205}} \succeq p(\bq) \bu_j  \stackrel{\eqref{79201}, \eqref{79202}} \approx \bq.
\]
The last step follows from $t(\bq)= t(\bu_j)$, $t(\bq)\notin c(p(\bu_j))$, and $c(\bu_j)\subseteq c(\bq)$.
This derives the inequality $\bu\succeq\bq$.
\end{proof}

\begin{remark}\label{remark26011201}
We note that $\mathsf{V}(S_{(4, 792)})=\mathsf{V}(M_2, S_{46})$,
since $S_{(4,792)}$ is isomorphic to a subdirect product of $M_2$ and $S_{46}$.
\end{remark}

\begin{cor}\label{coro26071210}
The ai-semiring $S_{(4, 591)}$ is finitely based.
\end{cor}
\begin{proof}
It is straightforward to verify that $S_{(4, 591)}$ satisfies the identities \eqref{79201}--\eqref{79202}.
By Proposition~\ref{pro79201}, $\mathsf V(S_{(4, 591)})$ is a subvariety of $\mathsf{V}(S_{(4, 792)})$.
Conversely, it is easy to check that $M_2$ is isomorphic to the subalgebra $\{1,4\}$ of $S_{(4, 591)}$,
and that $S_{46}$ is isomorphic to the subalgebra $\{2,3,4\}$ of $S_{(4, 591)}$.
By Remark~\ref{remark26011201}, $\mathsf{V}(S_{(4, 792)})$ is a subvariety of $\mathsf V(S_{(4, 591)})$.
Thus $V(S_{(4, 591)})=\mathsf{V}(S_{(4, 792)})$,
and so $S_{(4, 591)}$ and $S_{(4, 792)}$ have the same finite basis property.
By Proposition~\ref{pro79201}, $S_{(4, 792)}$ is finitely based.
We therefore obtain that $S_{(4, 591)}$ is finitely based.
\end{proof}

\begin{cor}
The ai-semiring $S_{(4, 588)}$ and $S_{(4, 750)}$ are both finitely based.
\end{cor}
\begin{proof}
It is easy to see that $S_{(4, 588)}$ and $S_{(4, 591)}$ have dual multiplications,
and that $S_{(4, 750)}$ and $S_{(4, 792)}$ have dual multiplications.
By Corollary~\ref{coro26071210} and Proposition \ref{pro79201},
both $S_{(4, 588)}$ and $S_{(4, 750)}$ are finitely based.
\end{proof}

\begin{pro}\label{pro79301}
The ai-semiring variety $\mathsf{V}(S_{(4, 793)})$ is defined by the identities
\begin{align}
&x^2y \approx xy; \label{79301}\\
&xyz \approx yxz; \label{79302}\\
&yx \preceq x; \label{79303}\\
&xz \preceq x+yz, \label{79304}
\end{align}
where $y$ may be empty in \eqref{79304}.
\end{pro}
\begin{proof}
It is easy to check that $S_{(4, 793)}$ satisfies the identities \eqref{79301}--\eqref{79304}.
It remains to prove that every inequality of $S_{(4, 793)}$ is derivable from \eqref{79301}--\eqref{79304}.
Let $\bq \preceq \bu$ be such a nontrivial inequality,
where $\bu=\bu_1+\bu_2+\cdots+\bu_n$ with $\bu_i, \bq\in X^+$ for $1 \leq i \leq n$.

Since $D_2$ is isomorphic to the subalgebra $\{1, 3\}$ of $S_{(4, 793)}$,
it follows that $D_2$ satisfies $\bq \preceq \bu$,
and so $c(\bu_i)\subseteq c(\bq)$ for some $\bu_i\in \bu$.
One can easily verify that $S_{(4,793)}$ is isomorphic to a subdirect product of $R_2$ and $S_{46}$
via two congruences with nontrivial blocks $\{1, 2, 3\}$ and $\{3, 4\}$.
Hence both $R_2$ and $S_{46}$ satisfy $\bq\preceq\bu$,
and so $t(\bu_j)=t(\bq)$ for some $\bu_j\in\bu$, and by Lemma~\ref{lem4601}, $\ell(\bq)\geq 2$.
Moreover, if $m(t(\bq),\bq)=1$, then there exists $\bu_k\in D_\bq(\bu)$ such that $t(\bq)\notin c(p(\bu_k))$;
in particular, $c(\bu_k) \subseteq c(\bq)$.
Thus $t(\bq)$ occurs in $\bu_k$ at most once, and if it occurs, it must be the tail of $\bu_k$.

\textbf{Case 1.} $m(t(\bq),\bq)\geq 2$.
Then
\[
\bu \succeq \bu_i+\bu_j = \bu_i+p(\bu_j)t(\bq) \stackrel{\eqref{79304}}\succeq \bu_i t(\bq)
\stackrel{\eqref{79303}}\succeq \bq\bu_i t(\bq) \stackrel{\eqref{79301}, \eqref{79302}}\approx \bq.
\]
The second step follows from $t(\bu_j)=t(\bq)$,
and the last step relies on $m(t(\bq),\bq)\geq 2$ and $c(\bu_i) \subseteq c(\bq)$.

\textbf{Case 2.} $m(t(\bq), \bq)=1$. Then $t(\bq) \notin c(p(\bq))$.

If $m(t(\bq),\bu_k)=0$, then
\[
\bu \succeq  \bu_k+\bu_j = \bu_k+p(\bu_j)t(\bq) \stackrel{\eqref{79304}} \succeq \bu_k t(\bq) \stackrel{\eqref{79303}} \succeq p(\bq) \bu_k t(\bq) \stackrel{\eqref{79301}, \eqref{79302}}\approx \bq.
\]
The last step follows from $c(\bu_k) \subseteq c(\bq)$ and $m(t(\bq), \bu_k)=0$.

If $m(t(\bq),\bu_k)=1$, then $t(\bu_k)=t(\bq)$, since $t(\bq)\notin c(p(\bu_k))$.
Now we have
\[
\bu \succeq \bu_k \stackrel{\eqref{79303}} \succeq p(\bq) \bu_k =p(\bq) p(\bu_k) t(\bq)
\stackrel{\eqref{79301}, \eqref{79302}} \approx \bq.
\]
The third step uses $t(\bu_k)=t(\bq)$,
and the last step follows from $c(p(\bu_k))\subseteq c(p(\bq))$.
This derives the inequality $\bu \succeq \bq$.
\end{proof}

\begin{remark}
We note that
$\mathsf{V}(S_{(4, 793)})=\mathsf{V}(R_2, S_{46})$,
since $S_{(4,793)}$ is isomorphic to a subdirect product of $R_2$ and $S_{46}$.
\end{remark}

\begin{cor}
The ai-semiring $S_{(4, 775)}$ is finitely based.
\end{cor}
\begin{proof}
Since $S_{(4, 775)}$ and $S_{(4, 793)}$ have dual multiplications, it follows from
Proposition {\ref{pro79301}} that $S_{(4, 775)}$ is finitely based.
\end{proof}

\begin{pro}\label{pro79901}
The ai-semiring variety $\mathsf{V}(S_{(4, 799)})$ is defined by the identities
\begin{align}
&x^2y \approx xy; \label{79901}\\
&xyzt \approx xzyt; \label{79902}\\
&xy \preceq x; \label{79903}\\
&xyz \preceq xz; \label{79905}\\
&xy \preceq y+xz, \label{79904}
\end{align}
where $z$ may be empty in \eqref{79904}.
\end{pro}
\begin{proof}
It is straightforward to check that $S_{(4, 799)}$ satisfies the identities \eqref{79901}--\eqref{79904}.
It remains to show that every inequality of $S_{(4, 799)}$ is derivable from \eqref{79901}--\eqref{79904}.
Let $\bq \preceq \bu$ be such a nontrivial inequality,
where $\bu=\bu_1+\cdots+\bu_n$ with $\bu_i, \bq \in X^+$ for $1 \le i \le n$.

Since $D_2$ is isomorphic to the subalgebra $\{1, 3\}$ of $S_{(4, 799)}$,
it follows that $D_2$ satisfies $\bq \preceq \bu$,
and so $c(\bu_i)\subseteq c(\bq)$ for some $\bu_i\in \bu$.
One can easily verify that $S_{(4,799)}$ is isomorphic to a subdirect product of $L_2$ and $S_{46}$
via two congruences with nontrivial blocks $\{1, 2, 3\}$ and $\{3, 4\}$.
Thus both $L_2$ and $S_{46}$ satisfy $\bq \preceq \bu$,
so $h(\bu_j)=h(\bq)$ for some $\bu_j\in\bu$,
and by Lemma~\ref{lem4601}, $\ell(\bq)\geq 2$.
Moreover, if $m(t(\bq),\bq)=1$,
then there exists $\bu_k\in D_\bq(\bu)$ such that $t(\bq)\notin c(p(\bu_k))$;
in particular, $c(\bu_k) \subseteq c(\bq)$.
Thus $t(\bq)$ occurs in $\bu_k$ at most once, and if it occurs, it must be the tail of $\bu_k$.

\textbf{Case 1.} $m(t(\bq),\bq)\geq 2$. Then
\[
\bu \succeq  \bu_i+\bu_j= \bu_i+h(\bq)s(\bu_j)\stackrel{\eqref{79904}} \succeq h(\bq) \bu_i \stackrel{\eqref{79903}} \succeq h(\bq) \bu_i \bq \stackrel{\eqref{79901}, \eqref{79902}}\approx \bq.
\]
The second step uses $h(\bu_j)=h(\bq)$,
while the last step follows from $c(\bu_i)\subseteq c(\bq)$ and $m(t(\bq),\bq)\geq 2$.

\textbf{Case 2.} $m(t(\bq),\bq)=1$. Then $m(t(\bq), \bu_k)\leq 1$.

If $m(t(\bq), \bu_k)=0$, then
\[
\bu \succeq  \bu_k+\bu_j= \bu_k + h(\bq)s(\bu_j) \stackrel{\eqref{79904}} \succeq h(\bq) \bu_k
\stackrel{\eqref{79903}} \succeq h(\bq) \bu_k \bq \stackrel{\eqref{79901}, \eqref{79902}}\approx \bq.
\]
The last step follows from $m(t(\bq), \bu_k)=0$ and $c(\bu_k) \subseteq c(\bq)$.

If $m(t(\bq), \bu_k)=1$, then $t(\bu_k)=t(\bq)$. Now we have
\[
\bu \succeq  \bu_k+\bu_j= \bu_k + h(\bq)s(\bu_j) \stackrel{\eqref{79904}} \succeq h(\bq) \bu_k
\stackrel{\eqref{79905}} \succeq h(\bq)p(\bq)\bu_k\stackrel{\eqref{79901}, \eqref{79902}}\approx \bq.
\]
The last step uses $t(\bu_k)=t(\bq)$ and $c(\bu_k) \subseteq c(\bq)$.
This derives the inequality $\bu \succeq \bq$.
\end{proof}

\begin{remark}
We have $\mathsf{V}(S_{(4, 799)})=\mathsf{V}(L_2, S_{46})$,
since $S_{(4,799)}$ is isomorphic to a subdirect product of $L_2$ and $S_{46}$.
\end{remark}

\begin{cor}
The ai-semiring $S_{(4, 751)}$ is finitely based.
\end{cor}
\begin{proof}
Since $S_{(4, 751)}$ and $S_{(4, 799)}$ have dual multiplications, it follows from
Proposition {\ref{pro79901}} that $S_{(4, 751)}$ is finitely based.
\end{proof}

\begin{table}[ht]
\caption{The Cayley tables of $S_{(4,380)}$} \label{tb438001}
\begin{tabular}{c|cccc}
$+$      &$1$ &$2$ &$3$ &$4$\\
\hline
$1$      &$1$ &$2$ &$1$ &$1$\\
$2$      &$2$ &$2$ &$2$ &$2$\\
$3$      &$1$ &$2$ &$3$ &$1$\\
$4$      &$1$ &$2$ &$1$ &$4$\\
\end{tabular}\qquad
\begin{tabular}{c|cccc}
$\cdot$      &$1$ &$2$ &$3$ &$4$\\
\hline
$1$      &$3$ &$3$ &$3$ &$3$\\
$2$      &$1$ &$2$ &$3$ &$1$\\
$3$      &$3$ &$3$ &$3$ &$3$\\
$4$      &$3$ &$3$ &$3$ &$3$\\
\end{tabular}
\end{table}

\begin{pro}
The ai-semirings $S_{(4, 749)}$ and $S_{(4, 791)}$ are both finitely based.
\end{pro}
\begin{proof}
It is easy to check that
$S_{(4,791)}$ is isomorphic to a subdirect product of $S_{46}$ and $T_2$
via two congruences with nontrivial blocks $\{3, 4\}$ and $\{1, 2, 3\}$,
and that $S_{(4,380)}$ (see Table~\ref{tb438001} for its Cayley tables)
is isomorphic to a subdirect product of $S_{46}$ and $T_2$
via two congruences with nontrivial blocks $\{1, 4\}$ and $\{1, 2, 3\}$.
Hence $\mathsf{V}(S_{(4, 791)})=\mathsf{V}(S_{(4, 380)})$,
so $S_{(4, 791)}$ and $S_{(4, 380)}$ have the same finite basis property.
By \cite[Proposition 5.6]{rlyc}, $S_{(4, 380)}$ is finitely based, and hence so is $S_{(4, 791)}$.

Moreover, $S_{(4, 749)}$ and $S_{(4, 791)}$ have dual multiplications.
Therefore, $S_{(4, 749)}$ is also finitely based.
\end{proof}

\section{Equational bases for 4-element ai-semirings related to $S_{47}$}
In this section, we study the finite basis problem for 4-element ai-semirings that relate to $S_{47}$,
whose Cayley tables are given in Table~\ref{tb4701}.

\begin{table}[ht]
\caption{The Cayley tables of $S_{47}$} \label{tb4701}
\begin{tabular}{c|ccc}
$+$      &$1$ &$2$ &$3$\\
\hline
$1$      &$1$ &$1$ &$3$\\
$2$      &$1$ &$2$ &$3$\\
$3$      &$3$ &$3$ &$3$\\
\end{tabular}\qquad
\begin{tabular}{c|ccc}
$\cdot$      &$1$ &$2$ &$3$\\
\hline
$1$      &$2$ &$2$ &$2$\\
$2$      &$2$ &$2$ &$2$\\
$3$      &$2$ &$2$ &$1$\\
\end{tabular}
\end{table}

The following result, which is due to Ren et al.~\cite[Lemma 5.8]{rlyc},
provides some information about the inequalities satisfied by $S_{47}$.

\begin{lem}\label{lem4701}
Let $\bq \preceq \bu$ be a nontrivial ai-semiring identity such that
$\bu=\bu_1+\bu_2+\cdots+\bu_n$ with $\bu_i, \bq \in X^+$ for $1\leq i \leq n$.
Suppose that  $\bq \preceq \bu$ is satisfied by $S_{47}$.
Then either $\ell(\bq)\geq 3$, or $\ell(\bq)=2$ and $L_{\leq2}(\bu) \cap D_\bq(\bu)$ is nonempty.
\end{lem}

\begin{pro}\label{pro85001}
The ai-semiring variety $\mathsf{V}(S_{(4, 850)})$ is defined by the identities
\begin{align}
&xy\approx yx;\label{8501}\\
&xy \preceq  x+y;\label{8502}\\
&xy \preceq  x+xyz;\label{8503}\\
&xy \preceq  x^2+xyz;\label{8504}\\
&xyz \preceq  xyzt.\label{8505}
\end{align}
\end{pro}
\begin{proof}
It is routine to check that $S_{(4, 850)}$ satisfies the identities \eqref{8501}--\eqref{8505}.
We now prove that every inequality of $S_{(4, 850)}$ is derivable from \eqref{8501}--\eqref{8505}.
Let $\bq\preceq \bu$ be such a nontrivial inequality,
where $\bu=\bu_1+\bu_2+\cdots+\bu_n$ with $\bu_i, \bq \in X^+$ for $1 \leq i \leq n$.
Observe that $S_{(4, 850)}$ is isomorphic to a subdirect product of $M_2$ and $S_{47}$
via the congruences defined by the nontrivial blocks $\{1, 2, 3\}$ and $\{3, 4\}$.
Hence both $M_2$ and $S_{47}$ satisfy $\bq\preceq \bu$, and so $c(\bq)\subseteq c(\bu)$,
and by Lemma~\ref{lem4701}, either $\ell(\bq)\geq 3$, or $\ell(\bq)=2$ and $L_{\leq2}(\bu) \cap D_\bq(\bu)$ is nonempty.
Consider the following two cases.

\textbf{Case 1.} $\ell(\bq) \geq 3$. Then
\[
\bu = \bu_1+\bu_2+ \cdots +\bu_n\stackrel{\eqref{8502}}\succeq \bu_1^k\bu_2^k \cdots \bu_n^k \stackrel{\eqref{8501}}\approx \bq\bq' \stackrel{\eqref{8505}}\succeq \bq,
\]
where $k$ is a sufficiently large positive integer.
The third step uses $c(\bq) \subseteq c(\bu)$, while the last step relies on $\ell(\bq) \geq 3$.
This derives the inequality $\bu\succeq \bq$.

\textbf{Case 2.} $\ell(\bq) = 2$.
Then $L_{\leq2}(\bu) \cap D_\bq(\bu)$ is nonempty, and so
there exists $\bu_i\in \bu$ such that $\ell(\bu_i) \leq 2$ and $c(\bu_i) \subseteq c(\bq)$.
Hence
\[
\bu = \bu_i+\bu_1+\bu_2 +\cdots + \bu_n \stackrel{\eqref{8502}} \succeq \bu_i+\bu_1^k\bu_2^k \cdots \bu_n^k
\stackrel{\eqref{8501}}\approx \bu_i + \bq\bq' \stackrel{\eqref{8501}, \eqref{8503},\eqref{8504}}\succeq \bq,
\]
where $k$ is a sufficiently large positive integer.
The third step uses $c(\bq) \subseteq c(\bu)$,
while the last step relies on $\ell(\bq) = 2$, $\ell(\bu_i) \leq 2$, and $c(\bu_i) \subseteq c(\bq)$.
This derives the inequality $\bu\succeq \bq$.
\end{proof}

\begin{remark}\label{remark2607305}
We note that
$\mathsf{V}(S_{(4, 850)})=\mathsf{V}(M_2, S_{47})$,
since $S_{(4, 850)}$ is isomorphic to a subdirect product of $M_2$ and $S_{47}$.
\end{remark}

\begin{cor}
The ai-semiring $S_{(4,600)}$ is finitely based.
\end{cor}
\begin{proof}
A straightforward verification shows that $S_{(4, 600)}$ satisfies the identities \eqref{8501}--\eqref{8505},
which form an equational basis for $S_{(4, 850)}$ by Proposition~\ref{pro85001}.
Thus $\mathsf{V}(S_{(4, 600)})$ is a subvariety of $\mathsf{V}(S_{(4, 850)})$.
Conversely,
since $M_2$ is isomorphic to the subalgebra $\{1,4\}$ of $S_{(4, 600)}$,
and $S_{47}$ is isomorphic to the subalgebra $\{2,3,4\}$ of $S_{(4, 600)}$,
it follows from Remark~\ref{remark2607305} that
$\mathsf{V}(S_{(4, 850)})$ is a subvariety of $\mathsf{V}(S_{(4, 600)})$.
Consequently, $\mathsf{V}(S_{(4, 600)})=\mathsf{V}(S_{(4, 850)})$,
so $S_{(4, 600)}$ and $S_{(4, 850)}$ have the same finite basis property.
Proposition~\ref{pro85001} tells us that $S_{(4, 850)}$ is finitely based.
Therefore, $S_{(4, 600)}$ is finitely based.
\end{proof}

\begin{pro}\label{pro85201}
The ai-semiring variety $\mathsf{V}(S_{(4, 852)})$ is defined by the identities
\begin{align}
&xy \preceq x;\label{8521}\\
&xyz \approx xzy;\label{8523}\\
&xyz \preceq xyzt;\label{8527}\\
&xy \preceq  x+xyz;\label{08503}\\
&xy \preceq  y+xyz;\label{085031}\\
&xy \preceq  x^2+xyz;\label{08504}\\
&xy \preceq  y^2+xyz.\label{085041}
\end{align}
\end{pro}
\begin{proof}
It is routine to check that $S_{(4,852)}$ satisfies the identities \eqref{8521}--\eqref{085041}.
It remains to show that every inequality of $S_{(4,852)}$ is derivable from \eqref{8521}--\eqref{085041}.
Let $\bq\preceq \bu$ be a nontrivial inequality,
where $\bu=\bu_1+\bu_2+\cdots+\bu_n$ with $\bu_i, \bq \in X^+$ for $1 \leq i \leq n$.
One can obtain that
$S_{(4, 852)}$ is isomorphic to a subdirect product of $L_2$ and $S_{47}$
via the congruences defined by the nontrivial blocks $\{1, 2, 3\}$ and $\{3, 4\}$.
Thus both $L_2$ and $S_{47}$ satisfy $\bq\preceq \bu$,
so $h(\bu_i)=h(\bq)$ for some $\bu_i\in \bu$,
and by Lemma~\ref{lem4701}, either $\ell(\bq)\geq 3$ or
$\ell(\bq)=2$ and $L_{\leq 2}(\bu) \cap D_\bq(\bu)$ is nonempty.

\textbf{Case 1.} $\ell(\bq) \geq 3$. Then
\[
\bu \succeq \bu_i \stackrel{\eqref{8521}} \succeq \bu_i\bq \stackrel{\eqref{8523}} \approx \bq\bq' \stackrel{\eqref{8527}}\succeq \bq .
\]
The third step uses $h(\bu_i)=h(\bq)$, while the last step follows from $\ell(\bq)\geq 3$.
This derives the inequality $\bu \succeq\bq$.

\textbf{Case 2.} $\ell(\bq)=2$. Since $L_{\leq 2}(\bu)\cap D_\bq(\bu) \neq \emptyset$,
there exists $\bu_j\in \bu$ such that $\ell(\bu_j) \leq 2$ and $c(\bu_j) \subseteq c(\bq)$.
Now we have
\[
\bu \succeq \bu_j+\bu_i \stackrel{\eqref{8521}} \succeq \bu_j+\bu_i\bq \stackrel{\eqref{8523}} \approx \bu_j+\bq\bq'
\stackrel{\eqref{08503}-\eqref{085041}} \succeq \bq.
\]
The last step follows from $\ell(\bq)=2$, $\ell(\bu_j) \leq 2$ and $c(\bu_j) \subseteq c(\bq)$.
This derives the inequality $\bu \succeq\bq$.
\end{proof}

\begin{remark}
We have $\mathsf{V}(S_{(4, 852)})=\mathsf{V}(L_2, S_{47})$,
since $S_{(4, 852)}$ is isomorphic to a subdirect product of $L_2$ and $S_{47}$.
\end{remark}

\begin{cor}
The ai-semiring $S_{(4, 851)}$ is finitely based.
\end{cor}
\begin{proof}
It is easy to see that $S_{(4, 851)}$ and $S_{(4, 852)}$ have dual multiplications.
By Proposition \ref{pro85201}, we immediately deduce $S_{(4, 851)}$ is finitely based.
\end{proof}

\begin{pro}\label{pro85301}
The ai-semiring variety $\mathsf{V}(S_{(4, 853)})$ is defined by the identities
\begin{align}
&xy     \preceq x;\label{8531}\\
&xy  \preceq x^2;\label{8532}\\
&xy    \approx yx;\label{8533}\\
&xyz^2 \approx xyz.\label{8534}
\end{align}
\end{pro}
\begin{proof}
It is easy to verify that $S_{(4,853)}$ satisfies the identities \eqref{8531}--\eqref{8534}.
Now we prove that every inequality of $S_{(4,853)}$ is derivable from \eqref{8531}--\eqref{8534}.
Let $\bq\preceq \bu$ be a nontrivial inequality,
where $\bu=\bu_1+\bu_2+\cdots+\bu_n$ with $\bu_i, \bq \in X^+$ for $1 \leq i \leq n$.
Observe that $S_{(4, 853)}$ is isomorphic to a subdirect product of $D_2$ and $S_{47}$
via the congruences defined by the nontrivial blocks $\{1, 2, 3\}$ and $\{3, 4\}$.
Hence both $D_2$ and $S_{47}$ satisfy the inequality $\bq \preceq \bu$,
so $c(\bu_i) \subseteq c(\bq)$ for some $\bu_i \in \bu$,
and by Lemma \ref{lem4701}, either $\ell(\bq)\geq 3$ or $\ell(\bq)= 2$ and $L_{\leq2}(\bu) \cap D_\bq(\bu)$ is nonempty.

\textbf{Case 1.} $\ell(\bq) \geq 3$. Then
\[
\bu \succeq \bu_i\stackrel{\eqref{8531}}\succeq \bu_i\bq \stackrel{\eqref{8533},\eqref{8534}}\approx \bq.
\]
The last step follows from $c(\bu_i) \subseteq c(\bq)$ and $\ell(\bq) \geq 3$.
This derives the inequality $\bu \succeq\bq$.

\textbf{Case 2.} $\ell(\bq) = 2$.
Since  $L_{\leq 2}(\bu) \cap D_\bq(\bu)$ is nonempty,
there exists $\bu_j\in \bu$ such that $\ell(\bu_j) \leq 2$ and $c(\bu_j) \subseteq c(\bq)$.
Hence
\[
\bu \succeq \bu_j \stackrel{\eqref{8531}-\eqref{8533}}\succeq \bq.
\]
This derives the inequality $\bu \succeq\bq$.
\end{proof}

\begin{remark}
We note that $\mathsf{V}(S_{(4, 850)})=\mathsf{V}(D_2, S_{47})$,
since $S_{(4, 853)}$ is isomorphic to a subdirect product of $D_2$ and $S_{47}$.
\end{remark}

\begin{table}[ht]
\caption{The Cayley tables of $S_{(4, 385)}$} \label{tb38501}
\begin{tabular}{c|cccc}
$+$      &$1$ &$2$ &$3$&$4$\\
\hline
$1$      &$1$ &$2$ &$1$ &$1$\\
$2$      &$2$ &$2$ &$2$ &$2$\\
$3$      &$1$ &$2$ &$3$ &$1$\\
$4$      &$1$ &$2$ &$1$ &$4$\\
\end{tabular}\qquad
\begin{tabular}{c|cccc}
$\cdot$      &$1$ &$2$ &$3$&$4$\\
\hline
$1$      &$3$ &$3$ &$3$ &$3$\\
$2$      &$3$ &$1$ &$3$ &$3$\\
$3$      &$3$ &$3$ &$3$ &$3$\\
$4$      &$3$ &$3$ &$3$ &$3$\\
\end{tabular}
\end{table}

\begin{pro}\label{pro84901}
The ai-semiring $\mathsf{V}(S_{(4, 849)})$ is finitely based.
\end{pro}
\begin{proof}
It is a routine matter to verify that $S_{(4, 849)}$ is isomorphic to a subdirect product of $S_{47}$ and $T_2$
via two congruences with nontrivial blocks $\{3, 4\}$ and $\{1, 2, 3\}$,
and that $S_{(4, 385)}$ (see Table~\ref{tb38501} for its Cayley tables) is isomorphic to a subdirect product of $S_{47}$ and $T_2$
via two congruences with nontrivial blocks $\{1, 4\}$ and $\{1, 2, 3\}$.
Thus $\mathsf{V}(S_{(4,849)})=\mathsf{V}(S_{(4,385)})$,
so $S_{(4, 849)}$ and $S_{(4, 385)}$ have the same finite basis property.
By \cite[Proposition 5.9]{rlyc}, $S_{(4, 385)}$ is finitely based.
Therefore, $\mathsf{V}(S_{(4, 849)})$ is finitely based.
\end{proof}

\section{Equational bases for 4-element ai-semirings related to $S_{53}$}
In this section, we study the finite basis problem for some 4-element ai-semirings that are related to $S_{53}$,
whose Cayley tables are given in Tbale~\ref{tb5301}.

\begin{table}[ht]
\caption{The Cayley tables of $S_{53}$} \label{tb5301}
\begin{tabular}{c|ccc}
$+$      &$1$ &$2$ &$3$\\
\hline
$1$      &$1$ &$1$ &$3$\\
$2$      &$1$ &$2$ &$3$\\
$3$      &$3$ &$3$ &$3$\\
\end{tabular}\qquad
\begin{tabular}{c|ccc}
$\cdot$      &$1$ &$2$ &$3$\\
\hline
$1$      &$3$ &$1$ &$3$\\
$2$      &$1$ &$2$ &$3$\\
$3$      &$3$ &$3$ &$3$\\
\end{tabular}
\end{table}

The following result, which is due to Ren et al.~\cite[Lemma 4.8]{rlyc},
provides some information about the inequalities satisfied by $S_{53}$.

\begin{lem}\label{lem5301}
Let $\bq\preceq\bu$ be a nontrivial inequality such that
$\bu=\bu_1+\bu_2+\cdots+\bu_n$ and $\bu_i, \bq \in X^+_c$ for all $1\leq i \leq n$.
Suppose that $\bq\preceq\bu$ holds in $S_{53}$.
Then $L_{\geq 2}(\bu)\neq \emptyset$, $c(\bq)\subseteq c(\bu)$,
and for every $\bw\in S_2(\bq)$ there exists $\bw'\in S_2(\bu)$ such that $c(\bw')\subseteq c(\bw)$.
\end{lem}

\begin{pro}\label{pro54601}
The ai-semiring variety $\mathsf{V}(S_{(4, 546)})$ is defined by the inequalities
\begin{align}
&xy        \approx yx;&\label{5461}\\
&xy     \preceq x^2+yz; &\label{5463}\\
&xyz \approx xy+xz+yz,&\label{5462}.
\end{align}
where $z$ may be empty in \eqref{5463}.
\end{pro}
\begin{proof}
It is routine to check that $S_{(4, 546)}$ satisfies the identities \eqref{5461}--\eqref{5462}.
It remains to show that every inequality of $S_{(4, 546)}$ is derivable from \eqref{5461}--\eqref{5462}.
Let $\bq \preceq \bu$ be a nontrivial inequality,
where $\bu=\bu_1+\cdots+\bu_n$ with $\bu_i, \bq\in X_c^+$ for $1\leq i\leq n$.
Since $N_2$ is isomorphic to the subalgebra $\{3, 4\}$ of $S_{(4, 546)}$,
it follows that $N_2$ satisfies $\bq \preceq \bu$, and so $\ell(\bq)\geq 2$.
Since $S_{53}$ is isomorphic to the subalgebra $\{1, 2, 4\}$ of $S_{(4, 546)}$,
$S_{53}$ also satisfies $\bq\preceq\bu$.
By Lemma~\ref{lem5301},
$L_{\geq2}(\bu)\neq\emptyset$, $c(\bq)\subseteq c(\bu)$,
and for every $\bw\in S_2(\bq)$ there exists $\bw'\in S_2(\bu)$ with $c(\bw')\subseteq c(\bw)$.

We may write $\bq=x_1x_2\cdots x_n$ with $n\geq2$, where the variables are not necessarily distinct.
By the identity \eqref{5462}, we have
\[
\bq \approx \sum_{1\leq i<j\leq n} x_ix_j.
\]
Thus it suffices to consider the case that $\ell(\bq)=2$.

\textbf{Case 1.} $\bq=x^2$.
Then there exists $\bw'\in L_2(\bu)$ such that $c(\bw')\subseteq c(x^2)$,
and so $\bw'=x^2$. We may assume that $\bw'\in L_2(\bu_i)$ for some $\bu_i\in \bu$.
Since $\bq \preceq \bu$ is nontrivial, $\ell(\bu_i)\geq 3$. Now we have
\[
\bu \succeq \bu_i \stackrel{\eqref{5462}, \eqref{5461}} \succeq \bq.
\]
This derives the inequality $\bu \succeq \bq$.

\textbf{Case 2.} $\bq=xy$, where $x$ and $y$ are distinct.
Then there exists $\bw'\in L_2(\bu)$ such that $c(\bw')\subseteq c(xy)$,
and so $\bw'\in\{x^2, y^2, xy\}$. We may assume that $\bw'\in L_2(\bu_i)$ for some $\bu_i\in \bu$.
If $\bw'=xy$, then $\ell(\bu_i)\geq 3$, since $\bq \preceq \bu$ is nontrivial.
Then
\[
\bu \succeq \bu_i \stackrel{\eqref{5462}, \eqref{5461}} \succeq \bq.
\]
Now suppose that $\bw'=x^2$. Since $c(\bq)\subseteq c(\bu)$,
there exists $\bu_j \in \bu$ such that $y\in c(\bu_j)$.
Write $\bu_j = yz\bu'_j$, where $z$ is empty if $\ell(\bu_j)=1$.
Then
\[
\bu \succeq \bu_i+\bu_j \stackrel{\eqref{5462}, \eqref{5461}} \succeq x^2+\bu_j
\stackrel{\eqref{5462}, \eqref{5461}} \succeq x^2+yz \stackrel{\eqref{5463}} \succeq xy=\bq.
\]
The case that $\bw'=x^2$ is similar to the preceding one.
\end{proof}

\begin{table}[ht]
\caption{The Cayley tables of $S_{53}$} \label{tbs52}
\begin{tabular}{c|ccc}
$+$      &$1$ &$2$ &$3$\\
\hline
$1$      &$1$ &$1$ &$3$\\
$2$      &$1$ &$2$ &$3$\\
$3$      &$3$ &$3$ &$3$\\
\end{tabular}\qquad
\begin{tabular}{c|ccc}
$\cdot$      &$1$ &$2$ &$3$\\
\hline
$1$      &$2$ &$2$ &$3$\\
$2$      &$2$ &$2$ &$3$\\
$3$      &$3$ &$3$ &$3$\\
\end{tabular}
\end{table}

\begin{remark}\label{remark26080101}
It is a routine matter to verify that $S_{(4, 546)}$ is isomorphic to a subdirect product of $S_{53}$ and $S_{52}$
(see Table~\ref{tbs52} for its Cayley tables)
via the congruences defined by the nontrivial blocks $\{3, 4\}$ and $\{2, 3\}$.
Thus $\mathsf{V}(S_{(4, 546)}) = \mathsf{V}(S_{53}, S_{52})$.
On the other hand,
it is easy to see that $M_2$ and $N_2$ are isomorphic to the subalgebras $\{2, 3\}$ and $\{1, 2\}$ of $S_{52}$, respectively.
Also, $S_{52}$ satisfies the identities
\[
xy \approx yx, \quad xy \preceq xyz, \quad xy \preceq x+y,
\]
which, by \cite{ryy}, form an equational basis of $\mathsf{V}(M_2, N_2)$.
Thus $\mathsf{V}(S_{52})=\mathsf{V}(M_2, N_2)$, and so $\mathsf{V}(S_{(4, 546)})=\mathsf{V}(S_{53}, M_2, N_2)$.
Since $M_2$ embeds into $S_{53}$, we obtain
\[
\mathsf{V}(S_{(4, 546)})=\mathsf{V}(S_{53}, N_2).
\]
\end{remark}

\begin{cor}
The ai-semiring $S_{(4,818)}$ is finitely based.
\end{cor}
\begin{proof}
It is easy to verify that $S_{(4, 818)}$ is isomorphic to a subdirect product of $S_{53}$ and $N_2$
via the congruences defined by the nontrivial blocks $\{1, 2\}$ and $\{2, 3, 4\}$.
Thus $\mathsf V(S_{(4, 818)}) = \mathsf V(S_{53}, N_2)$.
By Remark~\ref{remark26080101}, $\mathsf V(S_{(4, 818)}) =\mathsf{V}(S_{(4, 546)})$.
Proposition \ref{pro54601} tells us that $S_{(4, 546)}$ is finitely based.
Therefore, $S_{(4,818)}$ is finitely based.
\end{proof}

\begin{pro}\label{pro54401}
The ai-semiring variety $\mathsf{V}(S_{(4, 544)})$ is defined by the identities
\begin{align}
& xyz \approx xzy; \label{54401}\\
& x \preceq xy; \label{54402}\\
& xy \preceq x+y^2z; \label{54403}\\
& xy \preceq x+xyz; \label{54404}\\
& xyz \preceq xt+yxz; \label{54416}\\
& xy \preceq xt+y_1y^2y_2; \label{54410}\\
& xy \preceq xt+zxy; \label{54412}\\
& xz \preceq x^2y+zt; \label{54414}\\
& x^2y \preceq x^2z+y_1yy_2; \label{54406}\\
& xyz \preceq xz+xy+tyz; \label{54413}\\
& xyzt \approx xyz+xyt+xzt; \label{54415}\\
& xyz \preceq xy+xz+y_1y^2+z_1z;\label{54417}\\
& x_1^2x_2 \preceq x_1x_5+x_4x_1^2+x_6x_2x_3, \label{54408}
\end{align}
where $y_1$ and $y_2$ may be empty in \eqref{54406} and \eqref{54410},
$t$ may be empty in \eqref{54413}, \eqref{54410} and \eqref{54412},
$z$ and $t$ may be empty in \eqref{54416},
$y$ and $t$ may be empty in \eqref{54414},
$x_5$ and $x_6$ may be empty in \eqref{54408}.
\end{pro}
\begin{proof}
It is easy to check that both $S_{(4, 544)}$ satisfies the identities \eqref{54401}--\eqref{54408}.
In the remainder it is enough to prove that every identity that holds in $S_{(4, 544)}$
can be derived by \eqref{54401}--\eqref{54408}.
Let $\bq\preceq \bu$ be such a nontrivial inequality, where
$\bu=\bu_1+\bu_2+\cdots+\bu_n$ and $\bu_i, \bq \in X^+$, $1 \leq i \leq n$.
It follows that $L_2$ is isomorphic to $\{3, 4\}$,
we have that $L_2$ satisfies $\bq\preceq \bu$,
and so there exists $\bu_i\in \bu$ such that $h(\bu_i)=h(\bq)$.
Since $S_{53}$ is isomorphic to the quotient algebra $S_{(4, 544)}/\rho$,
where $\{3, 4\}$ is the nontrivial block of $\rho$,
we have that $S_{53}$ satisfies $\bq\preceq \bu$.
By Lemma \ref{lem5301},
$L_{\geq 2}(\bu)\neq \emptyset$, $c(\bq)\subseteq c(\bu)$,
and for every $\bw\in S_2(\bq)$ there exists $\bw'\in S_2(\bu)$ such that $c(\bw')\subseteq c(\bw)$.
The identity \eqref{54415} tells us that $\ell(\bq)\leq3$ and $\ell(\bu_i)\leq3$ for all $\bu_i\in \bu$.

\textbf{Case 1.} $\ell(\bq)=1$.
Then there exists $\bu_i\in L_{\geq 2}(\bu)$ such that $h(\bu_i)=h(\bq)$. Consequently,
\[
\bu \succeq \bu_i = h(\bq)s(\bu_i) \stackrel{\eqref{54402}}\succeq \bq,
\]
which yields the desired inequality $\bu\succeq\bq$.

\textbf{Case 2.} $\ell(\bq)=2$. Write $\bq=xy$.
By Lemma~\ref{lem5301}, there exists $\bw\in S_2(\bu)$ such that $c(\bw)\subseteq\{x,y\}$;
hence $\bw\in\{x^2,y^2,xy,yx\}$. We consider each possibility.

\textbf{Subcase 2.1.} $\bw=x^2$. Then $\bu_k=p_1x^2p_2$ for some $\bu_k\in\bu$ and $p_1,p_2\in X^*$.
Since $c(\bq)\subseteq c(\bu)$, we can choose $\bu_\ell\in\bu$ with $y\in c(\bu_\ell)$; write $\bu_\ell=p_3yp_4$.
Now we have
\[
\bu \succeq\bu_k+\bu_\ell=x^2p_2+ p_3yp_4 \stackrel{\eqref{54406}}\succeq x^2y\stackrel{\eqref{54402}}\succeq xy=\bq
\]
or
\[
\bu \succeq \bu_i+\bu_k+\bu_\ell=xs(\bu_i)+p_1x^2+ p_3yp_4 \stackrel{\eqref{54408}}\succeq x^2y\stackrel{\eqref{54402}}\succeq xy=\bq.
\]

\textbf{Subcase 2.2.} $\bw=y^2$. Then $\bu_k=p_1y^2p_2$ for some $\bu_k\in\bu$ and $p_1,p_2\in X^*$. Hence
\[
\bu \succeq \bu_i+\bu_k = xs(\bu_i) + p_1y^2p_2 \stackrel{\eqref{54410}}\succeq xy=\bq.
\]

\textbf{Subcase 2.3.} $\bw=xy$. Then $\bu_k=xy\bu_k'$ or $\bu_k=\bu_k'xy$ for some $\bu_k'\in X^*$.
If $\bu_k=xy\bu_k'$, then by \eqref{54402},
\[
\bu \succeq \bu_k = xy\bu_k' \succeq xy=\bq.
\]
If $\bu_k=\bu_k'xy$, then
\[
\bu \succeq \bu_i+\bu_k = xs(\bu_i) + \bu_k'xy \stackrel{\eqref{54412}}\succeq xy=\bq.
\]

\textbf{Subcase 2.4.} $\bw=yx$. Then $\bu_k=yx\bu_k'$ or $\bu_k=\bu_k'xy$ for some $\bu_k'\in X^*$.
If $\bu_k=\bu_k'xy$, the argument is identical to the second case in Subcase 2.3 and yields $\bu\succeq xy$.
It remains to consider $\bu_k=yx\bu_k'$. In this case,
\[
\bu \succeq \bu_i+\bu_k = xs(\bu_i) + yx\bu_k' \stackrel{\eqref{54416}}\succeq xy\bu_k'\stackrel{\eqref{54402}}\succeq xy=\bq.
\]
Thus, in all subcases, $\bu\succeq\bq$ is established.

\textbf{Case 3.} $\ell(\bq)=3$. Write $\bq=xyz$.
The identity \eqref{54401} implies that $S_2(\bq)=\{xy, yz, xz\}$.
By Lemma~\ref{lem5301}, for each $\bw\in S_2(\bq)$, there exists $\bw_1\in S_2(\bu)$ with $c(\bw_1)\subseteq c(\bw)$.
For $\bw=xy$ or $\bw=xz$, the derivation of $\bu\succeq \bw$ follows by an argument analogous to that in Case 2.
Thus it remains to consider the case $\bw=yz$.
For $\bw=yz$, there exists $\bw_1\in S_2(\bu)$ such that $c(\bw_1)\subseteq \{y,z\}$; hence $\bw_1\in\{y^2,z^2,yz,zy\}$.

\textbf{Subcase 3.1.} $\bw_1=y^2$ or $z^2$.
By symmetry, it suffices to consider $\bw_1=y^2$.
Then $\bu_k=p_1y^2p_2$ for some $\bu_k\in\bu$ and $p_1,p_2\in X^*$.
Since $c(\bq)\subseteq c(\bu)$, we have that there $\bu_\ell\in\bu$ such that $z\in c(\bu_\ell)$, and so $\bu_\ell=p_3zp_4$.
Now we have
\[
\bu \succeq\bu_k+\bu_\ell=y^2p_2+ p_3zp_4 \stackrel{\eqref{54406}}\succeq y^2z\stackrel{\eqref{54402}}\succeq yz
\]
or
\[
\bu \succeq \bu_k+\bu_\ell=p_1y^2+zp_4 \stackrel{\eqref{54410}}\succeq zy
\]
or
\[
\bu \succeq xy+xz+\bu_k+\bu_\ell=xy+xz+p_1y^2+p_3z \stackrel{\eqref{54417}}\succeq xyz.
\]

\textbf{Subcase 3.2.} $\bw_1=yz$ or $zy$.
By symmetry, it suffices to consider $\bw_1=yz$.
Then $\bu_k=yz\bu_k'$ or $\bu_k=\bu_k'yz$ for some $\bu_k'\in X^*$.
If $\bu_k=yz\bu_k'$, then by \eqref{54402},
\[
\bu \succeq \bu_k = yz\bu_k' \succeq yz.
\]
If $\bu_k=\bu_k'yz$, then
\[
\bu \succeq xy + xz + \bu_k = xy + xz + \bu_k'yz \stackrel{\eqref{54413}}\succeq xyz.
\]

In each of the above subcases, we have derived that either $\bu \succeq yz$, $\bu \succeq zy$, or directly $\bu \succeq xyz$.
If $\bu \succeq yz$ or $\bu \succeq zy$, then, together with the already established inequalities $\bu \succeq xy$ and $\bu \succeq xz$,
we obtain
\[
\bu \succeq xy+xz+yz \stackrel{\eqref{54413}}\succeq xyz
\]
or
\[
\bu \succeq xz+xy+zy \stackrel{\eqref{54413}}\succeq xzy\stackrel{\eqref{54401}} \approx xyz.
\]
Thus, in all instances, $\bu \succeq xyz = \bq$. This completes the proof of Case 3.
\end{proof}

\begin{table}[ht]
\caption{The Cayley tables of $S_{31}$} \label{tbs31}
\begin{tabular}{c|ccc}
$+$      &$1$ &$2$ &$3$\\
\hline
$1$      &$1$ &$1$ &$3$\\
$2$      &$1$ &$2$ &$3$\\
$3$      &$3$ &$3$ &$3$\\
\end{tabular}\qquad
\begin{tabular}{c|ccc}
$\cdot$      &$1$ &$2$ &$3$\\
\hline
$1$      &$1$ &$1$ &$3$\\
$2$      &$2$ &$2$ &$3$\\
$3$      &$3$ &$3$ &$3$\\
\end{tabular}
\end{table}

\begin{remark}\label{remark26080801}
It is easy to check that $S_{(4, 544)}$ is isomorphic to a subdirect product of $S_{31}$ and $S_{53}$
(see Table~\ref{tbs31} for the Cayley tables of $S_{31}$)
via the congruences defined by the nontrivial blocks $\{1, 2\}$ and $\{3, 4\}$.
Hence $\mathsf{V}(S_{(4, 544)}) = \mathsf{V}(S_{53}, S_{31})$.
On the other hand, $M_2$ and $L_2$ are isomorphic to the subalgebras $\{2, 3\}$ and $\{1, 2\}$ of $S_{31}$, respectively.
Moreover, $S_{31}$ satisfies the identities
\[
x \preceq xy, \quad xyz \approx xzy, \quad xy \preceq x+y,
\]
which, by \cite{ryy}, form an equational basis of $\mathsf{V}(M_2, L_2)$.
Consequently, $\mathsf{V}(S_{31})=\mathsf{V}(M_2, L_2)$,
and therefore $\mathsf{V}(S_{(4, 544)})=\mathsf{V}(S_{53}, M_2, L_2)$.
Since $M_2$ embeds into $S_{53}$,
we obtain
\[
\mathsf{V}(S_{(4, 544)})=\mathsf{V}(S_{53}, L_2).\qedhere
\]
\end{remark}

\begin{cor}
The ai-semiring $S_{(4, 697)}$ is finitely based.
\end{cor}
\begin{proof}
It is a routine matter to verify that $S_{(4, 697)}$ is isomorphic to a subdirect product of $S_{53}$ and $L_2$
via the congruences defined by the nontrivial blocks $\{1, 2\}$ and $\{2, 3, 4\}$.
Thus $\mathsf{V}(S_{(4, 697)})=\mathsf{V}(S_{53}, L_2)$.
By Remark~\ref{remark26080801}, $\mathsf{V}(S_{53}, L_2)=\mathsf{V}(S_{(4, 544)})$.
Hence $\mathsf{V}(S_{(4, 697)})=\mathsf{V}(S_{(4, 544)})$,
and so $S_{(4, 697)}$ and $S_{(4, 544)}$ have the same finite basis property.
Proposition \ref{pro54401} tells us that $S_{(4, 544)}$ is finitely based.
Therefore, $S_{(4, 697)}$ is finitely based.
\end{proof}

\begin{cor}\label{cor54301}
The ai-semiring $S_{(4, 543)}$ is finitely based.
\end{cor}
\begin{proof}
It is easy to see that $S_{(4, 543)}$ and $S_{(4, 544)}$ have dual multiplications.
By Proposition \ref{pro54401}, we immediately deduce $S_{(4, 543)}$ is finitely based.
\end{proof}

\begin{table}[ht]
\caption{The Cayley tables of $S_{38}$} \label{tb3801}
\begin{tabular}{c|ccc}
$+$      &$1$ &$2$ &$3$\\
\hline
$1$      &$1$ &$1$ &$3$\\
$2$      &$1$ &$2$ &$3$\\
$3$      &$3$ &$3$ &$3$\\
\end{tabular}\qquad
\begin{tabular}{c|ccc}
$\cdot$      &$1$ &$2$ &$3$\\
\hline
$1$      &$1$ &$2$ &$3$\\
$2$      &$1$ &$2$ &$3$\\
$3$      &$3$ &$3$ &$3$\\
\end{tabular}
\end{table}
\begin{remark}\label{remark26080810}
It is a routine matter to verify that $S_{(4, 543)}$ is isomorphic to a subdirect product of $S_{53}$ and $S_{38}$
((see Table~\ref{tb3801} for the Cayley tables of $S_{38}$))
via the congruences defined by the nontrivial blocks $\{1, 2\}$ and $\{3, 4\}$.
Thus $\mathsf{V}(S_{(4, 543)}) = \mathsf{V}(S_{53}, S_{38})$.
On the other hand,
it is easy to see that both $M_2$ and $R_2$ are isomorphic to
the subalgebras $\{1, 3\}$ and $\{1, 2\}$ of $S_{38}$, respectively.
Also, $S_{38}$ satisfies the identities
\[
x \preceq yx, \quad xyz \approx yxz, \quad xy \preceq x+y,
\]
which, by \cite{ryy}, form an equational basis of $\mathsf{V}(M_2, R_2)$.
Hence $\mathsf{V}(S_{38})=\mathsf{V}(M_2, R_2)$ and so $\mathsf{V}(S_{(4, 543)})=\mathsf{V}(S_{53}, M_2, R_2)$.
Since $M_2$ can be embedded into $S_{53}$, we therefore have
\[
\mathsf{V}(S_{(4, 543)})=\mathsf{V}(S_{53}, R_2).
\]
\end{remark}

\begin{cor}
The ai-semiring $S_{(4,608)}$ is finitely based.
\end{cor}
\begin{proof}
It is easy to verify that $S_{(4, 608)}$ is isomorphic to a subdirect product of $S_{53}$ and $R_2$
via the congruences defined by the nontrivial blocks $\{1, 2\}$ and $\{2, 3, 4\}$.
Hence $\mathsf{V}(S_{(4, 608)}) = \mathsf{V}(S_{53}, R_2)$.
By Remark~\ref{remark26080810}, $\mathsf{V}(S_{53}, R_2)=\mathsf{V}(S_{(4, 543)})$.
Thus $\mathsf{V}(S_{(4, 608)}) = \mathsf{V}(S_{(4, 543)})$,
and so $S_{(4, 608)}$ and $S_{(4, 543)}$ have the same finite basis property.
By Corollary~\ref{cor54301}, $S_{(4, 543)}$ is finitely based.
Therefore, $S_{(4,608)}$ is finitely based.
\end{proof}

\section{Equational bases for 4-element ai-semirings related to $S_{55}$}
In this section, we provide equational bases for some 4-element ai-semirings that are related to $S_{55}$,
whose Cayley tables are given in Table~\ref{tb5501}.

\begin{table}[ht]
\caption{The Cayley tables of $S_{55}$} \label{tb5501}
\begin{tabular}{c|ccc}
$+$      &$1$ &$2$ &$3$\\
\hline
$1$      &$1$ &$1$ &$3$\\
$2$      &$1$ &$2$ &$3$\\
$3$      &$3$ &$3$ &$3$\\
\end{tabular}\qquad
\begin{tabular}{c|ccc}
$\cdot$      &$1$ &$2$ &$3$\\
\hline
$1$      &$3$ &$2$ &$3$\\
$2$      &$2$ &$2$ &$2$\\
$3$      &$3$ &$2$ &$3$\\
\end{tabular}
\end{table}

The following result, which is due to Yue et al.~\cite[Lemma 1.2]{yrzs},
provides a solution of the equational problem for $S_{55}$.

\begin{lem}\label{lem5501}
Let $\bq \preceq \bu$ be a nontrivial inequality,
where $\bu=\bu_1+\bu_2+\cdots+\bu_n$ with $\bu_i, \bq \in X^+$ for $1 \leq i \leq n$.
Then $\bq \preceq \bu$ is satisfied by $S_{55}$
if and only if the intersection $D_\bq(\bu)\cap L_{\geq 2}(\bu)$ is nonempty.
\end{lem}

\begin{pro}\label{pro56601}
The ai-semiring variety $\mathsf{V}(S_{55}, M_2)$ is defined by the identities
\begin{align}
&x       \preceq x^2;\label{5661}\\
&x^3   \approx x^2;\label{5660}\\
&xy         \approx yx;\label{5662}\\
&xyz       \preceq x+yz;\label{5663}\\
&xyz  \preceq xy+xyzt.\label{5664}
\end{align}
\end{pro}

\begin{proof}
It is routine to check that both $S_{55}$ and $M_2$ satisfy the identities \eqref{5661}--\eqref{5664}.
In the remainder we show that every inequality holding in both $S_{55}$ and $M_2$ is derivable from \eqref{5661}--\eqref{5664}.
Let $\bq\preceq \bu$ be a nontrivial inequality,
where $\bu=\bu_1+\bu_2+\cdots+\bu_n$ with $\bu_i, \bq \in X^+$ for $1 \leq i \leq n$.
Since $M_2$ satisfies $\bq\preceq \bu$, we have that $c(\bq)\subseteq c(\bu)$.
Since $S_{55}$ satisfies $\bq\preceq \bu$,
it follows from Lemma~\ref{lem5501} that there exists $\bu_i\in D_\bq(\bu)$ such that $\ell(\bu_i)\geq 2$;
in particular, $c(\bu_i)\subseteq c(\bq)$.

\textbf{Case1.} $\ell(\bq)=1$. Then $\bq=x$ for some $x\in X$,
and so $\bu_i = x^k$ for $k \geq 2$.
Hence
\[
\bu \succeq \bu_i = x^k \stackrel{\eqref{5661}, \eqref{5660}}\succeq  x = \bq.
\]
This proves the inequality $\bu \succeq \bq$.

\textbf{Case2.} $\ell(\bq) \geq 2$. Then
\begin{align*}
&\bu \approx \bu_i+\bu_1 + \bu_2 + \cdots + \bu_n \stackrel{\eqref{5663}} \succeq \bu_i+\bu_1^k \bu_2^k \cdots \bu_n^k  \stackrel{\eqref{5662}} \approx \bu_i+\bq\bp\\
&\approx\bu_i+\bu_i+\bq\bp\stackrel{\eqref{5663}}\succeq \bu_i+\bu_i\bq\bp \stackrel{\eqref{5664}} \succeq \bu_i\bq \stackrel{\eqref{5661},\eqref{5662}} \succeq\bq,
\end{align*}
where $k$ is a sufficiently large positive integer.
Here the second step uses $\ell(\bu_i)\geq 2$,
the third step follows from $c(\bq)\subseteq c(\bu)$,
the sixth step uses $\ell(\bu_i)\geq 2$,
and the last step uses $c(\bu_i)\subseteq c(\bq)$.
This derives the inequality $\bu \succeq \bq$.
\end{proof}

\begin{remark}\label{remark26080820}
We note that $\mathsf{V}(S_{(4, 566)}) = \mathsf{V}(S_{55}, M_2)$.
Indeed, it is easy to verify that $S_{(4, 566)}$ satisfies the identities \eqref{5661}--\eqref{5664}.
By Proposition \ref{pro56601}, $\mathsf{V}(S_{(4, 566)})$ is a subvariety of $\mathsf{V}(S_{55}, M_2)$.
Conversely, $S_{55}$ is isomorphic to the subalgebra $\{2,3,4\}$ of $S_{(4, 566)}$,
and $M_2$ is isomorphic to the subalgebra $\{1,2\}$ of $S_{(4, 566)}$.
This implies that $\mathsf{V}(S_{55}, M_2)$ is a subvariety of $\mathsf{V}(S_{(4, 566)})$.
Consequently, $\mathsf{V}(S_{(4, 566)}) = \mathsf{V}(S_{55}, M_2)$.
\end{remark}

\begin{cor}\label{coro26080815}
The ai-semiring $S_{(4, 566)}$ is finitely based.
\end{cor}
\begin{proof}
This follows immediately from Proposition~\ref{pro56601} and Remark~\ref{remark26080820}.
\end{proof}

\begin{cor}
The ai-semiring $S_{(4, 656)}$ is finitely based.
\end{cor}
\begin{proof}
It is a routine matter to verify that $S_{(4, 656)}$ is isomorphic to a subdirect product of $S_{55}$ and $M_2$
via the congruences defined by the nontrivial blocks $\{3, 4\}$ and $\{1, 2, 3\}$.
Thus $\mathsf{V}(S_{(4, 656)})=\mathsf{V}(S_{55}, M_2)$.
By Remark~\ref{remark26080820}, $\mathsf{V}(S_{55}, M_2)=\mathsf{V}(S_{(4, 566)})$.
Hence $\mathsf{V}(S_{(4, 656)})=\mathsf{V}(S_{(4, 566)})$,
so $S_{(4, 656)}$ and $S_{(4, 566)}$ have the same finite basis property.
By Corollary~\ref{coro26080815}, $S_{(4, 566)}$ is finitely based.
Therefore, $S_{(4, 656)}$ is finitely based.
\end{proof}

\begin{pro}\label{pro66901}
The ai-semiring variety $\mathsf{V}(S_{(4, 669)})$ is defined by the identities
\begin{align}
&x       \preceq x^2;\label{6691}\\
&x^3   \approx x^2;\label{6690}\\
&xyz         \preceq xy;\label{6692}\\
&xyz     \preceq x+yz;\label{6693}\\
&zxy    \preceq ztxy+xy;\label{6694}\\
&zxy        \approx zyx.\label{6695}
\end{align}
\end{pro}
\begin{proof}
It is easy to verify that $S_{(4,669)}$ satisfies the inequalities \eqref{6691}--\eqref{6695}.
It remains to show that every inequality of $S_{(4,669)}$ is derivable from \eqref{6691}--\eqref{6695}.
Let $\bq \preceq \bu$ be such a nontrivial inequality,
where $\bu = \bu_1+\bu_2+\cdots+\bu_n$ with $\bu_i, \bq \in X^+$ for $1 \leq i \leq n$.
It is routine to check that $S_{(4, 669)}$ is isomorphic to a subdirect product of
$L_2$ and $S_{55}$
via the congruences defined by the nontrivial blocks $\{1, 2, 3\}$ and $\{3, 4\}$.
Thus both $L_2$ and $S_{55}$ satisfies $\bq\preceq \bu$,
so $h(\bu_i)=h(\bq)$ for some $\bu_i \in \bu$,
and by Lemma~\ref{lem5501}, there exists $\bu_j\in D_\bq(\bu)$ such that $\ell(\bu_j)\geq 2$;
in particular, $c(\bu_j) \subseteq c(\bq)$.

\textbf{Case1.} $\ell(\bq)=1$.
Then $\bq=x$ for some $x\in X$, and so $\bu_j = x^k$ for $k \geq 2$. Hence
\[
\bu \succeq \bu_i = x^k \stackrel{\eqref{6690}} \approx x^2 \stackrel{\eqref{6691}}\succeq  x = \bq .
\]
This derives the inequality $\bu \succeq \bq$.

\textbf{Case2.} $\ell(\bq) \geq 2$. Then
\begin{align*}
&\bu \succeq \bu_i + \bu_j = h(\bq)s(\bu_i) + \bu_j\approx h(\bq)s(\bu_i) + \bu_j+ \bu_j
\stackrel{\eqref{6692}} \succeq h(\bq)s(\bu_i) + \bu_j + \bu_j\bq\\
&\stackrel{\eqref{6693}} \succeq h(\bq)s(\bu_i)\bu_j\bq + \bu_j
\stackrel{\eqref{6695}}\approx h(\bq)s(\bu_i)\bq\bu_j + \bu_j
\stackrel{\eqref{6694}} \succeq h(\bq)\bu_j
\stackrel{\eqref{6692}} \succeq h(\bq)\bu_j\bq
\stackrel{\eqref{6695},\eqref{6691}} \succeq \bq.
\end{align*}
The second step uses $h(\bu_i)=h(\bq)$,
the fourth step relies on $\ell(\bu_j)\geq 2$,
and the last step follows from $c(\bu_j) \subseteq c(\bq)$.
This derives the inequality $\bu \succeq \bq$.
\end{proof}


\begin{remark}
We note that $\mathsf{V}(S_{(4, 669)})=\mathsf{V}(S_{55}, L_2)$,
since $S_{(4, 669)}$ is isomorphic to a subdirect product of $S_{55}$ and $L_2$.
\end{remark}

\begin{cor}
The ai-semiring $S_{(4, 657)}$ is finitely based.
\end{cor}
\begin{proof}
It is easy to see that $S_{(4, 657)}$ and $S_{(4, 669)}$ have dual multiplications.
By Proposition \ref{pro66901}, $S_{(4, 657)}$ is finitely based.
\end{proof}

\begin{table}[ht]
\caption{The Cayley tables of $S_{49}$} \label{tb4901}
\begin{tabular}{c|ccc}
$+$      &$1$ &$2$ &$3$\\
\hline
$1$      &$1$ &$1$ &$3$\\
$2$      &$1$ &$2$ &$3$\\
$3$      &$3$ &$3$ &$3$\\
\end{tabular}\qquad
\begin{tabular}{c|ccc}
$\cdot$      &$1$ &$2$ &$3$\\
\hline
$1$      &$2$ &$2$ &$2$\\
$2$      &$2$ &$2$ &$2$\\
$3$      &$2$ &$2$ &$3$\\
\end{tabular}
\end{table}

\begin{pro}\label{pro68501}
The ai-semirings $S_{(4, 685)}$ and $S_{(4, 831)}$ are both finitely based.
\end{pro}
\begin{proof}
It is easy to verify that $S_{(4, 685)}$ is isomorphic to a subdirect product of $S_{55}$ and $S_{49}$
(see Table~\ref{tb4901} for its Cayley tables)
via the congruences defined by the nontrivial blocks $\{1, 2\}$ and $\{3, 4\}$.
Hence $\mathsf{V}(S_{(4, 685)}) = \mathsf{V}(S_{55}, S_{49})$.
Now $N_2$ and $D_2$ are isomorphic to the subalgebras $\{1, 2\}$ and $\{2, 3\}$ of $S_{49}$, respectively.
Moreover, $S_{49}$ satisfies the identities
\[
xy\approx yx, \quad x^2 \preceq x, \quad x^2y \preceq xy,
\]
which, by \cite{ryy}, form an equational basis of $\mathsf{V}(N_2, D_2)$.
Thus $\mathsf{V}(S_{49})=\mathsf{V}(N_2, D_2)$, and so
\[
\mathsf{V}(S_{(4, 685)})=\mathsf{V}(S_{55}, N_2, D_2).
\]
Since $D_2$ is isomorphic to the subalgebra $\{2, 3\}$ of $S_{55}$, we therefore have
\[
\mathsf{V}(S_{(4, 685)})=\mathsf{V}(S_{55}, N_2).
\]

It is easy to verify that $S_{(4, 831)}$ is isomorphic to a subdirect product of $S_{55}$ and $N_2$
via the congruences defined by the nontrivial blocks $\{1, 2\}$ and $\{2, 3, 4\}$.
Thus
\[
\mathsf{V}(S_{(4, 831)})=\mathsf{V}(S_{55}, N_2).
\]

By \cite[Proposition 7.1, Remark 7.2]{yrzs}, $\mathsf{V}(T_2^0, N_2)$ is finitely based.
Since $T_2^0$ is isomorphic to $S_{55}$, we conclude that $\mathsf{V}(S_{55}, N_2)$ is finitely based.
Therefore, $S_{(4, 685)}$ and $S_{(4, 831)}$ are both finitely based.
\end{proof}

\section{Equational bases for 4-element ai-semirings related to $S_{57}$}
In this section, we provide equational bases for some 4-element ai-semirings that relate to $S_{57}$,
whose Cayley tables are given in Table~\ref{tb5701}.

\begin{table}[ht]
\caption{The Cayley tables of $S_{57}$} \label{tb5701}
\begin{tabular}{c|ccc}
$+$      &$1$ &$2$ &$3$\\
\hline
$1$      &$1$ &$1$ &$3$\\
$2$      &$1$ &$2$ &$3$\\
$3$      &$3$ &$3$ &$3$\\
\end{tabular}\qquad
\begin{tabular}{c|ccc}
$\cdot$      &$1$ &$2$ &$3$\\
\hline
$1$      &$3$ &$3$ &$3$\\
$2$      &$1$ &$2$ &$3$\\
$3$      &$3$ &$3$ &$3$\\
\end{tabular}
\end{table}

The following result, which is due to Yue et al.~\cite[Lemma 3.1]{yrzs},
provides some information about the inequalities satisfied by $S_{57}$.

\begin{lem}\label{lem5701}
Let $\bq \preceq \bu$ be a nontrivial inequality,
where $ \bu = \bu_1 + \bu_2 + \cdots + \bu_n$ with $\bu_i, \bq\in X^+$ for $1\leq i \leq n$.
If $\bq \preceq \bu$ holds in $S_{57}$, then $L_{\geq 2}(\bu)\neq \emptyset$,
$c(p(\bq))\subseteq c(p(\bu))$, and $t(\bq)\in c(\bu)$.
\end{lem}

\begin{table}[ht]
\caption{The Cayley tables of $S_{38}$} \label{tb3801}
\begin{tabular}{c|ccc}
$+$      &$1$ &$2$ &$3$\\
\hline
$1$      &$1$ &$1$ &$3$\\
$2$      &$1$ &$2$ &$3$\\
$3$      &$3$ &$3$ &$3$\\
\end{tabular}\qquad
\begin{tabular}{c|ccc}
$\cdot$      &$1$ &$2$ &$3$\\
\hline
$1$      &$1$ &$2$ &$3$\\
$2$      &$1$ &$2$ &$3$\\
$3$      &$3$ &$3$ &$3$\\
\end{tabular}
\end{table}

\begin{pro}
The ai-semirings $S_{(4, 536)}$, $S_{(4, 607)}$, $S_{(4, 509)}$, and $S_{(4, 695)}$ are all finitely based.
\end{pro}
\begin{proof}
It is easy to check that $S_{(4, 536)}$ is isomorphic to a subdirect product of
$S_{38}$ (see Table~\ref{tb3801} for its Cayley tables) and $S_{57}$
via the congruences defined by the nontrivial blocks $\{1, 2\}$ and $\{3, 4\}$.
Thus $\mathsf{V}(S_{(4, 536)}) = \mathsf{V}(S_{38}, S_{57})$.

Now $R_2$ and $M_2$ are isomorphic to the subalgebras $\{1, 2\}$ and $\{2, 3\}$ of $S_{38}$, respectively.
Moreover, $S_{38}$ satisfies the identities
\[
y \preceq xy, \quad xyz \approx yxz, \quad xy \preceq x+y,
\]
which, by \cite{ryy}, form an equational basis of $\mathsf{V}(R_2, M_2)$.
Thus $\mathsf{V}(S_{38})=\mathsf{V}(R_2, M_2)$, so $\mathsf{V}(S_{(4, 536)})=\mathsf{V}(S_{57}, R_2, M_2)$.
Since $M_2$ embeds into $S_{57}$ (indeed, $M_2$ is isomorphic to the subalgebra $\{2, 3\}$ of $S_{57}$),
we obtain
\[
\mathsf{V}(S_{(4, 536)})=\mathsf{V}(S_{57}, R_2).
\]

It is routine to verify that $S_{(4, 607)}$ is isomorphic to a subdirect product of $S_{57}$ and $R_2$
via the congruences defined by the nontrivial blocks $\{1, 2\}$ and $\{2, 3,4\}$.
Hence $\mathsf{V}(S_{(4, 607)})=\mathsf{V}(S_{57}, R_2)$.

By \cite[Proposition 3.2, Remark 3.3]{yrzs}, $\mathsf{V}(S_{57}, R_2)$ is finitely based.
Therefore, $S_{(4, 536)}$ and $S_{(4, 607)}$ are both finitely based.

Finally, $S_{(4, 509)}$ and $S_{(4, 536)}$ have dual multiplications,
and $S_{(4,695)}$ and $S_{(4,607)}$ have dual multiplications.
Hence $S_{(4,509)}$ and $S_{(4,695)}$ are also finitely based.
\end{proof}

\begin{pro}\label{pro312}
The ai-semiring variety $\mathsf{V}(S_{(4, 538)})$ is defined by the identities
\begin{align}
&x \preceq xy; \label{5381}\\
&x^2y \approx xy; \label{5382}\\
&xyzt \approx xzyt; \label{5383}\\
&yx \preceq x+yz; \label{5384}\\
&yz \preceq y+xyz; \label{5385}\\
&xy+xz \preceq xyz. \label{5386}
\end{align}
\end{pro}
\begin{proof}
It is routine to check that $S_{(4, 538)}$ satisfies the identities \eqref{5381}--\eqref{5386}.
It remains to show that every inequality of $S_{(4, 538)}$ is derivable from \eqref{5381}--\eqref{5386}.
Let $\bq \preceq \bu$ be such an inequality,
where $\bu=\bu_1+\bu_2+\cdots+\bu_n$ with $\bu_i, \bq \in X^+$ for $1 \leq i \leq n$.
Since $L_2$ is isomorphic to the subalgebra $\{3, 4\}$ of $S_{(4, 538)}$,
$L_2$ satisfies $\bq\preceq\bu$; hence $h(\bu_i)=h(\bq)$ for some $\bu_i \in \bu$.
Since $S_{57}$ is isomorphic to the subalgebra $\{1, 2, 3\}$ of $S_{(4, 538)}$,
it follows that $S_{57}$ satisfies $\bq\preceq\bu$.
By Lemma~\ref{lem5701}, $\ell(\bu_j)\geq 2$ for some $\bu_j\in\bu$,
$c(p(\bq))\subseteq c(p(\bu))$, and $t(\bq)\in c(\bu)$.

Suppose that $c(p(\bu))=\{x_1, x_2, \ldots, x_m\}$ and that $t(\bu)=\{y_1, y_2, \ldots, y_n\}$.
Then
\[
c(p(\bq)) \subseteq \{x_1, x_2, \ldots, x_m\}, t(\bq)\in \{x_1, x_2, \ldots, x_m,y_1, y_2, \ldots, y_n\}.
\]
Now we have
\[
\bu\succeq \bu_j+\bu_1+\bu_2+\cdots+\bu_n\stackrel{\eqref{5382},\eqref{5383},\eqref{5384}}\succeq x_j^2x_1^2x_2^2\cdots x_m^2(y_1+y_2+\cdots+y_n),
\]
where \(x_j = h(\bu_j)\).

\textbf{Case 1.} For every $\bu_k\in\bu$, there exists $\bu_\ell\in\bu$ such that $m(t(\bu_k),p(\bu_\ell))\geq 1$.
Then $c(\bq) \subseteq \{x_1, x_2, \ldots, x_m\}$.
If $\ell(\bu_i)=1$, then
\begin{align*}
\bu
&\succeq \bu_i+x_j^2x_1^2x_2^2\cdots x_m^2(y_1+y_2+\cdots+y_n)\\
&\approx \bu_i+\bq'\bq(y_1+y_2+\cdots+y_n)&&(\text{by}~{\eqref{5383}})\\
&\succeq \bq(y_1+y_2+\cdots+y_n)&&(\text{by}~{\eqref{5385}})\\
&\succeq\bq.&&(\text{by}~{\eqref{5386}})
\end{align*}
where $\bq'\bq= x_j^2x_1^2x_2^2\cdots x_m^2$.
If $\ell(\bu_i)\geq 2$, we obtain
\begin{align*}
\bu
&\succeq \bu_i+\bu_1+\bu_2+\cdots+\bu_n\\
&\succeq x_i^2x_1^2x_2^2\cdots x_m^2(y_1+y_2+\cdots+y_n)&&(\text{by}~{\eqref{5382},\eqref{5383},\eqref{5384}})\\
&\succeq \bq\bq'(y_1+y_2+\cdots+y_n)&&(\text{by}~{\eqref{5383}})\\
&\succeq\bq.&&(\text{by}~{\eqref{5386}, \eqref{5381}})
\end{align*}
where $\bq\bq'= x_i^2x_1^2x_2^2\cdots x_m^2$.

\textbf{Case 2.} There exists $\bu_k\in\bu$ such that for every $\bu_\ell\in\bu$, $m(t(\bu_k),p(\bu_\ell))=0$.

\textbf{Subcase 2.1.} $c(\bq)\subseteq c(p(\bu))$. This reduces to Case 1.

\textbf{Subcase 2.2.} $c(\bq) \nsubseteq c(p(\bu))$.
There exists $\bu_k\in \bu$ such that $t(\bu_k)=t(\bq)$ and $c(p(\bq))\subseteq c(p(\bu))$.
If $m(t(\bq),\bq)\geq 2$ , then $c(\bq)=c(p(\bq))\subseteq c(p(\bu))$, and this also reduces to Case 1.
Now suppose that $m(t(\bq),\bq)=1$.
Let $y_i\in c(\bq)$ be such that $t(\bu_k)=t(\bq)=y_i$.
If $\ell(\bu_i)=1$, then
\begin{align*}
\bu
&\succeq \bu_i+x_j^2x_1^2x_2^2\cdots x_m^2(y_1+y_2+\cdots+y_n)\\
&\approx \bu_i+x_j^2x_1^2x_2^2\cdots x_m^2y_i \\
&\approx \bu_i+\bq'\bq&&(\text{by}~{\eqref{5383}})\\
&\succeq \bq.&&(\text{by}~{\eqref{5385}})
\end{align*}
where $\bq'\bq= x_j^2x_1^2x_2^2\cdots x_m^2y_i$.
If $\ell(\bu_i)\geq 2$, then
\begin{align*}
\bu
&\succeq \bu_i+\bu_1+\bu_2+\cdots+\bu_n\\
&\succeq x_i^2x_1^2x_2^2\cdots x_m^2(y_1+y_2+\cdots+y_n)&&(\text{by}~\eqref{5382},\eqref{5383},\eqref{5384})\\
&\succeq \bq'p y_i&&(\text{by}~{\eqref{5383}})\\
&\succeq\bq,&&(\text{by}~{\eqref{5386}}or{\eqref{5381}})
\end{align*}
where $\bq'p= x_i^2x_1^2x_2^2\cdots x_m^2$.
This derives the identity $\bu \succeq \bq$.
\end{proof}

\begin{table}[ht]
\caption{The Cayley tables of $S_{31}$} \label{tb3101}
\begin{tabular}{c|ccc}
$+$      &$1$ &$2$ &$3$\\
\hline
$1$      &$1$ &$1$ &$3$\\
$2$      &$1$ &$2$ &$3$\\
$3$      &$3$ &$3$ &$3$\\
\end{tabular}\qquad
\begin{tabular}{c|ccc}
$\cdot$      &$1$ &$2$ &$3$\\
\hline
$1$      &$1$ &$1$ &$3$\\
$2$      &$2$ &$2$ &$3$\\
$3$      &$3$ &$3$ &$3$\\
\end{tabular}
\end{table}

\begin{remark}\label{remark26080910}
We note that $\mathsf{V}(S_{(4, 538)})=\mathsf{V}(S_{57}, L_2)$.
Indeed, it is easy to verify that $S_{(4, 538)}$ is isomorphic to a subdirect product of
$S_{31}$ (see Table~\ref{tb3101} for its Cayley tables) and $S_{57}$
via the congruences defined by the nontrivial blocks $\{1, 2\}$ and $\{3, 4\}$.
Thus $\mathsf V(S_{(4, 538)}) = \mathsf V(S_{31}, S_{57})$.

Now $L_2$ and $M_2$ are isomorphic to the subalgebras $\{1, 2\}$ and $\{1, 3\}$ of $S_{31}$, respectively.
Moreover, $S_{31}$ satisfies the identities
\[
x \preceq xy, \quad xyz \approx xzy, \quad xy \preceq x+y,
\]
which, by \cite{ryy}, form an equational basis of $\mathsf V(L_2, M_2)$.
Thus $\mathsf V(S_{31})=\mathsf V(L_2, M_2)$.

Therefore, $\mathsf{V}(S_{(4, 538)}) = \mathsf{V}(S_{57}, L_2, M_2)$.
Since $M_2$ embeds into $S_{57}$, we obtain
\[
\mathsf{V}(S_{(4, 538)})=\mathsf{V}(S_{57}, L_2).
\]
\end{remark}

\begin{cor}
The ai-semiring $S_{(4, 508)}$ is finitely based.
\end{cor}
\begin{proof}
It is readily seen that $S_{(4, 508)}$ and $S_{(4, 538)}$ have dual multiplications.
By Proposition~\ref{pro312}, $S_{(4, 538)}$ is finitely based.
Consequently, $S_{(4, 508)}$ is also finitely based.
\end{proof}

\begin{cor}
The ai-semirings $S_{(4, 696)}$ and $S_{(4, 605)}$ are both finitely based.
\end{cor}
\begin{proof}
It is routine to verify that $S_{(4, 696)}$ is isomorphic to a subdirect product of $S_{57}$ and $L_2$
via the congruences defined by the nontrivial blocks $\{1, 2\}$ and $\{2, 3, 4\}$.
Thus $\mathsf{V}(S_{(4, 696)})=\mathsf{V}(S_{57}, L_2)$.
By Remark~\ref{remark26080910}, $\mathsf{V}(S_{57}, L_2)=\mathsf{V}(S_{(4, 538)})$.
Hence $\mathsf{V}(S_{(4, 696)})=\mathsf{V}(S_{(4, 538)})$,
so $S_{(4, 696)}$ and $S_{(4, 538)}$ have the same finite basis property.
By Proposition~\ref{pro312}, $S_{(4, 538)}$ is finitely based.
Therefore, $S_{(4, 696)}$ is finitely based.

Finally, $S_{(4, 605)}$ and $S_{(4,696)}$ have dual multiplications.
Hence $S_{(4, 605)}$ is also finitely based.
\end{proof}

\begin{pro}\label{pro313}
The ai-semiring variety $\mathsf{V}(S_{(4, 539)})$ is defined by the identities
\begin{align}
&x\preceq x^{2};\label{5391}\\
&x^{2}y\approx xy; \label{5392}\\
&xyz \approx yxz;  \label{5393}\\
&yx \preceq x+yz; \label{5394}\\
&yz \preceq z+xyz; \label{5395}\\
&xy+xz \preceq x+xyz. \label{5396}
\end{align}
\end{pro}
\begin{proof}
It is easy to check that $S_{(4, 539)}$ satisfies the identities \eqref{5391}--\eqref{5396}.
In the remainder it is enough to prove that every ai-semiring identity of $S_{(4, 539)}$
can be derived by \eqref{5391}--\eqref{5396}.
Let $\bq\preceq\bu$ be such an inequality,
where $\bu=\bu_1+\bu_2+\cdots+\bu_n$ and $\bu_i, \bq \in X^+$, $1 \leq i \leq n$.
Since $D_2$ is isomorphic to $\{3,4\}$,
it satisfies $\bq\preceq\bu$; hence $c(\bu_i)\subseteq c(\bq)$ for some $\bu_i\in\bu$.
Also, as $S_{57}$ is isomorphic to $\{1,2,3\}$, it follows that $S_{57}$ satisfies $\bq\preceq\bu$.
By Lemma~\ref{lem5701}, we have $\ell(\bu_j)\geq 2$ for some $\bu_j\in\bu$,
$c(p(\bq))\subseteq c(p(\bu))$, and $t(\bq)\in c(\bu)$.

Suppose that $c(p(\bu))=\{x_1, x_2, \ldots, x_m\}$
and that $t(\bu)=\{y_1, y_2, \ldots, y_n\}$.
Then
\[
c(p(\bq)) \subseteq \{x_1, x_2, \ldots, x_m\}, t(\bq)\in \{x_1, x_2, \ldots, x_m,y_1, y_2, \ldots, y_n\}.
\]
If $\ell(\bq)=1$, then
\[
\bu\succeq \bu_i\approx \bq^k\stackrel{\eqref{5392}}\approx \bq^2 \stackrel{\eqref{5391}}\succeq \bq.
\]
If $\ell(\bq)\geq 2$, then
\[
\bu\succeq \bu_j+\bu_1+\bu_2+\cdots+\bu_n\stackrel{\eqref{5392},\eqref{5393},\eqref{5394}}\succeq x_1^2x_2^2\cdots x_m^2(y_1+y_2+\cdots+y_n).
\]

\textbf{Case 1.} For every $\bu_k\in\bu$, there exists $\bu_\ell\in\bu$ such that $m(t(\bu_k),p(\bu_\ell))\geq 1$.
Then $c(\bq)\subseteq \{x_1,\ldots,x_m\}$, and hence
\[
\bu\succeq \bu_i+x_1^2x_2^2\cdots x_m^2(y_1+y_2+\cdots+y_n)\stackrel{\eqref{5393}}\succeq \bu_i+\bq\bq'(y_1+y_2+\cdots+y_n)\stackrel{\eqref{5392},\eqref{5396}} \succeq\bq,
\]
where $\bq\bq'= x_1^2x_2^2\cdots x_m^2$. This derives the identity $\bu \succeq \bq$.

\textbf{Case 2.} There exists $\bu_k\in\bu$ such that for every $\bu_\ell\in\bu$, $m(t(\bu_k),p(\bu_\ell))=0$.

\textbf{Subcase 2.1.} $c(\bq)\subseteq c(p(\bu))$. This reduces to Case 1.

\textbf{Subcase 2.2.} $c(\bq)\nsubseteq c(p(\bu))$.
Then there exists $\bu_k\in\bu$ such that $t(\bu_k)=t(\bq)$ and $c(p(\bq))\subseteq c(p(\bu))$.
If $m(t(\bq),\bq)\geq 2$, then $c(\bq)=c(p(\bq))\subseteq c(p(\bu))$; this also reduces to Case 1.
Now suppose that $m(t(\bq),\bq)=1$. Let $y_i\in c(\bq)$ be such that $t(\bu_k)=t(\bq)=y_i$.
Then
\[
\bu\succeq \bu_i+x_1^2x_2^2\cdots x_m^2(y_1+y_2+\cdots+y_n)\approx \bu_i+x_1^2x_2^2\cdots x_m^2y_i \stackrel{\eqref{5393}}\approx \bu_i+\bq'\bq\stackrel{\eqref{5393},\eqref{5395},\eqref{5396}}\succeq \bq,
\]
where $\bq'\bq= x_1^2x_2^2\cdots x_m^2y_i$. This derives the identity $\bu \succeq \bq$.
\end{proof}

\begin{remark}
It is a routine matter to verify that $S_{(4, 539)}$ is isomorphic to a subdirect product of $S_{43}$ and $S_{57}$
via the congruences defined by the nontrivial blocks $\{1, 2\}$ and $\{3, 4\}$.
So $\mathsf V(S_{(4, 539)}) = \mathsf V(S_{43}, S_{57})$. On the other hand, it is easy
to see that both $D_2$ and $M_2$ can be embedded into $S_{43}$. Also,
$S_{43}$ satisfies an equational basis of $\mathsf V(D_2, M_2)$ that can be found in \cite[Table 3]{sr}.
It follows that $\mathsf V(S_{43})=\mathsf V(D_2, M_2)$, and so
$\mathsf V(S_{(4, 539)}) = \mathsf V(S_{57}, D_2, M_2)$. Since $M_2$ can be embedded into $S_{57}$, we therefore obtain
\[
\mathsf{V}(S_{(4, 539)})=\mathsf{V}(S_{57}, D_2).
\]
\end{remark}

\begin{cor}
The ai-semirings $S_{(4,709)}$, $S_{(4, 510)}$ and $S_{(4,707)}$ are all finitely based.
\end{cor}
\begin{proof}
It is easy to check that $S_{(4, 709)}$ is isomorphic to a subdirect product of $S_{57}$ and $S_{22}$ via the congruences defined by the nontrivial blocks $\{1, 2\}$ and $\{2, 3\}$.
This implies that $\mathsf{V}(S_{(4, 709)}) = \mathsf{V}(S_{57}, S_{22})$.
On the other hand, it is easy
to see that both $M_2$ and $D_2$ can be embedded into $S_{22}$.
Also, $S_{22}$ satisfies the equational basis of $\mathsf{V}(M_2, D_2)$ that can be found in \cite{sr}.
Hence $\mathsf{V}(S_{22})=\mathsf{V}(M_2, D_2)$ and so $\mathsf{V}(S_{(4, 709)})=\mathsf{V}(S_{57}, M_2, D_2)$.
Since $M_2$ can be embedded into $S_{57}$, we therefore obtain
\[
\mathsf{V}(S_{(4, 709)})=\mathsf{V}(S_{57}, D_2)=\mathsf{V}(S_{(4, 539)}).
\]
Since $S_{(4, 510)}$ and $S_{(4, 539)}$ have dual multiplications, it follows from
Proposition {\ref{pro313}} that $S_{(4, 510)}$ is finitely based.
And the multiplication Cayley tables of $S_{(4,707)}$ and $S_{(4,709)}$ are dual.
Therefore,  $S_{(4,707)}$ is finitely based.
\end{proof}

\begin{pro}\label{pro314}
The ai-semiring variety $\mathsf{V}(S_{(4, 540)})$ is defined by the identities
\begin{align}
&xz \preceq xy+zt; \label{5401}\\
&xt \preceq xy+zt; \label{5402}\\
&xz \preceq xy+z; \label{5403}\\
&xyz \approx xy+yz+xz. \label{5404}
\end{align}
\end{pro}
\begin{proof}
It is easy to check that $S_{(4, 540)}$ satisfies identities \eqref{5401}--\eqref{5404}.
It remains to show that every inequality of $S_{(4, 540)}$ is derivable from \eqref{5401}--\eqref{5404}.
Let $\bq \preceq \bu$ be such a nontrivial inequality,
where $\bu=\bu_1+\bu_2+\cdots+\bu_n$ with $\bu_i, \bq \in X^+$ for $1 \leq i \leq n$.
By the identity \eqref{5404}, every word of length at least $3$ can be reduced to a sum of words of length $2$.
Thus we may assume that $\bq$ and each $\bu_i$ have length at most $2$.

Since $N_2$ is isomorphic to the subalgebra $\{3, 4\}$ of $S_{(4, 540)}$,
$N_2$ satisfies $\bq \preceq \bu$, and so $\ell(\bq)=2$.
Write $\bq=xy$, where $x$ and $y$ may be equal.
Since $S_{57}$ is isomorphic to the subalgebra $\{1, 2, 4\}$ of $S_{(4, 540)}$,
$S_{57}$ satisfies $\bq \preceq \bu$.
By Lemma~\ref{lem5701}, $c(p(\bq))\subseteq c(p(\bu))$ and $t(\bq)\in c(\bu)$.
Hence there exist $\bu_i, \bu_j\in \bu$ such that $\ell(\bu_i)=2$, $h(\bu_i)=x$, and $y\in c(\bu_j)$.
Write $\bu_i=xz$, and $\bu_j=yt$ or $\bu_j=ty$, where $t$ may be empty.
If $\bu_j=yt$, then
\[
\bu \succeq \bu_i+\bu_j=xz+yt\stackrel{\eqref{5401}, \eqref{5403}}{\succeq} xy=\bq.
\]
If $\bu_j=ty$, then
\[
\bu \succeq \bu_i+\bu_j=xz+ty\stackrel{\eqref{5402}, \eqref{5403}}{\succeq} xy=\bq.
\]
This proves the inequality $\bu \succeq \bq$.
\end{proof}

\begin{cor}
The ai-semiring $S_{(4, 512)}$ is finitely based.
\end{cor}
\begin{proof}
It is readily seen that $S_{(4, 512)}$ and $S_{(4, 540)}$ have dual multiplications.
Thus they have the same finite basis property.
By Proposition~\ref{pro314}, $S_{(4, 540)}$ is finitely based.
Therefore, $S_{(4, 512)}$ is also finitely based.
\end{proof}

\begin{remark}\label{remark26080920}
We note that $\mathsf{V}(S_{(4, 540)})=\mathsf{V}(S_{57}, N_2)$.
Indeed, it is a routine matter to verify that $S_{(4, 540)}$ is isomorphic to a subdirect product of $S_{52}$ and $S_{57}$
via the congruences defined by the nontrivial blocks $\{1, 2\}$ and $\{3, 4\}$.
So $\mathsf V(S_{(4, 540)}) = \mathsf V(S_{52}, S_{57})$. On the other hand, it is easy
to see that both $M_2$ and $N_2$ can be embedded into $S_{52}$. Also,
$S_{52}$ satisfies an equational basis of $\mathsf V(M_2, N_2)$ that can be found in \cite[Table 3]{sr}.
It follows that $\mathsf V(S_{52})=\mathsf V(M_2, N_2)$, and so
$\mathsf V(S_{(4, 540)}) = \mathsf V(S_{57}, M_2, N_2)$. Since $M_2$ can be embedded into $S_{57}$, we therefore obtain
\[
\mathsf{V}(S_{(4, 540)})=\mathsf{V}(S_{57}, N_2).
\]
\end{remark}

\begin{table}[ht]
\caption{The Cayley tables of $S_{21}$} \label{tb2101}
\begin{tabular}{c|ccc}
$+$      &$1$ &$2$ &$3$\\
\hline
$1$      &$1$ &$1$ &$3$\\
$2$      &$1$ &$2$ &$3$\\
$3$      &$3$ &$3$ &$3$\\
\end{tabular}\qquad
\begin{tabular}{c|ccc}
$\cdot$      &$1$ &$2$ &$3$\\
\hline
$1$      &$1$ &$1$ &$1$\\
$2$      &$1$ &$2$ &$1$\\
$3$      &$1$ &$1$ &$1$\\
\end{tabular}
\end{table}

\begin{cor}
The ai-semirings $S_{(4,817)}$ and $S_{(4,815)}$ are both finitely based.
\end{cor}
\begin{proof}
It is easy to check that $S_{(4, 817)}$ is isomorphic to a subdirect product of $S_{57}$ and $S_{21}$
(see Tbale~\ref{tb2101} for its Cayley tables)
via the congruences defined by the nontrivial blocks $\{1, 2\}$ and $\{2, 3\}$.
Thus $\mathsf{V}(S_{(4, 817)}) = \mathsf{V}(S_{57}, S_{21})$.

Now both $M_2$ and $N_2$ are isomorphic to the subalgebras $\{1, 2\}$ and $\{1, 3\}$ of $S_{21}$.
Moreover, $S_{21}$ satisfies the identities
\[
xy\approx yx, \quad xy \preceq xyz, \quad xy \preceq x+y,
\]
which, by \cite{ryy}, form an equational basis of $\mathsf{V}(M_2, N_2)$.
Thus $\mathsf{V}(S_{21})=\mathsf{V}(M_2, N_2)$, so $\mathsf{V}(S_{(4, 817)})=\mathsf{V}(S_{57}, M_2, N_2)$.
Since $M_2$ embeds into $S_{57}$, we obtain $\mathsf{V}(S_{(4, 817)})=\mathsf{V}(S_{57}, N_2)$.
By Remark~\ref{remark26080920}, $\mathsf{V}(S_{57}, N_2)=\mathsf{V}(S_{(4, 540)})$.
Hence $\mathsf{V}(S_{(4, 817)})=\mathsf{V}(S_{(4, 540)})$,
so $S_{(4, 817)}$ and $S_{(4, 540)}$ have the same finite basis property.
By Proposition~\ref{pro314}, $S_{(4, 540)}$ is finitely based.
Therefore, $S_{(4, 817)}$ is finitely based.

Finally, $S_{(4,815)}$ and $S_{(4,817)}$ have dual multiplications.
Thus $S_{(4,815)}$ has the same finite basis property as $S_{(4,817)}$.
Therefore, $S_{(4,815)}$ is finitely based.
\end{proof}

\begin{pro}\label{pro63301}
The ai-semiring variety $\mathsf{V}(S_{(4, 633)})$ is defined by the identities
\begin{align}
& x^2y \approx xy; \label{63302}\\
& xyz \approx yxz; \label{63303}\\
& x^2y^2 \approx y^2x^2; \label{63304}\\
& x \preceq x^2; \label{63305}\\
& xyz \preceq xy+yz; \label{63301}\\
& xyz \preceq xz+yz; \label{63306}\\
& xyzt^2 \preceq xy^2+zt^2. \label{63307}
\end{align}
\end{pro}
\begin{proof}
It is straightforward to check that $S_{(4, 633)}$ satisfies the identities \eqref{63302}--\eqref{63307}.
It suffices to show that every inequality of $S_{(4, 633)}$
is derivable from \eqref{63302}--\eqref{63307}.
Let $\bq\preceq \bu$ be such a nontrivial inequality, where
$\bu=\bu_1+\bu_2+\cdots+\bu_n$ with $\bu_i, \bq \in X^+$ for $1 \leq i \leq n$.
Since $S_{57}^0$ is isomorphic to $S_{(4, 633)}$, it follows that
$S_{57}$ satisfies $\bq\preceq D_\bq(\bu)$.
By Lemma \ref{lem5701}, $\ell(\bu_k)\geq 2$ for some $\bu_k \in D_\bq(\bu)$,
$c(p(\bq))\subseteq c(p(D_\bq(\bu)))$, and $t(\bq)\in c(D_\bq(\bu))$.
This implies that $c(\bq)=c(D_\bq(\bu))$.
We may assume that $D_\bq(\bu)=\{\bu_1, \bu_2, \ldots, \bu_k\}$ and $c(D_\bq(\bu))=\{x_1, x_2, \dots, x_m\}$.

\textbf{Case 1.} $t(\bq)\in c(p(D_\bq(\bu)))$. Then $c(\bq)\subseteq c(p(D_\bq(\bu)))$.
Since $c(\bq)=c(D_\bq(\bu))$, we have that for any $\bu_i\in D_\bq(\bu)$, there exists $\bu_j\in D_\bq(\bu)$
such that $t(\bu_i)\in c(p(\bu_j))$.
Now we have
\begin{align*}
\bu
&\succeq \bu_1+\bu_2+\cdots+\bu_k+\bu_1 \\
&\succeq\bu_1\bu_2\cdots\bu_k\bu_1 &&(\text{by}~\eqref{63301}, \eqref{63307})\\
& \approx x_1^2x_2^2\cdots x_m^2 &&(\text{by}~\eqref{63302}, \eqref{63303}, \eqref{63304})\\
& \approx \bq^2 &&(\text{by}~\eqref{63302}, \eqref{63303}, \eqref{63304})\\
& \succeq \bq. &&(\text{by}~\eqref{63305})
\end{align*}
This derives the inequality $\bu \succeq\bq$.

\textbf{Case 2.} $t(\bq)\in t(D_\bq(\bu))$. Then we may assume that $t(\bq)=t(\bu_1)$.
Since $c(\bq)=c(D_\bq(\bu))$, for each $i\in\{2, 3, \ldots,k\}$,
either $t(\bu_i)=t(\bu_1)$ or there exists $\bu_j\in D_\bq(\bu)$ such that $m(t(\bu_i), p(\bu_j))\geq 1$.
Consequently,
\[
\bu \succeq \bu_2+\cdots+\bu_k+\bu_1 \stackrel{\eqref{63301}, \eqref{63306}, \eqref{63307}}\succeq\bu_2\cdots\bu_k\bu_1 \stackrel{\eqref{63302}, \eqref{63303}}\approx p(\bq)t(\bq)= \bq.
\]
This derives the inequality $\bu \succeq\bq$.
\end{proof}

\begin{cor}
The ai-semiring $S_{(4, 630)}$ is finitely based.
\end{cor}
\begin{proof}
It is easy to see that $S_{(4, 630)}$ and $S_{(4, 633)}$ have dual multiplications.
By Proposition~\ref{pro63301}, $S_{(4, 633)}$ is finitely based.
Therefore, $S_{(4, 630)}$ is also finitely based as required.
\end{proof}

\section{Equational bases for 4-element ai-semirings related to $S_{58}$}
In this section, we provide equational bases for some 4-element ai-semirings that relate to $S_{58}$,
whose Cayley tables are given in Table~\ref{tb5801}.

\begin{table}[ht]
\caption{The Cayley tables of $S_{58}$} \label{tb5801}
\begin{tabular}{c|ccc}
$+$      &$1$ &$2$ &$3$\\
\hline
$1$      &$1$ &$1$ &$3$\\
$2$      &$1$ &$2$ &$3$\\
$3$      &$3$ &$3$ &$3$\\
\end{tabular}\qquad
\begin{tabular}{c|ccc}
$\cdot$      &$1$ &$2$ &$3$\\
\hline
$1$      &$3$ &$3$ &$3$\\
$2$      &$2$ &$2$ &$2$\\
$3$      &$3$ &$3$ &$3$\\
\end{tabular}
\end{table}

The following result, which is due to Yue at al. \cite[Lemma 4.1]{yrzs},
provides some information about the identities of $S_{58}$.

\begin{lem}\label{lem5801}
Let $\bq\preceq \bu$ be a nontrivial inequality,
where $\bu = \bu_1 + \bu_2 + \cdots + \bu_n$ for $\bu_i, \bq\in X^+$ with $1\leq i \leq n$.
Suppose that $\bq\preceq \bu$ is satisfied by $S_{58}$.
Then $\bu$ and $\bq$ satisfy the following conditions:
\begin{itemize}
\item[$(1)$] $L_{\geq2}(\bu)\neq\emptyset$;

\item[$(2)$] $H_q(\bu)\neq\emptyset$;

\item[$(3)$]  If $\ell(\bq)\geq 2$, then $L_{\geq2}(\bu)\cap H_q(\bu)\neq\emptyset$.
\end{itemize}
\end{lem}

\begin{table}[ht]
\caption{The Cayley tables of $S_{50}$} \label{tb5001}
\begin{tabular}{c|ccc}
$+$      &$1$ &$2$ &$3$\\
\hline
$1$      &$1$ &$1$ &$3$\\
$2$      &$1$ &$2$ &$3$\\
$3$      &$3$ &$3$ &$3$\\
\end{tabular}\qquad
\begin{tabular}{c|ccc}
$\cdot$      &$1$ &$2$ &$3$\\
\hline
$1$      &$2$ &$2$ &$2$\\
$2$      &$2$ &$2$ &$2$\\
$3$      &$3$ &$3$ &$3$\\
\end{tabular}
\end{table}

\begin{pro}
The ai-semirings $S_{(4, 516)}$, $S_{(4, 683)}$, $S_{(4, 816)}$, and $S_{(4, 830)}$ are all finitely based.
\end{pro}
\begin{proof}
First, $S_{(4, 683)}$ is isomorphic to a subdirect product of $S_{50}$ and $S_{58}$ via the congruences defined by the nontrivial blocks $\{1, 2\}$ and $\{3, 4\}$; the Cayley tables of $S_{50}$ are given in Table~\ref{tb5001}.
Thus $\mathsf{V}(S_{(4, 683)}) = \mathsf{V}(S_{50}, S_{58})$.
Moreover, $L_2$ and $N_2$ are isomorphic to the subalgebras $\{2, 3\}$ and $\{1, 2\}$ of $S_{50}$, respectively.
$S_{50}$ satisfies the identities
\[
xy \preceq x, \quad xy \approx xz,
\]
which form an equational basis of $\mathsf{V}(L_2, N_2)$ (see \cite{ryy}).
Thus $\mathsf{V}(S_{50})=\mathsf{V}(L_2, N_2)$, so $\mathsf{V}(S_{(4, 683)})=\mathsf{V}(S_{58}, L_2, N_2)$.
Since $L_2$ embeds into $S_{58}$ (as the subalgebra $\{2, 3\}$), we obtain
\[
\mathsf{V}(S_{(4, 683)})=\mathsf{V}(S_{58}, N_2).
\]
By \cite[Remark~4.6]{yrzs}, $\mathsf{V}(S_{58}, N_2)=\mathsf{V}(S_{(4, 475)})$,
so $\mathsf{V}(S_{(4, 683)})=\mathsf{V}(S_{(4, 475)})$.
Consequently, $S_{(4, 683)}$ and $S_{(4, 475)}$ have the same finite basis property.
By \cite[Proposition~4.5]{yrzs}, $S_{(4, 475)}$ is finitely based.
Therefore, $S_{(4, 683)}$ is finitely based.

Next, $S_{(4, 830)}$ is isomorphic to a subdirect product of $S_{58}$ and $N_2$
via the congruences defined by the nontrivial blocks $\{1, 2\}$ and $\{2, 3, 4\}$.
Thus
\[
\mathsf{V}(S_{(4, 830)})=\mathsf{V}(S_{58}, N_2)= \mathsf{V}(S_{(4,683)}).
\]
Hence $S_{(4, 830)}$ is also finitely based.

Finally, $S_{(4, 516)}$ and $S_{(4, 683)}$ have dual multiplications,
and $S_{(4,816)}$ and $S_{(4,830)}$ have dual multiplications.
Therefore, $S_{(4,516)}$ and $S_{(4,816)}$ are both finitely based.
\end{proof}

\begin{pro}
The ai-semirings $S_{(4, 708)}$ and $S_{(4, 728)}$ are both finitely based.
\end{pro}
\begin{proof}
It is easy to check that $S_{(4, 728)}$ is isomorphic to a subdirect product of $S_{58}$ and $D_{2}$ via the congruences defined by the nontrivial blocks $\{1, 2\}$ and $\{2, 3, 4\}$.
Thus
\[
\mathsf{V}(S_{(4, 728)})=\mathsf{V}(S_{58}, D_2).
\]
By \cite[Remark 4.3]{yrzs}, $\mathsf{V}(S_{58}, D_2)=\mathsf{V}(S_{(4, 453)})$,
and so $\mathsf{V}(S_{(4, 728)})=\mathsf{V}(S_{(4, 453)})$.
Hence $S_{(4, 728)}$ and $S_{(4, 453)}$ have the same finite basis property.
\cite[Proposition 4.2]{yrzs} tells us that $S_{(4, 453)}$ is finitely based.
Consequently, $S_{(4, 728)}$ is finitely based.

Finally, $S_{(4, 708)}$ and $S_{(4, 728)}$ have dual multiplications.
Therefore, $S_{(4,708)}$ is also finitely based.
\end{proof}

\begin{pro}\label{pro322}
The ai-semiring variety $\mathsf{V}(S_{(4, 653)})$ is defined by the identities
\begin{align}
&x \preceq xy; \label{5651}\\
&xyz \approx xzy; \label{5652}\\
&xyz \preceq xy+z. \label{5653}
\end{align}
\end{pro}
\begin{proof}
It is easy to verify that $S_{(4, 653)}$ satisfies the identities \eqref{5651}--\eqref{5653}.
It remains to prove that every inequality of $S_{(4, 653)}$ is derivable from \eqref{5651}--\eqref{5653}.
Let $\bq \preceq \bu$ be such an inequality,
where $\bu=\bu_1+\bu_2+\cdots+\bu_n$ for $\bu_i, \bq \in X^+$ with $1 \leq i \leq n$.

It is routine to check that $S_{(4, 653)}$ is isomorphic to a subdirect product of $M_{2}$ and $S_{58}$
via the congruences defined by the nontrivial blocks $\{1, 2, 3\}$ and $\{3, 4\}$.
Then both $M_{2}$ and $S_{58}$ satisfy $\bq \preceq \bu$,
so $c(\bq) \subseteq c(\bu)$.
By Lemma~\ref{lem5801}, we have $L_{\geq2}(\bu)\neq\emptyset$ and $H_q(\bu)\neq\emptyset$;
moreover, if $\ell(\bq)\geq 2$, then $L_{\geq2}(\bu)\cap H_q(\bu)\neq\emptyset$.

\textbf{Case 1.} $\ell(\bq)=1$.
Then there exists $\bu_i\in \bu$ such that $h(\bu_i)=h(\bq)$.
Since $\ell(\bq)=1$ and $\bq \preceq \bu$ is nontrivial,
it follows that $\ell(\bu_i)\geq 2$. So we have
\[
\bu\succeq \bu_i = \bq s(\bu_i) \stackrel{\eqref{5651}} \succeq \bq.
\]
The second step uses $h(\bu_i)=h(\bq)$.
This derives the identity $\bu \succeq \bq$.

\textbf{Case 2.} $\ell(\bq)\geq 2$.
Then there exists $\bu_j \in \bu$ such that $\ell(\bu_j)\geq 2$ and $h(\bu_j)=h(\bq)$.
Now we have
\[
\bu \succeq\bu_j+\bu_1+\bu_2+\cdots+\bu_n\stackrel{\eqref{5653}}\succeq\bu_j\bu_1\bu_2\cdots\bu_n\stackrel{\eqref{5652}}\approx \bq\bq'\stackrel{\eqref{5651}}\succeq\bq,
\]
The second step relies on $\bu_j \in \bu$,
and the third step relies on $h(\bu_j)=h(\bq)$ and $c(\bq) \subseteq c(\bu)$.
This proves the inequality $\bu \succeq \bq$.
\end{proof}

\begin{table}[ht]
\centering
\caption{The Cayley tables of $S_{23}$} \label{tb2301}
\begin{tabular}{c|ccc}
$+$      &$1$ &$2$ &$3$\\
\hline
$1$      &$1$ &$1$ &$3$\\
$2$      &$1$ &$2$ &$3$\\
$3$      &$3$ &$3$ &$3$\\
\end{tabular}\qquad
\begin{tabular}{c|ccc}
$\cdot$      &$1$ &$2$ &$3$\\
\hline
$1$      &$1$ &$1$ &$1$\\
$2$      &$1$ &$2$ &$1$\\
$3$      &$3$ &$3$ &$3$\\
\end{tabular}
\end{table}

\begin{remark}\label{remark26081901}
We note that $\mathsf{V}(S_{(4, 653)})=\mathsf{V}(S_{58}, M_2)$.
Indeed, it is a routine matter to verify that $S_{(4, 653)}$ is isomorphic to a subdirect product of
$S_{23}$ and $S_{58}$
via the congruences defined by the nontrivial blocks $\{1, 2\}$ and $\{3, 4\}$;
the Cayley tables of $S_{23}$ are given in Table~\ref{tb2301}.
Thus $\mathsf V(S_{(4, 653)}) = \mathsf V(S_{23}, S_{58})$.
Moreover, $L_2$ and $M_2$ are isomorphic to the subalgebras $\{1, 3\}$ and $\{1, 2\}$ of $S_{23}$, respectively.
$S_{23}$ satisfies the identities
\[
x \preceq xy, \quad xyz \approx xzy, \quad xy \preceq x+y,
\]
which form an equational basis of $\mathsf V(L_2, M_2)$ (see \cite{ryy}).
It follows that $\mathsf V(S_{23})=\mathsf V(L_2, M_2)$, and so
$\mathsf V(S_{(4, 653)}) = \mathsf V(S_{58}, L_2, M_2)$.
Since $L_2$ embeds into $S_{58}$, we obtain
\[
\mathsf{V}(S_{(4, 653)})=\mathsf{V}(S_{58}, M_2).
\]
\end{remark}

\begin{cor}
The ai-semirings $S_{(4, 514)}$, $S_{(4, 551)}$, and $S_{(4, 565)}$ are all finitely based.
\end{cor}
\begin{proof}
First, one can see that $S_{(4, 514)}$ and $S_{(4, 653)}$ have dual multiplications,
and so they have the same finite basis property.
By Proposition~\ref{pro322}, $S_{(4, 514)}$ is finitely based.

Next,
a straightforward verification shows that $S_{(4, 565)}$ satisfies the identities \eqref{5651}--\eqref{5653},
which form an equational basis of $\mathsf{V}(S_{(4, 653)})$ by Proposition~\ref{pro322}.
Hence $\mathsf{V}(S_{(4, 565)})$ is a subvariety of $\mathsf{V}(S_{(4, 653)})$.
Conversely, $S_{58}$ and $M_2$ are isomorphic to the subalgebras $\{2,3,4\}$ and $\{1, 4\}$ of $S_{(4, 565)}$, respectively.
It follows that $\mathsf{V}(S_{58}, M_2)$ is a subvariety of $\mathsf{V}(S_{(4, 565)})$.
Remark~\ref{remark26081901} tells us that $\mathsf{V}(S_{(4, 653)})=\mathsf{V}(S_{58}, M_2)$.
Thus $\mathsf{V}(S_{(4, 653)})$ is a subvariety of $\mathsf{V}(S_{(4, 565)})$,
and so $\mathsf{V}(S_{(4, 565)})=\mathsf{V}(S_{(4, 653)})$.
Consequently, $S_{(4, 565)}$ and $S_{(4, 653)}$ have the same finite basis property.
By Proposition~\ref{pro322}, $S_{(4, 565)}$ is finitely based.

Finally, it is easy to see that
the multiplication Cayley tables of $S_{(4,551)}$ and $S_{(4,565)}$ are dual.
Therefore,  $S_{(4,551)}$ is finitely based.
\end{proof}

\begin{pro}\label{pro66801}
The ai-semiring variety $\mathsf{V}(S_{(4, 668)})$ is defined by the identities
\begin{align}
&x^2y\approx xy;\label{6681}\\
&x \preceq x^2;\label{6682}\\
&xyz \preceq xy;\label{6683}\\
&xyz  \approx xzy.\label{6684}
\end{align}
\end{pro}
\begin{proof}
It is easy to verify that $S_{(4,668)}$ satisfies the identities \eqref{6681}--\eqref{6684}.
It now suffices to prove that every inequality of $S_{(4,668)}$ is derivable from \eqref{6681}--\eqref{6684}.
Let $\bq \preceq \bu$ be such a nontrivial inequality,
where $\bu = \bu_1 + \bu_2 + \cdots + \bu_n$ with $\bu_i, \bq \in X^+$ for $1 \leq i \leq n$.

Observe that $S_{58}^0$ is isomorphic to $S_{(4,668)}$.
Then $S_{58}^0$ satisfies $\bq \preceq \bu$, and so $S_{58}$ satisfies $\bq \preceq D_\bq(\bu)$.
By Lemma~\ref{lem5801}, $L_{\geq2}(D_\bq(\bu))\neq\emptyset$ and $H_{\bq}(D_\bq(\bu))\neq\emptyset$;
moreover, if $\ell(\bq)\geq 2$, then $L_{\geq2}(D_\bq(\bu))\cap H_q(D_\bq(\bu))\neq\emptyset$.

\textbf{Case 1.} $\ell(\bq) = 1$. Then there exists $\bu_i\in D_\bq(\bu)$ such that $\ell(\bu_i)\geq 2$.
So we have
\[
\bu \succeq \bu_i \stackrel{\eqref{6681}}\approx \bq^2 \stackrel{\eqref{6682}}\succeq \bq.
\]
The second step relies on $\ell(\bq) = 1$, $\bu_i\in D_\bq(\bu)$, and $\ell(\bu_i)\geq 2$.

\textbf{Case 2.} $\ell(\bq) \geq 2$.
Then $L_{\geq 2}(D_\bq(\bu))\cap H_{\bq}(D_\bq(\bu))\neq\emptyset$,
and so there exists $\bu_i\in L_{\geq 2}(D_\bq(\bu))$ such that $h(\bu_i)=h(\bq)$.
Consequently,
\[
\bu \succeq \bu_i \stackrel{\eqref{6683}} \succeq \bu_i\bq \stackrel{\eqref{6681}, \eqref{6682}, \eqref{6684}}\succeq \bq.
\]
The third step follows from $h(\bu_i)=h(\bq)$ and $\bu_i\in L_{\geq 2}(D_\bq(\bu))$.
This proves $\bu \succeq \bq$.
\end{proof}

\begin{cor}
The ai-semiring $S_{(4, 631)}$ is finitely based.
\end{cor}
\begin{proof}
It is readily seen that $S_{(4, 631)}$ and $S_{(4, 668)}$ have dual multiplications,
and hence they have the same finite basis property.
By Proposition~\ref{pro66801}, $S_{(4, 668)}$ is finitely based.
Therefore, $S_{(4, 631)}$ is also finitely based.
\end{proof}

\section{Equational bases for 4-element ai-semirings related to $S_{59}$}
In this section, we provide equational bases for some 4-element ai-semirings that relate to $S_{59}$,
whose Cayley tables are given in Table~\ref{tb5901}.

\begin{table}[ht]
\caption{The Cayley tables of $S_{59}$} \label{tb5901}
\begin{tabular}{c|ccc}
$+$      &$1$ &$2$ &$3$\\
\hline
$1$      &$1$ &$1$ &$3$\\
$2$      &$1$ &$2$ &$3$\\
$3$      &$3$ &$3$ &$3$\\
\end{tabular}\qquad
\begin{tabular}{c|ccc}
$\cdot$      &$1$ &$2$ &$3$\\
\hline
$1$      &$3$ &$3$ &$3$\\
$2$      &$3$ &$1$ &$3$\\
$3$      &$3$ &$3$ &$3$\\
\end{tabular}
\end{table}

The following result, which is due to Yue et al. \cite[Lemma 5.1]{yrzs},
provides some information about the identities of $S_{59}$.

\begin{lem}\label{lem5901}
Let $\bq \preceq \bu$ be an inequality,
where $\bu = \bu_1 + \bu_2 + \cdots + \bu_n$ with $\bu_i, \bq\in X^+$ for $1\leq i \leq n$.
Suppose that $\bq \preceq \bu$ is satisfied by $S_{59}$.
Then $\bu$ and $\bq$ satisfy one of the following conditions:
\begin{itemize}
\item[$(1)$] $\ell(\bu_i) \geq 3$ for some $\bu_i \in \bu$;
\item[$(2)$] $\ell(\bu_i)\leq 2$ for all $\bu_i \in \bu$. Then $\ell(\bq)\leq 2$.
Moreover, if $\ell(\bq)=1$, then $c(\bq)\subseteq c(\bu)$; if $\ell(\bq)=2$, then $c(\bq)\subseteq c(L_2(\bu))$.
\end{itemize}
\end{lem}

\begin{pro}\label{pro332}
The ai-semiring variety $\mathsf{V}(S_{(4, 548)})$ is defined by the identities
\begin{align}
&x \preceq xy; \label{5481}\\
&xy \approx yx; \label{5482}\\
&x^{2}yz \approx xyz; \label{5483}\\
&xz \preceq xy+zt; \label{5484}\\
&xyz+t \approx xyzt. \label{5485}
\end{align}
\end{pro}
\begin{proof}
It is easy to check that $S_{(4, 548)}$ satisfies the identities \eqref{5481}--\eqref{5485}.
It remains to show that every inequality of $S_{(4, 548)}$ is derivable from \eqref{5481}--\eqref{5485}.
Let $\bq \preceq \bu$ be such a nontrivial inequality,
where $\bu=\bu_1+\bu_2+\cdots+\bu_n$ with $\bu_i, \bq \in X^+$ for $1 \leq i \leq n$.
Observe that $S_{(4, 548)}$ is isomorphic to a subdirect product of $M_{2}$ and $S_{59}$
via the congruences defined by the nontrivial blocks $\{2, 3, 4\}$ and $\{1, 2\}$.
Then both $M_2$ and $S_{59}$ satisfy $\bq \preceq \bu$, and so $c(\bq) \subseteq c(\bu)$.
By Lemma~\ref{lem5901}, we consider the following two cases:

\textbf{Case 1.} $\ell(\bu_i) \geq 3$ for some $\bu_i \in \bu$. Then
\[
\bu\approx\bu_i+\bu_1+\bu_2+\cdots+\bu_n \stackrel{\eqref{5485}}\approx \bu_i\bu_1\bu_2\cdots\bu_n\stackrel{\eqref{5482},\eqref{5483}}\approx \bq\bq'\stackrel{\eqref{5481}}\succeq \bq.
\]
The third step uses $c(\bq) \subseteq c(\bu)$.

\textbf{Case 2.} $\ell(\bu_i)\leq 2$ for all $\bu_i \in \bu$. Then $\ell(\bq)\leq 2$.

If $\ell(\bq)=1$, then there exists $\bu_j\in L_{2}(\bu)$ such that $c(\bq) \subseteq c(\bu_j)$. Now we have
\[
\bu\succeq\bu_j\stackrel{\eqref{5482}}\approx\bq\bu_j'\stackrel{\eqref{5481}}\succeq \bq.
\]
The second step relies on $c(\bq) \subseteq c(\bu_j)$, $\ell(\bu_j)=2$, and $\ell(\bq)=1$.
This derives the inequality $\bu \succeq \bq$.

If $\ell(\bq)= 2$, then $c(\bq)\subseteq c(L_2(\bu))$.
We may write $\bq=xy$. By \eqref{5482}, $xx_1, yy_1\in L_2(\bu)$ for some $x_1, y_1\in X$. We have
\[
\bu\stackrel{\eqref{5482}}\succeq xx_1+yy_1\stackrel{\eqref{5484}}\succeq xy = \bq.
\]
This derives the inequality $\bu \succeq \bq$.
\end{proof}

\begin{remark}
We note that $\mathsf{V}(S_{(4, 548)})=\mathsf{V}(S_{59}, M_2)$,
since $S_{(4, 548)}$ is isomorphic to a subdirect product of $S_{59}$ and $M_{2}$.
\end{remark}

\begin{pro}\label{pro333}
The ai-semiring variety $\mathsf{V}(S_{(4, 693)})$ is defined by the identities
\begin{align}
&x \preceq xy; \label{6931}\\
&xyz \approx xzy; \label{6932}\\
&x+yx \approx xy+yx; \label{6933}\\
&xz+xt \preceq xy+zt; \label{6934}\\
&x_1x_2x_3+y_1 \approx x_1x_4x_5+y_1y_2. \label{6935}
\end{align}
\end{pro}
\begin{proof}
It is routine to check that $S_{(4, 693)}$ satisfies identities \eqref{6931}--\eqref{6935}.
It suffices to show that every inequality of $S_{(4, 693)}$ is derivable from \eqref{6931}--\eqref{6935}.
Let $\bq \preceq \bu$ be such a nontrivial inequality,
where $\bu=\bu_1+\bu_2+\cdots+\bu_n$ with $\bu_i, \bq \in X^+$ for $1 \leq i \leq n$.
One can easily verify that $S_{(4, 693)}$ is isomorphic to a subdirect product of $L_2$ and $S_{59}$
via the congruences defined by the nontrivial blocks $\{2, 3, 4\}$ and $\{1, 2\}$.
Hence both $L_2$ and $S_{59}$ satisfy $\bq \preceq \bu$,
so $h(\bu_i)=h(\bq)$ for some $\bu_i \in \bu$.
By Lemma~\ref{lem5901}, we consider the following two cases:

\textbf{Case 1.} $\ell(\bu_j) \geq 3$ for some $\bu_j \in \bu$. Then
\[
\bu \succeq \bu_j+\bu_i = \bu_j+h(\bq)s(\bu_i)
\stackrel{\eqref{6935}}{\succeq} h(\bq)s(\bu_i)s(\bq)
\stackrel{\eqref{6932}}{\succeq} \bq\bq'
\stackrel{\eqref{6931}}{\succeq} \bq.
\]
The second step uses $h(\bu_i)=h(\bq)$.
This proves the inequality $\bu \succeq \bq$.

\textbf{Case 2.} $\ell(\bu_j)\leq 2$ for all $\bu_j \in \bu$. Then $\ell(\bq)\leq 2$.
If $\ell(\bq)=1$, then $\ell(\bu_i)=2$, so
\[
\bu \succeq \bu_i = \bq\bu_i' \stackrel{\eqref{6931}}{\succeq} \bq.
\]

Now suppose that $\ell(\bq)=2$. Then $c(\bq)\subseteq c(L_2(\bu))$.
Write $\bq=xy$.
Then there exist $\bu_p,\bu_q \in L_2(\bu)$ such that $x \in c(\bu_p)$ and $y \in c(\bu_q)$.
If $h(\bu_p)=x$, then $\bu_p=xx_1$ for some $x_1\in X$, and so
\[
\bu \succeq \bu_p+\bu_q = xx_1+\bu_q \stackrel{\eqref{6934}}{\succeq} xy = \bq.
\]
If $\ell(\bu_i)=2$ and $\bu_i=xx_2$ for some $x_2\in X$, then
\[
\bu \succeq \bu_i+\bu_q = xx_2+\bu_q \stackrel{\eqref{6934}}{\succeq} xy = \bq.
\]
In the remainder we assume that $t(\bu_p)=x$ and $\ell(\bu_i)=1$.
Then $\bu_p=x_3x$ for some $x_3 \in X$ and $\bu_i=x$.
Thus
\[
\bu \succeq \bu_i+\bu_p+\bu_q = x+x_3x+\bu_q
\stackrel{\eqref{6933}}{\succeq} xx_3+\bu_q
\stackrel{\eqref{6934}}{\succeq} xy=\bq.
\]
This proves the inequality $\bu \succeq \bq$.
\end{proof}

\begin{remark}
We note that $\mathsf{V}(S_{(4, 693)})=\mathsf{V}(S_{59}, L_2)$,
since $S_{(4, 693)}$ is isomorphic to a subdirect product of $L_{2}$ and $S_{59}$.
\end{remark}

\begin{cor}
The ai-semiring $S_{(4, 603)}$ is finitely based.
\end{cor}
\begin{proof}
Since $S_{(4, 603)}$ and $S_{(4, 693)}$ have dual multiplications, it follows from
Proposition {\ref{pro333}} that $S_{(4, 603)}$ is finitely based.
\end{proof}

\begin{pro}\label{pro334}
The ai-semiring variety $\mathsf{V}(S_{(4, 705)})$ is defined by the identities
\begin{align}
&x \preceq x^{2}; \label{7051}\\
&xy \approx yx; \label{7052}\\
&x^{4} \approx x^{3}; \label{7053}\\
&xz \preceq x^{2}+yz; \label{7054}\\
&x^{2} \preceq x+xy; \label{7055}\\
&xz \preceq x+xy+zt; \label{7056}\\
&x_1x_2x_3+y_1 \approx x_1x_4x_5+y_1y_2. \label{7057}
\end{align}
\end{pro}
\begin{proof}
It is easy to check that $S_{(4, 705)}$ satisfies the identities \eqref{7051}--\eqref{7057}.
It remains to show that every inequality of $S_{(4, 705)}$ is derivable from \eqref{7051}--\eqref{7057}.
Let $\bq \preceq \bu$ be such a nontrivial inequality,
where $\bu=\bu_1+\bu_2+\cdots+\bu_n$ with $\bu_i, \bq \in X^+$ for $1 \leq i \leq n$.
By \eqref{7052}, we may further assume that $\bq, \bu_i\in X^+_c$ for $1 \leq i \leq n$.
Since $D_2$ is isomorphic to the subalgebra $\{1, 2\}$ of $S_{(4, 705)}$,
it follows that $D_2$ satisfies $\bq \preceq \bu$,
and so there exists $\bu_i \in \bu$ such that $c(\bu_i)\subseteq c(\bq)$ .
Since $S_{59}$ is isomorphic to the subalgebra $\{2, 3, 4\}$ of $S_{(4, 705)}$, we have that $S_{59}$ satisfies $\bq \preceq \bu$.
By Lemma \ref{lem5901}, we consider the following two cases.

\textbf{Case 1.} $\ell(\bu_j) \geq 3$ for some $\bu_j \in \bu$. Then
\begin{align*}
\bu
&\succeq \bu_j+\bu_i\\
&\succeq \bu_j+\bu_i\bq^3&&(\text{by}~\eqref{7057})\\
&\succeq \bu_j+\bq^3&&(\text{by}~\eqref{7052}, \eqref{7053})\\
&\succeq \bq.&&(\text{by}~ \eqref{7051})
\end{align*}
The second step follows from $\ell(\bu_j) \geq 3$ for some $\bu_j \in \bu$,
while the third step uses $c(\bu_i)\subseteq c(\bq)$.

\textbf{Case 2.} $\ell(\bu_j)\leq 2$ for all $\bu_j \in \bu$. Then $\ell(\bq)\leq 2$.
If $\ell(\bq)=1$, then $c(\bu_i)=c(\bq)$, and so $\bu_i=\bq^k$ for some $k\geq 2$.
Hence
\[
\bu\succeq \bu_i=\bq^k \stackrel{\eqref{7051}}\succeq \bq.
\]
Now suppose that $\ell(\bq)= 2$. Then $c(\bq)\subseteq c(L_2(\bu))$.
If $\bq=x^2$ for some $x\in X$, then $\bu_i=x$, and $xx_1 \in L_2(\bu)$ for some $x_1 \in X$. Thus
\[
\bu\succeq x+xx_1 \stackrel{\eqref{7055}}\succeq x^2=\bq.
\]
In the remainder we assume that $\bq=xy$, where $x$ and $y$ are distinct.
Then $xx_1$ and $yy_1\in L_2(\bu)$ for some $x_1, y_1\in X$.
If $\ell(\bu_i)=1$, then $\bu_i=x$ or $y$, and so
\[
\bu\succeq\bu_i+xx_1+yy_1\stackrel{\eqref{7056}, \eqref{7052}}\succeq xy=\bq.
\]
If $\ell(\bu_i)=2$, then $\bu_i=x^2$ or $y^2$, and so
\[
\bu\succeq\bu_i+xx_1+yy_1\stackrel{\eqref{7054}, \eqref{7052}}\succeq xy=\bq.
\]
This derives the identity $\bu \succeq \bq$.
\end{proof}

\begin{table}[ht]
\caption{The Cayley tables of $S_{19}$} \label{tb1901}
\begin{tabular}{c|ccc}
$+$      &$1$ &$2$ &$3$\\
\hline
$1$      &$1$ &$1$ &$3$\\
$2$      &$1$ &$2$ &$3$\\
$3$      &$3$ &$3$ &$3$\\
\end{tabular}\qquad
\begin{tabular}{c|ccc}
$\cdot$      &$1$ &$2$ &$3$\\
\hline
$1$      &$1$ &$1$ &$1$\\
$2$      &$1$ &$1$ &$1$\\
$3$      &$1$ &$1$ &$3$\\
\end{tabular}
\end{table}

\begin{remark}
We note that $\mathsf{V}(S_{(4, 705)})=\mathsf{V}(S_{59}, D_2)$.
Indeed, $S_{(4, 705)}$ is isomorphic to a subdirect product of $S_{59}$ and $S_{19}$
via the congruences defined by the nontrivial blocks $\{1, 2\}$ and $\{2, 3\}$;
the Cayley tables of $S_{19}$ are given in Table~\ref{tb1901}.
Hence $\mathsf V(S_{(4, 705)}) = \mathsf V(S_{19}, S_{59})$.
Moreover, $D_2$ and $T_2$ are isomorphic to the subalgebras $\{1, 3\}$ and $\{1, 2\}$ of $S_{19}$.
$S_{19}$ satisfies the identities
\[
x \preceq x^2, \quad xy\approx yx, \quad x^2y\approx xy, \quad xt \preceq x+yz,
\]
which form an equational basis of $\mathsf V(D_2, T_2)$ (see \cite{ryy}).
It follows that $\mathsf V(S_{19})=\mathsf V(D_2, T_2)$, and so
$\mathsf V(S_{(4, 705)}) = \mathsf V(S_{59}, D_2, T_2)$.
Since $T_2$ can be embedded into $S_{59}$, we therefore obtain
\[
\mathsf{V}(S_{(4, 705)})=\mathsf{V}(S_{59}, D_2).
\]
\end{remark}

\begin{pro}\label{pro335}
The ai-semiring variety $\mathsf{V}(S_{(4, 813)})$ is defined by the identities
\begin{align}
&xy \preceq xyz; \label{8131}\\
&xy \approx x^{2}+y^{2}; \label{8132}\\
&x_1x_2x_3 \approx y_1y_2y_3. \label{8133}
\end{align}
\end{pro}
\begin{proof}
It is routine to check that $S_{(4, 813)}$ satisfies the identities \eqref{8131}--\eqref{8133}.
It suffices to show that every inequality of $S_{(4, 813)}$ is derivable from \eqref{8131}--\eqref{8133}.
Let $\bq \preceq \bu$ be such a nontrivial inequality,
where $\bu=\bu_1+\bu_2+\cdots+\bu_n$ with $\bu_i, \bq \in X^+$ for $1 \leq i \leq n$.
It is easy to see that
$S_{(4, 813)}$ is isomorphic to a subdirect product of $N_{2}$ and $S_{59}$
via the congruences defined by the nontrivial blocks $\{2, 3, 4\}$ and $\{1, 2\}$.
Hence both $N_{2}$ and $S_{59}$ satisfy $\bq \preceq \bu$,
and so $\ell(\bq)\geq 2$.
By Lemma \ref{lem5901}, we consider the following two cases.

\textbf{Case 1.} $\ell(\bu_i) \geq 3$ for some $\bu_i \in \bu$, then
\[
\bu\succeq\bu_i\stackrel{\eqref{8133}}\approx \bq\bp\stackrel{\eqref{8131}}\succeq \bq .
\]

\textbf{Case 2.} $\ell(\bu_i)\leq 2$ for all $\bu_i \in \bu$.
Then $\ell(\bq)= 2$, and so $c(\bq)\subseteq c(L_2(\bu))$.
We may write that $\bq=xy$.
Then there exist $\bu_k, \bu_\ell \in L_2(\bu)$ such that $x\in c(\bu_k)$ and $y\in c(\bu_\ell)$.
Now we have
\[
\bu\succeq \bu_k + \bu_\ell \stackrel{\eqref{8132}}\succeq x^2+y^2 \stackrel{\eqref{8132}}\approx xy = \bq.
\]
This derives the inequality $\bu \succeq \bq$.
\end{proof}

\begin{remark}
We note that $\mathsf{V}(S_{(4, 813)})=\mathsf{V}(S_{59}, N_2)$,
since $S_{(4, 813)}$ is isomorphic to a subdirect product of $S_{59}$ and $N_{2}$.
\end{remark}

\begin{pro}\label{pro62701}
The ai-semiring variety $\mathsf{V}(S_{(4, 627)})$ is defined by the identities
\begin{align}
&x \preceq x^2; \label{62701}\\
&xy \approx yx; \label{62702}\\
&xy \preceq x^2+y^2; \label{62703}\\
&x_1x_2x_3x_4 \preceq x_1x_2x_3. \label{62704}
\end{align}
\end{pro}
\begin{proof}
It is easy to verify that $S_{(4, 627)}$ satisfies the identities \eqref{62701}--\eqref{62704}.
It remains to show that every inequality of $S_{(4, 627)}$ is derivable from \eqref{62701}--\eqref{62704}.
Let $\bq \preceq \bu$ be such a nontrivial inequality,
where $\bu=\bu_1+\bu_2+\cdots+\bu_n$ with $\bu_i, \bq\in X^+$ for $1 \leq i \leq n$.
Observe that $S_{59}^0$ is isomorphic to $S_{(4,627)}$.
Then $S_{59}^0$ satisfies $\bq\preceq\bu$, and so $S_{59}$ satisfies $\bq\preceq D_\bq(\bu)$.
By Lemma~\ref{lem5901}, we consider the following two cases.

\textbf{Case 1.} $\ell(\bu_i) \geq 3$ for some $\bu_i \in D_\bq(\bu)$. Then
\[
\bu\succeq D_\bq(\bu) \succeq \bu_i \stackrel{\eqref{62704}} \succeq \bu_i \bq \stackrel{\eqref{62701}, \eqref{62702}} \succeq \bq.
\]

\textbf{Case 2.} $\ell(\bu_i) \leq 2$ for every $\bu_i \in D_\bq(\bu)$. Then $\ell(\bq) \leq 2$.
If $\ell(\bq)=1$, then $c(\bq)\subseteq c(D_\bq(\bu))$, and so $c(\bq)=(D_\bq(\bu))$.
Take a word $\bu_k$ in $D_\bq(\bu)$. Then $\bu_k=\bq^r$ for some $r\geq 2$, and so
\[
\bu\succeq D_\bq(\bu) \succeq \bu_k =\bq^r \stackrel{\eqref{62701}} \succeq \bq.
\]

Now suppose that $\ell(\bq) = 2$.
Then $c(\bq)\subseteq c(L_2(D_\bq(\bu)))$, and so $c(\bq)= c(L_2(D_\bq(\bu)))$.
Write $\bq=xy$. Since $\bq\preceq \bu$ is nontrivial, it follows that $x\neq y$.
This implies that there exist $\bu_i,\bu_j\in L_2(D_\bq(\bu))$ such that $\bu_i=x^2$ and $\bu_j=y^2$.
Now we have
\[
D_\bq(\bu) \succeq \bu_i+\bu_j =x^2+ y^2 \stackrel{\eqref{62703}} \succeq xy=\bq.
\]
This proves the inequality $\bu \succeq \bq$.
\end{proof}

\section{Equational bases for 4-element ai-semirings related to $S_{60}$}
In this section, we provide equational bases for some 4-element ai-semirings that relate to $S_{60}$,
whose Cayley tables are given in Table~\ref{tb6001}.
\begin{table}[ht]
\caption{The Cayley tables of $S_{60}$} \label{tb6001}
\begin{tabular}{c|ccc}
$+$      &$1$ &$2$ &$3$\\
\hline
$1$      &$1$ &$1$ &$3$\\
$2$      &$1$ &$2$ &$3$\\
$3$      &$3$ &$3$ &$3$\\
\end{tabular}\qquad
\begin{tabular}{c|ccc}
$\cdot$      &$1$ &$2$ &$3$\\
\hline
$1$      &$3$ &$3$ &$3$\\
$2$      &$3$ &$2$ &$3$\\
$3$      &$3$ &$3$ &$3$\\
\end{tabular}
\end{table}

\begin{lem}\label{lem6001}
Let $\bq\preceq \bu$ be a nontrivial inequality such that
$\bu=\bu_1+\bu_2+\cdots+\bu_n$ with $\bu_i, \bq \in X^+$ for $1\leq i \leq n$.
Suppose that $\bq\preceq \bu$ is satisfied by $S_{60}$.
Then $\bu$ and $\bq$ satisfy the following conditions:
\begin{itemize}
\item[$(1)$] $L_{\geq 2}(\bu) \neq \emptyset$;

\item[$(2)$] If $\ell(\bq)=1$, then $c(\bq)\subseteq c(\bu)$;

\item[$(3)$] If $\ell(\bq)\geq 2$, then $c(\bq)\subseteq c(L_{\geq2}(\bu))$.
\end{itemize}
\end{lem}

\begin{pro}\label{pro62901}
The ai-semiring variety $\mathsf{V}(S_{(4, 629)})$ is defined by the identities
\begin{align}
&x^2y\approx xy;\label{6291}\\
&x \preceq x^2;\label{6292}\\
&xy \approx yx;\label{6293}\\
&xyzt\preceq xy+zt.\label{6294}
\end{align}
\end{pro}
\begin{proof}
It is readily verified that $S_{(4,629)}$ satisfies the identities \eqref{6291}--\eqref{6294}.
It suffices to prove that every inequality of $S_{(4,629)}$ is derivable from \eqref{6291}--\eqref{6294}.
Let $\bq \preceq \bu$ be such a nontrivial inequality,
where $\bu = \bu_1 + \bu_2 + \cdots + \bu_n$ with $\bu_i,\bq \in X^+$ for $1 \leq i \leq n$.
Since $S_{60}^0$ is isomorphic to $S_{(4,629)}$,
it follows that $S_{60}^0$ satisfies $\bq \preceq \bu$, and so $S_{60}$ satisfies $\bq \preceq D_\bq(\bu)$.

If $\ell(\bq) = 1$, then for any $\bu_i \in D_\bq(\bu)$, $c(\bu_i) \subseteq c(\bq)$.
Thus $\bu_i = \bq^k$ for some $k \geq 2$, and so
\[
\bu \succeq \bu_i = \bq^k \stackrel{\eqref{6292}} \succeq \bq.
\]

If $\ell(\bq) \geq 2$, then by Lemma~\ref{lem6001}, $L_{\geq 2}(D_\bq(\bu))) \neq \emptyset$ and
$c(\bq) \subseteq c(L_{\geq 2}(D_\bq(\bu)))$, and so $c(\bq) = c(L_{\geq 2}(D_\bq(\bu)))$.
We may assume that $L_{\geq 2}(D_\bq(\bu))=\bu_1+\bu_2+\cdots+\bu_r$. Then

\begin{align*}
&\bu\succeq \bu_1+\bu_2+\cdots+\bu_r
\stackrel{\eqref{6294}}\succeq  \bu_1\bu_2\cdots\bu_r \stackrel{\eqref{6291}, \eqref{6293}}\approx \bq.
\end{align*}
This proves the inequality $\bu \succeq \bq$.
\end{proof}

\begin{table}[ht]
\caption{The Cayley tables of $S_{31}$} \label{tb3101}
\begin{tabular}{c|ccc}
$+$      &$1$ &$2$ &$3$\\
\hline
$1$      &$1$ &$1$ &$3$\\
$2$      &$1$ &$2$ &$3$\\
$3$      &$3$ &$3$ &$3$\\
\end{tabular}\qquad
\begin{tabular}{c|ccc}
$\cdot$      &$1$ &$2$ &$3$\\
\hline
$1$      &$1$ &$1$ &$3$\\
$2$      &$2$ &$2$ &$3$\\
$3$      &$3$ &$3$ &$3$\\
\end{tabular}
\end{table}

\begin{pro}
The ai-semirings $S_{(4, 500)}$, $S_{(4, 502)}$, $S_{(4, 604)}$, and $S_{(4, 694)}$ are all finitely based.
\end{pro}
\begin{proof}
It is easy to check that $S_{(4, 502)}$ is isomorphic to a subdirect product of $S_{31}$ and $S_{60}$
via the congruences defined by the nontrivial blocks $\{1, 2\}$ and $\{3, 4\}$;
the Cayley tables of $S_{31}$ are given in Table~\ref{tb3101}.
Thus $\mathsf{V}(S_{(4, 502)}) = \mathsf{V}(S_{31}, S_{60})$.
Moreover, $L_2$ and $M_2$ are isomorphic to the subalgebras $\{1, 2\}$ and $\{1, 3\}$ of $S_{31}$, respectively.
$S_{31}$ satisfies the identities
\[
x\preceq xy, \quad xyz \approx xzy, \quad xy \preceq x+y,
\]
which form an equational basis of $\mathsf{V}(L_2, M_2)$ (see \cite{ryy}).
Thus $\mathsf{V}(S_{31})=\mathsf{V}(L_2, M_2)$, and so $\mathsf{V}(S_{(4, 502)})=\mathsf{V}(S_{60}, L_2, M_2)$.
Since $M_2$ embeds into $S_{60}$ (as the subalgebra $\{2, 3\}$), we obtain
\[
\mathsf{V}(S_{(4, 502)})=\mathsf{V}(S_{60}, L_2).
\]
By \cite[Remark 6.3]{yrzs}, $\mathsf{V}(S_{60}, L_2)=\mathsf{V}(S_{(4, 459)})$,
so $\mathsf{V}(S_{(4, 502)})=\mathsf{V}(S_{(4, 459)})$.
Consequently, $S_{(4, 502)}$ and $S_{(4, 459)}$ have the same finite basis property.
By \cite[Proposition 6.2]{yrzs}, $S_{(4, 459)}$ is finitely based.
Therefore, $S_{(4, 502)}$ is finitely based.
Since $S_{(4,500)}$ and $S_{(4,502)}$ have dual multiplications,
$S_{(4,500)}$ is also finitely based.

It is easy to see that $S_{(4, 694)}$ is isomorphic to a subdirect product of $S_{60}$ and $L_2$
via the congruences defined by the nontrivial blocks $\{1, 2\}$ and $\{2, 3, 4\}$.
Thus $\mathsf{V}(S_{(4, 694)})=\mathsf{V}(S_{60}, L_2)$,
and so $\mathsf{V}(S_{(4, 694)})=\mathsf{V}(S_{(4,502)})$.
Consequently, $S_{(4, 694)}$ is finitely based.
Finally, $S_{(4, 604)}$ and $S_{(4, 694)}$ have dual multiplications.
Therefore, $S_{(4, 604)}$ is finitely based.
\end{proof}

\begin{table}[ht]
\caption{The Cayley tables of $S_{43}$} \label{tb4301}
\begin{tabular}{c|ccc}
$+$      &$1$ &$2$ &$3$\\
\hline
$1$      &$1$ &$1$ &$3$\\
$2$      &$1$ &$2$ &$3$\\
$3$      &$3$ &$3$ &$3$\\
\end{tabular}\qquad
\begin{tabular}{c|ccc}
$\cdot$      &$1$ &$2$ &$3$\\
\hline
$1$      &$1$ &$2$ &$3$\\
$2$      &$2$ &$2$ &$3$\\
$3$      &$3$ &$3$ &$3$\\
\end{tabular}
\end{table}

\begin{pro}
The ai-semirings $S_{(4, 503)}$ and $S_{(4, 706)}$ are both finitely based.
\end{pro}
\begin{proof}
It is routine to check that $S_{(4, 503)}$ is isomorphic to a subdirect product of $S_{43}$ and $S_{60}$
via the congruences defined by the nontrivial blocks $\{1, 2\}$ and $\{3, 4\}$;
the Cayley tables of $S_{43}$ are given in Table~\ref{tb4301}.
Thus $\mathsf{V}(S_{(4, 503)}) = \mathsf{V}(S_{43}, S_{60})$.
Moreover, $D_2$ and $M_2$ are isomorphic to the subalgebras $\{1, 2\}$ and $\{1, 3\}$ of $S_{43}$, respectively.
$S_{43}$ satisfies the identities
\[
x^2\approx x,\quad xy\approx yx, \quad xy\preceq x+y, \quad xz \preceq x+yz,
\]
which form an equational basis of $\mathsf{V}(D_2, M_2)$ (see \cite{ryy}).
Thus $\mathsf{V}(S_{43})=\mathsf{V}(D_2, M_2)$, so $\mathsf{V}(S_{(4, 503)})=\mathsf{V}(S_{60}, D_2, M_2)$.
Since $M_2$ embeds into $S_{60}$, we obtain
\[
\mathsf{V}(S_{(4, 503)})=\mathsf{V}(S_{60}, D_2).
\]
By \cite[Remark 6.6]{yrzs}, $\mathsf{V}(S_{60}, D_2)=\mathsf{V}(S_{(4, 467)})$,
so $\mathsf{V}(S_{(4, 503)})=\mathsf{V}(S_{(4, 467)})$.
Consequently, $S_{(4, 503)}$ and $S_{(4, 467)}$ have the same finite basis property.
By \cite[Proposition 6.5]{yrzs}, $S_{(4, 467)}$ is finitely based.
Therefore, $S_{(4, 503)}$ is finitely based.

Now, $S_{(4, 706)}$ is isomorphic to a subdirect product of $S_{60}$ and $D_2$
via the congruences defined by the nontrivial blocks $\{1, 2\}$ and $\{2, 3, 4\}$.
Thus $\mathsf{V}(S_{(4, 706)})=\mathsf{V}(S_{60}, D_2)$, so $\mathsf{V}(S_{(4, 706)})=\mathsf{V}(S_{(4,503)})$.
Therefore, $S_{(4, 706)}$ is also finitely based.
\end{proof}

\begin{table}[ht]
\caption{The Cayley tables of $S_{52}$} \label{tb5201}
\begin{tabular}{c|ccc}
$+$      &$1$ &$2$ &$3$\\
\hline
$1$      &$1$ &$1$ &$3$\\
$2$      &$1$ &$2$ &$3$\\
$3$      &$3$ &$3$ &$3$\\
\end{tabular}\qquad
\begin{tabular}{c|ccc}
$\cdot$      &$1$ &$2$ &$3$\\
\hline
$1$      &$2$ &$2$ &$3$\\
$2$      &$2$ &$2$ &$3$\\
$3$      &$3$ &$3$ &$3$\\
\end{tabular}
\end{table}

\begin{pro}
The ai-semirings $S_{(4, 505)}$ and $S_{(4, 814)}$ are both finitely based.
\end{pro}
\begin{proof}
It is easy to verify that $S_{(4, 505)}$ is isomorphic to a subdirect product of $S_{52}$ and $S_{60}$
via the congruences defined by the nontrivial blocks $\{1, 2\}$ and $\{3, 4\}$;
the Cayley tables of $S_{52}$ are given in Table~\ref{tb5201}.
Thus $\mathsf{V}(S_{(4, 505)}) = \mathsf{V}(S_{52}, S_{60})$.
Moreover, $M_2$ and $N_2$ are isomorphic to the subalgebras $\{2, 3\}$ and $\{1, 2\}$ of $S_{52}$.
$S_{52}$ satisfies the identities
\[
xy\approx yx, \quad xy \preceq xyz, \quad xy\preceq x+y,
\]
which form an equational basis of $\mathsf{V}(M_2, N_2)$ (see \cite{ryy}).
Thus $\mathsf{V}(S_{52})=\mathsf{V}(M_2, N_2)$, so $\mathsf{V}(S_{(4, 505)})=\mathsf{V}(S_{60}, M_2, N_2)$.
Since $M_2$ embeds into $S_{60}$, we obtain
\[
\mathsf{V}(S_{(4, 505)})=\mathsf{V}(S_{60}, N_2).
\]
By \cite[Remark 6.9]{yrzs}, $\mathsf{V}(S_{60}, N_2)=\mathsf{V}(S_{(4, 479)})$,
so $\mathsf{V}(S_{(4, 505)})=\mathsf{V}(S_{(4, 479)})$.
Consequently, $S_{(4, 505)}$ and $S_{(4, 479)}$ have the same finite basis property.
By \cite[Proposition 6.8]{yrzs}, $S_{(4, 479)}$ is finitely based.
Therefore, $S_{(4, 505)}$ is finitely based.

Finally, $S_{(4, 814)}$ is isomorphic to a subdirect product of $S_{60}$ and $N_2$
via the congruences defined by the nontrivial blocks $\{1, 2\}$ and $\{2, 3, 4\}$.
Thus $\mathsf{V}(S_{(4, 814)})=\mathsf{V}(S_{60}, N_2)=\mathsf{V}(S_{(4, 505)})$.
Therefore, $S_{(4, 814)}$ is also finitely based.
\end{proof}

\section{Equational bases for varieties generated by two $3$-element ai-semirings}
In this section, we present equational bases for some varieties generated by two $3$-element ai-semirings.
We begin by providing solutions of the equational problem for three $3$-element ai-semirings.

\begin{lem}\label{pro422}
Let $\bq\preceq \bu$ be an inequality,
where $\bu = \bu_1 + \bu_2 + \cdots + \bu_n $ with $\bu_i, \bq\in X^+$ for $1 \leq i \leq n$.
Then $\bq \preceq \bu$ is satisfied by $M_2^{0}$ if and only if $D_\bq(\bu)\neq\emptyset$, and $c(\bq)=c(D_\bq(\bu))$.
\end{lem}
\begin{proof}
This follows immediately from Lemmas~\ref{lem001} and~\ref{nlemma1}.
\end{proof}

\begin{lem}\label{pro423}
Let $\bq\preceq \bu$ be an ai-semiring inequality,
where $\bu = \bu_1 + \bu_2 + \cdots + \bu_n $ with $\bu_i, \bq\in X^+$ for $1\leq i \leq n$.
Then $\bq \preceq \bu$ is satisfied by $L_2^{0}$ if and only if $D_\bq(\bu)\neq\emptyset$, and $h(\bq)=h(\bu_i)$ for some $\bu_i\in D_\bq(\bu)$.
\end{lem}
\begin{proof}
This follows immediately from Lemmas~\ref{lem001} and~\ref{nlemma1}.
\end{proof}

\begin{lem}\label{cor424}
Let $\bq\preceq \bu$ be an ai-semiring inequality,
where $\bu = \bu_1 + \bu_2 + \cdots + \bu_n $ with $\bu_i, \bq\in X^+$ for $1\leq i \leq n$.
Then $\bq \preceq \bu$ is satisfied by $R_2^{0}$ if and only if $D_\bq(\bu)\neq\emptyset$, and $t(\bq)=t(\bu_i)$ for some $\bu_i\in D_\bq(\bu)$.
\end{lem}
\begin{proof}
This follows immediately from Lemmas~\ref{lem001} and~\ref{nlemma1}.
\end{proof}

\begin{pro}\label{pro425}
The ai-semiring variety $\mathsf{V}(S_{(4, 677)})$ is defined by the identities
\begin{align}
&x^{3} \approx x^{2}; \label{6771}\\
&xy \approx yx; \label{6772}\\
&xy \preceq x+y; \label{6773}\\
&xy \preceq x+y^{2}; \label{6774}\\
&xy \preceq y+xy^{2}; \label{6775}\\
&x^{2}y+xy^{2} \approx xy. \label{6776}
\end{align}
\end{pro}
\begin{proof}
It is routine to check that $S_{(4, 677)}$ satisfies the identities \eqref{6771}--\eqref{6776}.
It remains to show that every inequality of $S_{(4, 677)}$ is derivable from \eqref{6771}--\eqref{6776}.
Let $\bq \preceq \bu$ be such a nontrivial inequality,
where $\bu=\bu_1+\bu_2+\cdots+\bu_n$ with $\bu_i, \bq \in X^+$ for $1 \leq i \leq n$.
Observe that $S_{(4,677)}$ is isomorphic to a subdirect product of $S_{44}$ and $M_2^0$
via the congruences determined by the nontrivial blocks $\{1,2\}$ and $\{3,4\}$;
hence both $S_{44}$ and $M_2^0$ satisfy $\bq\preceq\bu$.
By Lemmas~\ref{pro422} and~\ref{lem4401},
we have that $\ell(\mathbf{q}) \geq 2$, $c(\bq)=c(D_\bq(\bu))$.
Moreover, if $M_1(\mathbf{q}) \neq \emptyset$, then for every $x \in M_1(\mathbf{q})$,
there exists $\mathbf{u}_i \in D_{\mathbf{q}}(\mathbf{u})$ such that $m(x, \mathbf{u}_i) \leq 1$.
We may assume that $D_\bq(\bu)=\bu_1+\bu_2+\cdots+\bu_k$. Then
\[
\bu\succeq \bu_1+\bu_2+\cdots+\bu_k\stackrel{\eqref{6773}}\succeq \bu_1^{2}\bu_2^{2}\cdots \bu_k^{2}\stackrel{\eqref{6771},\eqref{6772}}\approx \bq^2.
\]
The last step relies on $c(\bq)=c(D_\bq(\bu))$.
It remains to show that $\bq^2 \succeq \bq$ is derivable from \eqref{6771}--\eqref{6776}.

\textbf{Case 1.} $M_1(\bq)$ is empty. Then $m(x, \bq)\geq 2$ for all $x \in c(\bq)$.
By the identities \eqref{6771} and \eqref{6772}, one derives
\[
\bq^2 \approx \bq.
\]

\textbf{Case 2.} $M_1(\bq)$ is nonempty.
By \eqref{6776} we may assume that $|M_1(\bq)|=1$; that is, $M_1(\bq)=\{y\}$ for some $y \in X$.
Then there exists $\bu_i \in D_\bq(\bu)$ such that $m(y, \bu_i)\leq 1$.

If $m(y, \bu_i) = 0$, then
\[
\bu \succeq \bu_i+\bq^{2} \stackrel{\eqref{6774}}\succeq \bu_i \bq \stackrel{\eqref{6772},\eqref{6771}}\approx \bq.
\]
This derives $\bu \succeq \bq$.

If $m(y, \bu_i) = 1$, then
\[
\bu\succeq \bu_i+\bq^{2}\stackrel{\eqref{6772},\eqref{6771}}\approx \bu_i+\bp^{2}\bu_i^{2}\stackrel{\eqref{6775}}\succeq \bp^{2}\bu_i
\stackrel{\eqref{6772},\eqref{6771}}\approx \bq,
\]
where $y\notin c(\bp)$.
This derives $\bu \succeq \bq$.
\end{proof}


\begin{pro}\label{pro426}
The ai-semiring variety $\mathsf{V}(S_{(4, 688)})$ is defined by the identities
\begin{align}
&x^{2}y \approx xy; \label{6881}\\
&xyz \approx yxz; \label{6882}\\
&xy \preceq x+y; \label{6883}\\
&xy \preceq x+y^{2}; \label{6884}\\
&x^{2}y^{2} \approx y^{2}x^{2}; \label{6885}\\
&xy \preceq y+xy^{2}. \label{6886}
\end{align}
\end{pro}
\begin{proof}
It is easy to verify that $S_{(4, 688)}$ satisfies the identities \eqref{6881}--\eqref{6886}.
It suffices to show that every inequality of $S_{(4, 688)}$ is derivable from \eqref{6881}--\eqref{6886}.
Let $\bq \preceq \bu$ be such a nontrivial inequality,
where $\bu=\bu_1+\bu_2+\cdots+\bu_n$ with $\bu_i, \bq \in X^+$ for $1 \leq i \leq n$.

Observe that $S_{(4,688)}$ is isomorphic to a subdirect product of $S_{46}$ and $M_2^0$
via the congruences determined by the nontrivial blocks $\{1,2\}$ and $\{3,4\}$;
hence both $S_{46}$ and $M_2^0$ satisfy $\bq \preceq \bu$.
By Lemmas~\ref{lem4601} and~\ref{pro422}, we have $\ell(\bq)\geq 2$, $c(D_\bq(\bu))=c(\bq)$,
and if $m(t(\bq),\bq)=1$, then there exists $\bu_i\in D_\bq(\bu)$ such that $t(\bq)\notin c(p(\bu_i))$.
We may write $D_\bq(\bu)=\bu_1+\bu_2+\cdots+\bu_k$. Then
\[
\bu\succeq \bu_1+\bu_2+\cdots+\bu_k\stackrel{(\ref{6883})}\succeq \bu_1^{2}\bu_2^{2}\cdots \bu_k^{2}\stackrel{\eqref{6881},\eqref{6882},\eqref{6885}}\approx \bq^{2}.
\]
The third step relies on $c(D_\bq(\bu))=c(\bq)$.
This derives the inequality
\begin{equation}\label{26082901}
\bu \succeq \bq^{2}.
\end{equation}

\textbf{Case 1.} $m(t(\bq),\bq)\geq2$. By the identities \eqref{6881} and \eqref{6882}, one derives
\[
\bq^{2}\approx \bq.
\]
Combining this identity with \eqref{26082901}, one derives the inequality $\bu \succeq \bq$.

\textbf{Case 2.} $m(t(\bq),\bq)=1$.
Then $t(\bq) \notin c(p(\bq))$ and
there exists $\bu_i\in D_\bq(\bu)$ such that $t(\bq)\notin c(p(\bu_i))$;
in particular, $c(\bu_i) \subseteq c(\bq)$.
Thus $m(t(\bq), \bu_i)\leq 1$.

If $m(t(\bq), \bu_i)=0$, then $t(\bq) \notin c(\bu_i)$, and so $c(\bu_i) \subseteq c(p(\bq))$.
Now we have
\[
\bu \succeq \bu_i+\bu \stackrel{(\ref{26082901})}\succeq \bu_i+\bq^{2}
\stackrel{(\ref{6884})}\succeq \bu_i\bq \stackrel{\eqref{6881}, \eqref{6882}}\approx \bq.
\]
The last step uses $c(\bu_i) \subseteq c(p(\bq))$.
This derives $\bu \succeq \bq$.

If $m(t(\bq),\bu_i)=1$, then $t(\bq)=t(\bu_i)$, and so
\[
\begin{aligned}
\bu
&\succeq \bu_i+\bu
 \stackrel{(\ref{26082901})}{\succeq} \bu_i+\bq^{2}
 = p(\bu_i)t(\bq)+\bq^{2} \\
&\stackrel{\eqref{6881}, \eqref{6882}}\approx p(\bu_i)t(\bq)+p(\bq)^2(p(\bu_i)t(\bq))^2
 \stackrel{(\ref{6886})}\succeq p(\bq)^2p(\bu_i)t(\bq)
 \stackrel{\eqref{6881}, \eqref{6882}}\approx \bq.
\end{aligned}
\]
This derives $\bu \succeq \bq$.
\end{proof}


\begin{cor}
The ai-semiring $S_{(4, 678)}$ is finitely based.
\end{cor}
\begin{proof}
It is easy to see that $S_{(4, 678)}$ and $S_{(4, 688)}$ have dual multiplications,
and so they have the same finite basis property.
By Proposition~\ref{pro426}, $S_{(4, 688)}$ is finitely based.
Thus $S_{(4, 678)}$ is also finitely based.
\end{proof}

\begin{pro}\label{pro427}
The ai-semiring variety $\mathsf{V}(S_{(4, 756)})$ is defined by the identities
\begin{align}
&xy \preceq x; \label{7561}\\
&x^{3} \approx x^{2}; \label{7562}\\
&xyz \approx xzy; \label{7563}\\
&yx \preceq x+y^{2}; \label{7564}\\
&xy \preceq x+x^{2}y^{2}; \label{7565}\\
&xy \preceq y+x^{2}y^{2}; \label{7566}\\
&x^{2}y+xy^{2} \approx xy. \label{7567}
\end{align}
\end{pro}
\begin{proof}
It is easy to check that $S_{(4, 756)}$ satisfies the identities \eqref{7561}--\eqref{7567}.
In the remainder it is enough to prove that every ai-semiring identity of $S_{(4, 756)}$
can be derived by \eqref{7561}--\eqref{7567} and the identities defining $\mathbf{AI}$.
Let $\bq \preceq \bu$ be such an inequality,
where $\bu=\bu_1+\bu_2+\cdots+\bu_n$ and $\bu_i, \bq \in X^+$, $1 \leq i \leq n$.

It is a routine matter to verify that $S_{(4, 756)}$ is isomorphic to
a subdirect product of $S_{44}$ and $S_{36}$ via the congruences defined by the nontrivial blocks $\{1, 2\}$ and $\{3, 4\}$.
So both $S_{44}$ and $S_{36}$ satisfy $\bq \preceq \bu$.
By Lemmas~\ref{pro423} and~\ref{lem4401},
we have $D_\bq(\bu)\neq\emptyset$ and $H_\bq(\bu)\neq\emptyset$;
that is, there exists $\bu_i\in D_\bq(\bu)$ such that $h(\bu_i)=h(\bq)$ and $c(\bu_i)\subseteq c(\bq)$.
 Moreover, it suffices to consider one of the following cases.

If $\ell(\bq)=1$, then $\bu\succeq \bq$ is trivial. Thus if $\ell(\bq)\geq 2,D_\bq(\bu)\neq\emptyset$, we have

\textbf{Case 1.} $M_1(\bq)$ is empty. For all of $x\in \bq, m(x,\bq)\geq 2$. We have

\[
\bu\succeq \bu_i\stackrel{\eqref{7561}}\succeq \bu_i\bq\stackrel{\eqref{7562}}\approx \bq.
\]
This derives $\bu \succeq \bq$.

\textbf{Case 2.} $M_1(\bq)$ is nonempty.
Assume that
\[
c(\bq)=\{x_1, \ldots, x_m, y_1, \ldots, y_\ell\},
\]
where $m(x_i, \bq )\geq 2$, $m(y_j, \bq )=1$, $0 \leq i \leq m$, $1 \leq j \leq \ell$. We have

\[
\bu\succeq \bu_i\stackrel{\eqref{7561}}\succeq \bu_i\bq\stackrel{\eqref{7562},\eqref{7563}}\approx \bq^{2}.
\]

\textbf{Subcase 2.1.} $m(h(\bq), \bq) \geq 2$. Then
\begin{align*}
\bq
&\approx h(\bq)^{2}x_1^{2}x_2^{2}\ldots x_m^{2}y_1y_2\ldots y_\ell &&(\text{by}~ \eqref{7562}, \eqref{7563})\\
&\approx{\sum\limits_{1 \leq  j \leq \ell} h(\bq)^2x_1^2x_2^2\cdots x_m^2y_1^2y_2^2\cdots y_{j-1}^2y_{j+1}^2\cdots y_\ell^2y_j}.
&&(\text{by}~\eqref{7563}, \eqref{7567} )
\end{align*}

Let $\bq_j' = h(\bq)^{2}x_1^{2} x_2^{2} \cdots x_m^{2} y_\ell^{2} \cdots y_{j-1}^{2} y_{j+1}^{2} \cdots y_\ell^{2}$, then $\bq_j = \bq_j' y_j$, and $\bq_j^{2} \stackrel{\eqref{7562},\eqref{7563}}\approx \bq^{2}, h(\bq_j)=h(\bq)$. Since $\bq\approx \sum\limits_{1 \leq  j \leq \ell} \bq_j$, then it suffices to derive $\bq_j$.  It follows from Lemma \ref{lem4401} that for any $y_j$, there exists $\bu_j \in D_\bq(\bu)$ such that $m(y_j, \bu_j) \leq 1$, then

If $m(y_j, \bu_j) = 0$, then
\[
\bu\succeq \bu_j+\bq_j^{2}\stackrel{\eqref{7564}}\succeq \bq_j\bu_j\stackrel{\eqref{7562},\eqref{7563}}\approx \bq_j.
\]
This derives $\bu \succeq \bq$.

If $m(y_j, \bu_j) = 1$, we have
\[
\bu\succeq \bu_j+\bq_j^{2}\approx \bu_j+\bq_j^{'2}\bu_j^{2}\stackrel{\eqref{7566}}\succeq \bq_j.
\]

 This derives $\bu \succeq \bq$.

\textbf{Subcase 2.2.} $m(h(\bq), \bq)=1$. Then
\begin{align*}
\bq
&\approx h(\bq)x_1^{2}x_2^{2}\ldots x_m^{2}y_1y_2\ldots y_\ell\\
&\approx h(\bq)x_1^{2}x_2^{2}\ldots x_m^{2}y_1^{2}y_2^{2}\ldots y_\ell^{2}+\sum\limits_{1 \leq  j \leq \ell}  h(\bq)^{2}x_1^{2}x_2^{2}\cdots x_m^{2}y_1^{2}y_2^{2}\cdots y_{j-1}^{2}y_{j+1}^{2}\cdots y_\ell^{2}y_j.\\
\end{align*}

Let $\bq_H'=x_1^{2}x_2^{2}\cdots x_m^{2}y_1^{2}y_2^{2}\cdots y_\ell^{2}$,
$\bq_H = h(\bq)\bq_H'$,
then $h(\bq_H')\neq h(\bq_H)=h(\bq)$. Thus

\begin{align*}
\bu
&\succeq \bu_i+\bq_H^{2}\\
&\approx h(\bq)\bu_i'+h(\bq)^{2}\bu_i'^{2}\bq_H''^{2}\\
&\succeq h(\bq)\bu_i'\bq_H''&&(\text{by}~\eqref{7565})\\
&\approx \bq_H.
\end{align*}
where $\bu_i=h(\bq)\bu_i', \bq_H^{2}=h(\bq)^{2}\bu_i'^{2}\bq_H''^{2}$.
This derives the identity $\bu \succeq \bq$.
\end{proof}

%

\begin{cor}
The ai-semiring $S_{(4, 681)}$ is finitely based.
\end{cor}
\begin{proof}
Since $S_{(4, 681)}$ and $S_{(4, 756)}$ have dual multiplications, it follows from
Proposition {\ref{pro427}} that $S_{(4, 681)}$ is finitely based.
\end{proof}

\begin{pro}\label{pro428}
$\mathsf{V}(S_{(4, 780)})$ is the ai-semiring variety defined by the identities
\begin{align}
&xy \preceq x; \label{7801}\\
&x^{2}y \approx xy; \label{7802}\\
&xyzt \approx xzyt; \label{7803}\\
&xy \preceq y+x^{2}y^{2}; \label{7804}\\
&y+xyz^{2} \approx y+xyz. \label{7805}
\end{align}
\end{pro}
\begin{proof}
It is easy to check that $S_{(4, 780)}$ satisfies the identities \eqref{7801}--\eqref{7805}.
In the remainder it is enough to prove that every ai-semiring identity of $S_{(4, 780)}$
can be derived by \eqref{7801}--\eqref{7805}.
Let $\bq \preceq \bu$ be such an inequality,
where $\bu=\bu_1+\bu_2+\cdots+\bu_n$ and $\bu_i, \bq \in X^+$, $1 \leq i \leq n$.

It is a routine matter to verify that $S_{(4, 780)}$ is isomorphic to
a subdirect product of $S_{46}$ and $S_{36}$ via the congruences defined by the nontrivial blocks $\{1, 2\}$ and $\{3, 4\}$.
So both $S_{46}$ and $S_{36}$ satisfy $\bq \preceq \bu$.
By Lemmas~\ref{pro423} and~\ref{lem4601},
we have $\ell(\bq)\geq 2$, $H_\bq(D_\bq(\bu))\neq \emptyset$ (so there exists $\bu_i\in D_\bq(\bu)$ with $h(\bu_i)=h(\bq)$),
and if $m(t(\bq),\bq)=1$, then there exists $\bu_j\in D_\bq(\bu)$ such that $t(\bq)\notin c(p(\bu_j))$.

\textbf{Case 1.} $m(t(\bq), \bq)\geq 2$. We have

\[
\bu\succeq \bu_i\stackrel{\eqref{7801}}\succeq \bu_i\bq\stackrel{\eqref{7802},\eqref{7803}}\approx \bq.
\]
This derives the identity $\bu \succeq \bq$.

\textbf{Case 2.} $m(t(\bq), \bq)=1$. Then there exists $\bu_j \in D_\bq(\bu)$ such that $t(\bq)\notin c(p(\bu_j))$.

If $m(t(\bq),\bu_j)=0$, this implies
\[
\bu\succeq\bu_i+\bu_j\stackrel{\eqref{7801})}\succeq\bu_i\bq^2+\bu_j\stackrel{\eqref{7802},\eqref{7803}}\approx \bq^{2}+\bu_j\stackrel{\eqref{7805}} \approx(p(\bq))^{2}t(\bq)+\bu_j\stackrel{\eqref{7802}}\succeq \bq.
\]
This derives the identity $\bu \succeq \bq$.

If $m(t(\bq),\bu_j)=1$, then $t(\bu_j)=t(\bq)$. We have
\[
\bu\succeq\bu_i+\bu_j\stackrel{\eqref{7801}}\succeq\bu_j+\bu_i\bq^{2}\stackrel{\eqref{7802},\eqref{7803}}\approx \bu_j+\bq^{'2}(\bu_j)^{2}\stackrel{\eqref{7804}}\succeq\bq'\bu_j\stackrel{\eqref{7803}}\approx \bq.
\]
where $\bu_i\bq^{2}=\bq^{'2}(\bu_j)^{2}$.
This derives the identity $\bu \succeq \bq$.
\end{proof}

%
%

\begin{cor}
The ai-semiring $S_{(4, 682)}$ is finitely based.
\end{cor}
\begin{proof}
Since $S_{(4, 682)}$ and $S_{(4, 780)}$ have dual multiplications, it follows from
Proposition {\ref{pro428}} that $S_{(4, 682)}$ is finitely based.
\end{proof}

\begin{pro}\label{pro429}
$\mathsf{V}(S_{(4, 690)})$ is the ai-semiring variety defined by the identities
\begin{align}
&yx \preceq x; \label{6901}\\
&x^{2}y \approx xy; \label{6902}\\
&xyz \approx yxz; \label{6903}\\
&xy \preceq x+y^{2}; \label{6904}\\
&xy \preceq y+x^{2}y^{2}. \label{6905}
\end{align}
\end{pro}
\begin{proof}
It is easy to check that $S_{(4, 690)}$ satisfies the identities \eqref{6901}--\eqref{6905}.
In the remainder it is enough to prove that every ai-semiring identity of $S_{(4, 690)}$
can be derived by \eqref{6901}--\eqref{6905} and the identities defining $\mathbf{AI}$.
Let $\bq\preceq \bu$ be such an inequality,
where $\bu=\bu_1+\bu_2+\cdots+\bu_n$ and $\bu_i, \bq \in X^+$, $1 \leq i \leq n$.

Observe that $S_{(4,690)}$ is isomorphic to a subdirect product of $S_{46}$ and $S_{39}$ via the congruences determined by the nontrivial blocks $\{1,2\}$ and $\{2,3\}$; hence both $S_{46}$ and $S_{39}$ satisfy $\bq\preceq\bu$.
By Lemmas~\ref{cor424} and~\ref{lem4601}, we have $\ell(\bq)\geq 2$, $T_\bq(D_\bq(\bu))\neq\emptyset$
(so there exists $\bu_i\in D_\bq(\bu)$ with $t(\bu_i)=t(\bq)$),
and if $m(t(\bq),\bq)=1$, then there exists $\bu_j\in D_\bq(\bu)$ such that $t(\bq)\notin c(p(\bu_j))$.
 Now we have
\[
\bu\succeq \bu_i\stackrel{\eqref{6901}}\succeq \bq^{2}\bu_i\stackrel{\eqref{6902},\eqref{6903}}\approx \bq^{2}.
\]

We need to consider the following cases:

\textbf{Case 1.} $m(t(\bq), \bq)\geq 2$. We have
\[
\bq^{2}\stackrel{(\ref{6902})}\approx\bq.
\]
This derives $\bu \succeq \bq$.

\textbf{Case 2.} $m(t(\bq), \bq)=1$. Then there exists $\bu_j \in D_\bq(\bu)$ such that $t(\bq)\notin c(p(\bu_j))$.

If $m(t(\bq),\bu_j)=0$, this implies
\[
\bu\succeq\bu_j+\bq^{2}\stackrel{\eqref{6904}}\succeq\bu_j\bq\stackrel{\eqref{6902}),\eqref{6903}}\approx \bq.
\]
This derives $\bu \succeq \bq$.

If $m(t(\bq),\bu_j)=1$, then $t(\bu_j)=t(\bq)$. We have
\[
\bu\succeq\bu_j+\bq^{2}\stackrel{\eqref{6902},\eqref{6903}}\approx\bu_j+\bq^{'2}\bu_j^{2}\stackrel{\eqref{6905}}\succeq \bq'\bu_j\stackrel{\eqref{6903}}\approx \bq.
\]
where $\bq^{2}=\bq^{'2}\bu_j^{2}$ .
This derives $\bu \succeq \bq$.
\end{proof}


\begin{cor}
The ai-semiring $S_{(4, 757)}$ is finitely based.
\end{cor}
\begin{proof}
Notice that $S_{(4, 757)}$ and $S_{(4, 690)}$ have dual multiplications,
therefore, $S_{(4, 757)}$ is also finitely based.
\end{proof}

\begin{pro}\label{pro4210}
$\mathsf{V}(S_{(4, 671)})$ is the ai-semiring variety defined by the identities
\begin{align}
&x^{3} \approx x^{2}; \label{6711}\\
&xy \approx yx; \label{6712}\\
&xyz \preceq xy; \label{6713}\\
&xy \preceq y+xy^{2}. \label{6714}
\end{align}
\end{pro}
\begin{proof}
It is easy to check that $S_{(4, 671)}$ satisfies the identities \eqref{6711}--\eqref{6714}.
In the remainder it is enough to prove that every ai-semiring identity of $S_{(4, 671)}$
can be derived by \eqref{6711}--\eqref{6714}.
Let $\bq \preceq \bu$ be such an inequality,
where $\bu=\bu_1+\bu_2+\cdots+\bu_n$ and $\bu_i, \bq \in X^+$, $1 \leq i \leq n$.

Observe that $S_{(4,671)}$ is isomorphic to a subdirect product of $S_{44}$ and $S_{55}$ via the congruences determined by the nontrivial blocks $\{1,2\}$ and $\{3,4\}$; hence both $S_{44}$ and $S_{55}$ satisfy $\bq\preceq\bu$.
By Lemmas~\ref{lem5501} and~\ref{lem4401}, we have $L_{\geq2}(D_\bq(\bu))\neq\emptyset$,
and it suffices to consider one of the following cases.

If $\ell(\bq)=1$, then $\bu\succeq \bq$ is trivial.
If $\ell(\bq)\geq 2,$ there exists $\bu_i \in D_\bq(\bu)$, and $\ell(\bu_i)\geq 2$.

\textbf{Case 1.} $M_1(\bq)$ is empty. For all of $x\in \bq, m(x,\bq)\geq 2$. We have

\[
\bu\succeq \bu_i\stackrel{\eqref{6713}}\succeq \bu_i\bq\stackrel{\eqref{6711},\eqref{6712}}\approx \bq.
\]
This derives the identity $\bu \succeq \bq$.

\textbf{Case 2.} $M_1(\bq)$ is nonempty.
Then by Lemma \ref{lem4401} there exists $\bu_j \in D_\bq(\bu)$ such that $m(y, \bu_j)\leq 1$ for any $y \in M_1(\bq)$.
Now we have
\[
\bu\succeq \bu_j\stackrel{\eqref{6713}}\succeq \bu_j\bq.
\]

If $m(y, \bu_j) = 0$ for any $y \in M_1(\bq)$, then
\[
\bu\succeq \bu_j\bq\stackrel{\eqref{6711},\eqref{6712}}\approx \bq.
\]
This derives the identity $\bu \succeq \bq$.

If $m(y, \bu_j) = 1$ for some $y \in M_1(\bq)$, we have
\[
\bu\succeq \bu_j+\bu_j\bq\stackrel{\eqref{6712}}\approx \bu_j+\bq'\bu_j^{2}\stackrel{\eqref{6714}}\succeq\bq'\bu_j\stackrel{\eqref{6712}}\approx \bq.
\]
where $\bu_j\bq=\bq'\bu_j^{2}$.
This derives the identity $\bu \succeq \bq$.
\end{proof}


\begin{pro}\label{pro4211}
$\mathsf{V}(S_{(4, 684)})$ is the ai-semiring variety defined by the identities
\begin{align}
&xyz \preceq xy; \label{6841}\\
&x^{2}y \approx xy; \label{6842}\\
&xyz \approx yxz; \label{6843}\\
&zxy \preceq xy; \label{6844}\\
&xy \preceq x+x^{2}y^{2}; \label{6845}\\
&xyz \preceq x+yz; \label{6846}\\
&xy \preceq y+x^{2}y^{2}. \label{6847}
\end{align}
\end{pro}
\begin{proof}
It is easy to check that $S_{(4, 684)}$ satisfies the identities \eqref{6841}--\eqref{6847}.
In the remainder it is enough to prove that every ai-semiring identity of $S_{(4, 684)}$
can be derived by \eqref{6841}--\eqref{6847} and the identities defining $\mathbf{AI}$.
Let $\bq \preceq \bu$ be such an inequality,
where $\bu=\bu_1+\bu_2+\cdots+\bu_n$ and $\bu_i, \bq \in X^+$, $1 \leq i \leq n$.

Observe that $S_{(4,684)}$ is isomorphic to a subdirect product of $S_{46}$ and $S_{55}$ via the congruences determined by the nontrivial blocks $\{1,2\}$ and $\{3,4\}$; hence both $S_{46}$ and $S_{55}$ satisfy $\bq\preceq\bu$.
By Lemmas~\ref{lem5501} and~\ref{lem4601}, we have $\ell(\bq)\geq 2$, $L_{\geq2}(D_\bq(\bu))\neq\emptyset$,
and if $m(t(\bq),\bq)=1$, then there exists $\bu_j\in D_\bq(\bu)$ such that $t(\bq)\notin c(p(\bu_j))$.

\textbf{Case 1.} $m(t(\bq), \bq)\geq 2$. We have

\[
\bu\succeq \bu_i\stackrel{\eqref{6841}}\succeq \bu_i\bq\stackrel{\eqref{6842},\eqref{6843}}\approx \bq.
\]
This derives $\bu \succeq \bq$.

\textbf{Case 2.} $m(t(\bq), \bq)=1$. Then Lemma \ref{lem4601} there exists $\bu_j \in D_\bq(\bu)$ such that $t(\bq)\notin c(p(\bu_j))$.

\textbf{Subcase 2.1.} $m(t(\bq),\bu_j)=0$. This implies
\[
\bu\succeq\bu_i+\bu_j\stackrel{\eqref{6846}}\succeq\bu_i\bu_j\stackrel{\eqref{6841}}\succeq \bu_i\bu_j\bq.
\]

If $m(t(\bq),\bu_i)=0$, then $m(t(\bq),\bu_i\bu_j)=0$, by \eqref{6842}--\eqref{6843} we deduce the identity $\bu_i\bu_j\bq \approx \bq$.
This derives $\bu \succeq \bq$.

If $m(t(\bq),\bu_i)=1$, then $m(t(\bq),\bu_i\bu_j)=1$, we have
\[
\bu\succeq\bu_j+\bu_i\bu_j\bq\stackrel{\eqref{6842},\eqref{6843}}\approx \bu_j+(p(\bq))^{2}(t(\bq))^{2}\stackrel{\eqref{6845}}\succeq \bq.
\]
This derives $\bu \succeq \bq$.

If $m(t(\bq),\bu_i)\geq 2$, then $m(t(\bq),\bu_i\bu_j)\geq 2$, we have
\[
\bu\succeq\bu_i+\bu_i\bu_j\bq\stackrel{\eqref{6842},\eqref{6843}}\approx \bu_i+\bq^2\stackrel{\eqref{6845}}\succeq \bq.
\]
This derives $\bu \succeq \bq$.

\textbf{Subcase 2.2.} $m(t(\bq),\bu_j)=1$. Then $t(\bu_j)=t(\bq)$.

When $m(t(\bq),\bu_i)=0$. Then $m(t(\bq),\bu_i\bu_j)=1$, and $t(\bq)=t(\bu_i\bu_j)$. We have
\[
\bu\succeq\bu_i+\bu_j\stackrel{\eqref{6846}}\succeq\bu_i\bu_j\stackrel{\eqref{6844}}\succeq p(\bq)\bu_i\bu_j\stackrel{\eqref{6842},\eqref{6843}}\approx \bq.
\]
This derives $\bu \succeq \bq$.

When $m(t(\bq),\bu_i)=1$. Then $m(t(\bq),\bu_i\bu_j)=2$, we have

\[
\bu\succeq\bu_j+\bu_i\bu_j\bq\stackrel{\eqref{6842},\eqref{6843}}\approx \bu_j+\bq^{'2}(\bu_j))^{2}\stackrel{\eqref{6847}}\succeq \bq'\bu_j\stackrel{\eqref{6842},\eqref{6843}}\approx \bq.
\]
where $\bu_i\bu_j\bq=\bq^{'2}(\bu_j))^{2}$.
This derives $\bu \succeq \bq$.

When $m(t(\bq),\bu_i)\geq 2$. Then $m(t(\bq),\bu_i\bu_j)\geq 3$, by (\ref{6842}), (\ref{6843}) and (\ref{6845}) we deduce the identity $\bu_i\bu_j\bq \approx \bq$.
This derives $\bu \succeq \bq$.
\end{proof}

%

\begin{cor}
The ai-semiring $S_{(4, 672)}$ is finitely based.
\end{cor}
\begin{proof}
Since $S_{(4, 672)}$ and $S_{(4, 684)}$ have dual multiplications, it follows from
Proposition {\ref{pro4211}} that $S_{(4, 672)}$ is finitely based.
\end{proof}

\begin{pro}\label{pro4212}
$\mathsf{V}(S_{(4, 801)})$ is the ai-semiring variety defined by the identities
\begin{align}
&yx \preceq x; \label{8011}\\
&xy \preceq x; \label{8012}\\
&x^{3} \approx x^{2}; \label{8013}\\
&xyz \approx yxz; \label{8014}\\
&xy \preceq x+x^{2}y^{2}; \label{8015}\\
&xy \approx x^{2}y+xy^{2}. \label{8016}
\end{align}
\end{pro}
\begin{proof}
It is easy to check that $S_{(4, 801)}$ satisfies the identities \eqref{8011}--\eqref{8016}.
In the remainder it is enough to prove that every ai-semiring identity of $S_{(4, 801)}$
can be derived by \eqref{8011}--\eqref{8016} and the identities defining $\mathbf{AI}$.
Let $\bq \preceq \bu$ be such an inequality,
where $\bu=\bu_1+\bu_2+\cdots+\bu_n$ and $\bu_i, \bq \in X^+$, $1 \leq i \leq n$.

Observe that $S_{(4,801)}$ is isomorphic to a subdirect product of $S_{44}$ and $S_{56}$ via the congruences determined by the nontrivial blocks $\{2,3\}$ and $\{3,4\}$; hence both $S_{44}$ and $S_{56}$ satisfy $\bq\preceq\bu$.
By Lemmas~\ref{lem4601} and~\ref{lem4401}, we have $\ell(\bq)\geq 2$,
and if $m(t(\bq),\bq)=1$, then there exists $\bu_i\in D_\bq(\bu)$ such that $t(\bq)\notin c(p(\bu_i))$;
moreover, it suffices to consider one of the following cases.

\textbf{Case 1.}$M_1(\bq)$ is empty. For all of $x\in c(\bq), m(x,\bq)\geq 2$. There exists $\bu_j \in D_\bq(\bu)$. We have

\[
\bu\succeq \bu_j\stackrel{\eqref{8012}}\succeq \bu_j\bq\stackrel{\eqref{8013},\eqref{8014}}\approx \bq.
\]
This derives $\bu \succeq \bq$.

\textbf{Case 2.}$M_1(\bq)$ is nonempty.

\textbf{Subcase 2.1.} $m(t(\bq), \bq)\geq 2$. Assume that
\[
c(\bq)=\{x_1, \ldots, x_m, y_1, \ldots, y_\ell\},
\]
where $m(x_i, \bq )\geq 2$, $m(y_j, \bq )=1$, $0 \leq i \leq m$, $1 \leq j \leq \ell$.
Then
\begin{align*}
\bq
&\approx x_1^{2}x_2^{2}\ldots x_m^{2}y_1y_2\ldots y_\ell t^{2}(\bq)\\
&\approx x_1^{2}x_2^{2}\ldots x_m^{2}y_1^{2}y_2^{2}\ldots y_\ell^{2}t^{2}(\bq)\\
&\quad + \sum\limits_{1 \leq  j \leq \ell} x_1^{2}x_2^{2}\ldots x_m^{2}
   y_1^{2}\ldots y_{i-1}^{2}y_{i+1}^{2}\ldots y_\ell^{2}
   y_j t^{2}(\bq). \quad (\text{by}~\eqref{8016})
\end{align*}

Let $\bq_j' = x_1^{2}x_2^{2}\ldots x_m^{2}y_1^{2}\ldots y_{i-1}^{2}y_{i+1}^{2}\ldots y_\ell^{2},\bq_j=\bq_j'y_jt^{2}(\bq)$, then it suffices to derive $\bq_j$.
It follows from Lemma \ref{lem4401} that
for any $y \in M_{1}(\bq)$, there exists $\bu_k \in D_\bq(\bu)$ such that $m(y, \bu_k)\leq 1$.

If $m(y, \bu_k) = 0$ for all $y \in M_{1}(\bq)$. Then
\[
\bu\succeq \bu_k\stackrel{\eqref{8012}}\succeq \bu_k\bq_j\stackrel{\eqref{8013},\eqref{8014}}\approx \bq_j.
\]
This derives $\bu \succeq \bq$.

If $m(y, \bu_k) = 1$ for some $y \in M_{1}(\bq)$. We have
\[
\bu\succeq \bu_k\stackrel{\eqref{8012}}\succeq \bu_k\bq_j't^{2}(\bq)\stackrel{\eqref{8013},\eqref{8014}}\approx \bq_j.
\]
This derives $\bu \succeq \bq$.

\textbf{Subcase 2.2.}$m(t(\bq), \bq)= 1$. Then by Lemma \ref{lem4601} there exists $\bu_i \in D_\bq(\bu)$ such that
$t(\bq)\notin c(p(\bu_i))$. Then
\begin{align*}
\bq
&\approx x_1^{2}x_2^{2}\ldots x_m^{2}y_1y_2\ldots y_\ell t(\bq)\\
&\approx x_1^{2}x_2^{2}\ldots x_m^{2}y_1^{2}y_2^{2}\ldots y_\ell^{2}t(\bq)\\
&\quad + \sum\limits_{1 \leq  j \leq \ell}  x_1^{2}x_2^{2}\ldots x_m^{2}
   y_1^{2}\ldots y_{j-1}^{2}y_{j+1}^{2}\ldots y_\ell^{2}
   y_j t^{2}(\bq). \quad (\text{by}~\eqref{8016})
\end{align*}

Let $\bq_H = x_1^{2}x_2^{2}\ldots x_m^{2}y_1^{2} y_2^{2}\ldots y_\ell^{2}t(\bq)$, now $\bq= \bq_H + \bq'$, $\bq'$ can be converted to Subcase 2.1, it suffices to consider $\bq_H$.

When $m(t(\bq), \bu_i) = 0$. Then
\[
\bu\succeq \bu_i\stackrel{\eqref{8012}}\succeq\bu_i\bq_H\stackrel{\eqref{8013},\eqref{8014}}\approx \bq_H.
\]
This derives $\bu \succeq \bq$.

When $m(t(\bq), \bu_i) = 1$. We implies $t(\bu_i)=t(\bq)$. We have
\[
\bu\succeq \bu_i\stackrel{\eqref{8011}}\succeq \bq^{'}_H\bu_i\stackrel{\eqref{8013},\eqref{8014}}\approx \bq_H,
\]
where $ \bq_H=\bq^{'}_H t(\bq)$.
This derives $\bu \succeq \bq$.
\end{proof}


\begin{cor}
The ai-semiring $S_{(4, 777)}$ is finitely based.
\end{cor}
\begin{proof}
Since $S_{(4, 777)}$ and $S_{(4, 801)}$ have dual multiplications, it follows from
Proposition {\ref{pro4212}} that $S_{(4, 777)}$ is finitely based.
\end{proof}

\begin{pro}\label{pro77401}
$\mathsf{V}(S_{44}, S_{45})$ is the ai-semiring variety defined by the identities
\begin{align}
& x^3\approx x^2; \label{77402}\\
& xyz\approx xzy; \label{77401}\\
& x^2y^2 \approx y^2x^2; \label{77405}\\
& x_1xx_2 \preceq x; \label{77403}\\
& xy \approx x^2y+xy^2, \label{77404}
\end{align}
where $x_1$ and $x_2$ may be empty in \eqref{77403}.
\end{pro}
\begin{proof}
A direct verification shows that both $S_{44}$ and $S_{45}$ satisfy identities \eqref{77402}--\eqref{77404}.
It remains to prove that every identity holding in both $S_{44}$ and $S_{45}$ can be derived from \eqref{77402}--\eqref{77404}.
Let $\bq\preceq \bu$ be a nontrivial inequality, where
$\bu=\bu_1+\cdots+\bu_n$ and $\bu_i,\bq\in X^+$ for $1\leq i\leq n$.
Since the multiplications of $S_{45}$ and $S_{46}$ are dual, it follows from Lemmas~\ref{lem4401} and~\ref{lem4601} that
$\ell(\bq)\geq 2$ and $c(\bu_i)\subseteq c(\bq)$ for some $\bu_i\in\bu$.
Let $c(\bq)=\{x_1,x_2,\ldots,x_m\}$ and $h(\bq)=x_1$.

\textbf{Case 1.}
$M_1(\bq)$ is empty. Then $m(x, \bq)\geq 2$ for all $x\in c(\bq)$.
From \eqref{77402} and \eqref{77401}, we obtain $\bq\approx x_1^2x_2^2\cdots x_m^2$.
Consequently,
\[
\bu \succeq \bu_i\stackrel{\eqref{77403}}\succeq \bq\bu_i \stackrel{\eqref{77402}, \eqref{77401}}\approx \bq.
\]
This derives the inequality $\bu \succeq\bq$.

\textbf{Case 2.}
$M_1(\bq)$ is nonempty.
From \eqref{77402} and \eqref{77401}, we have either
\[
\bq\approx x_1\cdots x_k\, x_{k+1}^2\cdots x_m^2
\quad\text{or}\quad
\bq\approx x_1^2\cdots x_k^2\, x_{k+1}\cdots x_m.
\]
Furthermore, identity \eqref{77404} yields
\[
\bq
\approx x_1x_2^2\cdots x_m^2
+\sum_{2\leq i\leq k} x_1^2\cdots x_{i-1}^2\, x_{i+1}^2\cdots x_m^2\, x_i
\]
in the first case, and
\[
\bq
\approx \sum_{k+1\leq i\leq m} x_1^2\cdots x_{i-1}^2\, x_{i+1}^2\cdots x_m^2\, x_i
\]
in the second. Thus it suffices to consider the two forms
\[
\bq=xy_1^2\cdots y_k^2
\quad\text{or}\quad
\bq=y_1^2\cdots y_k^2x.
\]

\textbf{Subcase 2.1.}  $\bq=xy_1^2\cdots y_k^2$. Lemma~\ref{lem4601} tells us that
there exists $\bu_j \in D_\bq(\bu)$ such that
$h(\bq)\notin c(s(\bu_j))$, and so $\bu_j=h(\bu_j)s(\bu_j)$.
If $h(\bu_j)=x$, then
\[
\bu \succeq \bu_j=xs(\bu_j) \stackrel{\eqref{77403}}\succeq xs(\bu_j)s(\bq) \stackrel{\eqref{77402}, \eqref{77401}}\approx \bq.
\]
If $h(\bu_j)\neq x$, then
\[
\bu \succeq \bu_j\stackrel{\eqref{77403}}\succeq \bq\bu_j\stackrel{\eqref{77402}, \eqref{77401}}\approx \bq.
\]
This derives the inequality $\bu \succeq\bq$.

\textbf{Subcase 2.2.} $\bq=y_1^2\cdots y_k^2x$.
By Lemma~\ref{lem4401}, there exists $\bu_k\in D_\bq(\bu)$ such that $m(x,\bu_k)\leq 1$.
If $m(x, \bu_k)=0$, then
\[
\bu \succeq \bu_k \stackrel{\eqref{77403}}\succeq \bu_k\bq \stackrel{\eqref{77402}, \eqref{77401}, \eqref{77405}}\approx \bq.
\]
If $m(x, \bu_k)=1$, then
\[
\bu \succeq \bu_k \stackrel{\eqref{77403}}\succeq y_1^2\cdots y_k^2\bu_k \stackrel{\eqref{77402}, \eqref{77401}}\approx \bq.
\]
This derives the inequality $\bu \succeq\bq$.
\end{proof}

\begin{remark}
The ai-semiring $S_{(4, 774)}$ is finitely based.
\end{remark}
\begin{proof}
It is easy to see that both $S_{44}$ and $S_{45}$ can  be embedded into $S_{(4, 774)}$.
Also, $S_{(4, 774)}$ satisfies the identities \eqref{77402}--\eqref{77404},
and so $\mathsf{V}(S_{44}, S_{45})=\mathsf{V}(S_{(4, 774)})$.
By Proposition $\ref{pro77401}$ we immediately deduce that $S_{(4, 774)}$ is finitely based.
\end{proof}

\begin{cor}
The ai-semiring $S_{(4, 786)}$ is finitely based.
\end{cor}
\begin{proof}
Notice that $S_{(4, 786)}$ and $S_{(4, 774)}$ have dual multiplications,
therefore, $S_{(4, 786)}$ is also finitely based.
\end{proof}

\begin{pro}
The ai-semiring $S_{(4, 636)}$ is finitely based.
\end{pro}
\begin{proof}
It is a routine matter to verify that $S_{(4, 636)}$ is isomorphic to a subdirect product of $T_2^0$ and $M_2^0$
via the congruences defined by the nontrivial blocks $\{1, 2\}$ and $\{2, 3\}$. So
\[
\mathsf V(S_{(4, 636)}) = \mathsf V(T_2^0, M_2^0)
\]
By \cite[Proposition 7.4]{yrzs}, we have that $S_{(4, 431)}$ is finitely based, and
\[
\mathsf V(S_{(4, 431)}) = \mathsf V(T_2^0, M_2^0).
\]
Therefore,
\[
\mathsf V(S_{(4, 636)}) = \mathsf V(T_2^0, M_2^0)=\mathsf V(S_{(4, 431)}),
\]
and so $S_{(4, 636)}$ is also finitely based.
\end{proof}

\begin{pro}
The ai-semiring $S_{(4,517)}$ is  finitely based.
\end{pro}
\begin{proof}
It is easy to check that $S_{(4, 517)}$ is isomorphic to a subdirect product of $S_{60}$ and $S_{57}$ via the congruences defined by the nontrivial blocks $\{1, 2\}$ and $\{2, 3\}$.
This implies that
\[
\mathsf{V}(S_{(4, 517)}) = \mathsf{V}(S_{60}, S_{57}).
\]
By \cite[Proposition 6.11]{yrzs} it follows that $\mathsf{V}(S_{(4, 390)})$ is equal to $\mathsf{V}(S_{60}, S_{57})$ and is finitely based.
Thus $\mathsf{V}(S_{(4, 517)})=\mathsf{V}(S_{(4, 390)})$ and so $S_{(4, 517)}$ is also finitely based.
\end{proof}

\begin{pro}
The ai-semiring $S_{(4,518)}$ is  finitely based.
\end{pro}
\begin{proof}
It is easy to check that $S_{(4, 518)}$ is isomorphic to a subdirect product of $S_{60}$ and $S_{53}$
via the congruences defined by the nontrivial blocks $\{1, 2\}$ and $\{2, 3\}$.
This implies that
\[
\mathsf{V}(S_{(4, 518)}) = \mathsf{V}(S_{60}, S_{53}).
\]
By \cite[Proposition 6.13]{yrzs} it follows that $\mathsf{V}(S_{(4, 398)})$ is equal to $\mathsf{V}(S_{60}, S_{53})$ and is finitely based.
Thus $\mathsf{V}(S_{(4, 518)})=\mathsf{V}(S_{(4, 398)})$ and so $S_{(4, 518)}$ is also finitely based.
\end{proof}

\begin{pro}
The ai-semirings $S_{(4,522)}$, $S_{(4,519)}$, $S_{(4,522)}$ and $S_{(4, 496)}$ are finitely based.
\end{pro}
\begin{proof}
It is easy to check that $S_{(4, 522)}$ is isomorphic to a subdirect product of $S_{57}$ and $S_{53}$ via the congruences defined by the nontrivial blocks $\{1, 2\}$ and $\{2, 3\}$.
This implies that
\[
\mathsf{V}(S_{(4, 522)}) = \mathsf{V}(S_{57}, S_{53}).
\]
By \cite[Proposition 3.8]{yrzs} it follows that $\mathsf{V}(S_{(4, 401)})$ is equal to $\mathsf{V}(S_{57}, S_{53})$ and is finitely based.
Thus $\mathsf{V}(S_{(4, 522)})=\mathsf{V}(S_{(4, 401)})$ and so $S_{(4, 522)}$ is also finitely based.

Notice that $S_{(4, 519)}$ and $S_{(4, 522)}$ have dual multiplications.
Therefore, $S_{(4,519)}$ is also finitely based.

It is easy to see that both $S_{53}$ and $S_{57}$ can  be embedded into $S_{(4, 521)}$.
Also, $S_{(4, 521)}$ satisfies the equational basis of $\mathsf{V}(S_{57}, S_{53})$(see \cite[Proposition 3.8]{yrzs}),
and so $\mathsf{V}(S_{53}, S_{57})=\mathsf{V}(S_{(4, 521)})$.
By Proposition \cite[Proposition 3.8]{yrzs}, we immediately deduce that $S_{(4, 521)}$ is finitely based.

Notice that $S_{(4, 496)}$ and $S_{(4, 521)}$ have dual multiplications,
therefore, $S_{(4, 496)}$ is also finitely based.
\end{proof}

\begin{pro}
The ai-semirings $S_{(4,724)}$ and $S_{(4,648)}$ are  finitely based.
\end{pro}
\begin{proof}
It is easy to check that $S_{(4, 724)}$ is isomorphic to a subdirect product of $T_2^0$ and $L_2^0$ via the congruences defined by the nontrivial blocks $\{1, 2\}$ and $\{2, 3\}$.
This implies that
\[
\mathsf{V}(S_{(4, 724)}) = \mathsf{V}(T_2^0, L_2^0).
\]
By \cite[Proposition 7.5]{yrzs} it follows that $\mathsf{V}(S_{(4, 445)})$ is equal to $\mathsf{V}(T_2^0, L_2^0)$ and is finitely based.
Thus $\mathsf{V}(S_{(4, 724)})=\mathsf{V}(S_{(4, 445)})$ and so $S_{(4, 724)}$ is also finitely based.

Notice that $S_{(4, 648)}$ and $S_{(4, 724)}$ have dual multiplications.
Therefore, $S_{(4,648)}$ is also finitely based.
\end{proof}

\begin{pro}
The ai-semiring $S_{(4, 487)}$ is finitely based.
\end{pro}
\begin{proof}
It is a routine matter to verify that $S_{(4, 487)}$ is isomorphic to a subdirect product of $S_{54}$ and $S_{60}$
via the congruences defined by the nontrivial blocks $\{1, 2\}$ and $\{2, 3\}$. So
\[
\mathsf V(S_{(4, 487)}) = \mathsf V(S_{54}, S_{60})
\]
By \cite[Corollary 6.12]{yrzs}, we have that $S_{(4, 369)}$ is finitely based, and
\[
\mathsf V(S_{(4, 369)}) = \mathsf V(S_{54}, S_{60}).
\]
Therefore,
\[
\mathsf V(S_{(4, 487)}) = \mathsf V(S_{54}, S_{60})=\mathsf V(S_{(4, 369)}),
\]
and so $S_{(4, 487)}$ is also finitely based.
\end{proof}

\begin{pro}
The ai-semiring $S_{(4, 487)}$ is finitely based.
\end{pro}
\begin{proof}
It is a routine matter to verify that $S_{(4, 487)}$ is isomorphic to a subdirect product of $S_{54}$ and $S_{60}$
via the congruences defined by the nontrivial blocks $\{1, 2\}$ and $\{2, 3\}$. So
\[
\mathsf V(S_{(4, 487)}) = \mathsf V(S_{54}, S_{60})
\]
By \cite[Corollary 6.12]{yrzs}, we have that $S_{(4, 369)}$ is finitely based, and
\[
\mathsf V(S_{(4, 369)}) = \mathsf V(S_{54}, S_{60}).
\]
Therefore,
\[
\mathsf V(S_{(4, 487)}) = \mathsf V(S_{54}, S_{60})=\mathsf V(S_{(4, 369)}),
\]
and so $S_{(4, 487)}$ is also finitely based.
\end{proof}

Let $\bp$ be a word in $X^+$ and $Y$ a subset of $X$.
Then $i(\bp)$ denotes the the word that is obtained from $\bp$ by
retaining the first occurrence of each variable, and
$h_Y(\bp)$ denotes the first letter in the word
that is obtained from $\bp$ by deleting all variables in $Y$.
The following result, which is due to \cite[Lemma 2.4]{gpz}, provides a solution of the equational problem for $S_{41}$.

\begin{lem}\label{lem4101}
Let $\bq\preceq\bu$ be a nontrivial inequality such that
$\bu=\bu_1+\bu_2+\cdots+\bu_n$ and $\bu_i, \bq \in X^+$ for all $1\leq i \leq n$.
Then $\bq\preceq\bu$ holds in $S_{41}$ if and only if $D_\bq(\bu)\neq\emptyset$, and
for any $Y \subseteq c(\bq)$, there exists $\bu_j\in \bu$ such that $h_Y(\bu_j)=h_Y(\bq)$.
\end{lem}


\begin{pro}\label{pro63501}
$\mathsf{V}(S_{41}, T_2^0)$ is the ai-semiring variety defined by the identities
\begin{align}
& xyx \approx xy; \label{63501}\\
& x^2y \approx xy; \label{63502}\\
& xy^2\approx xy; \label{63503}\\
& x \preceq x^2; \label{63504}\\
& xy \preceq xz+yx; \label{63505}\\
& xy \preceq x^2+xyz; \label{63506}\\
& xyz \preceq xy+yzt; \label{63508}\\
& yx+xy \preceq x^2+yz, \label{63507}
\end{align}
where $z$ may be empty in \eqref{63506} or \eqref{63507}.
\end{pro}
\begin{proof}
It is easy to check that both $S_{41}$ and $T_2^0$ satisfy the identities \eqref{63501}--\eqref{63507}.
In the remainder it is enough to prove that every identity that holds in both $S_{41}$ and $T_2^0$
can be derived by \eqref{63501}--\eqref{63507}.
Let $\bq\preceq \bu$ be such a nontrivial inequality, where
$\bu=\bu_1+\bu_2+\cdots+\bu_n$ and $\bu_i, \bq \in X^+$, $1 \leq i \leq n$.
It follows from Lemmas \ref{lem4101} and \ref{lem001} that
for any $Y\subseteq c(\bq)$, there exists $\bu_\ell\in\bu$ such that $h_Y(\bu_\ell)=h_Y(\bq)$,
$D_{\bq}(\bu)\neq\emptyset$, and $\ell(\bu_i)\geq 2$ for some $\bu_i\in D_{\bq}(\bu)$.

If $|c(\bq)|=1$, then $\bu_i=x^k$ for some $k\geq2$, and so
\[
\bu \succeq x^k \stackrel{\eqref{63502}}\approx x^2\stackrel{\eqref{63502}, \eqref{63504}}\approx x^2+\bq.
\]

Now let $|c(\bq)|\geq 2$.
From \eqref{63501}--\eqref{63503} we derive the identity $\bq \approx i(\bq)$,
so $i(\bq) \preceq \bu$ is satisfied by both $S_{41}$ and $T_2^0$.
Let $i(\bq)=x_1x_2\cdots x_m$, $m\geq 2$.
By Lemma \ref{lem4101} it follows that for any $2\leq j\leq m$ and $Y_j=\{x_1, x_2, \ldots ,x_{j-1}\}$,
there exists $\bu_j \in \bu$
such that $\bu_j=\bu_{j_1}x_j\bu_{j_2}$ for some $\bu_{j_1}, \bu_{j_2} \in X^*$, where $c(\bu_{j_1})\subseteq Y_j$.

We shall show by induction on $j$ that $x_1x_2\cdots x_j\preceq \bu$
is derivable from \eqref{63501}--\eqref{63507}, $1 \leq j\leq m$.
For $j=2$, by \eqref{63501}--\eqref{63503} we have $\bu_i\approx x_1^2$, $x_2^2$, $x_1x_2$, or $x_2x_1$.
If $\bu_i\approx x_1^2$, then
\[
\bu \succeq x_1^2+\bu_2 =x_1^2+\bu_{2_1}x_2\bu_{2_2}\stackrel{\eqref{63502}, \eqref{63506}, \eqref{63507}}\succeq x_1x_2.
\]
If $\bu_i\approx x_2^2$, then
\[
\bu\succeq x_2^2+\bu_1= x_2^2+x_1\bu_{1_1}\stackrel{\eqref{63507}}\succeq x_1x_2.
\]
If $\bu_i\approx x_2x_1$, then
\[
\bu \succeq x_2x_1+\bu_1=x_2x_1+x_1\bu_{1_1}\stackrel{\eqref{63505}}\succeq x_1x_2.
\]
This implies the inequality $\bu \succeq x_1x_2$.

Let $2\leq j\leq m$.
Suppose that $\bu\succeq x_1x_2\cdots x_{j-1}$ is derivable from \eqref{63501}--\eqref{63507}.
By the identities \eqref{63501}--\eqref{63503}, we have that $x_1\dots x_{j-1}\approx x_1\dots x_{j-1}\bu_{j_1}$,
 where $\bu_j=\bu_{j_1}x_j\bu_{j_2}$, $c(\bu_{j_1})\subseteq\{x_1,\dots,x_{j-1}\}$.
Furthermore, we can deduce
\begin{align*}
\bu
&\succeq x_1\dots x_{j-1}+\bu_j \\
&\approx x_1\dots x_{j-1}\bu_{j_1}+\bu_{j_1}x_j\bu_{j_2}\\
&\succeq x_1\dots x_{j-1}\bu_{j_1}x_j &&(\text{by}~\eqref{63508})\\
&\approx x_1x_2\dots x_j. &&(\text{by}~\eqref{63501},\eqref{63502})
\end{align*}
Take $j=m$. We obtain the inequality $\bu \succeq x_1x_2\cdots x_m=i(\bq)$.
Since $i(\bq)=\bq$, the inequality $\bu\succeq\bq$ is derivable.
\end{proof}

\begin{remark}
The ai-semiring $S_{(4, 635)}$ is finitely based.
\end{remark}
\begin{proof}
It is easy to see that both $S_{41}$ and $T_2^0$ can  be embedded into $S_{(4, 635)}$.
Also, $S_{(4, 635)}$ satisfies the identities \eqref{63501}--\eqref{63507},
and so $\mathsf{V}(S_{41}, T_2^0)=\mathsf{V}(S_{(4, 635)})$.
By Proposition $\ref{pro63501}$ we immediately deduce that $S_{(4, 635)}$ is finitely based.
\end{proof}

\begin{cor}
The ai-semiring $S_{(4, 567)}$ is finitely based.
\end{cor}
\begin{proof}
Notice that $S_{(4, 567)}$ and $S_{(4, 635)}$ have dual multiplications,
therefore, $S_{(4, 567)}$ is also finitely based.
\end{proof}


\begin{pro}\label{pro68601}
$\mathsf{V}(S_{41}, S_{46})$ is the ai-semiring variety defined by the identities
\begin{align}
& x^2y \approx xy; \label{68602}\\
&(xy)^2 \approx x^2y^2; \label{68610}\\
& xyx \approx x^2y^2; \label{68603}\\
& x^2y^2 \approx x^2y^2x^2; \label{68609}\\
& x^2\preceq x; \label{68601}\\
& x^2y^2 \preceq x^2+xyz; \label{68604}\\
& x^2y^2 \preceq x^2+yz; \label{68605}\\
& x^2y^2 \preceq y^2+xz; \label{68606}\\
& x^2y^2 \preceq y^2x^2+xz; \label{68607}\\
& xyz^2 \approx xy+yzt; \label{68611}\\
& x^2y^2z \approx x^2y^2z^2+y^2; \label{68612}\\
& x^2y^2z \approx x^2y^2z^2+y^2z, \label{68608}
\end{align}
where $z$ may be empty in \eqref{68604}--\eqref{68607},
$y$ may be empty in \eqref{68608}, $t$ may be empty in \eqref{68611}.
\end{pro}
\begin{proof}
It is easy to check that both $S_{41}$ and $S_{46}$ satisfy the identities \eqref{68602}--\eqref{68608}.
In the remainder it is enough to prove that every identity that holds in both $S_{41}$ and $S_{46}$
can be derived by \eqref{68602}--\eqref{68608}.
Let $\bq\preceq \bu$ be such a nontrivial inequality, where
$\bu=\bu_1+\bu_2+\cdots+\bu_n$ and $\bu_i, \bq \in X^+$, $1 \leq i \leq n$.
It follows from Lemmas \ref{lem4101} and \ref{lem4601} that
for any $Y\subseteq c(\bq)$, there exists $\bu_j\in\bu$ such that $h_Y(\bu_j)=h_Y(\bq)$,
$D_{\bq}(\bu)\neq\emptyset$, and $\ell(\bq)\geq 2$.

If $|c(\bq)|=1$, then $\bq=x^k$ for some $k\geq2$.
Since $D_{\bq}(\bu)\neq\emptyset$, we have that there exists $\bu_i\in\bu$ such that $\bu_i=x^\ell$ for some $\ell\geq1$, and so
\[
\bu\succeq\bu_i =x^\ell\stackrel{\eqref{68601}, \eqref{68602}}\approx x^\ell+\bq.
\]

Now let $|c(\bq)|\geq 2$, and let $i(\bq)=x_1x_2\cdots x_m$, $m\geq 2$.
By the identities \eqref{68602}, \eqref{68603} and \eqref{68609} we deduce
$\bq \approx x_1^2x_2^2\dots x_m^2$ or $x_1^2x_2^2\dots x_{m-1}^2x_m$.

\textbf{Case 1.} $\bq \approx x_1^2x_2^2\dots x_m^2$.
By Lemma \ref{lem4101} it follows that for any $2\leq j\leq m$ and $Y_j=\{x_1, x_2, \ldots ,x_{j-1}\}$,
there exists $\bu_j \in \bu$
such that $\bu_j=\bu_{j_1}x_j\bu_{j_2}$ for some $\bu_{j_1}, \bu_{j_2} \in X^*$, where $c(\bu_{j_1})\subseteq Y_j$.
We shall show by induction on $j$ that $\bu \succeq x_1^2x_2^2\cdots x_j^2$
is derivable from \eqref{68602}--\eqref{68608}, $1 \leq j\leq m$.
If $j=2$, then there exists $\bu_i\in \bu$ such that $c(\bu_i)\subseteq \{x_1, x_2\}$.
Combined with the inequalities \eqref{68602}--\eqref{68601}, $\bu_i\approx x_1^2$ or $x_2^2$ or $x_2^2x_1^2$ or $x_1^2x_2^2$.
If $\bu_i\approx x_1^2$, then
\[
\bu \succeq x_1^2+\bu_2 =x_1^2+\bu_{2_1}x_2\bu_{2_2}\stackrel{\eqref{68604}, \eqref{68605}}\succeq x_1^2x_2^2.
\]
If $\bu_i\approx x_2^2$, then
\[
\bu \succeq x_2^2+\bu_1=x_2^2+x_1\bu_{1_1}\stackrel{\eqref{68606}}\succeq x_1^2x_2^2.
\]
If $\bu_i\approx x_2^2x_1^2$, then
\[
\bu \succeq x_2^2x_1^2+\bu_1=x_2^2x_1^2+x_1\bu_{1_1}\stackrel{\eqref{68607}}\succeq x_1^2x_2^2.
\]
This implies the identity $\bu \succeq x_1^2x_2^2$.

Let $2\leq j\leq m$.
Suppose that $\bu\succeq x_1^2x_2^2\cdots x_{j-1}^2$ is derivable from \eqref{68602}--\eqref{68608}.
By the inequalities \eqref{68609} and \eqref{68610},
we have that $x_1^2x_2^2\cdots x_{j-1}^2\approx x_1^2x_2^2\cdots x_{j-1}^2(\bu_j')^2$,
where $\bu_j=\bu_j'x_j\bu_j''$, $c(\bu_j')\subseteq\{x_1,\dots,x_{j-1}\}$.
Furthermore, we can deduce
\begin{align*}
\bu
&\succeq x_1^2x_2^2\cdots x_{j-1}^2+\bu_j \\
& = x_1^2x_2^2\cdots x_{j-1}^2(\bu_j')^2+\bu_j'x_j\bu_j''\\
&\approx x_1^2x_2^2\cdots x_{j-1}^2(\bu_j')^2+(\bu_j')^2x_j\bu_j'' &&(\text{by}~\eqref{68602}, \eqref{68610})\\
&\succeq x_1^2x_2^2\cdots x_{j-1}^2(\bu_j')^2x_j^2 &&(\text{by}~\eqref{68611})\\
&\approx x_1^2x_2^2\cdots x_{j-1}^2x_j^2. &&(\text{by}~\eqref{68609},\eqref{68610})
\end{align*}
Take $j=m$. We obtain the inequality $\bu\succeq x_1^2x_2^2\cdots x_m^2$.

\textbf{Case 2.} $\bq \approx x_1^2x_2^2\dots x_{m-1}^2x_m$.
From Case $1$ we can deduce the inequality $\bu\succeq x_1^2x_2^2\cdots x_m^2$.
By Lemma \ref{lem4601} it follows that if $ m(t(\bq), \bq)= 1$,
then there exists $\bu_k \in D_\bq(\bu)$ such that $t(\bq)\notin c(p(\bu_k))$.
This implies that $\bu_k=p(\bu_k)x_m$ or $m(x_m, \bu_k)=0$.
If $\bu_k=p(\bu_k)x_m$, then
\begin{align*}
\bu
&\succeq x_1^2\cdots x_m^2+p(\bu_k)x_m \\
&\approx x_1^2\cdots x_{m-1}^2p(\bu_k)^2x_m^2+p(\bu_k)^2x_m &&(\text{by}~\eqref{68609}, \eqref{68602}, \eqref{68610})\\
&\approx (x_1\cdots x_{m-1})^2p(\bu_k)^2x_m^2+p(\bu_k)^2x_m &&(\text{by}~\eqref{68610})\\
&\succeq (x_1\cdots x_{m-1})^2p(\bu_k)^2x_m &&(\text{by}~\eqref{68608})\\
&\approx x_1^2\dots x_{m-1}^2x_m. &&(\text{by}~\eqref{68609},\eqref{68610})
\end{align*}
If $m(x_m, \bu_k)=0$, then $c(\bu_k)\subseteq\{x_1, x_2,\dots, x_{m-1}\}$, and so
\[
\bu \succeq \bu_k \stackrel{\eqref{68601}}\succeq p(\bu_k)^2t(\bu_k)^2\stackrel{\eqref{68610}}\approx x_{i_1}^2\cdots x_{i_k}^2,
\]
where $\{x_{i_1},\dots,x_{i_k}\}\subseteq\{x_1,\dots,x_{m-1}\}$.
Furthermore, we have
\begin{align*}
\bu
&\succeq x_1^2\cdots x_m^2+\bu_k^2 \\
&\approx x_1^2\cdots x_{m-1}^2x_{i_1}^2\cdots x_{i_k}^2x_m^2+\bu_k^2 &&(\text{by}~\eqref{68609})\\
&\approx (x_1\cdots x_{m-1})^2\bu_k^2x_m^2+\bu_k^2 &&(\text{by}~\eqref{68610})\\
&\succeq (x_1\cdots x_{m-1})^2\bu_k^2x_m &&(\text{by}~\eqref{68612})\\
&\approx x_1^2\dots x_{m-1}^2x_m. &&(\text{by}~\eqref{68609},\eqref{68610})
\end{align*}
Therefore, we can obtain the inequality $\bu \succeq x_1^2\dots x_{m-1}^2x_m$, and so
$\bu \succeq \bq$ can be derived.
\end{proof}

\begin{remark}
The ai-semiring $S_{(4, 686)}$ is finitely based.
\end{remark}
\begin{proof}
It is easy to see that both $S_{41}$ and $S_{46}$ can  be embedded into $S_{(4, 686)}$.
Also, $S_{(4, 686)}$ satisfies the identities \eqref{68602}--\eqref{68608},
and so $\mathsf{V}(S_{41}, S_{46})=\mathsf{V}(S_{(4, 686)})$.
By Proposition $\ref{pro68601}$ we immediately deduce that $S_{(4, 686)}$ is finitely based.
\end{proof}

\begin{cor}
The ai-semiring $S_{(4, 589)}$ is finitely based.
\end{cor}
\begin{proof}
Notice that $S_{(4, 589)}$ and $S_{(4, 686)}$ have dual multiplications,
therefore, $S_{(4, 589)}$ is also finitely based.
\end{proof}

\section{Equational basis for subdirectly irreducible $4$-element ai-semirings}\label{section-si}
In this section we provide equational basis for some 4-element ai-semirings that are subdirectly irreducible.\\

\begin{pro}\label{pro48501}
$\mathsf{V}(S_{(4, 485)})$ is the ai-semiring variety defined by the identities
\begin{align}
&y \preceq x^3; \label{48501}\\
&x \preceq xy; \label{48502}\\
&y \preceq xy; \label{48503}\\
&x_1x_2x_3 \approx y_1y_2y_3; \label{48504}\\
&xt+yt \preceq xy+zt.\label{48505}
\end{align}
\end{pro}

\begin{proof}
It is easily verified that $S_{(4, 485)}$ satisfies the identities
\eqref{48501}--\eqref{48505}.
In the remainder we need only show that every ai-semiring identity of $S_{(4, 485)}$
can be derived by \eqref{48501}--\eqref{48505} and the identities defining $\mathbf{AI}$.
Let $\bq \preceq \bu$ be such a nontrivial identity,
where $\bu=\bu_1+\bu_2+\cdots+\bu_n$ and $\bu_i, \bq\in X^+$, $1 \leq i \leq n$.
It is easy to see that $S_{59}$ is isomorphic to $\{1, 2, 4\}$ and so $S_{59}$ satisfies $\bq \preceq \bu$.
By Lemma \ref{lem5901} we consider the following two cases.

\textbf{Case 1.} $\ell(\bu_i)\geq 3$ for some $\bu_i \in \bu$. Then
\[
\bu \succeq \bu_i \stackrel{\eqref{48501}, \eqref{48504}}\succeq \bq.
\]

\textbf{Case 2.} $\ell(\bu_i) \leq 2$ for all $\bu_i \in \bu$.
Then $\ell(\bq)\leq 2$.

\textbf{Subcase 2.1.} $\ell(\bq)=1$. Then there exists $\bu_j \in \bu$ such that
$c(\bq) \subseteq c(\bu_j)$. Now, from identities \eqref{48502}, \eqref{48503}, we have
$\bu \succeq \bu_j \succeq \bq.$

\textbf{Subcase 2.2.} $\ell(\bq)=2$ and so $c(\bq) \subseteq c(L_2(\bu))$.
Let $\bq = xy$ for some $x, y \in X$.
Then we claim that there exists $\bu_\ell \in L_2(\mathbf{u})$ with $t(\bu_\ell) = y$.
Otherwise, suppose that for every $\bu_\ell \in L_2(\mathbf{u})$, $t(\bu_\ell) \neq y$.
Consider the semiring substitution $\varphi: P_f(X^+) \to S_{(4,485)}$.
When $a = y$, set $\varphi(a) = 3$; when $a \neq y$, set $\varphi(a) = 4$.
Then $\varphi(\bq) = 1$ and $\varphi(\mathbf{u}) = 2$, a contradiction.
Therefore, $\bu_\ell = y' y$ for some $y' \in X$,
 and there exists $\bu_k \in \mathbf{u}$ such that $\bu_k = xx'$ or $x' x$ for some $x' \in X$.
 Consequently,
\[
\bu \succeq \bu_k + \bu_\ell \stackrel{\eqref{48505}}\succeq xy \approx \bq.
\]
This implies the identity $\bu \succeq \bq$ as required.
\end{proof}

\begin{cor}
The ai-semiring $S_{(4, 491)}$ is finitely based.
\end{cor}
\begin{proof}
Since $S_{(4, 491)}$ and $S_{(4, 485)}$ have dual multiplications, it follows from
Proposition {\ref{pro48501}} that $S_{(4, 491)}$ is finitely based.
\end{proof}

\begin{pro}\label{pro49201}
$\mathsf{V}(S_{(4, 492)})$ is the ai-semiring variety defined by the identities
\begin{align}
&xy \approx yx; \label{49201}\\
&x \preceq xy; \label{49202}\\
&y \preceq x^3; \label{49203}\\
&x_1x_2x_3 \approx y_1y_2y_3. \label{49204}
\end{align}
\end{pro}

\begin{proof}
It is easy to check that $S_{(4, 492)}$ satisfies the identities \eqref{49201}--\eqref{49204}.
In the remainder we need only prove that every ai-semiring identity of $S_{(4, 492)}$
is derivable from \eqref{49201}--\eqref{49204} and the identities defining $\mathbf{AI}$.
Let $\bq \preceq \bu$ be such a nontrivial identity,
where $\bu=\bu_1+\bu_2+\cdots+\bu_n$. Since the identity \eqref{49201} is satisfied by $S_{(4, 492)}$,
we may assume that $\bu_i, \bq\in X_c^+$, $1 \leq i \leq n$.
It is easy to see that $S_{59}$ is isomorphic to $\{1, 2, 4\}$ and so $S_{59}$ satisfies $\bq \preceq \bu$.
By Lemma \ref{lem5901} we consider the following two cases.

\textbf{Case 1.} $\ell(\bu_i) \geq 3$ for some $\bu_i \in \bu$. Then
\[
\bu \succeq \bu_i \stackrel{\eqref{49204}}\approx
x^3 \stackrel{\eqref{49203}}\succeq \bq.
\]

\textbf{Case 2.} $\ell(\bu_i)\leq 2$ for all $\bu_i \in \bu$. Then $\ell(\bq)\leq 2$.
If $\ell(\bq) = 1$, then there exists $\bu_i \in \bu$ such that $c(\bq) \subseteq c(\bu_i)$. Thus, we obtain
\[
\bu \succeq \bu_i \stackrel{\eqref{49201}, \eqref{49202}} \succeq \bq.
\]
If $\ell(\bq) = 2$, then $c(\bq) \subseteq c(L_2(\bu))$. Let $\bq = xy$.
If $x=y$, then $\bq=x^2$. There exists $\bu_\ell \in L_2(\bu)$ such that $\bu_\ell=x^2$. If not, let $\varphi: P_f(X^+) \to S_{(4,492)}$ be a substitution such that
$\varphi(a)=3$ for every $a=x,y$, and $\varphi(a)=4$ otherwise. It is easy to see that $\varphi(\bu)=2$ and so $\varphi(\bq)=1$, a contradiction.
Then we have that this case is trivial.
If $x\neq y$, then we claim that there exists $\bu_\ell \in L_2(\bu)$ such that $x, y \in c(\bu_\ell)$. Otherwise, assume that for any $\bu_\ell \in L_2(\bu)$, $x$ and $y$ cannot appear simultaneously in $c(\bu_\ell)$. If not, using the substitution $\varphi$ defined above, we obtain $\varphi(\bu)=2$ and so $\varphi(\bq)=1$, a contradiction. Then we have
\[
\bu \succeq \bu_\ell \stackrel{\eqref{49201}} \succeq \bq.
\]
This implies the identity $\bu \succeq \bq$ as required.
\end{proof}

\begin{pro}\label{pro49301}
$\mathsf{V}(S_{(4, 493)})$ is the ai-semiring variety defined by the identities
\begin{align}
&xy \approx yx; \label{49301}\\
&x \preceq xy; \label{49302}\\
&x_1 \preceq x_2x_3x_4x_5; \label{49303}\\
&xy \preceq x^2+yz; \label{49304}\\
&x_1^2 \preceq x_1x_2x_3; \label{49309}\\
&x_1^3 \preceq x_1x_2x_3; \label{49307}\\
&x_1^2x_2 \preceq x_1x_2x_3; \label{49306}\\
&x_1y_1 \preceq x_1x_2x_3+y_1y_2. \label{49305}\\
&x_1x_2^2 \preceq x_1^2+x_2^2+y_1y_2y_3; \label{49308}
\end{align}
\end{pro}
\begin{proof}
It is easily verified that $S_{(4, 493)}$ satisfies the identities
\eqref{49301}--\eqref{49308}.
In the remainder we need only show that every ai-semiring identity of $S_{(4, 493)}$
can be derived by \eqref{49301}--\eqref{49308}.
Let $\bq \preceq \bu$ be such a nontrivial identity,
where $\bu=\bu_1+\bu_2+\cdots+\bu_n$ and $\bu_i, \bq\in X^+$, $1 \leq i \leq n$.

\textbf{Case 1.} $\ell(\bu_i) \geq 4$ for some $\bu_i \in \bu$. Then
\[
\bu \succeq \bu_i \stackrel{\eqref{49303}} \succeq \bq.
\]

\textbf{Case 2.} $\ell(\bu_i) \leq 3$ for every $\bu_i \in \bu$. We first show that $c(\bq)\subseteq c(\bu)$.
If this is not the case, let $\varphi: P_f(X^+) \to S_{(4,493)}$ be a substitution such that
$\varphi(x)=4$ for every $x\in c(\bu)$, and $\varphi(x)=1$ otherwise. It is easy to see that $\varphi(\bu)=2$ and so $\varphi(\bq)=1$. Thus, we have $c(\bq)\subseteq c(\bu)$. Next, we prove that $\ell(\bq) \leq 3$.

It is easy to see that $\varphi(\bu)=2$ and so $\varphi(\bq)=2$.
Suppose by way of contradiction that $\ell(\bq)\geq 4$, since $S_{(4, 493)}$ satisfies the identity $x_1x_2x_3x_4\approx y_1y_2y_3y_4$ and $1$ is the multiplicative zero element of $S_{(4, 493)}$, then we have $\varphi(\bq)=1$.
This is a contradiction. So $\ell(\bq)\leq 3$.

\textbf{Subcase 2.1.} $\ell(\bq)=1$. We have $c(\bq)\subseteq c(\bu)$. Then there exists $\bu_i \in \bu$ such that $c(\bq) \subseteq c(\bu_i)$. Now we have
\[
\bu \succeq \bu_i \stackrel{\eqref{49301}, \eqref{49302}} \succeq \bq.
\]

\textbf{Subcase 2.2.} $\ell(\bq)=2$. Let $\bq=xy$.

\textbf{Subcase 2.2.1.} $c(\bq)\subseteq c(\bu_i)$ for some $\bu_i \in \bu$.

We first analyze the case where $x=y$, which implies that $\bq=x^2$.
Either $\ell(\bu_i)=3$, or $\bu_i=x^2$.
In the following, we will use a proof by contradiction.
Let $\psi: P_f(X^+) \to S_{(4,493)}$ be a substitution such that
$\psi(a)=3$ for every $a = x$, $\psi(a)=4$ for every $a \neq x$.
Then we have $\psi(\bu)=2$ and so $\psi(\bq)=1$. This is a contradiction. So
we can divide this into the following two cases.

$\ell(\bu_i)=3$. We have
\[
\bu \succeq \bu_i \stackrel{\eqref{49301}, \eqref{49309}} \succeq \bq.
\]

When $\bu_i=x^2$, this case is trivial.

We now turn to the analysis of the case $x \neq y$, which implies that $\bq=xy$. The remaining steps are similar to those in Subcase 2.1.

\textbf{Subcase 2.2.2.} There exists $\bu_i \in L_3(\bu)$, we have $x \in c(\bu_i)$ or $y \in c(\bu_i)$. Let $x \in c(\bu_i)$. If there exists $\bu_j \in \bu$ such that $y \in \bu_j$, then $\ell(\bu_j)\geq 2$. If not, let $\sigma: P_f(X^+) \to S_{(4,493)}$ be a substitution such that
$\sigma(a)=2$ for every $a=y$, and $\sigma(a)=4$ otherwise. It is easy to see that $\sigma(\bu)=2$ and so $\sigma(\bq)=1$, a contradiction.
Then
\[
\bu \succeq \bu_i + \bu_j \stackrel{\eqref{49301}} \approx x\bu_i' + y\bu_j' \stackrel{\eqref{49305}} \succeq \bq.
\]
Similarly, the same reasoning applies to $y \in c(\bu_i)$.

For all $\bu_i \in L_3(\bu)$, we have $x, y \notin c(\bu_i)$.
In this case, there exist $\bu_j,\bu_k \notin L_3(\bu)$ such that $\bu_j=x^2$ or $\bu_k=y^2$. If not, we introduce $\xi: P_f(X^+) \to S_{(4,493)}$ be a substitution such that
$\xi(a)=3$ for every $a=x, y$, and $\xi(a)=4$ otherwise. It is easy to see that $\xi(\bu)=2$ and so $\xi(\bq)=1$, a contradiction.
Let $\bu_j=x^2$. If there exists $\bu_\ell \in \bu$ such that $y \in \bu_\ell$, then $\ell(\bu_\ell) = 2$. If not, using the substitution $\sigma$ defined above, we obtain $\sigma(\bu)=2$ and so $\sigma(\bq)=1$, a contradiction.
Then
\[
\bu \succeq \bu_j + \bu_\ell \stackrel{\eqref{49301}} \approx x^2 + y\bu_\ell' \stackrel{\eqref{49304}} \succeq \bq.
\]
Similarly, the same reasoning applies to $\bu_k=y^2$.

\textbf{Subcase 2.3.} $\ell(\bq)=3$. Let $\bq=x_1x_2x_3$.
It follows that either $c(\bq) \subseteq c(L_3(\bu))$, or there exists an element $x_1 \in c(\bq)$ but $x_1 \notin c(L_3(\bu))$ such that $\bu_i=x_1^2$ for some $\bu_i \in L_2(\bu)$.
In the following, we will use a proof by contradiction. Let $\zeta: P_f(X^+) \to S_{(4,493)}$ be a substitution such that
$\zeta(a)=3$ for every $a = x_1$, $\zeta(a)=4$ for every $a \neq x_1$. We can obtain $\zeta(\bu)=2$ and so $\zeta(\bq)=1$. This is a contradiction. So we can divide this into the following two cases.

\textbf{Subcase 2.3.1.} $c(\bq)\subseteq c(L_3(\bu))$.

There exist $\bu_i \in L_3(\bu)$ and $x_1, x_2, x_3 \in c(\bu_i)$. If $x_1, x_2, x_3$ are all distinct, then we have
\[
\bu \succeq \bu_i \stackrel{\eqref{49301}} \succeq \bq.
\]
If exactly two of $x_1, x_2, x_3$ are equal, then we have
\[
\bu \succeq \bu_i \stackrel{\eqref{49301}, \eqref{49306}} \succeq \bq.
\]
If $x_1, x_2, x_3$ are all equal, then we have
\[
\bu \succeq \bu_i \stackrel{\eqref{49301}, \eqref{49307}} \succeq \bq.
\]

There exist $\bu_i, \bu_j \in L_3(\bu)$, $x_1, x_2 \in c(\bu_i)$, and $x_3 \in c(\bu_j)$, then
\[
\bu \succeq \bu_i + \bu_j \stackrel{\eqref{49301}} \approx x_1 x_2 \bu_i' + x_3 \bu_j' \stackrel{\eqref{49301}, \eqref{49305}} \succeq \bq.
\]

There exist $\bu_i, \bu_j, \bu_k \in L_3(\bu)$, $x_1 \in c(\bu_i)$, $x_2 \in c(\bu_j)$, and $x_3 \in c(\bu_k)$, then
\[
\bu \succeq \bu_i + \bu_j +\bu_k \stackrel{\eqref{49301}} \approx x_1 \bu'_i + x_2 \bu'_j + x_3 \bu'_k \stackrel{\eqref{49305}} \succeq x_1 x_2 x' + x_3 \bu'_k \stackrel{\eqref{49301}, \eqref{49305}} \succeq \bq.
\]

\textbf{Subcase 2.3.2.} There exists an element $x_1 \in c(\bq)$ but $x_1 \notin c(L_3(\bu))$ such that $ \bu_i = x_1^2$ for some $\bu_i \in L_2(\bu)$.

\textbf{Subcase 2.3.2.1} Suppose there exists $\bu_j \in \bu$, $x_2,x_3 \in c(\bu_j)$. Then there are two possibilities: either $x_2=x_3$ and $\bu_j=x_2^2$, or $\ell(\bu_j)=3$. If not, let $\delta: P_f(X^+) \to S_{(4,493)}$ be a substitution such that
$\delta(a)=3$ for every $a=x_2$, and $\delta(a)=4$ otherwise. It is easy to see that $\delta(\bu)=2$ and so $\delta(\bq)=1$, a contradiction.
If $x_2=x_3$ and $\bu_j=x_2^2$, for every $\bu_k \in L_3(\bu)$, then we have
\[
\bu \succeq \bu_i + \bu_j +\bu_k \stackrel{\eqref{49301}} \approx x_1^2 + x_2^2 + \bu_k \stackrel{\eqref{49308}} \succeq \bq.
\]
If $\ell(\bu_j)=3$, then we have
\[
\bu \succeq \bu_i + \bu_j \stackrel{\eqref{49301}} \approx x_1^2 + \bu_j \stackrel{\eqref{49304},\eqref{49305}, \eqref{49301}} \succeq \bq.
\]

\textbf{Subcase 2.3.2.2} Suppose there exists $\bu_j \in L_3(\bu)$, we have $x_2 \in c(\bu_j)$ or $x_3 \in c(\bu_j)$.
Let $x_2 \in c(\bu_j)$. If there exists $\bu_k \in \bu$ such that $x_3 \in \bu_k$, then either $\ell(\bu_k)=3$ or $\bu_k=x_3^2$. If not, let $\gamma: P_f(X^+) \to S_{(4,493)}$ be a substitution such that
$\gamma(a)=3$ for every $a=x_3$, and $\gamma(a)=4$ otherwise. It is easy to see that $\gamma(\bu)=2$ and so $\gamma(\bq)=1$, a contradiction.
If $\ell(\bu_k)=3$, then we have
\[
\bu \succeq \bu_i + \bu_j + \bu_k \stackrel{\eqref{49301}} \approx x_1^2 + x_2\bu_j' + x_3\bu_k' \stackrel{\eqref{49304},\eqref{49305},\eqref{49301}} \succeq \bq.
\]
If $\bu_k=x_3^2$, then we have
\[
\bu \succeq \bu_i + \bu_j + \bu_k \stackrel{\eqref{49301}} \approx x_1^2 + x_2\bu_j' + x_3^2 \stackrel{\eqref{49304},\eqref{49301}} \succeq \bq.
\]
Similarly, the same reasoning applies to $x_3 \in c(\bu_j)$.

For all $\bu_j \in L_3(\bu)$, we have $x_2, x_3 \notin c(\bu_j)$.
In this case, there exist $\bu_k,\bu_\ell \notin L_3(\bu)$ such that $\bu_k=x_2^2$ or $\bu_\ell=x_3^2$. If not, we introduce $\mu: P_f(X^+) \to S_{(4,493)}$ be a substitution such that
$\mu(a)=3$ for every $a=x_2, x_3$, and $\mu(a)=4$ otherwise. It is easy to see that $\mu(\bu)=2$ and so $\mu(\bq)=1$, a contradiction.
Let $\bu_k=x_2^2$. If there exists $\bu_t \in \bu$ such that $x_3 \in \bu_t$, for every $\bu_s \in L_3(u)$, then $\bu_t=x_3^2$. If not, using the substitution $\gamma$ defined above, we obtain $\gamma(\bu)=2$ and so $\gamma(\bq)=1$, a contradiction.
Then
\[
\bu \succeq \bu_i + \bu_k + \bu_t + \bu_s \stackrel{\eqref{49301}} \approx x_1^2 + x_2^2 + x_3^2 + \bu_s \stackrel{\eqref{49304},\eqref{49301}} \succeq \bq.
\]
Similarly, the same reasoning applies to $\bu_\ell=x_3^2$.
This implies the inequality $\bu \succeq \bq$.
\end{proof}

\begin{pro}\label{pro49401}
$\mathsf{V}(S_{(4, 494)})$ is the ai-semiring variety defined by the identities
\begin{align}
&xyz \approx yxz; \label{49401}\\
&x \preceq xy; \label{49402}\\
&y \preceq xy; \label{49403}\\
&x^2 \preceq xy; \label{49406}\\
&xyz \approx xy+yz+xz; \label{49404}\\
&xt+xz \preceq xy+zt. \label{49405}
\end{align}
\end{pro}

\begin{proof}
It is easily verified that $S_{(4, 494)}$ satisfies the identities
\eqref{49401}--\eqref{49405}.
In the remainder we need only show that every ai-semiring identity of $S_{(4, 494)}$
can be derived by \eqref{49401}--\eqref{49405} and the identities defining $\mathbf{AI}$.
Let $\bq \preceq \bu$ be such a nontrivial identity,
where $\bu=\bu_1+\bu_2+\cdots+\bu_n$ and $\bu_i, \bq\in X^+$, $1 \leq i \leq n$.
It is easy to see that $S_{57}$ is isomorphic to $\{1, 3, 4\}$ and so $S_{57}$ satisfies $\bq \preceq \bu$.
By Lemma \ref{lem5701} it follows that
$\ell(\bu_k)\geq 2$ for some $\bu_k \in \bu$,
$c(p(\bq))\subseteq c(p(\bu))$ and $t(\bq)\in c(\bu)$.

When $\ell(\bq) = 1$, there exists $\bu_i \in \bu$ such that $c(\bq) \subseteq c(u_i)$. We obtain
\[
\bu \succeq \bu_i \stackrel{\eqref{49401}, \eqref{49402}, \eqref{49403}} \succeq \bq.
\]

When $\ell(\bq) \geq 2$, we have $\bq \stackrel{\eqref{49404}}\approx \sum\limits_{1 \leq i < j \leq n} x_i x_j$.

\textbf{Case 1.} For any $\bu_i \in \bu$, there exists $u_j \in \bu$ such that $m(t(\bu_i), P(\bu_j)) \geq 1$. $c(\bq) \subseteq c(P(L_{\geq 2}(\bu))$.
Suppose there exists $\bu_k \in L_{\geq 2}(u)$ such that $x_i, x_j \in c(P(\bu_k))$. Then if $x_i = x_j$, we have
\[
\bu \succeq \bu_k \stackrel{\eqref{49401}} \approx x_i\bu_k' \stackrel{\eqref{49406}} \succeq x_i^2.
\]
If $x_i \neq x_j$, we have
\[
\bu \succeq \bu_k \stackrel{\eqref{49401}} \approx x_ix_j\bu_k' \stackrel{\eqref{49402}} \succeq x_ix_j.
\]
If there exist $\bu_k \in L_{\geq 2}(\bu)$ and $\bu_t \in L_{\geq 2}(\bu)$ such that $x_i \in c(P(\bu_k))$, $x_j \in c(P(\bu_t))$, then
\[
\bu \succeq \bu_k +\bu_t \stackrel{\eqref{49401}} \approx x_i\bu_k' + x_j\bu_t' \stackrel{\eqref{49405}} \succeq x_ix_j.
\]

\textbf{Case 2.}  There exist $\bu_i \in \bu$ and for any $\bu_j \in \bu$, $m(t(\bu_i), P(\bu_j)) = 0$. If $c(\bq) \subseteq c(P(\bu))$, then this can be reduced to Case 1. If there exists $\bu_i \in \bu$ with $t(\bq) = t(\bu_i)$, $c(P(\bq)) \subseteq c(P(\bu))$. When $\ell(\bu_i) \geq 2$, there exists $\bu_k \in \bu$ with $x_i \in c(P(\bu_k))$, thus we obtain
\[
\bu \succeq \bu_k +\bu_i \stackrel{\eqref{49401}} \approx x_i\bu_k' + \bu_it(\bq) \stackrel{\eqref{49405}} \succeq x_it(\bq).
\]
When $\ell(\bu_i) = 1$, we claim that there exists $\bu_k \in \bu$ such that $t(\bq) \in c(P(\bu_k))$.  Otherwise, when $\ell(\bu_i) = 1$, for every $\bu_k \in \bu$, we have $t(\bq) \notin c(P(\bu_k))$. If $t(\bq) = t(u_k)$, the proof is similar to the case where $\ell(\bu_i) \geq 2$.
If $t(\bq)\neq t(u_k)$, then we introduce $\varphi: P_f(X^+) \to S_{(4,494)}$ be a substitution such that $\varphi(x)=2$ for every $x = t(\bq)$, and $\varphi(x)=4$ otherwise. It is easy to see that $\varphi(\bu)=2$ and so $\varphi(\bq)=1$, a contradiction.
Moreover, there exists $\bu_t \in u$ with $x_i \in c(P(\bu_t))$, thus we obtain
\[
\bu \succeq \bu_t +\bu_k \stackrel{\eqref{49401}} \approx x_i\bu_t' + t(\bq)\bu_k' \stackrel{\eqref{49405}} \succeq x_it(\bq).
\]
This implies the identity $\bu \succeq \bq$.
\end{proof}

\begin{cor}
The ai-semiring $S_{(4, 486)}$ is finitely based.
\end{cor}
\begin{proof}
Since $S_{(4, 486)}$ and $S_{(4, 494)}$ have dual multiplications, it follows from
Proposition {\ref{pro49401}} that $S_{(4, 486)}$ is finitely based.
\end{proof}

\begin{pro}\label{pro53001}
$\mathsf{V}(S_{(4, 530)})$ is the ai-semiring variety defined by the identities
\begin{align}
&xyz \approx yxz; \label{53001}\\
&x \preceq xy; \label{53002}\\
&x^2 \preceq xy; \label{53006}\\
&yx \preceq x+yz; \label{53003}\\
&xyz \approx xy+yz+xz; \label{53004}\\
&xt+xz \preceq xy+zt. \label{53005}
\end{align}
\end{pro}

\begin{proof}
It is easily verified that $S_{(4, 530)}$ satisfies the identities
\eqref{53001}--\eqref{53005}.
In the remainder we need only show that every ai-semiring identity of $S_{(4, 530)}$
can be derived by \eqref{53001}--\eqref{53005} and the identities defining $\mathbf{AI}$.
Let $\bq \preceq \bu$ be such a nontrivial identity,
where $\bu=\bu_1+\bu_2+\cdots+\bu_n$ and $\bu_i, \bq\in X^+$, $1 \leq i \leq n$.
It is easy to see that $S_{57}$ is isomorphic to $\{1, 2, 4\}$ and so $S_{57}$ satisfies $\bq \preceq \bu$.
By Lemma \ref{lem5701} it follows that
$\ell(\bu_k)\geq 2$ for some $\bu_k \in \bu$,
$c(p(\bq))\subseteq c(p(\bu))$ and $t(\bq)\in c(\bu)$.

\textbf{Case 1.} For any $\bu_i \in \bu$, there exists $\bu_j \in \bu$ such that $m(t(\bu_i), P(\bu_j)) \geq 1$. Then $c(\bq) \subseteq c(P(L_{\geq 2}(\bu)))$.

$\ell(\bq) = 1$. Then there exists $\bu_i \in L_{\geq 2}(\bu)$ such that $c(\bq) \subseteq c(P(\bu_i))$, thus we have
\[
\bu \succeq \bu_i \stackrel{\eqref{53001}, \eqref{53002}} \succeq \bq.
\]

$\ell(\bq) \geq 2$. Then $\bq \stackrel{\eqref{53004}}\approx \sum\limits_{1 \leq i < j \leq n} x_i x_j$.
Suppose there exists $\bu_k \in L_{\geq 2}(\bu)$ with $x_i, x_j \in c(P(\bu_k))$. Then if $x_i=x_j$ and $\bu_k=x_i\bu_k'$, we have
\[
\bu \succeq \bu_k \approx x_i\bu_k' \stackrel{\eqref{53006}} \succeq x_i^2.
\]
If $x_i \neq x_j$, we have
\[
\bu \succeq \bu_k \stackrel{\eqref{53004}} \succeq x_ix_j.
\]

If there exist $\bu_k, \bu_t \in L_{\geq 2}(\bu)$ such that $x_i \in c(P(\bu_k))$, $x_j \in c(P(\bu_t))$, then we obtain
\[
\bu \succeq \bu_k +\bu_t \stackrel{\eqref{53001}} \approx x_i\bu_k' + x_j\bu_t' \stackrel{\eqref{53005}} \succeq x_ix_j.
\]

\textbf{Case 2.}  There exists $\bu_i \in \bu$ and for any $\bu_j \in \bu$, $m(t(\bu_i), P(\bu_j)) = 0$.

$\ell(\bq) \geq 2$. If $c(\bq) \subseteq c(P(\bu))$, then this can be reduced to Case 1. If there exists $\bu_i \in \bu$ with $t(\bq) = t(\bu_i)$, $c(P(\bq)) \subseteq c(P(\bu))$. When $\ell(\bu_i) \geq 2$, there exists $\bu_k \in \bu$ with $x_i \in c(P(\bu_k))$, thus we have
\[
\bu \succeq \bu_i +\bu_k \stackrel{\eqref{53001}} \approx \bu_it(\bq) + x_i\bu_k' \stackrel{\eqref{53005}} \succeq x_it(\bq).
\]
When $\ell(u_i) = 1$, $u_i = t(q)$. Thus we obtain
\[
\bu \succeq \bu_i +\bu_k \stackrel{\eqref{53001}} \approx t(\bq) + x_i\bu_k' \stackrel{\eqref{53003}} \succeq x_it(\bq).
\]

$\ell(\bq) = 1$. We claim that there exists $\bu_i \in \bu$ such that $t(\bq) \in c(P(\bu_i))$. If not, we introduce $\varphi: P_f(X^+) \to S_{(4,530)}$ be a substitution such that
$\varphi(x)=2$ for every $x = t(\bq)$, and $\varphi(x)=4$ otherwise. It is easy to see that $\varphi(\bu)=3$ and so $\varphi(\bq)=2$, a contradiction.
Thus we obtain
\[
\bu \succeq \bu_i \stackrel{\eqref{53001}, \eqref{53002}} \succeq t(\bq) \approx \bq.
\]

This implies the identity $\bu \succeq \bq$.
\end{proof}

\begin{cor}
The ai-semiring $S_{(4, 489)}$ is finitely based.
\end{cor}
\begin{proof}
Since $S_{(4, 489)}$ and $S_{(4, 530)}$ have dual multiplications, it follows from
Proposition {\ref{pro53001}} that $S_{(4, 489)}$ is finitely based.
\end{proof}

\begin{pro}\label{pro49501}
$\mathsf{V}(S_{(4, 495)})$ is the ai-semiring variety defined by the identities
\begin{align}
&xy \approx yx; \label{49501}\\
&x \preceq xy; \label{49502}\\
&xt \preceq xy+zt^2; \label{49504}\\
&xyz \approx xy+yz+xz. \label{49503}
\end{align}
\end{pro}
\begin{proof}
It is easy to verify that $S_{(4, 495)}$ satisfies the identities \eqref{49501}--\eqref{49504}.
In the remainder it is enough to show that every ai-semiring identity of $S_{(4, 495)}$
is derivable from \eqref{49501}--\eqref{49504} and the identities defining $\mathbf{AI}$.
Let $\bq \preceq \bu$ be such a nontrivial identity,
where $\bu=\bu_1+\bu_2+\cdots+\bu_n$ and $\bu_i, \bq\in X^+$, $1 \leq i \leq n$.
It is easy to see that $S_{53}$ is isomorphic to $\{1, 3, 4\}$ and so $S_{53}$ satisfies $\bq \preceq \bu$.
By Lemma \ref{lem5301} we need to consider the following two cases.

\textbf{Case 1.} $\ell(\bq)=1$. By Lemma \ref{lem5301} we obtain that $c(\bq)\subseteq c(\bu)$ and
so $c(\bq)\subseteq c(\bu_i)$ for some $\bu_i \in \bu$.
This implies that $\ell(\bu_i)\geq 2$.
By the identity $\eqref{49501}$ we deduce $\bu_i\approx \bq\bp$ for some nonempty word $\bp$ and so
\[
\bu \succeq \bu_i \approx
\bq\bp \stackrel{\eqref{49502}} \succeq \bq.
\]
This implies $\bu \succeq \bq$.

\textbf{Case 2.} $\ell(\bq)\geq 2$. Let $\bq=x_1x_2\cdots x_n$.
Then \eqref{49503} implies the identity
\[
\bq\approx \sum\limits_{1\leq i\textless j\leq n}x_ix_j.
\]
By Lemma \ref{lem5301} we need to consider the following subcases.

\textbf{Subcase 2.1.} $x_i=y_j$. Then $m(x_i, \bu_k)\geq 2$ for some $\bu_k\in \bu$. So we have
\[
\bu
\succeq \bu_k \stackrel{\eqref{49501}}\approx x_i^2\bu_k'
\stackrel{\eqref{49502}} \succeq  x_i^2.
\]

\textbf{Subcase 2.2.} $x_i \neq y_j$.
Then there exists $\bu_k\in \bu$ such that
either $x_i, y_j\in c(\bu_k)$ or $m(x_i, \bu_k)\geq 2$ or $m(y_j, \bu_k)\geq 2$.
If $x_i, y_j \in c(\bu_k)$, then
\[
\bu
\succeq \bu_k\stackrel{\eqref{49501}}\approx x_iy_j\bu_k'
\stackrel{\eqref{49502}}\approx x_iy_j.
\]

If there exists $u_k \in u$ such that $m(x_i, u_k) \geq 2$, then we assert that there exists $u_j \in u$ with $x_j \in c(u_j)$ and $\ell(u_j) \geq 2$. Otherwise,
suppose that for every $u_j \in u$, either $x_j \notin c(u_j)$ or $\ell(u_j) = 1$. Consider the semiring substitution $\varphi: P_f(X^+) \to S_{(4,495)}$,
where $\varphi(x) = 2$ when $x = x_j$, and $\varphi(x) = 4$ when $x \neq x_j$. Then $\varphi(q) = 1$ and $\varphi(u) = 2$, a contradiction. Thus we obtain
\[
\bu
\succeq \bu_k +\bu_j \stackrel{\eqref{49501}}\approx x_i^2 \bu_k'+x_j \bu_j'
\stackrel{\eqref{49501}, \eqref{49504}}\approx x_iy_j.
\]

If $m(y_j, \bu_k)\geq 2$, then the remaining steps are similar to the preceding case.
This implies the identity $\bu \succeq \bq$.
\end{proof}

\begin{pro}\label{pro53201}
$\mathsf{V}(S_{(4, 532)})$ is the ai-semiring variety defined by the inequalities
\begin{align}
&xyz     \approx xy+xz+yz;\label{5321}\\
&xy      \approx yx;\label{5322}\\
&xz  \preceq x+yz^2;\label{5323}\\
&xt \preceq xy+zt^2;\label{5324}\\
&x^3     \approx x^2;\label{5325}\\
&x    \preceq x^2.\label{5326}
\end{align}
\end{pro}

\begin{proof}
It is easy to verify that $S_{(4,532)}$ satisfies the inequalities \eqref{5321}--\eqref{5326}.
Now we only need to prove that every ineuqality satisfied by $S_{(4,532)}$ can be deduced from \eqref{5321}--\eqref{5326}.
Suppose that $\bq \preceq \bu$ is an arbitrary semiring inequality satisfied by $S_{(4,532)}$,
where $\bu = \bu_1+\bu_2+\cdots+\bu_n$, $\bu_i, \bq \in X^+$, $1 \leq i \leq n$.
Since $S_{53}$, $S_{61}$, $M_2$ and $T_2$ can be embedded into $S_{(4,532)}$, $S_{53}$, $S_{61}$, $M_2$ and $T_2$ satisfy the inequality $\bq \preceq \bu$.

\textbf{Case 1.}~If $\ell(\bq)\geq 2$, then $\bq\stackrel{\eqref{5321}}\approx\sum\limits_{1\leq i< j\leq n}x_ix_j$. If $x_i=x_j$, then there exists $\bu_k\in \bu$ such that $m(x_i,\bu_k)\geq 2$. Hence
\[
\bu\succeq \bu_k\stackrel{\eqref{5322}}\approx x_i^2\bu_k'\stackrel{\eqref{5321}}\succeq x_i^2\approx x_ix_j\approx \bq.
\]
If $x_i\neq x_j$, then there exists $\bu_k\in \bu$ such that $x_i,x_j\in c(\bu_k)$. Thus
\[
\bu\succeq \bu_k\stackrel{\eqref{5322}}\approx x_ix_j\bu_k'\stackrel{\eqref{5321}}\succeq x_ix_j\approx \bq.
\]
If there exists $\bu_k\in u$ such that $m(x_i,\bu_k)\geq 2$, then there exists $\bu_j\in \bu$ such that $x_j\in c(\bu_j)$. Consequently,
\[
\bu\succeq \bu_k+\bu_j\stackrel{\eqref{5322}}\approx x_i^2\bu_k'+x_ju_j'\stackrel{\eqref{5323},\eqref{5324}}\succeq x_ix_j \approx \bq.
\]

\textbf{Case 2.}~If $\ell(\bq)=1$, then there exists $\bu_i\in \bu$ such that $c(\bq) \subseteq c(\bu_i)$.
If $\bu_i=\bq^k.$~Then
\[
\bu\succeq \bu_i\approx \bq^k\stackrel{\eqref{5325}}\approx \bq^2\stackrel{\eqref{5326}}\succeq \bq.
\]
If $\bu_i=\bq'\bq^k for~k>1.$~Then
\[
\bu\succeq \bu_i\stackrel{\eqref{5321}}\succeq \bq^k \stackrel{\eqref{5325},\eqref{5326}}\approx \bq.
\]
If $\bu_i=\bq'\bq.$~Consider the semiring substitution $\varphi:P_f(X^+)\to S_{(4,532)}$, where $\varphi(x)=2$ if $x=\bq$ and $\varphi(x)=4$ if $x\neq \bq$. Then $\varphi(\bq)=2$ and $\varphi(\bu)=3$.

This derives the inequality $\bq\preceq \bu$.
\end{proof}

\begin{lem}\label{lem85401}
Let $\bq\preceq \bu$ be a semiring inequality, where $\bu = \bu_1 + \bu_2 + \cdots + \bu_n$, and $\bu_i, \bq \in X^+$ for $1 \leq i \leq n$. If $S_{(4,854)}$ satisfies $\bq\preceq \bu$, then

\begin{itemize}
  \item[(1)] $\ell(\bq) = 1$;
  \item[(2)] $\ell(\bq) = 2$. Let $\bq = xy$. Either there exists $\bu_i \in \bu$ such that $\bu_i \leq xy$ or $\bu_i \leq yx$, or there exist $\bu_i, \bu_j \in \bu$ such that $\bu_i = x^m$, $\bu_j = y^n$, where $m + n \leq 5$;
  \item[(3)] $\ell(\bq) = 3$. Let $\bq = xyz$. Then $D_q(\bu) \neq \emptyset$. Either there exists $\bu_i \in \bu$ such that $\ell(\bu_i) \leq 3$, or there exists $\bu_i \in \bu$ with $\ell(\bu_i) \geq 4$, and $\bu_j \in D_q(\bu)$ with $\ell(\bu_j) \leq 3$;
  \item[(4)] $\ell(\bq) \geq 4$.
\end{itemize}
\end{lem}

\begin{proof}

\textbf{Case 1.}
$\ell(\bq) = 1$. Since $ N_2 $ can be embedded into $ S_{(4,854)}$, then $\bq\preceq \bu$ is trivial.

\textbf{Case 2.}
$\ell(\bq) = 2 $. Let $\bq = xy $.

\textbf{Subcase 2.1.} Suppose that $\bu_i \not\preceq xy, or ~\bu_i \not\preceq yx$ for every $\bu_i \in \bu$.
Consider the semiring substitution $\varphi: P_f(X^+) \to S_{(4,854)} $ defined by $ \varphi(a) = 1 $ for $a \in c(\bq) $
and $ \varphi(a) = 4 $ for $ a \notin c(\bq) $.
Then $ \varphi(\bq) = 2 $, but $ \varphi(\bu) = 4 $.

\textbf{Subcase 2.2.} If $ \bu_i = x^m $ and $ \bu_j = y^n $ with $ m + n \geq 6 $, then let $ \bu_i = x^4 $ and $ \bu_j = y^2 $.
Consider the substitution $ \varphi: P_f(X^+) \to S_{(4,854)} $ defined by $ \varphi(a) = 1 $ for $a = x $,
$ \varphi(a) = 2 $ for $a = y $, and $ \varphi(a) = 4 $ for $ a \neq x, y $.
Then $ \varphi(\bq) = 3 $, while $ \varphi(\bu) = 4 $.

\textbf{Case 3.}
 $ \ell(\bq) = 3$, let $ \bq = xyz $.
 Suppose that $ D_q(\bu) = \emptyset $.
 Let $ \varphi $ be a semiring substitution defined by $ \varphi(a) = 1 $ for $a \in c(\bq) $ and $ \varphi(a) = 4 $ for $ a \notin c(\bq) $.
 Then $ \varphi(\bq) = 3 $ and $ \varphi(\bu) = 4 $.
 Thus $ D_q(\bu) \neq \emptyset$.
 Suppose that there exists $ \bu_i \in \bu $ with $ \ell(\bu_i) \geq 4 $, and $ \ell(\bu_j) \geq 4 $ for every $ \bu_j \in D_q(\bu) $.
 Then $\varphi(\bq) = 3 $ and $ \varphi(\bu) = 4 $.

\textbf{Case4.} $ \ell(\bq) \geq 4$, it is trivial.

 This derives the inequality $\bq\preceq \bu$.
\end{proof}

\begin{pro}\label{pro85402}
$\mathsf{V}(S_{(4, 854)})$ is the ai-semiring variety defined by the inequalities
\begin{align}
&xy                \preceq x;\label{8541}\\
&xy               \approx yx;\label{8542}\\
&xy        \preceq x^2+y^2;\label{8543}\\
&xy         \preceq x^2+y^3;\label{8544}\\
&xyz           \preceq x^2+t;\label{8545}\\
&xyz          \preceq x^2y+t;\label{8546}\\
&xyz \preceq x_1x_2x_3x_4+x^2;\label{8547}\\
&xyz\preceq x_1x_2x_3x_4+x^2y;\label{8548}\\
&x_1x_2x_3x_4    \preceq x.\label{8549}
\end{align}
\end{pro}

\begin{proof}
It is readily verified that $S_{(4,854)}$ satisfies the inequalities \eqref{8541}--\eqref{8549}.
Therefore, it suffices to prove that every inequality satisfied by $S_{(4,854)}$ can be derived from \eqref{8541}--\eqref{8549}.
Suppose that $\bq \preceq \bu$ is an arbitrary semiring inequality satisfied by $S_{(4,854)}$,
where $\bu = \bu_1 + \bu_2 + \cdots + \bu_n$, $\bu_i, \bq \in X^+$ for $1 \leq i \leq n$.

\textbf{Case1.} $\ell(\bq) = 1$, $\bq \preceq \bu$ is trivial.

\textbf{Case2.} $\ell(\bq) = 2$, let $\bq = xy$. If there exists $\bu_i \in \bu$ such that $\bu_i \leq xy$ or $\bu_i \leq yx$, then
\[
\bu \succeq \bu_i \stackrel{\eqref{8541},\eqref{8542}}\succeq \bq.
\]
If there exist $\bu_i, \bu_j \in \bu$ with $\bu_i = x^m$, $\bu_j = y^n$ and $m + n \leq 5$, then
\[
\bu \succeq \bu_i + \bu_j \stackrel{\eqref{8543},\eqref{8544}}\succeq xy.
\]

\textbf{Case3.} $\ell(\bq) = 3$, let $\bq = xyz$ such that $D_q(\bu) \neq \emptyset$. Then $\ell(\bu_i) \leq 3$ for every $\bu_i \in \bu$, and there exists $\bu_j \in D_q(\bu)$ with $\ell(\bu_j) \leq 3$, then
\[
\bu \succeq \bu_j \stackrel{\eqref{8545},\eqref{8546}}\succeq \bq.
\]
If there exists $\bu_i \in \bu$ with $\ell(\bu_i) \geq 4$, and there exists $\bu_j \in D_q(\bu)$ such that $\ell(\bu_j) \leq 3$, then
\[
\bu \succeq \bu_i + \bu_j \stackrel{\eqref{8547},\eqref{8548}}\succeq \bq.
\]

\textbf{Case4.} $\ell(\bq) \geq 4$, $\bu \stackrel{\eqref{8549}}\succeq \bq$.

This derives the inequality $\bq\preceq \bu$.
\end{proof}

\begin{pro}\label{pro86201}
$\mathsf{V}(S_{(4, 862)})$ is the ai-semiring variety defined by the inequalities
\begin{align}
&x_1x_2x_3  \preceq x;\label{8621}\\
&xy       \preceq x;\label{8622}\\
&xy      \approx yx;\label{8623}\\
&xy \preceq x^2+y^2.\label{8624}
\end{align}

\end{pro}

\begin{proof}
It is easy to verify that $S_{(4,862)}$ satisfies the inequalities \eqref{8621}--\eqref{8624}.
Therefore, it suffices to prove that every inequality satisfied by $S_{(4,862)}$ can be derived from \eqref{8621}--\eqref{8624}.
Suppose that $\bq \preceq \bu$ is an arbitrary semiring inequality satisfied by $S_{(4,862)}$,
where $\bu = \bu_1 + \bu_2 + \cdots + \bu_n$, $\bu_i, \bq \in X^+$ for $1 \leq i \leq n$.
Since $S_{48}$, $S_{47}$ and $N_2$ can be embedded into $S_{(4,862)}$, inequality $\bq \preceq \bu$ is also satisfied by $S_{48}$, $S_{47}$ and $N_2$.

\textbf{Case1.} If $\ell(\bq) = 1$, then $\bq \preceq \bu$ is trivial.

\textbf{Case2.} If $\ell(\bq) = 2$, there exists $\bu_i \in \bu$ such that $c(\bu_i) \subseteq c(\bq)$ and $\ell(\bu_i) \leq 2$. Let $\bq = xy$.
If $\ell(\bu_i) = 1$, then $\bu_i = x$ or $\bu_i = y$. Then
  \[
  \bu \succeq \bu_i \stackrel{\eqref{8622},\eqref{8623}}{\succeq} \bq.
  \]
If $\ell(\bu_i) = 2$, then $\bu_i = xy$ or $\bu_i = yx$. Consequently,
  \[
  \bu \succeq \bu_i \stackrel{\eqref{8623}}{\approx} \bq.
  \]
If $\bu_i = x^2$, there exists $\bu_j \in \bu$ such that $\bu_j = y^2$. Otherwise, suppose that $\bu_j \neq y^2$ for every $\bu_j \in \bu$. Consider the semiring substitution $\varphi: P_f(X^+) \to S_{(4,862)}$ defined by $\varphi(a) = 2$ if $a = x$, $\varphi(a) = 1$ if $a = y$, and $\varphi(a) = 4$ if $a \neq x, y$. Then $\varphi(\bq) = 3$ but $\varphi(\bu) = 4$. Hence
  \[
  \bu \succeq \bu_i + \bu_j \stackrel{\eqref{8624}}\succeq xy \approx \bq.
  \]

\textbf{Case3.} If $\ell(\bq) \geq 3$, then $\bu \stackrel{\eqref{8621}}\succeq \bq$.

This derives the inequality $\bq\preceq \bu$.
\end{proof}

\begin{pro}\label{pro86301}
$\mathsf{V}(S_{(4, 863)})$ is the ai-semiring variety defined by the inequalities
\begin{align}
&x_1x_2x_3\preceq x;\label{8631}\\
&xy       \preceq x;\label{8632}\\
&yx \preceq x;\label{8633}\\
&xy  \preceq y^2.\label{8634}
\end{align}

\end{pro}

\begin{proof}
It is easy to verify that $S_{(4,863)}$ satisfies the inequalities \eqref{8631}--\eqref{8634}.
Therefore, it remains only to prove that every inequality satisfied by $S_{(4,863)}$ can be deduced from \eqref{8631}--\eqref{8634}.
Suppose that $\bq \preceq \bu$ is an arbitrary semiring inequality satisfied by $S_{(4,863)}$,
where $\bu=\bu_1+\bu_2+\cdots+\bu_n$, $\bu_i,\bq\in X^+$ for $1\leq i\leq n$.
Since $S_{48}$, $S_{47}$, and $N_2$ can be embedded into $S_{(4,863)}$, $S_{48}$, $S_{47}$, and $N_2$ satisfy $\bq \preceq \bu$.

\textbf{Case1.} If $\ell(\bq)=1$, then it is trivial.

\textbf{Case2.} If $\ell(\bq)=2$, then there exists $\bu_i \in \bu$ such that $c(\bu_i)\subseteq c(\bq)$ and $\ell(\bu_i)\leq 2$.
If $\ell(\bu_i)=1$, then
\[
\bu\succeq \bu_i \stackrel{\eqref{8632},\eqref{8633}}\succeq xy \approx \bq.
\]
If $\ell(\bu_i)=2$, there exists $\bu_j\in L_2(u)$ and $\bu_j=y^2$. Otherwise, suppose that $\bu_j\neq y^2$ for every $\bu_j\in L_2(\bu)$. Consider the semiring substitution $\varphi:P_f(X^+)\to S_{(4,863)}$ defined by $\varphi(a)=2$ for $a=x$, $\varphi(a)=1$ for $a=y$, and $\varphi(a)=4$ for otherwise. Then $\varphi(\bq)=3$ and $\varphi(\bu)=4$.  Consequently,
\[
\bu\succeq \bu_j\stackrel{\eqref{8634}}\succeq \bq.
\]
This derives the inequality $\bq\preceq \bu$.
\end{proof}

\begin{cor}
The ai-semiring $S_{(4, 864)}$ is finitely based.
\end{cor}
\begin{proof}
Notice that $S_{(4, 864)}$ and $S_{(4, 863)}$ have dual multiplications,
therefore, $S_{(4, 864)}$ is also finitely based.
\end{proof}

\begin{pro}\label{pro80701}
$\mathsf{V}(S_{(4, 807)})$ is the ai-semiring variety defined by the inequalities
\begin{align}
&xyz\approx yxz;\label{8071}\\
&x^2y\approx xy;\label{8072}\\
&xy \preceq x;\label{8073}\\
&zxy \preceq xy.\label{8074}
\end{align}
\end{pro}

\begin{proof}
It is readily verified that $S_{(4,807)}$ satisfies the inequalities \eqref{8071}--\eqref{8074}.
Hence, it suffices to prove that every inequality satisfied by $S_{(4,807)}$ can be derived from \eqref{8071}--\eqref{8074}.
Suppose that $\bq \preceq \bu$ is an arbitrary semiring inequality satisfied by $S_{(4,807)}$,
where $\bu = \bu_1 + \bu_2 + \cdots + \bu_n$ with $\bu_i,\bq \in X^+$ for $1 \leq i \leq n$.
Since $S_{48}$, $S_{47}$, $N_2$, and $D_2$ can be embedded into $S_{(4,807)}$, $S_{48}$, $S_{47}$, $N_2$, and $D_2$ satisfy $\bq \preceq \bu$.

\textbf{Case1.} If $\ell(\bq) = 1$, then the inequality $\bu \succeq \bq$ is trivial.

\textbf{Case2.} If $\ell(\bq) \geq 2$, then $D_q(\bu) \neq \emptyset$.

\textbf{Subcase 2.1.} If $m(t(\bq),\bq) \geq 2$, then $\bq \stackrel{\eqref{8071},\eqref{8072}}\approx x_1^2 x_2^2 \cdots x_n^2 t^2(\bq)$. Since there exists $\bu_i \in D_q(\bu)$ such that $c(\bu_i) \subseteq c(\bq)$. Then,
\[
\bu \succeq \bu_i \stackrel{\eqref{8073}}\succeq \bu_i \bq \stackrel{\eqref{8071},\eqref{8072}}\succeq \bq.
\]

\textbf{Subcase 2.2.} If $m(t(\bq),\bq) = 1$, then $\bq \stackrel{\eqref{8071},\eqref{8072}}\approx x_1^2 x_2^2 \cdots x_n^2 t^2(\bq)$.
If there exists $\bu_i \in D_q(\bu)$ such that $m(t(\bq),\bu_i) = 0$, then $\bu_i \bq \approx \bq$.
This reduces to Case 1.
If there exists $\bu_i \in D_q(\bu)$ such that $m(t(\bq),\bu_i) = 1$ and $t(\bu_i) = t(\bq)$,
we claim that $\ell(\bu_i) \geq 2$.
Otherwise, suppose that either $\bu_i = t(\bq)$ or $\bu_i = \bu_i' t(\bq) \bu_i''$ for any $\bu_i \in D_q(\bu)$, where $\bu_i''$ is non-empty.
Consider the semiring substitution $\varphi: P_f(X^+) \to S_{(4,807)}$, defined by $\varphi(x) = 3$ if $x = t(\bq)$,
$\varphi(x) = 1$ if $x = x_i$, and $\varphi(x) = 4$ otherwise.
Then $\varphi(\bq) = 2$ but $\varphi(\bu) = 3$. Therefore,
\[
\bu \succeq \bu_i \stackrel{\eqref{8074}} \succeq \bq' \bu_i \stackrel{\eqref{8072},\eqref{8071}} \succeq \bq.
\]

This derives the inequality $\bq\preceq \bu$.
\end{proof}

\begin{cor}
The ai-semiring $S_{(4, 735)}$ is finitely based.
\end{cor}
\begin{proof}
Notice that $S_{(4, 735)}$ and $S_{(4, 807)}$ have dual multiplications,
therefore, $S_{(4, 735)}$ is also finitely based.
\end{proof}

\qquad

\newpage
\begin{pro}\label{pro72101}
$\mathsf{V}(S_{(4, 721)})$ is the ai-semiring variety defined by the identities
\begin{align}
x^2y &\approx xy; \label{72102}\\
xyz &\approx yxz; \label{72101}\\
yx &\preceq x; \label{72103}\\
xz &\preceq x+yz; \label{72104}\\
zy &\preceq xy+yz; \label{72105}\\
x_1x_2x_3 &\preceq x_1x_2x_3^2+x_4x_1. \label{72106}
\end{align}
\end{pro}
\begin{proof}
It is straightforward to verify that $S_{(4, 721)}$ satisfies identities \eqref{72102}--\eqref{72106}.
It remains to show that every ai-semiring identity of $S_{(4, 721)}$
can be derived from \eqref{72102}--\eqref{72106}.
Let $\bq\preceq \bu$ be a nontrivial inequality, where
$\bu=\bu_1+\bu_2+\cdots+\bu_n$ and $\bu_i, \bq \in X^+$, $1 \leq i \leq n$.
Observe that $N_2$ is isomorphic to $\{3, 4\}$, hence $N_2$ satisfies $\bq\preceq \bu$.
Thus $\ell(\bq)\geq 2$.
Since $R_2$ is isomorphic to $\{2, 4\}$,
we have that $R_2$ satisfies $\bq\preceq \bu$, and so
there exists $\bu_j \in \bu$ such that $t(\bu_j)=t(\bq)$.
Also, as $D_2$ is isomorphic to $\{1, 2\}$,
we obtain that $D_2$ satisfies $\bq\preceq \bu$, hence
there exists $\bu_i \in \bu$ such that $c(\bu_i)\subseteq c(\bq)$.

If $\ell(\bu_j)=1$, then $\bu_j=t(\bq)$, and so
\[
\bu \succeq \bu_j=t(\bq)\stackrel{\eqref{72103}} \succeq p(\bq)t(\bq)=\bq.
\]
We now consider the case where $\ell(\bu_j)\geq 2$ for every $\bu_j$ with $t(\bu_j)=t(\bq)$.

\textbf{Case 1.} There exists $\bu_i\in D_\bq(\bu)$ such that $m(t(\bq), \bu_i)\leq1$.
If $m(t(\bq), \bu_i)=0$, then
\[
\bu \succeq \bu_i+\bu_j=\bu_i+p(\bu_j)t(\bq)\stackrel{\eqref{72104}} \succeq \bu_it(\bq) \stackrel{\eqref{72102}, \eqref{72101}, \eqref{72103}}\approx p(\bq)t(\bq) = \bq.
\]
If $m(t(\bq), \bu_i)=1$ and $t(\bq)=t(\bu_i)$, then
\[
\bu \succeq \bu_i =p(\bu_i)t(\bq) \stackrel{\eqref{72102}, \eqref{72101}, \eqref{72103}}\succeq p(\bq)t(\bq) = \bq.
\]
If $m(t(\bq), \bu_i)=1$ and $t(\bq)\neq t(\bu_i)$, then identity \eqref{72101} implies that $\bu_i\approx t(\bq)\bu_i'$, and so
\[
\bu \succeq \bu_i+\bu_j=t(\bq)\bu_i'+p(\bu_j)t(\bq) \stackrel{\eqref{72105}}\succeq \bu_i't(\bq) \stackrel{\eqref{72102}, \eqref{72101}, \eqref{72103}}\succeq p(\bq)t(\bq) = \bq.
\]
This yields the inequality $\bu\succeq \bq$.

\textbf{Case 2.} $m(t(\bq), \bu_i)\geq2$ whenever $\bu_i\in D_\bq(\bu)$.
If $t(\bq)\neq t(\bu_i)$, then
\[
\bu_i \stackrel{\eqref{72101}}\approx t(\bq)^2\bu_i'\stackrel{\eqref{72102}}\approx t(\bq)\bu_i'.
\]
The remaining steps are similar to Case 1.
If $t(\bq)=t(\bu_i)$, then $\bu_i\approx \bu_i't(\bq)^2$.

\textbf{Case 2.1.} There exists $\bu_k\in \bu\setminus D_\bq(\bu)$ such that
$\bu_k=p(\bu_k)x$ for some $x\in c(\bq)\setminus \{t(\bq)\}$.Then
\[
\bu \succeq \bu_i+\bu_k\approx\bu_i't(\bq)^2+p(\bu_k)x \stackrel{\eqref{72103}}\succeq x\bu_i't(\bq)^2+p(\bu_k)x
\stackrel{\eqref{72106}}\succeq x\bu_i't(\bq) \stackrel{\eqref{72102}, \eqref{72101}, \eqref{72103}}\succeq \bq.
\]

\textbf{Case 2.2.} $\bu_k \neq p(\bu_k)x$ for all $x\in c(\bq)\setminus \{t(\bq)\}$ if $\bu_k\in \bu\setminus D_\bq(\bu)$.
In this case, we shall show that $m(t(\bq), \bq)\geq2$.
Suppose that this is not true.
Consider the semiring homomorphism $\varphi: P_f(X^+) \to S_{(4, 721)}$ defined by
$\varphi(t(\bq))=3$, $\varphi(y)=1$ if $y\in c(\bq)\setminus \{t(\bq)\}$, and $\varphi(y)=4$ otherwise.
It follows that $\varphi(\bu)=4$ and $\varphi(\bq)=3$, a contradiction.
Hence $m(t(\bq), \bq)\geq2$.
Therefore,
\[
\bu \succeq \bu_i\approx \bu_i't(\bq)^2\stackrel{\eqref{72102}, \eqref{72101}, \eqref{72103}}\succeq p(\bq)t(\bq)=\bq.
\]
This implies the inequality $\bu \succeq\bq$.
\end{proof}

\begin{cor}
The ai-semiring $S_{(4, 765)}$ is finitely based.
\end{cor}
\begin{proof}
Notice that $S_{(4, 765)}$ and $S_{(4, 721)}$ have dual multiplications,
therefore, $S_{(4, 765)}$ is also finitely based.
\end{proof}

\begin{pro}\label{pro72701}
$\mathsf{V}(S_{(4, 727)})$ is the ai-semiring variety defined by the identities
\begin{align}
x^3&\approx x^2; \label{72701}\\
xy &\approx xyz; \label{72702}\\
x^2 &\preceq x; \label{72703}\\
x^2 &\preceq xy; \label{72704}\\
x_1x_3 &\preceq x_1x_2+x_3x_4, \label{72705}
\end{align}
where $x_2$ and $x_4$ may be empty in \eqref{72705}.
\end{pro}
\begin{proof}
It is easy to verify that $S_{(4, 727)}$ satisfies the identities \eqref{72701}--\eqref{72705}.
In the remainder we need only show that every ai-semiring identity of $S_{(4, 727)}$
can be derived by \eqref{72701}--\eqref{72705}.
Let $\bq\preceq \bu$ be such a nontrivial inequality, where
$\bu=\bu_1+\bu_2+\cdots+\bu_n$ and $\bu_i, \bq \in X^+$, $1 \leq i \leq n$.
It is easy to see that $N_2$ is isomorphic to $\{3, 4\}$ and so $N_2$ satisfies $\bq\preceq \bu$.
This implies that $\ell(\bq)\geq 2$.
Since $L_2$ is isomorphic to $\{1, 2\}$,
it follows that $L_2$ satisfies $\bq\preceq \bu$, and so
there exists $\bu_i \in \bu$ such that $h(\bu_i)=h(\bq)$.

If $|c(\bq)|=1$, then
\[
\bu \succeq \bu_i = h(\bq)s(\bu_i) \stackrel{\eqref{72703}, \eqref{72704}}\succeq h(\bq)^2\stackrel{\eqref{72701}}\approx \bq.
\]
Now we consider the case that $|c(\bq)|\geq2$.
Write \(\bq=x_1x_2\cdots x_m\), with \(m\geq2\).

\textbf{Case 1.} $x_1=x_2$. Then
\[
\bu \succeq \bu_i = h(\bq)s(\bu_i) \stackrel{\eqref{72703}, \eqref{72704}}\succeq h(\bq)^2\stackrel{\eqref{72702}}\approx h(\bq)^2\bq_1= \bq.
\]
This derives the inequality $\bu\succeq \bq$.

\textbf{Case 2.} $x_1\neq x_2$.
We shall show that there exists $\bu_j\in\bu$ such that $\bu_j=x_1x_2\bu_j'$ or $x_2\bu_j'$, where $\bu_j'\in X^*$.
Suppose that this is not true.
Consider the semiring homomorphism $\varphi: P_f(X^+) \to S_{(4, 727)}$ defined by
$\varphi(x_1)=3$, $\varphi(x_2)=1$, and $\varphi(x)=4$ otherwise.
It follows that $\varphi(\bu)=4$ and $\varphi(\bq)=2$, a contradiction.
Hence such a $\bu_j$ exists.

If $\bu_j=x_1x_2\bu_j'$, then
\[
\bu \succeq \bu_j=x_1x_2\bu_j' \stackrel{\eqref{72702}} \approx x_1x_2\cdots x_m=\bq.
\]
If $\bu_j=x_2\bu_j'$, then
\[
\bu \succeq \bu_i+\bu_j=h(\bq)s(\bu_i)+x_2\bu_j' \stackrel{\eqref{72705}}\succeq h(\bq)x_2 \stackrel{\eqref{72702}}\approx x_1x_2\cdots x_m=\bq.
\]
This proves the inequality $\bu \succeq\bq$.
\end{proof}

\begin{cor}
The ai-semiring $S_{(4, 616)}$ is finitely based.
\end{cor}
\begin{proof}
Notice that $S_{(4, 616)}$ and $S_{(4, 727)}$ have dual multiplications,
therefore, $S_{(4, 616)}$ is also finitely based.
\end{proof}

\begin{pro}\label{pro52501}
$\mathsf{V}(S_{(4, 525)})$ is the ai-semiring variety defined by the identities
\begin{align}
& xy \approx yx; \label{52502}\\
& x \preceq xy; \label{52503}\\
& xy \preceq x^2+y^2y_1; \label{52504}\\
& xyy_1 \preceq x^3+yy_1; \label{52507}\\
& x^2y \preceq x^2+y^3; \label{52500}\\
& x^2y \preceq x^2+xy^2; \label{52506}\\
& x^2y \preceq x^3+y_1yy_2; \label{52505}\\
& xyz \preceq x^2y+xz+yz; \label{52508}\\
& xyzt \approx xyz+xyt+xzt+yzt, \label{52501}
\end{align}
where $y_1$ may be empty in \eqref{52504}, \eqref{52507} and \eqref{52505},
$y_2$ may be empty in \eqref{52505}.
\end{pro}
\begin{proof}
It is easy to check that both $S_{(4, 525)}$ satisfies the identities \eqref{52502}--\eqref{52501}.
In the remainder it is enough to prove that every identity that holds in $S_{(4, 525)}$
can be derived by \eqref{52502}--\eqref{52501}.
Let $\bq\preceq \bu$ be such an inequality, where
$\bu=\bu_1+\bu_2+\cdots+\bu_n$ and $\bu_i, \bq \in X^+$, $1 \leq i \leq n$.
Since $S_{53}$ and $S_{59}$ can be embedded in $S_{(4, 525)}$,
we have that both $S_{53}$ and $S_{59}$ satisfy $\bq\preceq \bu$.
By Lemma \ref{lem5301},
$L_{\geq 2}(\bu)\neq \emptyset$, $c(\bq)\subseteq c(\bu)$,
and for every $\bw\in S_2(\bq)$ there exists $\bw'\in S_2(\bu)$ such that $c(\bw')\subseteq c(\bw)$.
Moreover, the identity \eqref{52501} tells us that $\ell(\bq)\leq3$ and $\ell(\bu_i)\leq3$ for all $\bu_i\in \bu$.

\textbf{Case 1.} $\ell(\bu_i)\leq2$ for all $\bu_i\in\bu$.
By Lemma~\ref{lem5901}, we have $\ell(\bq)\leq 2$, and if $\ell(\bq)=2$, then $c(\bq)\subseteq c(L_2(\bu))$.
If $\ell(\bq)=1$, then $c(\bq)\subseteq c(\bu_j)$ for some $\bu_j\in L_2(\bu)$, and thus
\[
\bu \succeq \bu_j \stackrel{\eqref{52502}}\approx \bq\bu_j' \stackrel{\eqref{52503}}\succeq \bq.
\]
Now suppose that $\ell(\bq)=2$. Then $\bq=x^2$ or $xy$.
If $\bq=x^2$, Lemma~\ref{lem5301} yields that $\bq \preceq \bu$ is trivial.
Let $\bq=xy$. Lemma~\ref{lem5301} gives some $\bu_k\in \bu$ such that $\bu_k=x^2$, $y^2$, or $xy$.
If $\bu_k=xy$, then $\bu\succeq xy=\bq$ is immediate.
If $\bu_k=y^2$, the argument is symmetric to the case $\bu_k=x^2$; hence we may assume $\bu_k=x^2$.
In this case, we claim that there exists $\bu_\ell\in \bu$ with $\bu_\ell=y^2$.
Suppose that this is not true.
Consider the semiring homomorphism $\varphi: P_f(X^+) \to S_{(4, 525)}$ defined by
$\varphi(x)=3$, $\varphi(y)=2$, $\varphi(z)=4$ otherwise.
Then $\varphi(\bq)=1$ and $\varphi(\bu)=2$, a contradiction.
Hence such a $\bu_\ell$ exists. Consequently,
\[
\bu \succeq \bu_k+\bu_\ell=x^2+y^2 \stackrel{\eqref{52504}}\succeq xy=\bq.
\]
This derives the inequality $\bu \succeq\bq$.

\textbf{Case 2.} $\ell(\bu_i)=3$ for some $\bu_i\in \bu$.
If $\ell(\bq)=1$, then the remaining steps are analogous to Case 1.
Next, we consider $\ell(\bq)=2$. Then $\bq=x^2$ or $xy$.
If $\bq=x^2$, then the remaining steps are similar to Case 1.

Now let $\bq=xy$. By Lemma~\ref{lem5301}, there is $\bw\in S_2(\bu_k)$ such that $c(\bw)\subseteq\{x, y\}$ for some $\bu_k\in \bu$,
and so $\bw=x^2$ or $y^2$ or $xy$.
If $\bw=xy$, then $\bu_k=xy\bu_k'$ for some $\bu_k'\in X^*$, and so
\[
\bu \succeq \bu_k=xy\bu_k'  \stackrel{\eqref{52503}}\succeq xy=\bq.
\]
If $\bw=y^2$, the argument is symmetric to the case $\bw=x^2$; thus it suffices to consider $\bw=x^2$.
Assume that $\bw=x^2$. Then there is $\bu_s\in\bu$ such that $\bu_s=x^3$, $y^2\bu_s'$, or $xy\bu_s'$ for some $\bu_s'\in X^*$.
Suppose otherwise. Consider the semiring homomorphism $\varphi: P_f(X^+) \to S_{(4, 525)}$ defined by
$\varphi(x)=3$, $\varphi(y)=2$, and $\varphi(z)=4$ otherwise.
Then $\varphi(\bq)=1$ and $\varphi(\bu)=2$, a contradiction.
Hence such a $\bu_s$ exists.
If $\bu_s=x^3$, then there exists $\bu_r\in\bu$ such that $\bu_r=y\bu_r'$ for some $\bu_r'\in X^*$, and so
\[
\bu \succeq \bu_s+\bu_r=x^3+y\bu_r'\stackrel{\eqref{52507}} \succeq xy\bu_r'\stackrel{\eqref{52503}} \succeq xy=\bq
\]
If $\bu_s=y^2\bu_s'$ or $xy\bu_s'$, then
\[
\bu \succeq \bu_k+\bu_s \stackrel{\eqref{52503}} \succeq x^2+y^2\bu_s'\stackrel{\eqref{52504}} \succeq xy=\bq
\]
or
\[
\bu \succeq \bu_s=xy\bu_s'  \stackrel{\eqref{52503}}\succeq xy=\bq.
\]
This derives the inequality $\bu \succeq\bq$.

Finally, let $\ell(\bq)=3$. Then $\bq=x^3$ or $x^2y$ or $xyz$.
If $\bq=x^3$, then there is $\bu_r\in\bu$ such that $\bu_r=x^3$.
Suppose that this is not true.
Consider the semiring homomorphism $\varphi: P_f(X^+) \to S_{(4, 525)}$ defined by
$\varphi(x)=3$, $\varphi(y)=4$ otherwise.
Then $\varphi(\bq)=1$ and $\varphi(\bu)=2$, a contradiction.
Hence there is $\bu_r\in\bu$ such that $\bu_r=x^3$, and so $\bq\preceq \bu$ is trivial.

If $\bq=x^2y$, then Lemma~\ref{lem5301} tells us that there is $\bu_k\in \bu$ such that $\bu_k=x^2\bu_k'$ for some $\bu_k'\in X^*$,
and so
\[
\bu \succeq \bu_k=x^2\bu_k'  \stackrel{\eqref{52503}}\succeq x^2.
\]
Furthermore, we shall show that there is $\bu_t\in\bu$ such that $\bu_t=x^3$ or $xy^2$ or $x^2y$ or $y^3$.
Suppose that this is not true.
Consider the semiring homomorphism $\varphi: P_f(X^+) \to S_{(4, 525)}$ defined by
$\varphi(x)=\varphi(y)=3$, $\varphi(z)=4$ otherwise.
Then $\varphi(\bq)=1$ and $\varphi(\bu)=2$, a contradiction.
Hence such a $\bu_t$ exists.
If $\bu_t=x^3$, then there exists $\bu_r\in\bu$ such that $\bu_r=y\bu_r'$ for some $\bu_r'\in X^*$, and so
\[
\bu \succeq \bu_t+\bu_r=x^3+y\bu_r'\stackrel{\eqref{52505}} \succeq x^2y=\bq.
\]
If $\bu_t=xy^2$, then
\[
\bu \succeq x^2+\bu_t=x^2+xy^2\stackrel{\eqref{52506}} \succeq x^2y=\bq.
\]
If $\bu_t=y^3$, then
\[
\bu \succeq x^2+\bu_t=x^2+y^3\stackrel{\eqref{52500}} \succeq x^2y=\bq.
\]

Now let $\bq=xyz$. The identity \eqref{52502} implies that $S_2(\bq)=\{xy, yz, xz\}$.
Using the same reasoning as for $\bq=xy$ in Case 2, we derive the inequalities
$\bu\succeq xy$, $\bu\succeq yz$, and $\bu\succeq xz$.
Moreover, we shall show that there is $\bu_m\in\bu$ such that $\bu_m=x^3$ or $y^3$ or $z^3$ or $xyz$ or $x^2y$ or $x^2z$ or $y^2z$
or $y^2x$ or $z^2x$ or $z^2y$.
Suppose that this is not true.
Consider the semiring homomorphism $\varphi: P_f(X^+) \to S_{(4, 525)}$ defined by
$\varphi(x)=\varphi(y)=\varphi(z)=3$, $\varphi(t)=4$ otherwise.
Then $\varphi(\bq)=1$ and $\varphi(\bu)=2$, a contradiction.
Hence such a $\bu_m$ exists.
If $\bu_m=x^3$, then
\[
\bu \succeq \bu_m+yz=x^3+yz \stackrel{\eqref{52507}}\succeq xyz=\bq.
\]
The same reasoning applies to the cases $\bu_m = y^3$ and $\bu_m = z^3$.
If $\bu_m=x^2y$, then
\[
\bu \succeq \bu_m+yz+xz=x^2y+xz+yz \stackrel{\eqref{52508}}\succeq xyz=\bq.
\]
The remaining cases ($\bu_m = x^2z, y^2x, y^2z, z^2x, z^2y$) are analogous.

Hence, we obtain the inequality $\bu \succeq \bq$, which completes the proof.
\end{proof}

\begin{pro}\label{pro52401}
$\mathsf{V}(S_{(4, 524)})$ is the ai-semiring variety defined by the identities
\begin{align}
& xyz \approx yxz; \label{52403}\\
& x^3 \approx x^2; \label{52411}\\
& y \preceq xy; \label{52402}\\
& y^2 \preceq x_1y^2x_2; \label{52404}\\
& xyz \preceq x^2+yz; \label{52408}\\
& x^2y \preceq x^2+zy; \label{52405}\\
& xy^2 \preceq x_1xx_2+y^2; \label{52406}\\
& xy \preceq x_1yxx_2+x_3y; \label{52407}\\
& xyz \preceq xyt+yz+xz; \label{52409}\\
& xyz \preceq yx+yz+xz; \label{52410}\\
& xyzt \approx xyt+xzt+yzt, \label{52401}
\end{align}
where $x_1$ and $x_2$ may be empty in \eqref{52404}, \eqref{52406} and \eqref{52407},
$z$ may be empty in \eqref{52405},
$t$ may be empty in \eqref{52409},
$x_3$ may be empty in \eqref{52407}.
\end{pro}
\begin{proof}
It is easy to check that both $S_{(4, 524)}$ satisfies the identities \eqref{52403}--\eqref{52401}.
In the remainder it is enough to prove that every identity that holds in $S_{(4, 524)}$
can be derived by \eqref{52403}--\eqref{52401}.
Let $\bq\preceq \bu$ be such a nontrivial inequality, where
$\bu=\bu_1+\bu_2+\cdots+\bu_n$ and $\bu_i, \bq \in X^+$, $1 \leq i \leq n$.
Since $S_{53}$ is isomorphic to the quotient algebra $S_{(4, 524)}/\rho$,
where $\{2, 3\}$ is the nontrivial block of $\rho$,
we have that $S_{53}$ satisfies $\bq\preceq \bu$.
By Lemma \ref{lem5301},
$L_{\geq 2}(\bu)\neq \emptyset$, $c(\bq)\subseteq c(\bu)$,
and for every $\bw\in S_2(\bq)$ there exists $\bw'\in S_2(\bu)$ such that $c(\bw')\subseteq c(\bw)$.

The identity \eqref{52401} tells us that $\ell(\bq)\leq3$ and $\ell(\bu_i)\leq3$ for all $\bu_i\in \bu$.
We claim that there exists some $\bu_\ell\in \bu$ such that either $t(\bq)=t(\bu_\ell)$ or $m(t(\bq), \bu_\ell)\geq 2$.
Suppose otherwise.
Consider the semiring homomorphism $\varphi: P_f(X^+) \to S_{(4, 524)}$ defined by
$\varphi(t(\bq))=2$, $\varphi(t)=4$ otherwise.
Then $\varphi(\bq)=2$ or $1$, and $\varphi(\bu)=3$, a contradiction.
Hence such a $\bu_\ell$ does exist.

\textbf{Case 1.} $\ell(\bq)=1$, so $\bq=t(\bq)$.
If $t(\bq)=t(\bu_\ell)$, then we immediately obtain
\[
\bu \succeq \bu_\ell = p(\bu_\ell)\bq \stackrel{\eqref{52402}}\succeq \bq.
\]
Otherwise, we have $m(t(\bq), \bu_\ell)\geq 2$, and hence
\[
\bu \succeq \bu_\ell \stackrel{\eqref{52403}}\approx \bu_\ell'\bq^2\bu_\ell'' \stackrel{\eqref{52404}}\succeq \bq^2 \stackrel{\eqref{52402}}\succeq \bq.
\]
This proves $\bu\succeq\bq$ in this case.

\textbf{Case 2.} $\ell(\bq)=2$. Then $\bq=x^2$ or $\bq=xy$.

If $\bq=x^2$, Lemma~\ref{lem5301} ensures that there exists $\bu_k\in\bu$ with $\bu_k=\bu_k'x^2\bu_k''$ for some $\bu_k',\bu_k''\in X^*$. Hence
\[
\bu \succeq \bu_k = \bu_k'x^2\bu_k'' \stackrel{\eqref{52404}}\succeq x^2=\bq.
\]

Now assume $\bq=xy$. By Lemma~\ref{lem5301}, there is $\bw\in S_2(\bu)$ such that $c(\bw)\subseteq\{x,y\}$; thus $\bw\in\{x^2,y^2,xy,yx\}$.

\textbf{Subcase 2.1.} $\bw=x^2$. Then $\bu_k=p_1x^2p_2$ for some $\bu_k\in\bu$ and $p_1,p_2\in X^*$; by \eqref{52404}, $\bu_k\succeq x^2$.
Since $c(\bq)\subseteq c(\bu)$, there is $\bu_\ell\in\bu$ such that $y\in c(\bu_\ell)$.
If $\bu_\ell=p(\bu_\ell)y$, then
\[
\bu \succeq x^2+\bu_\ell=x^2+p(\bu_\ell)y \stackrel{\eqref{52405}, \eqref{52402}}\succeq xy=\bq.
\]
Otherwise, $\bu_\ell=\bu_\ell' y^2\bu_\ell''$, and thus
\[
\bu \succeq x^2+\bu_\ell=x^2+\bu_\ell' y^2\bu_\ell'' \stackrel{\eqref{52404}}\succeq x^2+y^2 \stackrel{\eqref{52405}, \eqref{52402}}\succeq xy=\bq.
\]

\textbf{Subcase 2.2.} $\bw=y^2$. Then $\bu_k=p_1y^2p_2$ for some $\bu_k\in\bu$ and $p_1,p_2\in X^*$;
hence $\bu_k\succeq y^2$ by \eqref{52404}.
As $c(\bq)\subseteq c(\bu)$,
we can choose $\bu_r\in\bu$ such that $x\in c(\bu_r)$;
write $\bu_r=p_3xp_4$. Then
\[
\bu \succeq y^2+\bu_r = y^2+p_3xp_4 \stackrel{\eqref{52406}, \eqref{52402}}\succeq xy=\bq.
\]

\textbf{Subcase 2.3.} $\bw=xy$. Then $\bu_k=xy\bu_k'$ or $\bu_k=\bu_k'xy$ for some $\bu_k'\in X^*$.
If $\bu_k=\bu_k'xy$, then by \eqref{52402},
\[
\bu \succeq \bu_k \succeq xy=\bq.
\]
If $\bu_k=xy\bu_k'$, we choose $\bu_\ell\in\bu$ such that $\bu_\ell=\bu_\ell' y^2\bu_\ell''$ (which exists by the same argument as above).
Then
\[
\bu \succeq \bu_\ell+\bu_k=\bu_\ell'y^2\bu_\ell''+xy\bu_k' \stackrel{\eqref{52404}}\succeq y^2+xy\bu_k' \stackrel{\eqref{52406}, \eqref{52402}}\succeq xy=\bq.
\]
If instead no such $\bu_\ell$ exists, then we must have $\bu_\ell=\bu_\ell' y$ for some $\bu_\ell'\in X^*$, and hence
\[
\bu \succeq \bu_\ell+\bu_k=\bu_\ell'y+xy\bu_k' \stackrel{\eqref{52407}}\succeq xy=\bq.
\]

\textbf{Subcase 2.4.} $\bw=yx$. Then $\bu_k=\bu_k'yx\bu_k''$ for some $\bu_k',\bu_k''\in X^*$.
If there is $\bu_\ell\in\bu$ such that $\bu_\ell=p(\bu_\ell)y$, then
\[
\bu \succeq \bu_k+\bu_\ell=\bu_k'yx\bu_k''+p(\bu_\ell)y \stackrel{\eqref{52407}}\succeq xy=\bq.
\]
Otherwise, we may take $\bu_\ell=\bu_\ell' y^2\bu_\ell''$, and then
\[
\bu \succeq \bu_\ell+\bu_k=\bu_\ell' y^2\bu_\ell''+\bu_k'yx\bu_k''
\stackrel{\eqref{52404}}\succeq y^2+\bu_k'yx\bu_k''
\stackrel{\eqref{52407}}\succeq xy=\bq.
\]

Thus, in every subcase, the inequality $\bu\succeq\bq$ is derived.

\textbf{Case 3.} $\ell(\bq)=3$. Then $\bq=x^3$ or $x^2y$ or $xy^2$ or $xyz$.

First, suppose that $\bq=x^3$.
By Lemma~\ref{lem5301}, there exists $\bu_k\in\bu$ such that $\bu_k=\bu_k'x^2\bu_k''$ for some $\bu_k',\bu_k''\in X^*$.
Hence, using \eqref{52404},
\[
\bu \succeq \bu_k = \bu_k'x^2\bu_k'' \succeq x^2.
\]
Since identity \eqref{52411} gives $x^2\approx x^3=\bq$, we obtain $\bu\succeq\bq$.

Next, assume that $\bq=x^2y$.
Since $c(\bq)\subseteq c(\bu)$, it follows that there is $\bu_r\in\bu$ such that $y\in c(\bu_r)$, and so $\bu_r=p_3yp_4$.
Hence we have
\[
\bu\succeq x^2+\bu_r=x^2+p_3y \stackrel{\eqref{52405}}\succeq x^2y=\bq
\]
or
\[
\bu\succeq x^2+\bu_r=x^2+p_3y^2p_4 \stackrel{\eqref{52404}}\succeq x^2+y^2 \stackrel{\eqref{52405}}\succeq x^2y=\bq.
\]

Now suppose that $\bq=xy^2$.
Lemma \ref{lem5301} tells us that there is $\bu_k\in \bu$ such that $\bu_k=\bu_k'y^2\bu_k''$, and so
\[
\bu \succeq \bu_k =\bu_k'y^2\bu_k'' \stackrel{\eqref{52404}}\succeq y^2.
\]
Since $c(\bq)\subseteq c(\bu)$, it follows that there is $\bu_r\in\bu$ such that $x\in c(\bu_r)$, and so $\bu_r=p_3xp_4$.
Hence we have
\[
\bu\succeq y^2+\bu_r=y^2+p_3xp_4 \stackrel{\eqref{52406}}\succeq xy^2=\bq.
\]

Finally, consider $\bq=xyz$. The identity \eqref{52403} implies that $S_2(\bq)=\{xy, yz, xz\}$.
Using a method similar to Case 2, we can obtain $\bu \succeq xyy_1$ or $\bu \succeq yx$, $\bu \succeq xz$, and $\bu \succeq yz$.
Furthermore, we have
\[
\bu \succeq xyy_1+xz+yz \stackrel{\eqref{52409}}\succeq xyz
\]
or
\[
\bu \succeq xz+yz+yx \stackrel{\eqref{52410}}\succeq xyz.
\]
This completes the proof.
\end{proof}

\begin{cor}
The ai-semiring $S_{(4, 531)}$ is finitely based.
\end{cor}
\begin{proof}
Notice that $S_{(4, 531)}$ and $S_{(4, 524)}$ have dual multiplications,
therefore, $S_{(4, 531)}$ is also finitely based.
\end{proof}

\begin{pro}\label{pro52701}
$\mathsf{V}(S_{(4, 527)})$ is the ai-semiring variety defined by the identities
\begin{align}
& xy \approx yx; \label{52701}\\
& x \preceq xy; \label{52702}\\
& xy \preceq x^2+y^2; \label{52704}\\
& xyz \approx xy+yz+xz. \label{52703}
\end{align}
\end{pro}
\begin{proof}
It is easy to check that $S_{(4, 527)}$ satisfies the identities \eqref{52701}--\eqref{52703}.
In the remainder it is enough to prove that every identity that holds in $S_{(4, 527)}$
can be derived by \eqref{52701}--\eqref{52703}.
Let $\bq\preceq \bu$ be such a nontrivial inequality, where
$\bu=\bu_1+\bu_2+\cdots+\bu_n$ and $\bu_i, \bq \in X_c^+$, $1 \leq i \leq n$.
Since $S_{53}$ and $S_{60}$ can be embedded in $S_{(4, 527)}$,
we have both $S_{53}$ and $S_{60}$ satisfy $\bq\preceq \bu$.
By Lemmas \ref{lem5301} and \ref{lem6001},
$L_{\geq 2}(\bu)\neq \emptyset$, $c(\bq)\subseteq c(\bu)$,
and for any $\bw\in S_2(\bq)$, there exists $\bw'\in S_2(\bu)$ such that $c(\bw')\subseteq c(\bw)$.

\textbf{Case 1.} $\ell(\bq)=1$.
Then $c(\bq)\subseteq c(\bu_j)$ for some $\bu_j\in L_{\ge2}(\bu)$, and hence
\[
\bu \succeq \bu_j \stackrel{\eqref{52701}}\approx \bq\bu_j' \stackrel{\eqref{52702}}\succeq \bq,
\]
which yields the desired inequality $\bu\succeq\bq$.

\textbf{Case 2.} $\ell(\bq)\geq 2$.
Let $\bq=x_1x_2\cdots x_n$, where $n\geq 2$. Then
by the identity \eqref{52703} we deduce the identity
\[
\bq \approx\sum\limits_{1\leq i\textless j\leq n}x_ix_j.
\]
So we only need to consider the case that $\bq=xy$.
Similarly, we can obtain $\bu=L_1(\bu)+L_2(\bu)$.

\textbf{Subcase 2.1.} $x=y$. Then $\bq=x^2$.
By Lemma~\ref{lem5301}, there exists $\bu_j\in L_2(\bu)$ such that $c(\bu_j)\subseteq\{x\}$.
Since $\ell(\bu_j)=2$, this forces $\bu_j=x^2=\bq$.
Hence $\bu\succeq\bq$ is immediate.

\textbf{Subcase 2.2.} $x\neq y$.
By Lemma~\ref{lem5301}, there exists $\bu_k\in L_2(\bu)$ such that $c(\bu_k)\subseteq\{x,y\}$.
Thus $\bu_k$ is one of $x^2$, $y^2$, or $xy$.
If $\bu_k=xy$, then $\bu\succeq xy=\bq$ is immediate.
If $\bu_k=y^2$, the argument is symmetric to the case $\bu_k=x^2$; hence we may assume $\bu_k=x^2$.
In this case, we claim that there exists $\bu_\ell\in \bu$ such that $\bu_\ell=y^2$.
Suppose otherwise.
Consider the semiring homomorphism $\varphi: P_f(X^+) \to S_{(4, 527)}$ defined by
$\varphi(x)=3$, $\varphi(y)=2$, $\varphi(z)=4$ otherwise.
It is easy to see that $\varphi(\bu)=2$ and $\varphi(\bq)=1$, a contradiction.
Therefore such a $\bu_\ell$ exists. Consequently,
\[
\bu \succeq x^2+y^2 \stackrel{\eqref{52704}}\succeq xy=\bq.
\]
This proves the desired inequality $\bu\succeq\bq$.
\end{proof}

\begin{pro}\label{pro50101}
$\mathsf{V}(S_{(4, 501)})$ is the ai-semiring variety defined by the identities
\begin{align}
& xyz \approx xzy; \label{50103}\\
& y \preceq xyz; \label{50101}\\
& yt \preceq xy+zt; \label{50104}\\
& xyz \approx xy+yz+xz, \label{50102}
\end{align}
where $x$ and $z$ may be empty in \eqref{50102}.
\end{pro}
\begin{proof}
It is easy to check that $S_{(4, 501)}$ satisfies the identities \eqref{50103}--\eqref{50102}.
In the remainder it is enough to prove that every identity that holds in $S_{(4, 501)}$
can be derived by \eqref{50103}--\eqref{50102}.
Let $\bq\preceq \bu$ be such a nontrivial inequality, where
$\bu=\bu_1+\bu_2+\cdots+\bu_n$ and $\bu_i, \bq \in X^+$, $1 \leq i \leq n$.
Since $S_{54}$ and $S_{60}$ can be embedded in $S_{(4, 501)}$, they both satisfy $\bq\preceq\bu$.
Note that the multiplications of $S_{54}$ and $S_{57}$ are dual to each other.
By Lemmas~\ref{lem5701} and~\ref{lem6001}, we have that
$L_{\geq2}(\bu)\neq\emptyset$, $c(s(\bq))\subseteq c(s(\bu))$, $h(\bq)\in c(\bu)$,
and, if $\ell(\bq)\geq2$, then $c(\bq)\subseteq c(L_{\geq2}(\bu))$.

\textbf{Case 1.} $\ell(\bq)=1$.
Then $c(\bq)\subseteq c(\bu_j)$ for some $\bu_j\in L_{\geq2}(\bu)$, and hence
\[
\bu \succeq \bu_j = \bu_j'\bq\bu_j'' \stackrel{\eqref{50101}}\succeq \bq,
\]
which yields the desired inequality $\bu\succeq\bq$.

\textbf{Case 2.} $\ell(\bq)\geq 2$.
Write $\bq=x_1x_2\cdots x_n$ with $n\geq2$.
From identity \eqref{50102}, we derive
\[
\bq \approx \sum_{1\leq i<j\leq n} x_ix_j.
\]
Thus it suffices to consider the case $\bq=xy$.
Similarly, we can obtain $\bu=L_1(\bu)+L_2(\bu)$.

\textbf{Subcase 2.1.} There exists $\bu_k\in L_2(\bu)$ such that $x,y\in c(\bu_k)$.
Then $\bu_k=xy$ or $\bu_k=yx$.
If $\bu_k=yx$, then by the dual of Lemma~\ref{lem5701}, we have $c(s(\bq))\subseteq c(s(\bu))$.
Consequently, there exists $\bu_\ell\in L_2(\bu)$ such that $\bu_\ell=s(\bu_\ell)y$, and hence
\[
\bu \succeq \bu_k+\bu_\ell = yx+s(\bu_\ell)y \stackrel{\eqref{50104}}\succeq xy=\bq.
\]
Thus $\bu\succeq\bq$ in this subcase.

\textbf{Subcase 2.2.} For every $\bu_k\in L_2(\bu)$, either $x\notin c(\bu_k)$ or $y\notin c(\bu_k)$.
We claim that there exists $\bu_\ell\in L_2(\bu)$ with $t(\bu_\ell)=x$.
Suppose otherwise. Define a semiring homomorphism $\varphi:P_f(X^+)\to S_{(4,501)}$ by
$\varphi(x)=2$, $\varphi(y)=3$, and $\varphi(z)=4$ for all other $z$.
Then $\varphi(\bu)=2$ while $\varphi(\bq)=1$, a contradiction.
Hence such $\bu_\ell$ exists; write $\bu_\ell=h(\bu_\ell)x$.
Since $c(s(\bq))\subseteq c(s(\bu))$, there also exists $\bu_r\in L_2(\bu)$ such that $\bu_r=h(\bu_r)y$.
Consequently,
\[
\bu \succeq \bu_\ell+\bu_r = h(\bu_\ell)x + h(\bu_r)y \stackrel{\eqref{50104}}\succeq xy=\bq.
\]
Thus the inequality $\bu\succeq\bq$ is derived.
\end{proof}

\begin{cor}
The ai-semiring $S_{(4, 526)}$ is finitely based.
\end{cor}
\begin{proof}
Notice that $S_{(4, 526)}$ and $S_{(4, 501)}$ have dual multiplications,
therefore, $S_{(4, 526)}$ is also finitely based.
\end{proof}

\begin{pro}\label{pro53701}
$\mathsf{V}(S_{(4, 537)})$ is the ai-semiring variety defined by the identities
\begin{align}
& y \preceq xyz; \label{53702}\\
& xy \preceq x^2+y_1yy_2; \label{53704}\\
& xyz \approx xy+yz+xz, \label{53701}
\end{align}
where $x$ and $z$ may be empty in \eqref{53702},
$y_1$ and $y_2$ may be empty in \eqref{53704}.
\end{pro}
\begin{proof}
It is easy to check that $S_{(4, 537)}$ satisfies the identities \eqref{53702}--\eqref{53701}.
In the remainder it is enough to prove that every identity that holds in $S_{(4, 537)}$
can be derived by \eqref{53702}--\eqref{53701}.
Let $\bq\preceq \bu$ be such a nontrivial inequality, where
$\bu=\bu_1+\bu_2+\cdots+\bu_n$ and $\bu_i, \bq \in X^+$, $1 \leq i \leq n$.
Since $S_{53}$ and $S_{57}$ can be embedded in $S_{(4, 537)}$,
we have both $S_{53}$ and $S_{57}$ satisfy $\bq\preceq \bu$.
By Lemmas \ref{lem5301} and \ref{lem5701},
$L_{\geq 2}(\bu)\neq \emptyset$, $c(p(\bq))\subseteq c(p(\bu))$, $t(\bq)\in c(\bu)$,
and for any $\bw\in S_2(\bq)$, there exists $\bw'\in S_2(\bu)$ such that $c(\bw')\subseteq c(\bw)$.

\textbf{Case 1.} $\ell(\bq)=1$.
Then $c(\bq)\subseteq c(\bu_j)$ for some $\bu_j\in L_{\geq2}(\bu)$, and hence
\[
\bu \succeq \bu_j = \bu_j'\bq\bu_j'' \stackrel{\eqref{53702}}\succeq \bq,
\]
which yields the desired inequality $\bu\succeq\bq$.

\textbf{Case 2.} $\ell(\bq)\geq 2$.
Let $\bq=x_1x_2\cdots x_n$, where $n\geq 2$.
From identity \eqref{53701} we derive
\[
\bq \approx\sum\limits_{1\leq i\textless j\leq n}x_ix_j.
\]
So we only need to consider the case that $\bq=xy$.
Similarly, we can obtain $\bu=L_1(\bu)+L_2(\bu)$.

\textbf{Subcase 2.1.} $x=y$. Then $\bq=x^2$.
By Lemma~\ref{lem5301}, there exists $\bu_i\in L_2(\bu)$ such that $c(\bu_i)\subseteq\{x\}$.
Since $\ell(\bu_i)=2$, this forces $\bu_i=x^2=\bq$.
Hence $\bu\succeq\bq$ is immediate.

\textbf{Subcase 2.2.} $x\neq y$.
By Lemma~\ref{lem5301}, there exists $\bu_k\in L_2(\bu)$ such that $c(\bu_k)\subseteq\{x,y\}$.
Thus $\bu_k$ is one of $x^2$, $y^2$, $xy$ or $yx$.
If $\bu_k=xy$, then $\bu\succeq xy=\bq$ is immediate.
If $\bu_k=x^2$, then, since $t(\bq)\in c(\bu)$, there exists $\bu_r\in\bu$ such that $y\in c(\bu_r)$.
Consequently,
\[
\bu \succeq x^2+\bu_r= x^2+\bu_r'y\bu_r'' \stackrel{\eqref{53704}}\succeq xy=\bq.
\]
If $\bu_k=y^2$ or $yx$,  we claim that there exists $\bu_\ell\in\bu$ with $\bu_\ell=x^2$.
Suppose otherwise.
Consider the semiring homomorphism $\varphi: P_f(X^+) \to S_{(4, 537)}$ defined by
$\varphi(x)=2$, $\varphi(y)=3$, $\varphi(z)=4$ otherwise.
It is easy to see that $\varphi(\bu)=2$ and $\varphi(\bq)=1$, a contradiction.
Hence such a $\bu_\ell$ exists.
The remaining argument is analogous to the case $\bu_k=x^2$ above, and yields $\bu\succeq xy=\bq$.
\end{proof}

\begin{cor}
The ai-semiring $S_{(4, 528)}$ is finitely based.
\end{cor}
\begin{proof}
Notice that $S_{(4, 528)}$ and $S_{(4, 537)}$ have dual multiplications,
therefore, $S_{(4, 528)}$ is also finitely based.
\end{proof}

\begin{pro}\label{pro62801}
$\mathsf{V}(S_{(4, 628)})$ is the ai-semiring variety defined by the identities
\begin{align}
& xyx \approx xy; \label{62804}\\
& x^2y \approx xy;  \label{62802}\\
& xy^2 \approx xy; \label{62805}\\
& xy \approx (xy)^2; \label{62808}\\
& x \preceq x^2; \label{62801}\\
& yx \preceq x^2+yz; \label{62803}\\
& yx \preceq x+yxz; \label{62806}\\
& x+xy \approx x^2+xy. \label{62807}
\end{align}
\end{pro}
\begin{proof}
It is easy to check that $S_{(4, 628)}$ satisfies the identities \eqref{62804}--\eqref{62807}.
In the remainder it is enough to show that every ai-semiring identity of $S_{(4, 628)}$
is derivable from \eqref{62804}--\eqref{62807}.
Let $\bq\preceq \bu$ be such a nontrivial inequality, where
$\bu=\bu_1+\bu_2+\cdots+\bu_n$ and $\bu_i, \bq \in X^+$, $1 \leq i \leq n$.
Since $D_2$ is isomorphic to $\{3, 4\}$, it follows that $D_2$ satisfies $\bq\preceq \bu$
and so $D_{\bq}(\bu)$ is nonempty. So there exists $\bu_{i_1} \in \bu$ such that $c(\bu_{i_1})\subseteq c(\bq)$.
It is easy to see that $S_{58}$ is isomorphic to $\{1, 2, 4\}$
and so $S_{58}$ satisfies $\bq\preceq \bu$.
If $\ell(\bq)\geq 2$, then by Lemma \ref{lem5801} $L_{\geq 2}(\bu)\cap H_{\bq}(\bu)$ is nonempty
and so $\ell(\bu_k)\geq 2$ and $h(\bu_k)=h(\bq)$ for some $\bu_k \in \bu$.

Let $|c(\bq)|=1$.
If $\ell(\bq)=1$, then $\bu_{i_1}=\bq^\ell$ for some $\ell\geq 2$. So we have
\[
\bu \succeq \bu_{i_1} =\bq^\ell \stackrel{\eqref{62801}}\succeq \bq.
\]
Now assume that $\ell(\bq)\geq 2$. Then $\bq=x^r$ for some $x \in X$ and some $r \geq 2$.
This implies that $\bu_{i_1}=x^s$ for some $s \geq 1$ and $\bu_k=xs(\bu_k)$, where $s(\bu_k)$ is nonempty.
If $s=1$, then $\bu_{i_1}=x$ and so
\[
\bu\succeq \bu_{i_1}+\bu_k=x+xs(\bu_k) \stackrel{\eqref{62807}}\approx x^2+xs(\bu_k)\stackrel{\eqref{62802}}
\approx x^r+xs(\bu_k)=\bq+xs(\bu_k).
\]
If $s\geq 2$, then
\[
\bu \succeq \bu_{i_1}\approx x^s \stackrel{\eqref{62802}}\approx x^r=\bq.
\]
This implies the inequality $\bq\preceq \bu$.

Now let $|c(\bq)|\geq 2$.
By \eqref{62804}--\eqref{62805} we deduce the identity $\bq \approx i(\bq)$
and so $i(\bq) \preceq \bu$ is satisfied by $S_{(4, 628)}$.
Let $i(\bq)=x_1x_2\cdots x_m$, $m\geq 2$.
One can check that $S_{41}$ is isomorphic to the quotient algebra $S_{(4, 628)}/\rho$,
where $\{1, 2\}$ is the nontrivial block of $\rho$. So $S_{41}$ satisfies $i(\bq) \preceq \bu$.
By Lemma \ref{lem4101} it follows that for any $2\leq j\leq m$ and $Y_j=\{x_1, x_2, \ldots ,x_{j-1}\}$,
there exists $\bu_j \in \bu$
such that $\bu_j=\bu_{j_1}x_j\bu_{j_2}$ for some $\bu_{j_1}, \bu_{j_2} \in X^*$, where $c(\bu_{j_1})\subseteq Y_j$.

\textbf{Case 1.}
$D_\bq(\bu)\cap L_{\geq 2}(\bu)$ is nonempty.
Choose a word $\bp$ in $D_\bq(\bu)\cap L_{\geq 2}(\bu)$.
We shall show by induction on $j$ that $x_1x_2\cdots x_j\bp \preceq \bu$
is derivable from \eqref{62804}--\eqref{62807}, $1 \leq j\leq m$.
If $j=1$, then
\[
\bu \succeq \bp+\bu_k \stackrel{\eqref{62808}}\approx \bp^2+x_1s(\bu_k) \stackrel{\eqref{62803}}\succeq x_1\bp.
\]
This implies $\bu \succeq x_1\bp$.
Let $2\leq j\leq m$.
Suppose that $\bu\succeq x_1x_2\cdots x_{j-1}\bp$ is derivable from \eqref{62804}--\eqref{62807}.
If $h(\bp)=x_j$, then
\[
\bu\approx \bu+x_1x_2\cdots x_{j-1}\bp \stackrel{\eqref{62802}}\approx x_1x_2\cdots x_{j-1}x_j\bp.
\]
Now assume that $h(\bp)\neq x_j$.
We shall show that there exists $\bu_j\in L_{\geq2}(\bu)$ such that $h_{Y_j}(\bq)=h_{Y_j}(\bu_j)$.
If this is not true,
then we consider the semiring homomorphism $\varphi: P_f(X^+) \to S_{(4, 628)}$ defined by
$\varphi(x)=3$ if $x\in Y_j$, $\varphi(x)=2$ if $x=x_j$, and $\varphi(x)=4$ otherwise.
It is easy to see that $\varphi(\bu)=2$ and $\varphi(i(\bq))=1$, a contradiction.
Hence there exists $\bu_j\in L_{\geq2}(\bu)$ such that $h_{Y_j}(\bq)=h_{Y_j}(\bu_j)$.
Then $\bu_{j}=\bu_{j_1}x_j\bu_{j_2}$,
where $c(\bu_{j_1})\subseteq Y_j$,
$\bu_{j_1}$ and $\bu_{j_2}$ can not be empty words simultaneously.
Now we have
\begin{align*}
\bu
&\succeq \bu_j+x_1x_2\cdots x_{j-1}\bp  \\
&\approx \bu_j^2+x_1x_2\cdots x_{j-1}\bp  &&(\text{by}~\eqref{62808})\\
&= (\bu_{j_1}x_j\bu_{j_2})^2+x_1x_2\cdots x_{j-1}\bp \\
&\succeq x_1x_2\cdots x_{j-1}\bu_{j_1}x_j\bu_{j_2} &&(\text{by}~\eqref{62803})\\
&\approx x_1x_2\cdots x_{j-1}x_j\bu_{j_2}.
&&(\text{by}~\eqref{62804}, \eqref{62802})
\end{align*}
This implies the inequality
\[
\bu\succeq x_1x_2\cdots x_{j-1}x_j\bu_{j_2}.
\]
Furthermore, we can deduce
\begin{align*}
\bu
&\succeq x_1x_2\cdots x_{j-1}\bp+x_1x_2\cdots x_j\bu_{j_2}\\
&\approx (x_1x_2\cdots x_{j-1}\bp)^2+x_1x_2\cdots x_j\bu_{j_2} &&(\text{by}~\eqref{62808})\\
&\succeq x_1x_2\cdots x_j\bp.
&&(\text{by}~\eqref{62803}, \eqref{62804}, \eqref{62802})
\end{align*}
So we can obtain $\bu \succeq x_1x_2\cdots x_{j-1}x_j\bp$.

Take $j=m$. We obtain the inequality $\bu\succeq x_1x_2\cdots x_m\bp$. So we deduce
\[
\bu\succeq x_1x_2\cdots x_m\bp\approx i(\bq)\bp
\stackrel{\eqref{62804}, \eqref{62802}, \eqref{62805}}\approx i(\bq)\approx \bq.
\]
This derives $\bu\succeq \bq$.

\textbf{Case 2.} $D_\bq(\bu)\cap L_{\geq 2}(\bu)$ is empty. Then $D_\bq(\bu)\subseteq L_{1}(\bu)$.
If $\bu_i=x_1$ for some $\bu_i \in D_\bq(\bu)$, then
\[
\bu\succeq \bu_i+\bu_k =x_1+x_1s(\bu_k) \stackrel{\eqref{62807}}\approx x_1^2+x_1s(\bu_k).
\]
The remaining steps are similar to Case 1.
If $\bu_i\neq x_1$ for all $\bu_i \in D_\bq(\bu)$,
then there exists $2\leq i\leq n$ such that $x_i \in D_\bq(\bu)$, but $x_1, \ldots, x_{i-1} \notin D_\bq(\bu)$.
We shall prove that there exists $\bu_s \in \bu$ such that either $x_jx_i$ is a prefix of $\bu_s$ for some $1\leq j\leq i-1$
or $\bu_s=x_is(\bu_s)$ and $ \ell(\bu_s)\geq 2$.
Consider the semiring homomorphism $\varphi: P_f(X^+) \to S_{(4, 628)}$ defined by
$\varphi(x_i)=2$, $\varphi(x)=3$ if $x\in Y_i$, and $\varphi(x)=4$ otherwise.
It is easy to see that $\varphi(\bu)=2$ and $\varphi(\bq)=1$, a contradiction.
If there exists $\bu_s \in \bu$ such that $x_jx_i$ is a prefix of $\bu_s$ for some $1\leq j\leq i-1$,
then $\bu_s=x_jx_i\bu_s'$ for some word $\bu_s'$ and so
\[
\bu\succeq x_i+\bu_s=x_i+x_jx_i\bu_s' \stackrel{\eqref{62806}}\succeq x_jx_i.
\]
The remaining steps are similar to Case 1.
If there exists $\bu_s \in \bu$ such that $\bu_s=x_is(\bu_s)$ and $ \ell(\bu_s)\geq 2$,
then we have
\[
\bu\succeq x_i+\bu_s=x_i+x_is(\bu_s) \stackrel{\eqref{62803}}\approx {x_i}^2+x_is(\bu_s).
\]
The remaining steps are similar to Case 1.
\end{proof}

\begin{cor}
The ai-semiring $S_{(4, 504)}$ is finitely based.
\end{cor}
\begin{proof}
Notice that $S_{(4, 504)}$ and $S_{(4, 628)}$ have dual multiplications,
therefore, $S_{(4, 504)}$ is also finitely based.
\end{proof}

\begin{pro}\label{pro65401}
$\mathsf{V}(S_{(4, 654)})$ is the ai-semiring variety defined by the identities
\begin{align}
& x \preceq xy; \label{65401}\\
& xyz \approx xy+xz; \label{65402}\\
& xy \preceq x^2+xx_1+yy_1; \label{65403}\\
& xy \preceq x^2+xx_1+y_1y, \label{65404}
\end{align}
where $y_1$ may be empty in \eqref{65403}  and \eqref{65404}.
\end{pro}
\begin{proof}
It is easy to check that $S_{(4, 654)}$ satisfies the identities \eqref{65401}--\eqref{65404}.
In the remainder it is enough to prove that every identity that holds in $S_{(4, 654)}$
can be derived by \eqref{65401}--\eqref{65404}.
Let $\bq\preceq \bu$ be such a nontrivial inequality, where
$\bu=\bu_1+\bu_2+\cdots+\bu_n$ and $\bu_i, \bq \in X^+$, $1 \leq i \leq n$.
It is easy to check that $M_2$ is isomorphic to $\{3, 4\}$
and so $M_2$ satisfies $\bq\preceq \bu$.
This implies that $c(\bq)\subseteq c(\bu)$.
Since $S_{58}$ is isomorphic to $\{1, 2, 3\}$, it follows that
$S_{58}$ satisfies $\bq\preceq \bu$.
By Lemmas \ref{lem5801},
$L_{\geq 2}(\bu) \neq \emptyset$, $H_{\bq}(\bu) \neq \emptyset$.

\textbf{Case 1.} $\ell(\bq)=1$.
It follows that $h(\bq)=h(\bu_j)$ for some $\bu_j \in L_{\geq 2}(\bu)$, and so
\[
\bu \succeq \bu_j =h(\bq)s(\bu_j)=\bq s(\bu_j)\stackrel{\eqref{65401}}\succeq \bq.
\]
This derives the inequality $\bu \succeq\bq$.

\textbf{Case 2.} $\ell(\bq)\geq 2$.
Lemma \ref{lem5801} tells us that there exists $\bu_j\in L_{\geq 2}(\bu)$ such that $h(\bu_j)=h(\bq)$.
Let $\bq=x_1x_2\cdots x_n$, where $n\geq 2$. Then
by the identity \eqref{65402} we deduce the identity
\[
\bq \approx\sum\limits_{i=2}^nx_1x_i.
\]
So we only need to consider the case that $\bq=xy$.
Similarly, we can obtain $\bu=L_1(\bu)+L_2(\bu)$.

\textbf{Subcase 2.1.} $x=y$.
We claim that there exists $\bu_k\in\bu$ such that $\bu_k=x^2$.
Suppose otherwise. Define a semiring homomorphism $\varphi:P_f(X^+)\to S_{(4,654)}$ by
$\varphi(x)=2$ and $\varphi(z)=4$ for all $z\neq x$.
Then $\varphi(\bu)=2$, $\varphi(\bq)=1$, a contradiction.
Hence such $\bu_k$ exists; consequently $\bu\succeq x^2=\bq$.

\textbf{Subcase 2.2.} $x\neq y$.
Since $\bq\preceq \bu$ be such a nontrivial inequality,
we shall show that there exists $\bu_\ell\in\bu$ such that $\bu_\ell=x^2$.
Suppose otherwise.
Consider the semiring homomorphism $\varphi: P_f(X^+) \to S_{(4, 654)}$ defined by
$\varphi(x)=2$, $\varphi(y)=3$, $\varphi(z)=4$ otherwise.
Then $\varphi(\bu)=2$ and $\varphi(\bq)=1$, a contradiction.
Thus such $\bu_\ell$ exists.
Since $c(\bq)\subseteq c(\bu)$, we have that there is $\bu_r\in\bu$ such that $y\in \bu_r$.
Now we have
\[
\bu \succeq x^2+\bu_r= x^2+y\bu_r' \stackrel{\eqref{65403}}\succeq xy=\bq
\]
or
\[
\bu \succeq x^2+\bu_r= x^2+\bu_r'y \stackrel{\eqref{65404}}\succeq xy=\bq.
\]
In either case, the inequality $\bu\succeq xy=\bq$ is derived.
This completes the proof.
\end{proof}

\begin{cor}
The ai-semiring $S_{(4, 529)}$ is finitely based.
\end{cor}
\begin{proof}
Notice that $S_{(4, 529)}$ and $S_{(4, 654)}$ have dual multiplications,
therefore, $S_{(4, 529)}$ is also finitely based.
\end{proof}

\begin{pro}\label{pro66701}
$\mathsf{V}(S_{(4, 667)})$ is the ai-semiring variety defined by the identities
\begin{align}
& x^3 \approx x^2; \label{66702}\\
& xyz \approx xyt; \label{66704}\\
& xy \preceq x^2; \label{66703}\\
& x^2 \preceq x^2y, \label{66701}
\end{align}
where $y$ may be empty in \eqref{66703}.
\end{pro}
\begin{proof}
It is easy to check that $S_{(4, 667)}$ satisfies the identities \eqref{66702}--\eqref{66701}.
In the remainder it is enough to prove that every identity that holds in $S_{(4, 667)}$
can be derived by \eqref{66702}--\eqref{66701}.
Let $\bq\preceq \bu$ be such a nontrivial inequality, where
$\bu=\bu_1+\bu_2+\cdots+\bu_n$ and $\bu_i, \bq \in X^+$, $1 \leq i \leq n$.
Since $S_{58}$ is isomorphic to $\{1, 2, 3\}$, it follows that
$S_{58}$ satisfies $\bq\preceq \bu$.
By Lemmas \ref{lem5801},
$L_{\geq 2}(\bu) \neq \emptyset$, $H_{\bq}(\bu) \neq \emptyset$.

\textbf{Case 1.} $\ell(\bq)=1$.
We claim that there exists $\bu_k\in\bu$ such that $\bu_k=\bq^2\bu_k'$ for some $\bu_k'\in X^*$.
Suppose otherwise. Define a semiring homomorphism $\varphi:P_f(X^+)\to S_{(4,667)}$ by
$\varphi(\bq)=2$ and $\varphi(z)=4$ for all $z\neq \bq$.
Then $\varphi(\bu)=3$ and $\varphi(\bq)=2$, a contradiction.
Hence such $\bu_k$ exists. Consequently,
\[
\bu \succeq \bu_k = \bq^2\bu_k' \stackrel{\eqref{66701}}\succeq \bq^2 \stackrel{\eqref{66703}}\succeq \bq.
\]
This derives the inequality $\bu \succeq\bq$.

\textbf{Case 2.} $\ell(\bq)\geq 2$.
Lemma \ref{lem5801} tells us that there exists $\bu_j\in L_{\geq 2}(\bu)$ such that $h(\bu_j)=h(\bq)$.
Let $\bq=x_1x_2\cdots x_n$, where $n\geq 2$.

\textbf{Subcase 2.1.} $x_1=x_2$. Then $\bq=x_1^2\bq_1$.
We shall show that there exists $\bu_k\in\bu$ such that $\bu_k=x_1^2\bu_k'$.
Suppose that this is not true.
Consider the semiring homomorphism $\varphi: P_f(X^+) \to S_{(4, 667)}$ defined by
$\varphi(x_1)=2$, $\varphi(z)=4$ otherwise.
It follows that $\varphi(\bu)=3$ or $2$, and $\varphi(\bq)=1$, a contradiction.
Hence there exists $\bu_k\in\bu$ such that $\bu_k=x_1^2\bu_k'$.
Then
\[
\bu \succeq \bu_k=x_1^2\bu_k'\stackrel{\eqref{66704}}\approx x_1^2\bq_1=\bq.
\]

\textbf{Subcase 2.2.} $x_1\neq x_2$.
We claim that there exists $\bu_\ell\in\bu$ such that either $\bu_\ell=x_1x_2\bu_\ell'$ or $\bu_\ell=x_1^2\bu_\ell''$ for some $\bu_\ell',\bu_\ell''\in X^*$.
Suppose not. Define $\varphi:P_f(X^+)\to S_{(4,667)}$ by
$\varphi(x_1)=2$, $\varphi(x_2)=3$, and $\varphi(z)=4$ for all other $z$.
Then $\varphi(\bu)\in\{2,3\}$ while $\varphi(\bq)=1$, a contradiction.
Hence such $\bu_\ell$ exists.
If $\bu_\ell=x_1x_2\bu_\ell'$, then
\[
\bu \succeq \bu_\ell = x_1x_2\bu_\ell' \stackrel{\eqref{66704}}\approx x_1x_2\cdots x_n = \bq.
\]
If $\bu_\ell=x_1^2\bu_\ell''$, then
\[
\bu \succeq \bu_\ell = x_1^2\bu_\ell'' \stackrel{\eqref{66701}}\succeq x_1^2 \stackrel{\eqref{66703}}\succeq x_1x_2\cdots x_n = \bq.
\]
This derives the inequality $\bu \succeq\bq$.
\end{proof}

\begin{cor}
The ai-semiring $S_{(4, 533)}$ is finitely based.
\end{cor}
\begin{proof}
Notice that $S_{(4, 533)}$ and $S_{(4, 667)}$ have dual multiplications,
therefore, $S_{(4, 533)}$ is also finitely based.
\end{proof}

\begin{pro}\label{pro63201}
$\mathsf{V}(S_{(4, 632)})$ is the ai-semiring variety defined by the identities
\begin{align}
& x^2y \approx xy; \label{63202}\\
&(xy)^2 \approx x^2y^2; \label{63205}\\
& x^2y^2 \approx x^2y^2x^2; \label{63206}\\
& x\preceq x^2; \label{63209}\\
& xy\preceq xyx; \label{63201}\\
& xy \preceq xy^2; \label{63203}\\
& xz \preceq xy+z; \label{63204}\\
& yx \preceq x+yxz; \label{63207}\\
& x+xy \approx x^2+xy. \label{63208}
\end{align}
\end{pro}
\begin{proof}
It is easy to check that $S_{(4, 632)}$ satisfies the identities \eqref{63202}--\eqref{63208}.
In the remainder it is enough to prove that every identity that holds in $S_{(4, 632)}$
can be derived by \eqref{63202}--\eqref{63208}.
Let $\bq\preceq \bu$ be such a nontrivial inequality, where
$\bu=\bu_1+\bu_2+\cdots+\bu_n$ and $\bu_i, \bq \in X^+$, $1 \leq i \leq n$.
Since $S_{41}$, $S_{57}$, and $S_{58}$ embed into $S_{(4,632)}$,
it follows from Lemmas~\ref{lem4101}, \ref{lem5701}, and \ref{lem5801} that
$L_{\geq2}(\bu)\neq\emptyset$, $H_{\bq}(\bu)\neq\emptyset$, and $D_{\bq}(\bu)\neq\emptyset$;
moreover, for every $Y\subseteq c(\bq)$, there exists $\bu_j\in\bu$ such that $h_Y(\bu_j)=h_Y(\bq)$.
If $\ell(\bq)\geq2$, then $L_{\geq2}(\bu)\cap H_{\bq}(\bu)\neq\emptyset$.
Since $D_{\bq}(\bu)\neq\emptyset$, we can choose $\bu_i\in\bu$ such that $c(\bu_i)\subseteq c(\bq)$.

Let $|c(\bq)|=1$.
If $\ell(\bq)=1$, then $\bu_{i}=\bq^\ell$ for some $\ell\geq 2$. So we have
\[
\bu\succeq\bu_{i}=\bq^\ell \stackrel{\eqref{63202}}\approx \bq^2\stackrel{ \eqref{63209}}\succeq \bq.
\]
Now assume that $\ell(\bq)\geq 2$. Then $\bq=x^r$ for some $x \in X$ and some $r \geq 2$.
This implies that $\bu_{i}=x^s$ for some $s \geq 1$ and $\bu_k=xs(\bu_k)$, where $s(\bu_k)$ is nonempty.
If $s=1$, then $\bu_{i}=x$ and so
\[
\bu\succeq x+xs(\bu_k)\stackrel{\eqref{63208}}\approx x^2+xs(\bu_k)\stackrel{\eqref{63202}}\approx x^r+xs(\bu_k)=\bq+xs(\bu_k).
\]
If $s\geq 2$, then
\[
\bu \succeq\bu_{i_1}=x^s\stackrel{\eqref{63202}}\approx x^r=\bq.
\]
This derives the inequality $\bu\succeq\bq$.

Now let $|c(\bq)|\geq 2$, and let $i(\bq)=x_1x_2\cdots x_m$, $m\geq 2$.
By the identities \eqref{63201}, \eqref{63202} and \eqref{63209} we deduce
$i(\bq) \preceq\bq$.
Since $\bq\preceq\bu$ holds in $S_{(4, 632)}$, we have $i(\bq)\preceq\bu$ holds in $S_{(4, 632)}$.
This implies that $i(\bq)\preceq\bu$ holds in  $S_{41}$, $S_{57}$ and $S_{58}$.

We shall show that for any $2\leq j\leq m-1$, there exists $\bu_{i_j}\in\bu$ such that
$\bu_{i_j}=\bu_{i_j}'x_j\bu_{i_j}''$, where $c(\bu_{i_j}')\subseteq Y_j$ and $\bu_{i_j}''\neq\varepsilon$.
Indeed, suppose that this is not true.
Then there exists $2\leq j\leq m-1$ such that $\bu_{i_j}=\bu_{i_j}'x_j$ for some $\bu_{i_j}' \in X^*$,
where $c(\bu_{i_j}')\subseteq Y_j$,
if $\bu_{i_j} \in \bu$ and $h_{Y_j}(\bu_{i_j})=h_{Y_j}(\bq)$.
Choose a minimum $j$.
Consider the semiring homomorphism $\varphi: P_f(X^+) \to S_{(4, 632)}$ defined by
$\varphi(x_j)=2$, $\varphi(x)=3$ if $x\in Y_{j}$, and $\varphi(x)=4$ otherwise.
It is easy to see that $\varphi(\bu)=2$ and $\varphi(i(\bq))=1$, a contradiction.
Combining Lemma~\ref{lem5801}, for any $1\leq j\leq m-1$, there exists $\bu_{i_j}\in\bu$ such that
$\bu_{i_j}=\bu_{i_j}'x_j\bu_{i_j}''$, where $c(\bu_{i_j}')\subseteq Y_j$ and $\bu_{i_j}''\neq\varepsilon$.
Now we have
\[
\bu\succeq x_1\bu_{i_1}''+\bu_{i_2}'x_2\bu_{i_2}''\stackrel{\eqref{63204}}\succeq x_1\bu_{i_2}'x_2\bu_{i_2}''\stackrel{\eqref{63202}}\approx x_1x_2\bu_{i_2}''.
\]
This implies the inequality
\[
\bu\succeq x_1x_2\bu_{i_2}''.
\]
Repeat this process and one can obtain the inequality
\[
\bu\succeq x_1x_2\cdots x_m\bu_{i_m}''.
\]

\textbf{Case 1.} $\bu_{i_m}''$ is nonempty. Then
\[
\bu \succeq \bu_i+x_1x_2\cdots x_m\bu_{i_m}''\stackrel{\eqref{63202}-\eqref{63206}}\approx \bu_i+\bq^2\bu_i^2\bu_{i_m}''
\stackrel{\eqref{63202}}\approx \bu_i+\bq\bu_i\bu_{i_m}''\stackrel{\eqref{63207}}\succeq\bq\bu_i\stackrel{\eqref{63201}, \eqref{63203}}\succeq\bq.
\]

\textbf{Case 2.} $\bu_{i_m}''$ is empty. Then one can obtain the inequality
\[
\bu\succeq x_1x_2\cdots x_m=i(\bq).
\]
In this case, we shall show that for any $t(\bq)=x_m$ and $m(t(\bq), \bq)=1$.
Suppose that this is not true.
Consider the semiring homomorphism $\varphi: P_f(X^+) \to S_{(4, 632)}$ defined by
$\varphi(x_m)=2$, $\varphi(x)=3$ if $x\in Y_m$, and $\varphi(x)=4$ otherwise.
It follows that $\varphi(\bu)=2$ and $\varphi(i(\bq))=1$, a contradiction.
Hence the identities \eqref{63202}, \eqref{63205} and \eqref{63206} imply that $i(\bq)\approx\bq$.
Therefore, we can derive the inequality $\bu \succeq \bq$.
\end{proof}

\begin{cor}
The ai-semiring $S_{(4, 511)}$ is finitely based.
\end{cor}
\begin{proof}
Notice that $S_{(4, 511)}$ and $S_{(4, 632)}$ have dual multiplications,
therefore, $S_{(4, 511)}$ is also finitely based.
\end{proof}



\begin{pro}\label{pro77101}
$\mathsf{V}(S_{(4, 771)})$ is the ai-semiring variety defined by the identities
\begin{align}
& x^3\approx x^4; \label{77103}\\
& xy\approx yx; \label{77101}\\
& x_1xx_2 \preceq x; \label{77102}\\
& xy \preceq x^2+y^2; \label{77105}\\
& xy \preceq x^2+xy^2; \label{77106}\\
& xy \preceq x^2+xy^3; \label{77107}\\
& xyz^3 \preceq x^2z^3+xy^3z^3; \label{77110}\\
& xy^2y_1^3 \preceq xy^3y_1^3+x^3y^2y_1^3; \label{77108}\\
& x^2y^2y_1^3 \preceq x^2y^3y_1^3+x^3y^2y_1^3; \label{77100}\\
& xy^2z^2z_1^3 \preceq xy^2z^3z_1^3+x^3y^3z^2z_1^3; \label{77109}\\
& xyz \approx x^2yz+xy^2z+xyz^2, \label{77104}
\end{align}
where $x_1$ and $x_2$ may be empty in \eqref{77102},
$y_1$ may be empty in \eqref{77108} and \eqref{77100}, $z_1$ may be empty in \eqref{77109}.
\end{pro}
\begin{proof}
It is easy to check that $S_{(4, 771)}$ satisfies the identities \eqref{77103}--\eqref{77104}.
In the remainder it is enough to prove that every identity that holds in $S_{(4, 771)}$
can be derived by \eqref{77103}--\eqref{77104}.
Let $\bq\preceq \bu$ be such a nontrivial inequality, where
$\bu=\bu_1+\bu_2+\cdots+\bu_n$ and $\bu_i, \bq \in X^+$, $1 \leq i \leq n$.
Since $S_{44}$ and $S_{47}$ can be embedded in $S_{(4, 771)}$,
we have that both $S_{44}$ and $S_{47}$ satisfy $\bq\preceq \bu$.
By Lemma~\ref{lem4701},  it follows that
either $\ell(\bq)\geq 3$ or $\ell(\bq)=2$, $L_{\leq2}(\bu) \cap D_\bq(\bu)$ is nonempty.

\textbf{Case 1.} $\ell(\bq)=2$. Then $\bq=xy$.
Lemma~\ref{lem4701} tells us that there is $\bu_i\in D_\bq(\bu)$ such that $\ell(\bu_i)\leq2$.
If $\ell(\bu_i)=1$, then
\[
\bu \succeq \bu_i \stackrel{\eqref{77102}}\succeq \bu_i\bq_1 \stackrel{\eqref{77101}}\approx \bq.
\]
Now assume $\ell(\bu_i)=2$. Then $\bu_i\in\{x^2,y^2,xy\}$.
If $\bu_i=xy$, then $\bu\succeq xy=\bq$ is immediate.
If $\bu_i=y^2$, the argument is symmetric to the case $\bu_i=x^2$; hence we may assume $\bu_i=x^2$.

By Lemma~\ref{lem4401}, for every $x\in M_1(\bq)$, there exists $\bu_k\in D_\bq(\bu)$ such that $m(x,\bu_k)\leq 1$.
Together with identity \eqref{77103}, this implies that $\bu_k$ is one of
\[
\{y, y^2, y^3, xy^2, xy^3\}.
\]
For each of these five possibilities, using identities \eqref{77101} and \eqref{77102} (and the fact that $\bu\succeq\bu_k$),
one can derive $\bu\succeq xy^3$ in a straightforward manner.
Hence it suffices to consider the case $\bu_k=xy^3$.
In this case,
\[
\bu \succeq \bu_i+\bu_k = x^2+xy^3 \stackrel{\eqref{77107}}\succeq xy=\bq.
\]
Thus the inequality $\bu\succeq\bq$ is established.

\textbf{Case 2.} $\ell(\bq)\geq3$. By Lemma~\ref{lem4401}, we consider the following cases.

\textbf{Subcase 2.1.}
$M_1(\bq)$ is empty. Then $m(x, \bq)\geq 2$ for all $x\in c(\bq)$.
By identities \eqref{77101} and \eqref{77103}, we may assume that
\[
\bq \approx x_1^2\cdots x_m^2\, y_1^3\cdots y_n^3,
\]
where $m,n\geq 0$ and $c(\bq)=\{x_1,\ldots,x_m,y_1,\ldots,y_n\}$.

We claim that for each $x_\ell\in M_2(\bq)$, there is $\bu_{i_\ell}\in D_\bq(\bu)$ such that $m(x_\ell,\bu_{i_\ell})\leq 2$.
Suppose otherwise. Define a semiring homomorphism $\varphi:P_f(X^+)\to S_{(4,771)}$ by
$\varphi(x_\ell)=2$, $\varphi(x)=1$ if $x\in c(\bq)\backslash\{x_\ell\}$, and $\varphi(x)=4$ otherwise.
It is easy to see that $\varphi(\bq)=3$ and $\varphi(\bu)=4$, a contradiction.
Hence such a $\bu_{i_\ell}$ exists.
Consequently, for each $1\leq \ell\leq m$,
\[
\bu \succeq \bu_{i_\ell} \stackrel{\eqref{77103}, \eqref{77101}, \eqref{77102}}\succeq x_1^3\cdots x_{\ell-1}^3x_\ell^2x_{\ell+1}^3\cdots x_m^3y_1^3\cdots y_n^3.
\]
Furthermore, we have
\begin{align*}
\bu
&\succeq \bu_{i_1}+\bu_{i_2}+\dots +\bu_{i_m} \\
&\approx x_1^2(x_2\cdots x_my_1\cdots y_n)^3+x_2^2(x_1x_3\cdots x_my_1\cdots y_n)^3+\dots+\bu_{i_m} &&(\text{by}~\eqref{77101})\\
&\succeq x_1^2x_2^2(x_3\cdots x_my_1\cdots y_n)^3+\dots+\bu_{i_m} &&(\text{by}~\eqref{77100})\\
&\succeq  \dots \\
&\succeq x_1^2x_2^2\cdots x_m^2y_1^3\cdots y_n^3. &&(\text{by}~\eqref{77100})
\end{align*}
This derives the inequality $\bu \succeq\bq$.

\textbf{Subcase 2.2.}
$M_1(\bq)$ is nonempty. Lemma~\ref{lem4401} tells us that for any $x \in M_{1}(\bq)$,
there exists $\bu_k \in D_\bq(\bu)$ such that $m(x, \bu_k)\leq 1$.
Let
$c(\bq)=\{x_1, \ldots, x_k, \ldots, x_m, $ $y_1, \ldots, y_n\}$,
where $m(x_s, \bq )= 1$, $m(x_i, \bq )= 2$, $m(y_j, \bq )\geq3$, $1 \leq s \leq k$, $k+1 \leq i \leq m$, $1 \leq j \leq n$.
Then
\begin{align*}
\bq
&\approx x_1\cdots x_kx_{k+1}^2\cdots x_m^2y_1^3\cdots y_n^3 &&(\text{by}~ \eqref{77101}, \eqref{77103})\\
&\approx\sum\limits_{1\leq i<j\leq k} x_ix_jx_1^2\cdots x_{i-1}^2x_{i+1}^2\cdots x_{j-1}^2x_{j+1}^2\cdots x_m^2y_1^3\cdots y_n^3.
&&(\text{by}~\eqref{77101}, \eqref{77104})
\end{align*}
So we only need to consider $\bq=x_1x_2^2\cdots x_m^2y_1^3\cdots y_n^3$ or $x_1x_2x_3^2\cdots x_m^2y_1^3\cdots y_n^3$,
where $m\geq1$, $n\geq0$.

If $\bq=x_1x_2^2\cdots x_m^2y_1^3\cdots y_n^3$, then
we shall show that for every $x\in M_2(\bq)$, there exists $\bu_\ell \in D_\bq(\bu)$ such that $ m(x, \bu_\ell)\leq 2$.
Suppose that this is not true.
Let us consider the semiring homomorphism $\varphi: P_f(X^+) \to S_{(4, 771)}$ defined by
$\varphi(x)=2$, $\varphi(y)=1$ if $y\in c(\bq)\backslash\{x\}$, and $\varphi(y)=4$ otherwise.
It is easy to see that $\varphi(\bq)=3$ and $\varphi(\bu)=4$, a contradiction.
So for every $x\in M_2(\bq)$, there exists $\bu_\ell \in D_\bq(\bu)$ such that $ m(x, \bu_\ell)\leq 2$.
Now we have
\begin{align*}
\bu
&\succeq \bu_k+\bu_\ell \\
&\succeq x_1(x_2\cdots x_my_1\cdots y_n)^3+x_2^2(x_1x_3\cdots x_my_1\cdots y_n)^3 &&(\text{by}~\eqref{77103}, \eqref{77101}, \eqref{77102})\\
&\approx x_1x_2^3(x_3\cdots x_my_1\cdots y_n)^3+x_2^2x_1^3(x_3\cdots x_my_1\cdots y_n)^3 &&(\text{by}~\eqref{77101})\\
&\succeq x_1x_2^2(x_3\cdots x_my_1\cdots y_n)^3. &&(\text{by}~\eqref{77108})
\end{align*}
This proves the inequality
\[
\bu\succeq x_1x_2^2(x_3\cdots x_my_1\cdots y_n)^3.
\]
Furthermore, we have
\begin{align*}
\bu
&\succeq x_1x_2^2(x_3\cdots x_my_1\cdots y_n)^3+\bu_\ell \\
&\succeq x_1x_2^2x_3^3(x_4\cdots x_my_1\cdots y_n)^3+x_3^2(x_4\cdots x_my_1\cdots y_n)^3x_1^3x_2^3 &&(\text{by}~\eqref{77103}, \eqref{77101}, \eqref{77102})\\
&\approx x_1x_2^2x_3^2(x_4\cdots x_my_1\cdots y_n)^3. &&(\text{by}~\eqref{77109})
\end{align*}
Repeat this process and one can obtain the inequality
\[
\bu\succeq x_1x_2^2\cdots x_m^2y_1^3\cdots y_n^3.
\]
This derives the inequality $\bu \succeq\bq$.

If $\bq=x_1x_2x_3^2\cdots x_m^2y_1^3\cdots y_n^3$, then we shall show that there exists $\bu_p\in D_\bq(\bu)$
such that $m(x_1, \bu_p)+m(x_2, \bu_p)\leq 2$.
Suppose that this is not true.
Let us consider the semiring homomorphism $\varphi: P_f(X^+) \to S_{(4, 771)}$ defined by
$\varphi(x_1)=\varphi(x_2)=2$, $\varphi(z)=1$ if $z\in c(\bq)\backslash\{x_1, x_2\}$, and $\varphi(z)=4$ otherwise.
It is easy to see that $\varphi(\bq)=3$ and $\varphi(\bu)=4$, a contradiction.
Hence such a $\bu_p$ exists.
Consequently, one of the following holds:
\[
m(x_1,\bu_p)\leq 1,\ m(x_2,\bu_p)\leq 1;
\]
\[
m(x_1,\bu_p)=2,\ m(x_2,\bu_p)=0;
\]
or
\[
m(x_1,\bu_p)=0,\ m(x_2,\bu_p)=2.
\]
If $m(x_1,\bu_p)\leq1$, $m(x_2,\bu_p)\leq1$, then
\[
\bu\succeq \bu_p \stackrel{\eqref{77103}, \eqref{77101}, \eqref{77102}}\succeq x_1x_2x_3^3\cdots x_m^3y_1^3\cdots y_n^3.
\]
If $m(x_1,\bu_p)=2$ and $m(x_2,\bu_p)=0$, then by Lemma~\ref{lem4401},
there exists $\bu_q\in D_\bq(\bu)$ such that $m(x_1,\bu_q)\leq1$.
Thus
\begin{align*}
\bu
&\succeq \bu_p+\bu_q \\
&\succeq x_1^2(x_3\cdots x_my_1\cdots y_n)^3+x_1x_2^3(x_3\cdots x_my_1\cdots y_n)^3 &&(\text{by}~\eqref{77103}, \eqref{77101}, \eqref{77102})\\
&\succeq x_1x_2(x_3\cdots x_my_1\cdots y_n)^3. &&(\text{by}~\eqref{77110})
\end{align*}
The remainder of the argument follows the same pattern as $\bq=x_1x_2^2\cdots x_m^2y_1^3\cdots y_n^3$.
By symmetry, the case $m(x_2, \bu_p)=2, m(x_1, \bu_p)=0$ follows the same reasoning as the case $m(x_1, \bu_p)=2, m(x_2, \bu_p)=0$.
This completes the proof.
\end{proof}

\begin{pro}\label{pro73301}
$\mathsf{V}(S_{(4, 733)})$ is the ai-semiring variety defined by the identities
\begin{align}
& x^3\approx x^2; \label{73302}\\
& xy\approx yx; \label{73303}\\
& x^2 \preceq x; \label{73304}\\
& xyz \preceq xy; \label{73301}\\
& xy \approx x^2y+xy^2. \label{73305}
\end{align}
\end{pro}
\begin{proof}
A direct verification shows that $S_{(4,733)}$ satisfies identities \eqref{73302}--\eqref{73305}.
It remains to prove that every identity holding in $S_{(4,733)}$ can be derived from \eqref{73302}--\eqref{73305}.
Let $\bq\preceq \bu$ be a nontrivial inequality, where
$\bu=\bu_1+\cdots+\bu_n$ and $\bu_i,\bq\in X^+$ for $1\leq i\leq n$.
Since $S_{44}$ is isomorphic to the quotient $S_{(4,733)}/\rho$,
where $\{2,3\}$ is the nontrivial $\rho$-block, it follows that $S_{44}$ satisfies $\bq\preceq\bu$.
By Lemma~\ref{lem4401}, there exists $\bu_i\in\bu$ such that $c(\bu_i)\subseteq c(\bq)$ and $\ell(\bq)\geq 2$.

\textbf{Case 1.}
$M_1(\bq)$ is empty. Then $m(x, \bq)\geq 2$ for all $x\in c(\bq)$.
If $\ell(\bu_i)\geq 2$ for some $\bu_i\in D_\bq(\bu)$, then
\[
\bu \succeq \bu_i \stackrel{\eqref{73301}}\succeq \bu_i\bq \stackrel{\eqref{73302}, \eqref{73303}}\approx \bq.
\]
Otherwise, $\ell(\bu_i)=1$ for every $\bu_i\in D_\bq(\bu)$; then $\bu_i\in c(\bq)$, and hence
\[
\bu \succeq \bu_i \stackrel{\eqref{73304}}\succeq \bu_i^2 \stackrel{\eqref{73301}}\succeq \bu_i^2\bq \stackrel{\eqref{73302}, \eqref{73303}}\approx \bq.
\]
Thus the desired inequality $\bu\succeq\bq$ is established in this case.

\textbf{Case 2.}
$M_1(\bq)$ is nonempty. Lemma~\ref{lem4401} tells us that for any $x \in M_{1}(\bq)$,
there exists $\bu_j \in D_\bq(\bu)$ such that $m(x, \bu_j)\leq 1$.
Write
\[
c(\bq)=\{x_1,\ldots,x_r,y_1,\ldots,y_s\},
\]
where $m(x_i,\bq)\geq 2$ for $0\leq i\leq r$, and $m(y_j,\bq)=1$ for $1\leq j\leq s$.
Then, by \eqref{73302} and \eqref{73303},
\[
\bq \approx x_1^2x_2^2\cdots x_r^2\, y_1y_2\cdots y_s,
\]
and applying \eqref{73302}, \eqref{73303}, and \eqref{73305} yields
\[
\bq \approx \sum_{1\leq j\leq s} x_1^2x_2^2\cdots x_r^2\, y_1^2\cdots y_{j-1}^2\, y_{j+1}^2\cdots y_s^2\, y_j.
\]
Thus it suffices to consider the case $\bq=x_1^2x_2^2\cdots x_k^2y$.

We shall show that there exists $\bu_j \in D_\bq(\bu)$ such that
either $ m(y, \bu_j)\leq 1$, $\ell(\bu_j)\geq 2$ or $\bu_j=x_i$ for some $1\leq i\leq k$.
Suppose that this is not true.
Let us consider the semiring homomorphism $\varphi: P_f(X^+) \to S_{(4, 733)}$ defined by
$\varphi(y)=3$, $\varphi(x_i)=1$, and $\varphi(x)=4$ otherwise.
It is easy to see that $\varphi(\bq)=2$ and $\varphi(\bu)=3$ or $4$, a contradiction.
So there exists $\bu_j \in D_\bq(\bu)$ such that
either $m(y, \bu_j)\leq 1$, $\ell(\bu_j)\geq 2$ or $\bu_j=x_i$ with $1\leq i\leq k$.
If $m(y, \bu_j)\leq 1$, $\ell(\bu_j)\geq 2$ for some $\bu_j \in D_\bq(\bu)$, then
\[
\bu \succeq \bu_j \stackrel{\eqref{73301}}\succeq \bu_j\bq_1 \stackrel{\eqref{73302}, \eqref{73303}}\approx \bq.
\]
If $\bu_j=x_i$ for some $\bu_j \in D_\bq(\bu)$, then
\[
\bu \succeq \bu_j=x_i \stackrel{\eqref{73304}}\succeq x_i^2 \stackrel{\eqref{73301}}\succeq x_i^2\bq_1
\stackrel{\eqref{73302}, \eqref{73303}}\approx \bq.
\]
This derives the inequality $\bu \succeq\bq$.
\end{proof}

\begin{pro}\label{pro73401}
$\mathsf{V}(S_{(4, 734)})$ is the ai-semiring variety defined by the identities
\begin{align}
& x^3\approx x^2; \label{73403}\\
& xyz \approx yxz; \label{73401}\\
& x^2y^2 \approx y^2x^2; \label{73402}\\
& yx \preceq x; \label{73404}\\
& x^2y \preceq x; \label{73407}\\
& x^2y \preceq x^2; \label{73405}\\
& xy \approx x^2y+xy^2; \label{73406}\\
&  x^2y \preceq yx^2. \label{73408}
\end{align}
\end{pro}
\begin{proof}
It is easy to check that $S_{(4, 734)}$ satisfies the identities \eqref{73403}--\eqref{73408}.
In the remainder it is enough to prove that every identity that holds in $S_{(4, 734)}$
can be derived by \eqref{73403}--\eqref{73408}.
Let $\bq\preceq \bu$ be such a nontrivial inequality, where
$\bu=\bu_1+\bu_2+\cdots+\bu_n$ and $\bu_i, \bq \in X^+$, $1 \leq i \leq n$.
Since $S_{44}$ is isomorphic to the quotient algebra $S_{(4, 734)}/\rho$,
where $\{2, 3\}$ is the nontrivial block of $\rho$,
we have that $S_{44}$ satisfies $\bq\preceq \bu$.
By Lemma \ref{lem4401},  there exists $\bu_i\in \bu$ such that $c(\bu_i)\subseteq c(\bq)$ and $\ell(\bq)\geq 2$.

\textbf{Case 1.}
$M_1(\bq)$ is empty. Then $m(x, \bq)\geq 2$ for all $x\in c(\bq)$.
Consequently,
\[
\bu \succeq \bu_i \stackrel{\eqref{73404}}\succeq \bq_1\bu_i
\stackrel{\eqref{73403}, \eqref{73401}, \eqref{73402}}\approx \bq.
\]
This derives the inequality $\bu \succeq\bq$.

\textbf{Case 2.}
$M_1(\bq)$ is nonempty.
Let
$c(\bq)=\{x_1, \ldots, x_r, y_1, \ldots, y_s\}$,
where $m(x_i, \bq )\geq 2$, $m(y_j, \bq )=1$, $0 \leq i \leq r$, $1 \leq j \leq s$.
By the identity \eqref{73401}, $\bq \approx x_1^2x_2^2\cdots x_r^2y_1y_2\cdots y_s$ or $y_1y_2\cdots y_sx_1^2x_2^2\cdots x_r^2$.
If $\bq \approx x_1^2x_2^2\cdots x_r^2y_1y_2\cdots y_s$, then
\begin{align*}
\bq
&\approx x_1^2x_2^2\cdots x_r^2y_1y_2\cdots y_s \\
&\approx\sum\limits_{j=1}^{s-1}x_1^2\cdots x_r^2y_1^2\cdots y_{j-1}^2y_jy_{j+1}^2\cdots y_s^2+x_1^2\cdots x_r^2y_1^2\cdots y_{s-1}^2y_s
&&(\text{by}~\eqref{73403}, \eqref{73401}, \eqref{73406})\\
&\approx\sum\limits_{j=1}^{s-1}y_jx_1^2\cdots x_r^2y_1^2\cdots y_{j-1}^2y_{j+1}^2\cdots y_s^2+x_1^2\cdots x_r^2y_1^2\cdots y_{s-1}^2y_s.&&(\text{by}~\eqref{73401})
\end{align*}
If $\bq \approx y_1y_2\cdots y_sx_1^2x_2^2\cdots x_r^2$, then
\begin{align*}
\bq
&\approx y_1y_2\cdots y_sx_1^2x_2^2\cdots x_r^2 \\
&\approx\sum\limits_{1\leq j\leq s}y_jy_1^2\cdots y_{j-1}^2y_{j+1}^2\cdots y_s^2x_1^2x_2^2\cdots x_r^2.
&&(\text{by}~\eqref{73403}, \eqref{73401}, \eqref{73406})
\end{align*}
So we only need to consider $\bq=x_1^2x_2^2\cdots x_k^2y$ or $yx_1^2x_2^2\cdots x_k^2$.

\textbf{Subcase 2.1.} $\bq=x_1^2x_2^2\cdots x_k^2y$.
Lemma~\ref{lem4401} tells us that for any $y \in M_{1}(\bq)$,
there exists $\bu_j \in D_\bq(\bu)$ such that $m(y, \bu_j)\leq 1$.
If $m(y, \bu_j)=0$, then
\begin{align*}
\bu
&\succeq \bu_j \\
&\succeq x_1^2x_2^2\cdots x_k^2y\bu_j &&(\text{by}~{\eqref{73404}}) \\
&\approx yx_1^2x_2^2\cdots x_k^2 &&(\text{by}~{\eqref{73403}, \eqref{73401}, \eqref{73402}}) \\
&\approx y(x_1x_2\cdots x_k)^2 &&(\text{by}~{\eqref{73401}})\\
&\succeq (x_1x_2\cdots x_k)^2y &&(\text{by}~{\eqref{73408}})\\
&\approx x_1^2x_2^2\cdots x_k^2y. &&(\text{by}~{\eqref{73401}})
\end{align*}
If $m(y, \bu_j)=1$, then
\[
\bu \succeq \bu_j=p(\bu_j)y \stackrel{\eqref{73404}}\succeq x_1^2x_2^2\cdots x_k^2p(\bu_j)y \stackrel{\eqref{73403}, \eqref{73401}}\approx x_1^2x_2^2\cdots x_k^2y
\]
or
\[
\bu \succeq \bu_j=y\bu_j' \stackrel{\eqref{73404}}\succeq yx_1^2\cdots x_k^2\bu_j' \stackrel{\eqref{73403}, \eqref{73401}, \eqref{73402}}\approx y(x_1\cdots x_k)^2  \stackrel{\eqref{73408}}\succeq (x_1\cdots x_k)^2y \stackrel{\eqref{73401}}\approx \bq.
\]
This derives the inequality $\bu \succeq\bq$.

\textbf{Subcase 2.2.} $\bq=yx_1^2x_2^2\cdots x_k^2$.
We claim that there exists $\bu_\ell\in D_\bq(\bu)$ such that
$m(y,\bu_\ell)\leq 1$ and $t(\bu_\ell)\neq y$.
Suppose otherwise. Define $\varphi:P_f(X^+)\to S_{(4,734)}$ by
$\varphi(y)=3$, $\varphi(x_i)=1$ for $1\leq i\leq k$, and $\varphi(z)=4$ for all other $z$.
Then $\varphi(\bq)=2$ while $\varphi(\bu)\in\{3,4\}$, a contradiction.
Hence such $\bu_\ell$ exists.
Furthermore, we have
\[
\bu \succeq \bu_k=p(\bu_k)t(\bu_k) \stackrel{\eqref{73404}}\succeq \bq_1p(\bu_k)t(\bu_k) \stackrel{\eqref{73403}, \eqref{73401}, \eqref{73402}}\approx yx_1^2x_2^2\cdots x_k^2=\bq.
\]
This derives the inequality $\bu \succeq\bq$.
\end{proof}

\begin{cor}
The ai-semiring $S_{(4, 772)}$ is finitely based.
\end{cor}
\begin{proof}
Notice that $S_{(4, 772)}$ and $S_{(4, 734)}$ have dual multiplications,
therefore, $S_{(4, 772)}$ is also finitely based.
\end{proof}

\begin{pro}\label{pro76301}
$\mathsf{V}(S_{(4, 763)})$ is the ai-semiring variety defined by the identities
\begin{align}
& x^3 \approx x^2; \label{76302}\\
& xyx \approx x^2y; \label{76301}\\
& x^2y^2 \approx (xy)^2; \label{76304}\\
& x^2y^2 \approx y^2x^2; \label{76305}\\
& x_1xx_2 \preceq x; \label{76303}\\
& x_1^2xx_2^2 \preceq x_1^2x^2x_2^2+x_1^2x_2^2; \label{76306}\\
& x_1^2xx_2^2yx_3^2 \approx x_1^2xx_2^2y^2x_3^2+x_1^2x^2x_2^2yx_3^2, \label{76307}
\end{align}
where $x_1$ and $x_2$ may be empty in \eqref{76303} and \eqref{76306},
$x_1$, $x_2$ and $x_3$ may be empty in \eqref{76307}.
\end{pro}
\begin{proof}
It is easy to check that $S_{(4, 763)}$ satisfies the identities \eqref{76302}--\eqref{76307}.
In the remainder it is enough to prove that every identity that holds in $S_{(4, 763)}$
can be derived by \eqref{76302}--\eqref{76307}.
Let $\bq\preceq \bu$ be such a nontrivial inequality, where
$\bu=\bu_1+\bu_2+\cdots+\bu_n$ and $\bu_i, \bq \in X^+$, $1 \leq i \leq n$.
It is easy to see that $S_{44}$ and $S_{45}$ can be embedded into $S_{(4, 763)}$.
So both $S_{44}$ and $S_{45}$ satisfy $\bq\preceq \bu$.
By Lemma~\ref{lem4401}, $\ell(\bq)\geq 2$ and there exists $\bu_i\in\bu$ such that $c(\bu_i)\subseteq c(\bq)$.

\textbf{Case 1.}
$M_1(\bq)$ is empty. Then $m(x, \bq)\geq 2$ for all $x\in c(\bq)$.
Let $i(\bq)=x_1x_2\cdots x_m$.
By \eqref{76301} and \eqref{76302}, we have $\bq\approx x_1^2x_2^2\cdots x_m^2$.
Then
\[
\bu \succeq \bu_i\stackrel{\eqref{76303}}\succeq \bq\bu_i \stackrel{\eqref{76301}, \eqref{76302}}\approx \bq.
\]
This derives the inequality $\bu \succeq\bq$.

\textbf{Case 2.}
$M_1(\bq)$ is nonempty. Let
$c(\bq)=\{x_1, x_2, \ldots, x_m\}$.
Then
\begin{align*}
\bq
&\approx x_1^{l_1}x_2^{l_1}\cdots x_m^{l_m} &&(\text{by}~ \eqref{76301}, \eqref{76302})\\
&\approx\sum\limits_{1\leq k\leq m}x_1^2\cdots x_{k-1}^2x_kx_{k+1}^2\cdots x_m^2,
&&(\text{by}~\eqref{76307})
\end{align*}
where $l_i=1$ or $2$.
So we only need to consider $\bq=x_1^2\cdots x_{k-1}^2x_kx_{k+1}^2\cdots x_m^2$ for all $1\leq k\leq m$.
We shall show that for every $x_k\in M_1(\bq)$,
there exists $\bu_k \in D_\bq(\bu)$ such that
either $m(x_k, \bu_k)=0$ or $m(x_k, \bu_k)=1$$(\bu_k=\bu_k'x_k\bu_k'')$ and $x_{k+i}\not\in c(\bu_k')$ for all $x_{k+i}\in\{x_{k+1},\dots,x_m\}$.
Suppose that this is not true.
Then there exists $x_k\in M_1(\bq)$, for any $\bu_k \in D_\bq(\bu)$,
either $m(x_k, \bu_k)\geq2$ or $m(x_k, \bu_k)=1$ and $x_{k+i}\in c(\bu_k')$ for some $x_{k+i}\in\{x_{k+1},\dots,x_m\}$.

If $m(x_k, \bu_k)\geq2$,
let us consider the semiring homomorphism $\varphi: P_f(X^+) \to S_{(4, 763)}$ defined by
$\varphi(x_k)=3$, $\varphi(x)=1$ if $x\in c(\bq)\setminus{\{x_k\}}$, and $\varphi(x)=4$ otherwise.
It is easy to see that $\varphi(\bu)=4$ and $\varphi(\bq)=3$, a contradiction.
If $m(x_k, \bu_k)=1$ and $x_{k+i}\in c(\bu_k')$ for some $x_{k+i}\in\{x_{k+1},\dots,x_m\}$,
 consider the homomorphism $\varphi: P_f(X^+) \to S_{(4, 763)}$ defined by
$\varphi(x_k)=3$, $\varphi(x_{k+i})=2$, $\varphi(x)=1$ if $x\in c(\bq)\setminus{\{x_k, x_{k+i}\}}$, and $\varphi(x)=4$ otherwise.
It follows that $\varphi(\bu)=4$ and $\varphi(\bq)=3$, a contradiction.
So there exists $\bu_k \in D_\bq(\bu)$ such that
either $m(x_k, \bu_k)=0$ or $m(x_k, \bu_k)=1$$(\bu_k=\bu_k'x_k\bu_k'')$ and $x_{k+i}\not\in c(\bu_k')$ for all $x_{k+i}\in\{x_{k+1},\dots,x_m\}$.

\textbf{Subcase 2.1.} There exists $\bu_k \in D_\bq(\bu)$ such that
$ m(x_k, \bu_k)=0$. Then
\[
\bu \succeq \bu_k\stackrel{\eqref{76303}}\succeq x_1^2\cdots x_{k-1}^2x_{k+1}^2\cdots x_m^2\bu_k \stackrel{\eqref{76301}, \eqref{76302}}\approx x_1^2\cdots x_{k-1}^2x_{k+1}^2\cdots x_m^2
\]
and
\[
\bu \succeq \bu_i\stackrel{\eqref{76303}}\succeq x_1^2x_2^2\cdots x_m^2\bu_i \stackrel{\eqref{76301}, \eqref{76302}}\approx x_1^2x_2^2\cdots x_m^2.
\]
This implies the inequalities
$\bu \succeq x_1^2\cdots x_{k-1}^2x_{k+1}^2\cdots x_m^2$ and $\bu \succeq x_1^2x_2^2\cdots x_m^2$.
Furthermore, we have
\begin{align*}
\bu
&\succeq x_1^2x_2^2\cdots x_m^2+x_1^2\cdots x_{k-1}^2x_{k+1}^2\cdots x_m^2  \\
&\approx (x_1\cdots x_{k-1})^2x_k^2(x_{k+1}\cdots x_m)^2+(x_1\cdots x_{k-1})^2(x_{k+1}\cdots x_m)^2 &&(\text{by}~\eqref{76304})\\
&\succeq (x_1\cdots x_{k-1})^2x_k(x_{k+1}\cdots x_m)^2 &&(\text{by}~\eqref{76306})\\
&\approx x_1^2\cdots x_{k-1}^2x_kx_{k+1}^2\cdots x_m^2. &&(\text{by}~\eqref{76304})
\end{align*}

\textbf{Subcase 2.2.} There exists $\bu_k \in D_\bq(\bu)$ such that
$ m(x_k, \bu_k)=1$ and $x_{k+i}\not\in c(\bu_k')$ for all $x_{k+i}\in\{x_{k+1},\dots,x_m\}$. Then
\begin{align*}
\bu
&\succeq \bu_k \\
&\succeq x_1^2\cdots x_{k-1}^2\bu_k'x_k\bu_k''x_{k+1}^2\cdots x_m^2  \\
&\approx x_1^2\cdots x_{k-1}^2x_k\bu_k''x_{k+1}^2\cdots x_m^2  &&(\text{by}~\eqref{76301}, \eqref{76302})\\
&\approx x_1^2\cdots x_{k-1}^2x_kx_{k+1}^2\cdots x_m^2. &&(\text{by}~\eqref{76301}, \eqref{76302}, \eqref{76305})
\end{align*}
This completes the proof.
\end{proof}

\begin{cor}
The ai-semiring $S_{(4, 766)}$ is finitely based.
\end{cor}
\begin{proof}
Notice that $S_{(4, 766)}$ and $S_{(4, 763)}$ have dual multiplications,
therefore, $S_{(4, 766)}$ is also finitely based.
\end{proof}

\subsection*{Acknowledgment}
This work is supported by National Natural Science Foundation of China (Grant Nos. 12371024, 12571020).

\bibliographystyle{amsplain}

\begin{table}[htbp]
\caption{The multiplicative tables of $4$-element ai-semirings whose additive reducts form a chain} \label{tb1}
\begin{tabular}{cccccc}
\hline
Semiring & $\cdot$ & Semiring & $\cdot$ & Semiring & $\cdot$\\
\hline
$S_{(4, 481)}$
&
\begin{tabular}{cccc}
1 & 1 & 1 & 1\\
1 & 1 & 1 & 1 \\
1 & 1 & 1 & 1 \\
1 & 1 & 1 & 1 \\
\end{tabular}
&
$S_{(4, 482)}$
&
\begin{tabular}{cccc}
1 & 1 & 1 & 1\\
1 & 1 & 1 & 1 \\
1 & 1 & 1 & 1 \\
1 & 1 & 1 & 2 \\
\end{tabular}
&
$S_{(4, 483)}$
&
\begin{tabular}{cccc}
1 & 1 & 1 & 1\\
1 & 1 & 1 & 1 \\
1 & 1 & 1 & 1 \\
1 & 1 & 1 & 3 \\
\end{tabular}\\
\hline
$S_{(4, 484)}$
&
\begin{tabular}{cccc}
1 & 1 & 1 & 1\\
1 & 1 & 1 & 1 \\
1 & 1 & 1 & 1 \\
1 & 1 & 1 & 4 \\
\end{tabular}
&
$S_{(4, 485)}$
&
\begin{tabular}{cccc}
1 & 1 & 1 & 1\\
1 & 1 & 1 & 1 \\
1 & 1 & 1 & 2 \\
1 & 1 & 1 & 2 \\
\end{tabular}
&
$S_{(4, 486)}$
&
\begin{tabular}{cccc}
1 & 1 & 1 & 1\\
1 & 1 & 1 & 1 \\
1 & 1 & 1 & 3 \\
1 & 1 & 1 & 4 \\
\end{tabular}\\
\hline
$S_{(4, 487)}$
&
\begin{tabular}{cccc}
1 & 1 & 1 & 1\\
1 & 1 & 1 & 2 \\
1 & 1 & 1 & 2 \\
1 & 1 & 1 & 4 \\
\end{tabular}
&
$S_{(4, 488)}$
&
\begin{tabular}{cccc}
1 & 1 & 1 & 1\\
1 & 1 & 1 & 2 \\
1 & 1 & 1 & 3 \\
1 & 1 & 1 & 4 \\
\end{tabular}
&
$S_{(4, 489)}$
&
\begin{tabular}{cccc}
1 & 1 & 1 & 1\\
1 & 1 & 1 & 3 \\
1 & 1 & 1 & 3 \\
1 & 1 & 1 & 4 \\
\end{tabular}\\
\hline
$S_{(4, 490)}$
&
\begin{tabular}{cccc}
1 & 1 & 1 & 4\\
1 & 1 & 1 & 4 \\
1 & 1 & 1 & 4 \\
1 & 1 & 1 & 4 \\
\end{tabular}
&
$S_{(4, 491)}$
&
\begin{tabular}{cccc}
1 & 1 & 1 & 1\\
1 & 1 & 1 & 1 \\
1 & 1 & 1 & 1 \\
1 & 1 & 2 & 2 \\
\end{tabular}
&
$S_{(4, 492)}$
&
\begin{tabular}{cccc}
1 & 1 & 1 & 1\\
1 & 1 & 1 & 1 \\
1 & 1 & 1 & 2 \\
1 & 1 & 2 & 2 \\
\end{tabular}\\
\hline
$S_{(4, 493)}$
&
\begin{tabular}{cccc}
1 & 1 & 1 & 1\\
1 & 1 & 1 & 1 \\
1 & 1 & 1 & 2 \\
1 & 1 & 2 & 3 \\
\end{tabular}
&
$S_{(4, 494)}$
&
\begin{tabular}{cccc}
1 & 1 & 1 & 1\\
1 & 1 & 1 & 1 \\
1 & 1 & 1 & 1 \\
1 & 1 & 3 & 4 \\
\end{tabular}
&
$S_{(4, 495)}$
&
\begin{tabular}{cccc}
1 & 1 & 1 & 1\\
1 & 1 & 1 & 1 \\
1 & 1 & 1 & 3 \\
1 & 1 & 3 & 4 \\
\end{tabular}\\
\hline
$S_{(4, 496)}$
&
\begin{tabular}{cccc}
1 & 1 & 1 & 1\\
1 & 1 & 1 & 2 \\
1 & 1 & 1 & 3 \\
1 & 1 & 3 & 4 \\
\end{tabular}
&
$S_{(4, 497)}$
&
\begin{tabular}{cccc}
1 & 1 & 1 & 1\\
1 & 1 & 1 & 1 \\
1 & 1 & 2 & 2 \\
1 & 1 & 2 & 2 \\
\end{tabular}
&
$S_{(4, 498)}$
&
\begin{tabular}{cccc}
1 & 1 & 1 & 1\\
1 & 1 & 1 & 1 \\
1 & 1 & 3 & 3 \\
1 & 1 & 3 & 3 \\
\end{tabular}\\
\hline
$S_{(4, 499)}$
&
\begin{tabular}{cccc}
1 & 1 & 1 & 1\\
1 & 1 & 1 & 1 \\
1 & 1 & 3 & 3 \\
1 & 1 & 3 & 4 \\
\end{tabular}
&
$S_{(4, 500)}$
&
\begin{tabular}{cccc}
1 & 1 & 1 & 1\\
1 & 1 & 1 & 1 \\
1 & 1 & 3 & 4 \\
1 & 1 & 3 & 4 \\
\end{tabular}
&
$S_{(4, 501)}$
&
\begin{tabular}{cccc}
1 & 1 & 1 & 1\\
1 & 1 & 1 & 2 \\
1 & 1 & 3 & 3 \\
1 & 1 & 3 & 4 \\
\end{tabular}\\
\hline
$S_{(4, 502)}$
&
\begin{tabular}{cccc}
1 & 1 & 1 & 1\\
1 & 1 & 1 & 1 \\
1 & 1 & 3 & 3 \\
1 & 1 & 4 & 4 \\
\end{tabular}
&
$S_{(4, 503)}$
&
\begin{tabular}{cccc}
1 & 1 & 1 & 1\\
1 & 1 & 1 & 1 \\
1 & 1 & 3 & 4 \\
1 & 1 & 4 & 4 \\
\end{tabular}
&
$S_{(4, 504)}$
&
\begin{tabular}{cccc}
1 & 1 & 1 & 4\\
1 & 1 & 1 & 4 \\
1 & 1 & 3 & 4 \\
1 & 1 & 4 & 4 \\
\end{tabular}\\
\hline
$S_{(4, 505)}$
&
\begin{tabular}{cccc}
1 & 1 & 1 & 1\\
1 & 1 & 1 & 1 \\
1 & 1 & 4 & 4 \\
1 & 1 & 4 & 4 \\
\end{tabular}
&
$S_{(4, 506)}$
&
\begin{tabular}{cccc}
1 & 1 & 1 & 1\\
1 & 1 & 2 & 2 \\
1 & 1 & 3 & 3 \\
1 & 1 & 3 & 3 \\
\end{tabular}
&
$S_{(4, 507)}$
&
\begin{tabular}{cccc}
1 & 1 & 1 & 1\\
1 & 1 & 2 & 2 \\
1 & 1 & 3 & 3 \\
1 & 1 & 3 & 4 \\
\end{tabular}\\
\hline
$S_{(4, 508)}$
&
\begin{tabular}{cccc}
1 & 1 & 1 & 1\\
1 & 1 & 2 & 2 \\
1 & 1 & 3 & 4 \\
1 & 1 & 3 & 4 \\
\end{tabular}
&
$S_{(4, 509)}$
&
\begin{tabular}{cccc}
1 & 1 & 1 & 1\\
1 & 1 & 2 & 2 \\
1 & 1 & 3 & 3 \\
1 & 1 & 4 & 4 \\
\end{tabular}
&
$S_{(4, 510)}$
&
\begin{tabular}{cccc}
1 & 1 & 1 & 1\\
1 & 1 & 2 & 2 \\
1 & 1 & 3 & 4 \\
1 & 1 & 4 & 4 \\
\end{tabular}\\
\hline
$S_{(4, 511)}$
&
\begin{tabular}{cccc}
1 & 1 & 1 & 4\\
1 & 1 & 2 & 4 \\
1 & 1 & 3 & 4 \\
1 & 1 & 4 & 4 \\
\end{tabular}
&
$S_{(4, 512)}$
&
\begin{tabular}{cccc}
1 & 1 & 1 & 1\\
1 & 1 & 2 & 2 \\
1 & 1 & 4 & 4 \\
1 & 1 & 4 & 4 \\
\end{tabular}
&
$S_{(4, 513)}$
&
\begin{tabular}{cccc}
1 & 1 & 3 & 3 \\
1 & 1 & 3 & 3 \\
1 & 1 & 3 & 3 \\
1 & 1 & 3 & 3 \\
\end{tabular}\\
\hline
$S_{(4, 514)}$
&
\begin{tabular}{cccc}
1 & 1 & 3 & 3\\
1 & 1 & 3 & 3 \\
1 & 1 & 3 & 3 \\
1 & 1 & 3 & 4 \\
\end{tabular}
&
$S_{(4, 515)}$
&
\begin{tabular}{cccc}
1 & 1 & 3 & 4 \\
1 & 1 & 3 & 4 \\
1 & 1 & 3 & 4 \\
1 & 1 & 3 & 4 \\
\end{tabular}
&
$S_{(4, 516)}$
&
\begin{tabular}{cccc}
1 & 1 & 4 & 4 \\
1 & 1 & 4 & 4 \\
1 & 1 & 4 & 4 \\
1 & 1 & 4 & 4 \\
\end{tabular}\\
\hline
\end{tabular}
\end{table}

\begin{table}[htbp]
\label{tb2}
\begin{tabular}{cccccc}
\hline
Semiring & $\cdot$ & Semiring & $\cdot$ & Semiring & $\cdot$\\
\hline
$S_{(4, 517)}$
&
\begin{tabular}{cccc}
1 & 1 & 1 & 1\\
1 & 1 & 1 & 1 \\
1 & 1 & 1 & 1 \\
1 & 2 & 2 & 4 \\
\end{tabular}
&
$S_{(4, 518)}$
&
\begin{tabular}{cccc}
1 & 1 & 1 & 1 \\
1 & 1 & 1 & 2 \\
1 & 1 & 1 & 2 \\
1 & 2 & 2 & 4 \\
\end{tabular}
&
$S_{(4, 519)}$
&
\begin{tabular}{cccc}
1 & 1 & 1 & 1 \\
1 & 1 & 1 & 2 \\
1 & 1 & 1 & 3 \\
1 & 2 & 2 & 4 \\
\end{tabular}\\
\hline
$S_{(4, 520)}$
&
\begin{tabular}{cccc}
1 & 1 & 1 & 1\\
1 & 1 & 1 & 1 \\
1 & 1 & 1 & 1 \\
1 & 2 & 3 & 4 \\
\end{tabular}
&
$S_{(4, 521)}$
&
\begin{tabular}{cccc}
1 & 1 & 1 & 1 \\
1 & 1 & 1 & 1 \\
1 & 1 & 1 & 3 \\
1 & 2 & 3 & 4 \\
\end{tabular}
&
$S_{(4, 522)}$
&
\begin{tabular}{cccc}
1 & 1 & 1 & 1 \\
1 & 1 & 1 & 2 \\
1 & 1 & 1 & 2 \\
1 & 2 & 3 & 4 \\
\end{tabular}\\
\hline
$S_{(4, 523)}$
&
\begin{tabular}{cccc}
1 & 1 & 1 & 1\\
1 & 1 & 1 & 2 \\
1 & 1 & 1 & 3 \\
1 & 2 & 3 & 4 \\
\end{tabular}
&
$S_{(4, 524)}$
&
\begin{tabular}{cccc}
1 & 1 & 1 & 1 \\
1 & 1 & 1 & 3 \\
1 & 1 & 1 & 3 \\
1 & 2 & 3 & 4 \\
\end{tabular}
&
$S_{(4, 525)}$
&
\begin{tabular}{cccc}
1 & 1 & 1 & 1 \\
1 & 1 & 1 & 2 \\
1 & 1 & 2 & 3 \\
1 & 2 & 3 & 4 \\
\end{tabular}\\
\hline
$S_{(4, 526)}$
&
\begin{tabular}{cccc}
1 & 1 & 1 & 1\\
1 & 1 & 1 & 1 \\
1 & 1 & 3 & 3 \\
1 & 2 & 3 & 4 \\
\end{tabular}
&
$S_{(4, 527)}$
&
\begin{tabular}{cccc}
1 & 1 & 1 & 1 \\
1 & 1 & 1 & 2 \\
1 & 1 & 3 & 3 \\
1 & 2 & 3 & 4 \\
\end{tabular}
&
$S_{(4, 528)}$
&
\begin{tabular}{cccc}
1 & 1 & 1 & 1 \\
1 & 1 & 2 & 2 \\
1 & 1 & 3 & 3 \\
1 & 2 & 3 & 4 \\
\end{tabular}\\
\hline
$S_{(4, 529)}$
&
\begin{tabular}{cccc}
1 & 1 & 3 & 3\\
1 & 1 & 3 & 3 \\
1 & 1 & 3 & 3 \\
1 & 2 & 3 & 4 \\
\end{tabular}
&
$S_{(4, 530)}$
&
\begin{tabular}{cccc}
1 & 1 & 1 & 1 \\
1 & 1 & 1 & 1 \\
1 & 1 & 1 & 1 \\
1 & 3 & 3 & 4 \\
\end{tabular}
&
$S_{(4, 531)}$
&
\begin{tabular}{cccc}
1 & 1 & 1 & 1 \\
1 & 1 & 1 & 2 \\
1 & 1 & 1 & 3 \\
1 & 3 & 3 & 4 \\
\end{tabular}\\
\hline
$S_{(4, 532)}$
&
\begin{tabular}{cccc}
1 & 1 & 1 & 1\\
1 & 1 & 1 & 3 \\
1 & 1 & 1 & 3 \\
1 & 3 & 3 & 4 \\
\end{tabular}
&
$S_{(4, 533)}$
&
\begin{tabular}{cccc}
1 & 1 & 3 & 4 \\
1 & 1 & 3 & 4 \\
1 & 1 & 3 & 4 \\
1 & 3 & 3 & 4 \\
\end{tabular}
&
$S_{(4, 534)}$
&
\begin{tabular}{cccc}
1 & 1 & 1 & 1 \\
1 & 1 & 1 & 1 \\
1 & 2 & 3 & 3 \\
1 & 2 & 3 & 3 \\
\end{tabular}\\
\hline
$S_{(4, 535)}$
&
\begin{tabular}{cccc}
1 & 1 & 1 & 1\\
1 & 1 & 1 & 1 \\
1 & 2 & 3 & 3 \\
1 & 2 & 3 & 4 \\
\end{tabular}
&
$S_{(4, 536)}$
&
\begin{tabular}{cccc}
1 & 1 & 1 & 1 \\
1 & 1 & 1 & 1 \\
1 & 2 & 3 & 4 \\
1 & 2 & 3 & 4 \\
\end{tabular}
&
$S_{(4, 537)}$
&
\begin{tabular}{cccc}
1 & 1 & 1 & 1 \\
1 & 1 & 1 & 2 \\
1 & 2 & 3 & 3 \\
1 & 2 & 3 & 4 \\
\end{tabular}\\
\hline
$S_{(4, 538)}$
&
\begin{tabular}{cccc}
1 & 1 & 1 & 1 \\
1 & 1 & 1 & 1 \\
1 & 2 & 3 & 3 \\
1 & 2 & 4 & 4 \\
\end{tabular}
&
$S_{(4, 539)}$
&
\begin{tabular}{cccc}
1 & 1 & 1 & 1\\
1 & 1 & 1 & 1 \\
1 & 2 & 3 & 4 \\
1 & 2 & 4 & 4 \\
\end{tabular}
&
$S_{(4, 540)}$
&
\begin{tabular}{cccc}
1 & 1 & 1 & 1 \\
1 & 1 & 1 & 1 \\
1 & 2 & 4 & 4 \\
1 & 2 & 4 & 4 \\
\end{tabular}\\
\hline
$S_{(4, 541)}$
&
\begin{tabular}{cccc}
1 & 1 & 1 & 1 \\
1 & 1 & 2 & 2 \\
1 & 2 & 3 & 3 \\
1 & 2 & 3 & 3 \\
\end{tabular}
&
$S_{(4, 542)}$
&
\begin{tabular}{cccc}
1 & 1 & 1 & 1 \\
1 & 1 & 2 & 2 \\
1 & 2 & 3 & 3 \\
1 & 2 & 3 & 4 \\
\end{tabular}
&
$S_{(4, 543)}$
&
\begin{tabular}{cccc}
1 & 1 & 1 & 1 \\
1 & 1 & 2 & 2 \\
1 & 2 & 3 & 4 \\
1 & 2 & 3 & 4 \\
\end{tabular}\\
\hline
$S_{(4, 544)}$
&
\begin{tabular}{cccc}
1 & 1 & 1 & 1 \\
1 & 1 & 2 & 2 \\
1 & 2 & 3 & 3 \\
1 & 2 & 4 & 4 \\
\end{tabular}
&
$S_{(4, 545)}$
&
\begin{tabular}{cccc}
1 & 1 & 1 & 1 \\
1 & 1 & 2 & 2 \\
1 & 2 & 3 & 4 \\
1 & 2 & 4 & 4 \\
\end{tabular}
&
$S_{(4, 546)}$
&
\begin{tabular}{cccc}
1 & 1 & 1 & 1 \\
1 & 1 & 2 & 2 \\
1 & 2 & 4 & 4 \\
1 & 2 & 4 & 4 \\
\end{tabular}\\
\hline
$S_{(4, 547)}$
&
\begin{tabular}{cccc}
1 & 1 & 1 & 1 \\
1 & 2 & 2 & 2 \\
1 & 2 & 2 & 2 \\
1 & 2 & 2 & 2 \\
\end{tabular}
&
$S_{(4, 548)}$
&
\begin{tabular}{cccc}
1 & 1 & 1 & 1 \\
1 & 2 & 2 & 2 \\
1 & 2 & 2 & 2 \\
1 & 2 & 2 & 3 \\
\end{tabular}
&
$S_{(4, 549)}$
&
\begin{tabular}{cccc}
1 & 1 & 1 & 1 \\
1 & 2 & 2 & 2 \\
1 & 2 & 2 & 2 \\
1 & 2 & 2 & 4 \\
\end{tabular}\\
\hline
$S_{(4, 550)}$
&
\begin{tabular}{cccc}
1 & 1 & 1 & 1 \\
1 & 2 & 2 & 2 \\
1 & 2 & 2 & 3 \\
1 & 2 & 2 & 4 \\
\end{tabular}
&
$S_{(4, 551)}$
&
\begin{tabular}{cccc}
1 & 1 & 1 & 1 \\
1 & 2 & 2 & 4 \\
1 & 2 & 2 & 4 \\
1 & 2 & 2 & 4 \\
\end{tabular}
&
$S_{(4, 552)}$
&
\begin{tabular}{cccc}
1 & 1 & 1 & 1 \\
1 & 2 & 2 & 2 \\
1 & 2 & 2 & 2 \\
1 & 2 & 3 & 4 \\
\end{tabular}\\
\hline
$S_{(4, 553)}$
&
\begin{tabular}{cccc}
1 & 1 & 1 & 1 \\
1 & 2 & 2 & 2 \\
1 & 2 & 2 & 3 \\
1 & 2 & 3 & 4 \\
\end{tabular}
&
$S_{(4, 554)}$
&
\begin{tabular}{cccc}
1 & 1 & 1 & 1 \\
1 & 2 & 2 & 2 \\
1 & 2 & 3 & 3 \\
1 & 2 & 3 & 3 \\
\end{tabular}
&
$S_{(4, 555)}$
&
\begin{tabular}{cccc}
1 & 1 & 1 & 1 \\
1 & 2 & 2 & 2 \\
1 & 2 & 3 & 3 \\
1 & 2 & 3 & 4 \\
\end{tabular}\\
\hline
\end{tabular}
\end{table}

\begin{table}[htbp]
\label{tb3}
\begin{tabular}{cccccc}
\hline
Semiring & $\cdot$ & Semiring & $\cdot$ & Semiring & $\cdot$\\
\hline
$S_{(4, 556)}$
&
\begin{tabular}{cccc}
1 & 1 & 1 & 1 \\
1 & 2 & 2 & 2 \\
1 & 2 & 3 & 4 \\
1 & 2 & 3 & 4 \\
\end{tabular}
&
$S_{(4, 557)}$
&
\begin{tabular}{cccc}
1 & 1 & 1 & 1 \\
1 & 2 & 2 & 2 \\
1 & 2 & 3 & 3 \\
1 & 2 & 4 & 4 \\
\end{tabular}
&
$S_{(4, 558)}$
&
\begin{tabular}{cccc}
1 & 1 & 1 & 1 \\
1 & 2 & 2 & 2 \\
1 & 2 & 3 & 4 \\
1 & 2 & 4 & 4 \\
\end{tabular}\\
\hline
$S_{(4, 559)}$
&
\begin{tabular}{cccc}
1 & 1 & 1 & 1 \\
1 & 2 & 2 & 4 \\
1 & 2 & 3 & 4 \\
1 & 2 & 4 & 4 \\
\end{tabular}
&
$S_{(4, 560)}$
&
\begin{tabular}{cccc}
1 & 1 & 1 & 1 \\
1 & 2 & 2 & 2 \\
1 & 2 & 4 & 4 \\
1 & 2 & 4 & 4 \\
\end{tabular}
&
$S_{(4, 561)}$
&
\begin{tabular}{cccc}
1 & 1 & 1 & 1 \\
1 & 2 & 3 & 3 \\
1 & 2 & 3 & 3 \\
1 & 2 & 3 & 3 \\
\end{tabular}\\
\hline
$S_{(4, 562)}$
&
\begin{tabular}{cccc}
1 & 1 & 1 & 1 \\
1 & 2 & 3 & 3 \\
1 & 2 & 3 & 3 \\
1 & 2 & 3 & 4 \\
\end{tabular}
&
$S_{(4, 563)}$
&
\begin{tabular}{cccc}
1 & 1 & 1 & 1 \\
1 & 2 & 3 & 4 \\
1 & 2 & 3 & 4 \\
1 & 2 & 3 & 4 \\
\end{tabular}
&
$S_{(4, 564)}$
&
\begin{tabular}{cccc}
1 & 1 & 1 & 1 \\
1 & 2 & 4 & 4 \\
1 & 2 & 4 & 4 \\
1 & 2 & 4 & 4 \\
\end{tabular}\\
\hline
$S_{(4, 565)}$
&
\begin{tabular}{cccc}
1 & 1 & 1 & 1 \\
1 & 2 & 2 & 2 \\
1 & 2 & 2 & 2 \\
1 & 4 & 4 & 4 \\
\end{tabular}
&
$S_{(4, 566)}$
&
\begin{tabular}{cccc}
1 & 1 & 1 & 1 \\
1 & 2 & 2 & 4 \\
1 & 2 & 2 & 4 \\
1 & 4 & 4 & 4 \\
\end{tabular}
&
$S_{(4, 567)}$
&
\begin{tabular}{cccc}
1 & 1 & 1 & 4 \\
1 & 2 & 2 & 4 \\
1 & 2 & 2 & 4 \\
1 & 4 & 4 & 4 \\
\end{tabular}\\
\hline
$S_{(4, 568)}$
&
\begin{tabular}{cccc}
1 & 1 & 1 & 1 \\
1 & 2 & 2 & 2 \\
1 & 2 & 3 & 4 \\
1 & 4 & 4 & 4 \\
\end{tabular}
&
$S_{(4, 569)}$
&
\begin{tabular}{cccc}
1 & 1 & 1 & 1 \\
1 & 2 & 2 & 4 \\
1 & 2 & 3 & 4 \\
1 & 4 & 4 & 4 \\
\end{tabular}
&
$S_{(4, 570)}$
&
\begin{tabular}{cccc}
1 & 1 & 1 & 4 \\
1 & 2 & 2 & 4 \\
1 & 2 & 3 & 4 \\
1 & 4 & 4 & 4 \\
\end{tabular}\\
\hline
$S_{(4, 571)}$
&
\begin{tabular}{cccc}
1 & 1 & 1 & 1 \\
1 & 2 & 3 & 4 \\
1 & 2 & 3 & 4 \\
1 & 4 & 4 & 4 \\
\end{tabular}
&
$S_{(4, 572)}$
&
\begin{tabular}{cccc}
1 & 1 & 1 & 4 \\
1 & 2 & 3 & 4 \\
1 & 2 & 3 & 4 \\
1 & 4 & 4 & 4 \\
\end{tabular}
&
$S_{(4, 573)}$
&
\begin{tabular}{cccc}
1 & 1 & 1 & 1 \\
1 & 2 & 2 & 2 \\
1 & 3 & 3 & 3 \\
1 & 3 & 3 & 3 \\
\end{tabular}\\
\hline
$S_{(4, 574)}$
&
\begin{tabular}{cccc}
1 & 1 & 1 & 1 \\
1 & 2 & 2 & 2 \\
1 & 3 & 3 & 3 \\
1 & 3 & 3 & 4 \\
\end{tabular}
&
$S_{(4, 575)}$
&
\begin{tabular}{cccc}
1 & 1 & 1 & 1 \\
1 & 2 & 3 & 3 \\
1 & 3 & 3 & 3 \\
1 & 3 & 3 & 3 \\
\end{tabular}
&
$S_{(4, 576)}$
&
\begin{tabular}{cccc}
1 & 1 & 1 & 1 \\
1 & 2 & 3 & 3 \\
1 & 3 & 3 & 3 \\
1 & 3 & 3 & 4 \\
\end{tabular}\\
\hline
$S_{(4, 577)}$
&
\begin{tabular}{cccc}
1 & 1 & 1 & 1 \\
1 & 2 & 3 & 4 \\
1 & 3 & 3 & 4 \\
1 & 3 & 3 & 4 \\
\end{tabular}
&
$S_{(4, 578)}$
&
\begin{tabular}{cccc}
1 & 1 & 3 & 3 \\
1 & 2 & 3 & 3 \\
1 & 3 & 3 & 3 \\
1 & 3 & 3 & 3 \\
\end{tabular}
&
$S_{(4, 579)}$
&
\begin{tabular}{cccc}
1 & 1 & 3 & 3 \\
1 & 2 & 3 & 3 \\
1 & 3 & 3 & 3 \\
1 & 3 & 3 & 4 \\
\end{tabular}\\
\hline
$S_{(4, 580)}$
&
\begin{tabular}{cccc}
1 & 1 & 3 & 4 \\
1 & 2 & 3 & 4 \\
1 & 3 & 3 & 4 \\
1 & 3 & 3 & 4 \\
\end{tabular}
&
$S_{(4, 581)}$
&
\begin{tabular}{cccc}
1 & 1 & 1 & 1 \\
1 & 2 & 2 & 2 \\
1 & 3 & 3 & 3 \\
1 & 4 & 4 & 4 \\
\end{tabular}
&
$S_{(4, 582)}$
&
\begin{tabular}{cccc}
1 & 1 & 1 & 1 \\
1 & 2 & 2 & 4 \\
1 & 3 & 3 & 4 \\
1 & 4 & 4 & 4 \\
\end{tabular}\\
\hline
$S_{(4, 583)}$
&
\begin{tabular}{cccc}
1 & 1 & 1 & 4 \\
1 & 2 & 2 & 4 \\
1 & 3 & 3 & 4 \\
1 & 4 & 4 & 4 \\
\end{tabular}
&
$S_{(4, 584)}$
&
\begin{tabular}{cccc}
1 & 1 & 1 & 1 \\
1 & 2 & 3 & 3 \\
1 & 3 & 3 & 3 \\
1 & 4 & 4 & 4 \\
\end{tabular}
&
$S_{(4, 585)}$
&
\begin{tabular}{cccc}
1 & 1 & 1 & 1 \\
1 & 2 & 3 & 4 \\
1 & 3 & 3 & 4 \\
1 & 4 & 4 & 4 \\
\end{tabular}\\
\hline
$S_{(4, 586)}$
&
\begin{tabular}{cccc}
1 & 1 & 1 & 4 \\
1 & 2 & 3 & 4 \\
1 & 3 & 3 & 4 \\
1 & 4 & 4 & 4 \\
\end{tabular}
&
$S_{(4, 587)}$
&
\begin{tabular}{cccc}
1 & 1 & 1 & 1 \\
1 & 2 & 3 & 4 \\
1 & 3 & 4 & 4 \\
1 & 4 & 4 & 4 \\
\end{tabular}
&
$S_{(4, 588)}$
&
\begin{tabular}{cccc}
1 & 1 & 1 & 1 \\
1 & 2 & 4 & 4 \\
1 & 3 & 4 & 4 \\
1 & 4 & 4 & 4 \\
\end{tabular}\\
\hline
$S_{(4, 589)}$
&
\begin{tabular}{cccc}
1 & 1 & 4 & 4 \\
1 & 2 & 4 & 4 \\
1 & 3 & 4 & 4 \\
1 & 4 & 4 & 4 \\
\end{tabular}
&
$S_{(4, 590)}$
&
\begin{tabular}{cccc}
1 & 1 & 1 & 1 \\
1 & 2 & 2 & 2 \\
1 & 4 & 4 & 4 \\
1 & 4 & 4 & 4 \\
\end{tabular}
&
$S_{(4, 591)}$
&
\begin{tabular}{cccc}
1 & 1 & 1 & 1 \\
1 & 2 & 3 & 4 \\
1 & 4 & 4 & 4 \\
1 & 4 & 4 & 4 \\
\end{tabular}\\
\hline
$S_{(4, 592)}$
&
\begin{tabular}{cccc}
1 & 1 & 1 & 1 \\
1 & 2 & 4 & 4 \\
1 & 4 & 4 & 4 \\
1 & 4 & 4 & 4 \\
\end{tabular}
&
$S_{(4, 593)}$
&
\begin{tabular}{cccc}
1 & 1 & 4 & 4 \\
1 & 2 & 4 & 4 \\
1 & 4 & 4 & 4 \\
1 & 4 & 4 & 4 \\
\end{tabular}
&
$S_{(4, 594)}$
&
\begin{tabular}{cccc}
1 & 1 & 1 & 1 \\
1 & 3 & 3 & 3 \\
1 & 3 & 3 & 3 \\
1 & 3 & 3 & 3 \\
\end{tabular}\\
\hline
\end{tabular}
\end{table}

\begin{table}[htbp]
\label{tb4}
\begin{tabular}{cccccc}
\hline
Semiring & $\cdot$ & Semiring & $\cdot$ & Semiring & $\cdot$\\
\hline
$S_{(4, 595)}$
&
\begin{tabular}{cccc}
1 & 1 & 1 & 1 \\
1 & 3 & 3 & 3 \\
1 & 3 & 3 & 3 \\
1 & 3 & 3 & 4 \\
\end{tabular}
&
$S_{(4, 596)}$
&
\begin{tabular}{cccc}
1 & 1 & 1 & 1 \\
1 & 3 & 3 & 4 \\
1 & 3 & 3 & 4 \\
1 & 3 & 3 & 4 \\
\end{tabular}
&
$S_{(4, 597)}$
&
\begin{tabular}{cccc}
1 & 1 & 1 & 1 \\
1 & 3 & 3 & 3 \\
1 & 3 & 3 & 3 \\
1 & 4 & 4 & 4 \\
\end{tabular}\\
\hline
$S_{(4, 598)}$
&
\begin{tabular}{cccc}
1 & 1 & 1 & 1 \\
1 & 3 & 3 & 4 \\
1 & 3 & 3 & 4 \\
1 & 4 & 4 & 4 \\
\end{tabular}
&
$S_{(4, 599)}$
&
\begin{tabular}{cccc}
1 & 1 & 1 & 4 \\
1 & 3 & 3 & 4 \\
1 & 3 & 3 & 4 \\
1 & 4 & 4 & 4 \\
\end{tabular}
&
$S_{(4, 600)}$
&
\begin{tabular}{cccc}
1 & 1 & 1 & 1 \\
1 & 3 & 4 & 4 \\
1 & 4 & 4 & 4 \\
1 & 4 & 4 & 4 \\
\end{tabular}\\
\hline
$S_{(4, 601)}$
&
\begin{tabular}{cccc}
1 & 1 & 1 & 1 \\
1 & 4 & 4 & 4 \\
1 & 4 & 4 & 4 \\
1 & 4 & 4 & 4 \\
\end{tabular}
&
$S_{(4, 602)}$
&
\begin{tabular}{cccc}
1 & 2 & 2 & 2 \\
1 & 2 & 2 & 2 \\
1 & 2 & 2 & 2 \\
1 & 2 & 2 & 2 \\
\end{tabular}
&
$S_{(4, 603)}$
&
\begin{tabular}{cccc}
1 & 2 & 2 & 2 \\
1 & 2 & 2 & 2 \\
1 & 2 & 2 & 2 \\
1 & 2 & 2 & 3 \\
\end{tabular}\\
\hline
$S_{(4, 604)}$
&
\begin{tabular}{cccc}
1 & 2 & 2 & 2 \\
1 & 2 & 2 & 2 \\
1 & 2 & 2 & 2 \\
1 & 2 & 2 & 4 \\
\end{tabular}
&
$S_{(4, 605)}$
&
\begin{tabular}{cccc}
1 & 2 & 2 & 2 \\
1 & 2 & 2 & 2 \\
1 & 2 & 2 & 3 \\
1 & 2 & 2 & 4 \\
\end{tabular}
&
$S_{(4, 606)}$
&
\begin{tabular}{cccc}
1 & 2 & 2 & 4 \\
1 & 2 & 2 & 4 \\
1 & 2 & 2 & 4 \\
1 & 2 & 2 & 4 \\
\end{tabular}\\
\hline
$S_{(4, 607)}$
&
\begin{tabular}{cccc}
1 & 2 & 2 & 2 \\
1 & 2 & 2 & 2 \\
1 & 2 & 2 & 2 \\
1 & 2 & 3 & 4 \\
\end{tabular}
&
$S_{(4, 608)}$
&
\begin{tabular}{cccc}
1 & 2 & 2 & 2 \\
1 & 2 & 2 & 2 \\
1 & 2 & 2 & 3 \\
1 & 2 & 3 & 4 \\
\end{tabular}
&
$S_{(4, 609)}$
&
\begin{tabular}{cccc}
1 & 2 & 2 & 2 \\
1 & 2 & 2 & 2 \\
1 & 2 & 3 & 3 \\
1 & 2 & 3 & 3 \\
\end{tabular}\\
\hline
$S_{(4, 610)}$
&
\begin{tabular}{cccc}
1 & 2 & 2 & 2 \\
1 & 2 & 2 & 2 \\
1 & 2 & 3 & 3 \\
1 & 2 & 3 & 4 \\
\end{tabular}
&
$S_{(4, 611)}$
&
\begin{tabular}{cccc}
1 & 2 & 2 & 2 \\
1 & 2 & 2 & 2 \\
1 & 2 & 3 & 4 \\
1 & 2 & 3 & 4 \\
\end{tabular}
&
$S_{(4, 612)}$
&
\begin{tabular}{cccc}
1 & 2 & 2 & 2 \\
1 & 2 & 2 & 2 \\
1 & 2 & 3 & 3 \\
1 & 2 & 4 & 4 \\
\end{tabular}\\
\hline
$S_{(4, 613)}$
&
\begin{tabular}{cccc}
1 & 2 & 2 & 2 \\
1 & 2 & 2 & 2 \\
1 & 2 & 3 & 4 \\
1 & 2 & 4 & 4 \\
\end{tabular}
&
$S_{(4, 614)}$
&
\begin{tabular}{cccc}
1 & 2 & 2 & 4 \\
1 & 2 & 2 & 4 \\
1 & 2 & 3 & 4 \\
1 & 2 & 4 & 4 \\
\end{tabular}
&
$S_{(4, 615)}$
&
\begin{tabular}{cccc}
1 & 2 & 2 & 2 \\
1 & 2 & 2 & 2 \\
1 & 2 & 4 & 4 \\
1 & 2 & 4 & 4 \\
\end{tabular}\\
\hline
$S_{(4, 616)}$
&
\begin{tabular}{cccc}
1 & 2 & 2 & 4 \\
1 & 2 & 4 & 4 \\
1 & 2 & 4 & 4 \\
1 & 2 & 4 & 4 \\
\end{tabular}
&
$S_{(4, 617)}$
&
\begin{tabular}{cccc}
1 & 2 & 3 & 3 \\
1 & 2 & 3 & 3 \\
1 & 2 & 3 & 3 \\
1 & 2 & 3 & 3 \\
\end{tabular}
&
$S_{(4, 618)}$
&
\begin{tabular}{cccc}
1 & 2 & 3 & 3 \\
1 & 2 & 3 & 3 \\
1 & 2 & 3 & 3 \\
1 & 2 & 3 & 4 \\
\end{tabular}\\
\hline
$S_{(4, 619)}$
&
\begin{tabular}{cccc}
1 & 2 & 3 & 4 \\
1 & 2 & 3 & 4 \\
1 & 2 & 3 & 4 \\
1 & 2 & 3 & 4 \\
\end{tabular}
&
$S_{(4, 620)}$
&
\begin{tabular}{cccc}
1 & 2 & 4 & 4 \\
1 & 2 & 4 & 4 \\
1 & 2 & 4 & 4 \\
1 & 2 & 4 & 4 \\
\end{tabular}
&
$S_{(4, 621)}$
&
\begin{tabular}{cccc}
1 & 3 & 3 & 3 \\
1 & 3 & 3 & 3 \\
1 & 3 & 3 & 3 \\
1 & 3 & 3 & 3 \\
\end{tabular}\\
\hline
$S_{(4, 622)}$
&
\begin{tabular}{cccc}
1 & 3 & 3 & 3 \\
1 & 3 & 3 & 3 \\
1 & 3 & 3 & 3 \\
1 & 3 & 3 & 4 \\
\end{tabular}
&
$S_{(4, 623)}$
&
\begin{tabular}{cccc}
1 & 3 & 3 & 4 \\
1 & 3 & 3 & 4 \\
1 & 3 & 3 & 4 \\
1 & 3 & 3 & 4 \\
\end{tabular}
&
$S_{(4, 624)}$
&
\begin{tabular}{cccc}
1 & 4 & 4 & 4 \\
1 & 4 & 4 & 4 \\
1 & 4 & 4 & 4 \\
1 & 4 & 4 & 4 \\
\end{tabular}\\
\hline
$S_{(4, 625)}$
&
\begin{tabular}{cccc}
1 & 1 & 1 & 1 \\
1 & 1 & 1 & 1 \\
1 & 1 & 1 & 1 \\
4 & 4 & 4 & 4 \\
\end{tabular}
&
$S_{(4, 626)}$
&
\begin{tabular}{cccc}
1 & 1 & 1 & 4 \\
1 & 1 & 1 & 4 \\
1 & 1 & 1 & 4 \\
4 & 4 & 4 & 4 \\
\end{tabular}
&
$S_{(4, 627)}$
&
\begin{tabular}{cccc}
1 & 1 & 1 & 4 \\
1 & 1 & 1 & 4 \\
1 & 1 & 2 & 4 \\
4 & 4 & 4 & 4 \\
\end{tabular}\\
\hline
$S_{(4, 628)}$
&
\begin{tabular}{cccc}
1 & 1 & 1 & 1 \\
1 & 1 & 1 & 1 \\
1 & 1 & 3 & 4 \\
4 & 4 & 4 & 4 \\
\end{tabular}
&
$S_{(4, 629)}$
&
\begin{tabular}{cccc}
1 & 1 & 1 & 4 \\
1 & 1 & 1 & 4 \\
1 & 1 & 3 & 4 \\
4 & 4 & 4 & 4 \\
\end{tabular}
&
$S_{(4, 630)}$
&
\begin{tabular}{cccc}
1 & 1 & 1 & 4 \\
1 & 1 & 2 & 4 \\
1 & 1 & 3 & 4 \\
4 & 4 & 4 & 4 \\
\end{tabular}\\
\hline
$S_{(4, 631)}$
&
\begin{tabular}{cccc}
1 & 1 & 3 & 4 \\
1 & 1 & 3 & 4 \\
1 & 1 & 3 & 4 \\
4 & 4 & 4 & 4 \\
\end{tabular}
&
$S_{(4, 632)}$
&
\begin{tabular}{cccc}
1 & 1 & 1 & 1 \\
1 & 1 & 1 & 1 \\
1 & 2 & 3 & 4 \\
4 & 4 & 4 & 4 \\
\end{tabular}
&
$S_{(4, 633)}$
&
\begin{tabular}{cccc}
1 & 1 & 1 & 4 \\
1 & 1 & 1 & 4 \\
1 & 2 & 3 & 4 \\
4 & 4 & 4 & 4 \\
\end{tabular}\\
\hline
\end{tabular}
\end{table}

\begin{table}[htbp]
\label{tb5}
\begin{tabular}{cccccc}
\hline
Semiring & $\cdot$ & Semiring & $\cdot$ & Semiring & $\cdot$\\
\hline
$S_{(4, 634)}$
&
\begin{tabular}{cccc}
1 & 1 & 1 & 4 \\
1 & 1 & 2 & 4 \\
1 & 2 & 3 & 4 \\
4 & 4 & 4 & 4 \\
\end{tabular}
&
$S_{(4, 635)}$
&
\begin{tabular}{cccc}
1 & 1 & 1 & 1 \\
1 & 2 & 2 & 4 \\
1 & 2 & 2 & 4 \\
4 & 4 & 4 & 4 \\
\end{tabular}
&
$S_{(4, 636)}$
&
\begin{tabular}{cccc}
1 & 1 & 1 & 4 \\
1 & 2 & 2 & 4 \\
1 & 2 & 2 & 4 \\
4 & 4 & 4 & 4 \\
\end{tabular}\\
\hline
$S_{(4, 637)}$
&
\begin{tabular}{cccc}
1 & 1 & 1 & 1 \\
1 & 2 & 2 & 4 \\
1 & 2 & 3 & 4 \\
4 & 4 & 4 & 4 \\
\end{tabular}
&
$S_{(4, 638)}$
&
\begin{tabular}{cccc}
1 & 1 & 1 & 4 \\
1 & 2 & 2 & 4 \\
1 & 2 & 3 & 4 \\
4 & 4 & 4 & 4 \\
\end{tabular}
&
$S_{(4, 639)}$
&
\begin{tabular}{cccc}
1 & 1 & 1 & 1 \\
1 & 2 & 3 & 4 \\
1 & 2 & 3 & 4 \\
4 & 4 & 4 & 4 \\
\end{tabular}\\
\hline
$S_{(4, 640)}$
&
\begin{tabular}{cccc}
1 & 1 & 1 & 4 \\
1 & 2 & 3 & 4 \\
1 & 2 & 3 & 4 \\
4 & 4 & 4 & 4 \\
\end{tabular}
&
$S_{(4, 641)}$
&
\begin{tabular}{cccc}
1 & 1 & 1 & 1 \\
1 & 2 & 2 & 4 \\
1 & 3 & 3 & 4 \\
4 & 4 & 4 & 4 \\
\end{tabular}
&
$S_{(4, 642)}$
&
\begin{tabular}{cccc}
1 & 1 & 1 & 4 \\
1 & 2 & 2 & 4 \\
1 & 3 & 3 & 4 \\
4 & 4 & 4 & 4 \\
\end{tabular}\\
\hline
$S_{(4, 643)}$
&
\begin{tabular}{cccc}
1 & 1 & 1 & 1 \\
1 & 2 & 3 & 4 \\
1 & 3 & 3 & 4 \\
4 & 4 & 4 & 4 \\
\end{tabular}
&
$S_{(4, 644)}$
&
\begin{tabular}{cccc}
1 & 1 & 1 & 4 \\
1 & 2 & 3 & 4 \\
1 & 3 & 3 & 4 \\
4 & 4 & 4 & 4 \\
\end{tabular}
&
$S_{(4, 645)}$
&
\begin{tabular}{cccc}
1 & 1 & 3 & 4 \\
1 & 2 & 3 & 4 \\
1 & 3 & 3 & 4 \\
4 & 4 & 4 & 4 \\
\end{tabular}\\
\hline
$S_{(4, 646)}$
&
\begin{tabular}{cccc}
1 & 1 & 1 & 1 \\
1 & 3 & 3 & 4 \\
1 & 3 & 3 & 4 \\
4 & 4 & 4 & 4 \\
\end{tabular}
&
$S_{(4, 647)}$
&
\begin{tabular}{cccc}
1 & 1 & 1 & 4 \\
1 & 3 & 3 & 4 \\
1 & 3 & 3 & 4 \\
4 & 4 & 4 & 4 \\
\end{tabular}
&
$S_{(4, 648)}$
&
\begin{tabular}{cccc}
1 & 2 & 2 & 4 \\
1 & 2 & 2 & 4 \\
1 & 2 & 2 & 4 \\
4 & 4 & 4 & 4 \\
\end{tabular}\\
\hline
$S_{(4, 649)}$
&
\begin{tabular}{cccc}
1 & 2 & 2 & 4 \\
1 & 2 & 2 & 4 \\
1 & 2 & 3 & 4 \\
4 & 4 & 4 & 4 \\
\end{tabular}
&
$S_{(4, 650)}$
&
\begin{tabular}{cccc}
1 & 2 & 3 & 4 \\
1 & 2 & 3 & 4 \\
1 & 2 & 3 & 4 \\
4 & 4 & 4 & 4 \\
\end{tabular}
&
$S_{(4, 651)}$
&
\begin{tabular}{cccc}
1 & 3 & 3 & 4 \\
1 & 3 & 3 & 4 \\
1 & 3 & 3 & 4 \\
4 & 4 & 4 & 4 \\
\end{tabular}\\
\hline
$S_{(4, 652)}$
&
\begin{tabular}{cccc}
1 & 1 & 1 & 1 \\
1 & 1 & 1 & 1 \\
3 & 3 & 3 & 3 \\
3 & 3 & 3 & 3 \\
\end{tabular}
&
$S_{(4, 653)}$
&
\begin{tabular}{cccc}
1 & 1 & 1 & 1 \\
1 & 1 & 1 & 1 \\
3 & 3 & 3 & 3 \\
3 & 3 & 3 & 4 \\
\end{tabular}
&
$S_{(4, 654)}$
&
\begin{tabular}{cccc}
1 & 1 & 1 & 1 \\
1 & 1 & 1 & 2 \\
3 & 3 & 3 & 3 \\
3 & 3 & 3 & 4 \\
\end{tabular}\\
\hline
$S_{(4, 655)}$
&
\begin{tabular}{cccc}
1 & 1 & 3 & 3 \\
1 & 1 & 3 & 3 \\
3 & 3 & 3 & 3 \\
3 & 3 & 3 & 3 \\
\end{tabular}
&
$S_{(4, 656)}$
&
\begin{tabular}{cccc}
1 & 1 & 3 & 3 \\
1 & 1 & 3 & 3 \\
3 & 3 & 3 & 3 \\
3 & 3 & 3 & 4 \\
\end{tabular}
&
$S_{(4, 657)}$
&
\begin{tabular}{cccc}
1 & 1 & 3 & 4 \\
1 & 1 & 3 & 4 \\
3 & 3 & 3 & 4 \\
3 & 3 & 3 & 4 \\
\end{tabular}\\
\hline
$S_{(4, 658)}$
&
\begin{tabular}{cccc}
1 & 1 & 1 & 1 \\
1 & 2 & 3 & 3 \\
3 & 3 & 3 & 3 \\
3 & 3 & 3 & 3 \\
\end{tabular}
&
$S_{(4, 659)}$
&
\begin{tabular}{cccc}
1 & 1 & 1 & 1 \\
1 & 2 & 3 & 3 \\
3 & 3 & 3 & 3 \\
3 & 3 & 3 & 4 \\
\end{tabular}
&
$S_{(4, 660)}$
&
\begin{tabular}{cccc}
1 & 1 & 3 & 3 \\
1 & 2 & 3 & 3 \\
3 & 3 & 3 & 3 \\
3 & 3 & 3 & 3 \\
\end{tabular}\\
\hline
$S_{(4, 661)}$
&
\begin{tabular}{cccc}
1 & 1 & 3 & 3 \\
1 & 2 & 3 & 3 \\
3 & 3 & 3 & 3 \\
3 & 3 & 3 & 4 \\
\end{tabular}
&
$S_{(4, 662)}$
&
\begin{tabular}{cccc}
1 & 1 & 3 & 4 \\
1 & 2 & 3 & 4 \\
3 & 3 & 3 & 4 \\
3 & 3 & 3 & 4 \\
\end{tabular}
&
$S_{(4, 663)}$
&
\begin{tabular}{cccc}
1 & 2 & 3 & 3 \\
1 & 2 & 3 & 3 \\
3 & 3 & 3 & 3 \\
3 & 3 & 3 & 3 \\
\end{tabular}\\
\hline
$S_{(4, 664)}$
&
\begin{tabular}{cccc}
1 & 2 & 3 & 3 \\
1 & 2 & 3 & 3 \\
3 & 3 & 3 & 3 \\
3 & 3 & 3 & 4 \\
\end{tabular}
&
$S_{(4, 665)}$
&
\begin{tabular}{cccc}
1 & 2 & 3 & 4 \\
1 & 2 & 3 & 4 \\
3 & 3 & 3 & 4 \\
3 & 3 & 3 & 4 \\
\end{tabular}
&
$S_{(4, 666)}$
&
\begin{tabular}{cccc}
1 & 1 & 1 & 1 \\
1 & 1 & 1 & 1 \\
3 & 3 & 3 & 3 \\
4 & 4 & 4 & 4 \\
\end{tabular}\\
\hline
$S_{(4, 667)}$
&
\begin{tabular}{cccc}
1 & 1 & 1 & 1 \\
1 & 1 & 1 & 3 \\
3 & 3 & 3 & 3 \\
4 & 4 & 4 & 4 \\
\end{tabular}
&
$S_{(4, 668)}$
&
\begin{tabular}{cccc}
1 & 1 & 1 & 4 \\
1 & 1 & 1 & 4 \\
3 & 3 & 3 & 4 \\
4 & 4 & 4 & 4 \\
\end{tabular}
&
$S_{(4, 669)}$
&
\begin{tabular}{cccc}
1 & 1 & 3 & 3 \\
1 & 1 & 3 & 3 \\
3 & 3 & 3 & 3 \\
4 & 4 & 4 & 4 \\
\end{tabular}\\
\hline
$S_{(4, 670)}$
&
\begin{tabular}{cccc}
1 & 1 & 3 & 4 \\
1 & 1 & 3 & 4 \\
3 & 3 & 3 & 4 \\
4 & 4 & 4 & 4 \\
\end{tabular}
&
$S_{(4, 671)}$
&
\begin{tabular}{cccc}
1 & 1 & 3 & 4 \\
1 & 1 & 3 & 4 \\
3 & 3 & 4 & 4 \\
4 & 4 & 4 & 4 \\
\end{tabular}
&
$S_{(4, 672)}$
&
\begin{tabular}{cccc}
1 & 1 & 4 & 4 \\
1 & 1 & 4 & 4 \\
3 & 3 & 4 & 4 \\
4 & 4 & 4 & 4 \\
\end{tabular}\\
\hline
\end{tabular}
\end{table}

\begin{table}[htbp]
\label{tb6}
\begin{tabular}{cccccc}
\hline
Semiring & $\cdot$ & Semiring & $\cdot$ & Semiring & $\cdot$\\
\hline
$S_{(4, 673)}$
&
\begin{tabular}{cccc}
1 & 1 & 1 & 1 \\
1 & 2 & 3 & 3 \\
3 & 3 & 3 & 3 \\
4 & 4 & 4 & 4 \\
\end{tabular}
&
$S_{(4, 674)}$
&
\begin{tabular}{cccc}
1 & 1 & 1 & 4 \\
1 & 2 & 3 & 4 \\
3 & 3 & 3 & 4 \\
4 & 4 & 4 & 4 \\
\end{tabular}
&
$S_{(4, 675)}$
&
\begin{tabular}{cccc}
1 & 1 & 3 & 3 \\
1 & 2 & 3 & 3 \\
3 & 3 & 3 & 3 \\
4 & 4 & 4 & 4 \\
\end{tabular}\\
\hline
$S_{(4, 676)}$
&
\begin{tabular}{cccc}
1 & 1 & 3 & 4 \\
1 & 2 & 3 & 4 \\
3 & 3 & 3 & 4 \\
4 & 4 & 4 & 4 \\
\end{tabular}
&
$S_{(4, 677)}$
&
\begin{tabular}{cccc}
1 & 1 & 3 & 4 \\
1 & 2 & 3 & 4 \\
3 & 3 & 4 & 4 \\
4 & 4 & 4 & 4 \\
\end{tabular}
&
$S_{(4, 678)}$
&
\begin{tabular}{cccc}
1 & 1 & 4 & 4 \\
1 & 2 & 4 & 4 \\
3 & 3 & 4 & 4 \\
4 & 4 & 4 & 4 \\
\end{tabular}\\
\hline
$S_{(4, 679)}$
&
\begin{tabular}{cccc}
1 & 2 & 3 & 3 \\
1 & 2 & 3 & 3 \\
3 & 3 & 3 & 3 \\
4 & 4 & 4 & 4 \\
\end{tabular}
&
$S_{(4, 680)}$
&
\begin{tabular}{cccc}
1 & 2 & 3 & 4 \\
1 & 2 & 3 & 4 \\
3 & 3 & 3 & 4 \\
4 & 4 & 4 & 4 \\
\end{tabular}
&
$S_{(4, 681)}$
&
\begin{tabular}{cccc}
1 & 2 & 3 & 4 \\
1 & 2 & 3 & 4 \\
3 & 3 & 4 & 4 \\
4 & 4 & 4 & 4 \\
\end{tabular}\\
\hline
$S_{(4, 682)}$
&
\begin{tabular}{cccc}
1 & 2 & 4 & 4 \\
1 & 2 & 4 & 4 \\
3 & 3 & 4 & 4 \\
4 & 4 & 4 & 4 \\
\end{tabular}
&
$S_{(4, 683)}$
&
\begin{tabular}{cccc}
1 & 1 & 1 & 1 \\
1 & 1 & 1 & 1 \\
4 & 4 & 4 & 4 \\
4 & 4 & 4 & 4 \\
\end{tabular}
&
$S_{(4, 684)}$
&
\begin{tabular}{cccc}
1 & 1 & 3 & 4 \\
1 & 1 & 3 & 4 \\
4 & 4 & 4 & 4 \\
4 & 4 & 4 & 4 \\
\end{tabular}\\
\hline
$S_{(4, 685)}$
&
\begin{tabular}{cccc}
1 & 1 & 4 & 4 \\
1 & 1 & 4 & 4 \\
4 & 4 & 4 & 4 \\
4 & 4 & 4 & 4 \\
\end{tabular}
&
$S_{(4, 686)}$
&
\begin{tabular}{cccc}
1 & 1 & 1 & 1 \\
1 & 2 & 3 & 4 \\
4 & 4 & 4 & 4 \\
4 & 4 & 4 & 4 \\
\end{tabular}
&
$S_{(4, 687)}$
&
\begin{tabular}{cccc}
1 & 1 & 1 & 1 \\
1 & 2 & 4 & 4 \\
4 & 4 & 4 & 4 \\
4 & 4 & 4 & 4 \\
\end{tabular}\\
\hline
$S_{(4, 688)}$
&
\begin{tabular}{cccc}
1 & 1 & 3 & 4 \\
1 & 2 & 3 & 4 \\
4 & 4 & 4 & 4 \\
4 & 4 & 4 & 4 \\
\end{tabular}
&
$S_{(4, 689)}$
&
\begin{tabular}{cccc}
1 & 1 & 4 & 4 \\
1 & 2 & 4 & 4 \\
4 & 4 & 4 & 4 \\
4 & 4 & 4 & 4 \\
\end{tabular}
&
$S_{(4, 690)}$
&
\begin{tabular}{cccc}
1 & 2 & 3 & 4 \\
1 & 2 & 3 & 4 \\
4 & 4 & 4 & 4 \\
4 & 4 & 4 & 4 \\
\end{tabular}\\
\hline
$S_{(4, 691)}$
&
\begin{tabular}{cccc}
1 & 2 & 4 & 4 \\
1 & 2 & 4 & 4 \\
4 & 4 & 4 & 4 \\
4 & 4 & 4 & 4 \\
\end{tabular}
&
$S_{(4, 692)}$
&
\begin{tabular}{cccc}
1 & 1 & 1 & 1 \\
2 & 2 & 2 & 2 \\
2 & 2 & 2 & 2 \\
2 & 2 & 2 & 2 \\
\end{tabular}
&
$S_{(4, 693)}$
&
\begin{tabular}{cccc}
1 & 1 & 1 & 1 \\
2 & 2 & 2 & 2 \\
2 & 2 & 2 & 2 \\
2 & 2 & 2 & 3 \\
\end{tabular}\\
\hline
$S_{(4, 694)}$
&
\begin{tabular}{cccc}
1 & 1 & 1 & 1 \\
2 & 2 & 2 & 2 \\
2 & 2 & 2 & 2 \\
2 & 2 & 2 & 4 \\
\end{tabular}
&
$S_{(4, 695)}$
&
\begin{tabular}{cccc}
1 & 1 & 1 & 1 \\
2 & 2 & 2 & 2 \\
2 & 2 & 2 & 3 \\
2 & 2 & 2 & 4 \\
\end{tabular}
&
$S_{(4, 696)}$
&
\begin{tabular}{cccc}
1 & 1 & 1 & 1 \\
2 & 2 & 2 & 2 \\
2 & 2 & 2 & 2 \\
2 & 2 & 3 & 4 \\
\end{tabular}\\
\hline
$S_{(4, 697)}$
&
\begin{tabular}{cccc}
1 & 1 & 1 & 1 \\
2 & 2 & 2 & 2 \\
2 & 2 & 2 & 3 \\
2 & 2 & 3 & 4 \\
\end{tabular}
&
$S_{(4, 698)}$
&
\begin{tabular}{cccc}
1 & 1 & 1 & 1 \\
2 & 2 & 2 & 2 \\
2 & 2 & 3 & 3 \\
2 & 2 & 3 & 3 \\
\end{tabular}
&
$S_{(4, 699)}$
&
\begin{tabular}{cccc}
1 & 1 & 1 & 1 \\
2 & 2 & 2 & 2 \\
2 & 2 & 3 & 3 \\
2 & 2 & 3 & 4 \\
\end{tabular}\\
\hline
$S_{(4, 700)}$
&
\begin{tabular}{cccc}
1 & 1 & 1 & 1 \\
2 & 2 & 2 & 2 \\
2 & 2 & 3 & 4 \\
2 & 2 & 3 & 4 \\
\end{tabular}
&
$S_{(4, 701)}$
&
\begin{tabular}{cccc}
1 & 1 & 1 & 1 \\
2 & 2 & 2 & 2 \\
2 & 2 & 3 & 3 \\
2 & 2 & 4 & 4 \\
\end{tabular}
&
$S_{(4, 702)}$
&
\begin{tabular}{cccc}
1 & 1 & 1 & 1 \\
2 & 2 & 2 & 2 \\
2 & 2 & 3 & 4 \\
2 & 2 & 4 & 4 \\
\end{tabular}\\
\hline
$S_{(4, 703)}$
&
\begin{tabular}{cccc}
1 & 1 & 1 & 1 \\
2 & 2 & 2 & 2 \\
2 & 2 & 4 & 4 \\
2 & 2 & 4 & 4 \\
\end{tabular}
&
$S_{(4, 704)}$
&
\begin{tabular}{cccc}
1 & 2 & 2 & 2 \\
2 & 2 & 2 & 2 \\
2 & 2 & 2 & 2 \\
2 & 2 & 2 & 2 \\
\end{tabular}
&
$S_{(4, 705)}$
&
\begin{tabular}{cccc}
1 & 2 & 2 & 2 \\
2 & 2 & 2 & 2 \\
2 & 2 & 2 & 2 \\
2 & 2 & 2 & 3 \\
\end{tabular}\\
\hline
$S_{(4, 706)}$
&
\begin{tabular}{cccc}
1 & 2 & 2 & 2 \\
2 & 2 & 2 & 2 \\
2 & 2 & 2 & 2 \\
2 & 2 & 2 & 4 \\
\end{tabular}
&
$S_{(4, 707)}$
&
\begin{tabular}{cccc}
1 & 2 & 2 & 2 \\
2 & 2 & 2 & 2 \\
2 & 2 & 2 & 3 \\
2 & 2 & 2 & 4 \\
\end{tabular}
&
$S_{(4, 708)}$
&
\begin{tabular}{cccc}
1 & 2 & 2 & 4 \\
2 & 2 & 2 & 4 \\
2 & 2 & 2 & 4 \\
2 & 2 & 2 & 4 \\
\end{tabular}\\
\hline
$S_{(4, 709)}$
&
\begin{tabular}{cccc}
1 & 2 & 2 & 2 \\
2 & 2 & 2 & 2 \\
2 & 2 & 2 & 2 \\
2 & 2 & 3 & 4 \\
\end{tabular}
&
$S_{(4, 710)}$
&
\begin{tabular}{cccc}
1 & 2 & 2 & 2 \\
2 & 2 & 2 & 2 \\
2 & 2 & 2 & 3 \\
2 & 2 & 3 & 4 \\
\end{tabular}
&
$S_{(4, 711)}$
&
\begin{tabular}{cccc}
1 & 2 & 2 & 2 \\
2 & 2 & 2 & 2 \\
2 & 2 & 3 & 3 \\
2 & 2 & 3 & 3 \\
\end{tabular}\\
\hline
\end{tabular}
\end{table}

\begin{table}[htbp]
\label{tb7}
\begin{tabular}{cccccc}
\hline
Semiring & $\cdot$ & Semiring & $\cdot$ & Semiring & $\cdot$\\
\hline
$S_{(4, 712)}$
&
\begin{tabular}{cccc}
1 & 2 & 2 & 2 \\
2 & 2 & 2 & 2 \\
2 & 2 & 3 & 3 \\
2 & 2 & 3 & 4 \\
\end{tabular}
&
$S_{(4, 713)}$
&
\begin{tabular}{cccc}
1 & 2 & 2 & 2 \\
2 & 2 & 2 & 2 \\
2 & 2 & 3 & 4 \\
2 & 2 & 3 & 4 \\
\end{tabular}
&
$S_{(4, 714)}$
&
\begin{tabular}{cccc}
1 & 2 & 2 & 2 \\
2 & 2 & 2 & 2 \\
2 & 2 & 3 & 3 \\
2 & 2 & 4 & 4 \\
\end{tabular}\\
\hline
$S_{(4, 715)}$
&
\begin{tabular}{cccc}
1 & 2 & 2 & 2 \\
2 & 2 & 2 & 2 \\
2 & 2 & 3 & 4 \\
2 & 2 & 4 & 4 \\
\end{tabular}
&
$S_{(4, 716)}$
&
\begin{tabular}{cccc}
1 & 2 & 2 & 4 \\
2 & 2 & 2 & 4 \\
2 & 2 & 3 & 4 \\
2 & 2 & 4 & 4 \\
\end{tabular}
&
$S_{(4, 717)}$
&
\begin{tabular}{cccc}
1 & 2 & 2 & 2 \\
2 & 2 & 2 & 2 \\
2 & 2 & 4 & 4 \\
2 & 2 & 4 & 4 \\
\end{tabular}\\
\hline
$S_{(4, 718)}$
&
\begin{tabular}{cccc}
1 & 2 & 3 & 3 \\
2 & 2 & 3 & 3 \\
2 & 2 & 3 & 3 \\
2 & 2 & 3 & 3 \\
\end{tabular}
&
$S_{(4, 719)}$
&
\begin{tabular}{cccc}
1 & 2 & 3 & 3 \\
2 & 2 & 3 & 3 \\
2 & 2 & 3 & 3 \\
2 & 2 & 3 & 4 \\
\end{tabular}
&
$S_{(4, 720)}$
&
\begin{tabular}{cccc}
1 & 2 & 3 & 4 \\
2 & 2 & 3 & 4 \\
2 & 2 & 3 & 4 \\
2 & 2 & 3 & 4 \\
\end{tabular}\\
\hline
$S_{(4, 721)}$
&
\begin{tabular}{cccc}
1 & 2 & 3 & 4 \\
2 & 2 & 4 & 4 \\
2 & 2 & 4 & 4 \\
2 & 2 & 4 & 4 \\
\end{tabular}
&
$S_{(4, 722)}$
&
\begin{tabular}{cccc}
1 & 2 & 4 & 4 \\
2 & 2 & 4 & 4 \\
2 & 2 & 4 & 4 \\
2 & 2 & 4 & 4 \\
\end{tabular}
&
$S_{(4, 723)}$
&
\begin{tabular}{cccc}
1 & 1 & 1 & 1 \\
2 & 2 & 2 & 2 \\
2 & 2 & 2 & 2 \\
4 & 4 & 4 & 4 \\
\end{tabular}\\
\hline
$S_{(4, 724)}$
&
\begin{tabular}{cccc}
1 & 1 & 1 & 4 \\
2 & 2 & 2 & 4 \\
2 & 2 & 2 & 4 \\
4 & 4 & 4 & 4 \\
\end{tabular}
&
$S_{(4, 725)}$
&
\begin{tabular}{cccc}
1 & 1 & 1 & 1 \\
2 & 2 & 2 & 2 \\
2 & 2 & 3 & 4 \\
4 & 4 & 4 & 4 \\
\end{tabular}
&
$S_{(4, 726)}$
&
\begin{tabular}{cccc}
1 & 1 & 1 & 4 \\
2 & 2 & 2 & 4 \\
2 & 2 & 3 & 4 \\
4 & 4 & 4 & 4 \\
\end{tabular}\\
\hline
$S_{(4, 727)}$
&
\begin{tabular}{cccc}
1 & 1 & 1 & 1 \\
2 & 2 & 2 & 2 \\
2 & 4 & 4 & 4 \\
4 & 4 & 4 & 4 \\
\end{tabular}
&
$S_{(4, 728)}$
&
\begin{tabular}{cccc}
1 & 2 & 2 & 2 \\
2 & 2 & 2 & 2 \\
2 & 2 & 2 & 2 \\
4 & 4 & 4 & 4 \\
\end{tabular}
&
$S_{(4, 729)}$
&
\begin{tabular}{cccc}
1 & 2 & 2 & 4 \\
2 & 2 & 2 & 4 \\
2 & 2 & 2 & 4 \\
4 & 4 & 4 & 4 \\
\end{tabular}\\
\hline
$S_{(4, 730)}$
&
\begin{tabular}{cccc}
1 & 2 & 2 & 2 \\
2 & 2 & 2 & 2 \\
2 & 2 & 3 & 4 \\
4 & 4 & 4 & 4 \\
\end{tabular}
&
$S_{(4, 731)}$
&
\begin{tabular}{cccc}
1 & 2 & 2 & 4 \\
2 & 2 & 2 & 4 \\
2 & 2 & 3 & 4 \\
4 & 4 & 4 & 4 \\
\end{tabular}
&
$S_{(4, 732)}$
&
\begin{tabular}{cccc}
1 & 2 & 3 & 4 \\
2 & 2 & 3 & 4 \\
2 & 2 & 3 & 4 \\
4 & 4 & 4 & 4 \\
\end{tabular}\\
\hline
$S_{(4, 733)}$
&
\begin{tabular}{cccc}
1 & 2 & 2 & 4 \\
2 & 4 & 4 & 4 \\
2 & 4 & 4 & 4 \\
4 & 4 & 4 & 4 \\
\end{tabular}
&
$S_{(4, 734)}$
&
\begin{tabular}{cccc}
1 & 2 & 3 & 4 \\
2 & 4 & 4 & 4 \\
2 & 4 & 4 & 4 \\
4 & 4 & 4 & 4 \\
\end{tabular}
&
$S_{(4, 735)}$
&
\begin{tabular}{cccc}
1 & 4 & 4 & 4 \\
2 & 4 & 4 & 4 \\
2 & 4 & 4 & 4 \\
4 & 4 & 4 & 4 \\
\end{tabular}\\
\hline
$S_{(4, 736)}$
&
\begin{tabular}{cccc}
1 & 1 & 1 & 1 \\
2 & 2 & 2 & 2 \\
3 & 3 & 3 & 3 \\
3 & 3 & 3 & 3 \\
\end{tabular}
&
$S_{(4, 737)}$
&
\begin{tabular}{cccc}
1 & 1 & 1 & 1 \\
2 & 2 & 2 & 2 \\
3 & 3 & 3 & 3 \\
3 & 3 & 3 & 4 \\
\end{tabular}
&
$S_{(4, 738)}$
&
\begin{tabular}{cccc}
1 & 1 & 3 & 3 \\
2 & 2 & 3 & 3 \\
3 & 3 & 3 & 3 \\
3 & 3 & 3 & 3 \\
\end{tabular}\\
\hline
$S_{(4, 739)}$
&
\begin{tabular}{cccc}
1 & 1 & 3 & 3 \\
2 & 2 & 3 & 3 \\
3 & 3 & 3 & 3 \\
3 & 3 & 3 & 4 \\
\end{tabular}
&
$S_{(4, 740)}$
&
\begin{tabular}{cccc}
1 & 1 & 3 & 4 \\
2 & 2 & 3 & 4 \\
3 & 3 & 3 & 4 \\
3 & 3 & 3 & 4 \\
\end{tabular}
&
$S_{(4, 741)}$
&
\begin{tabular}{cccc}
1 & 2 & 2 & 2 \\
2 & 2 & 2 & 2 \\
3 & 3 & 3 & 3 \\
3 & 3 & 3 & 3 \\
\end{tabular}\\
\hline
$S_{(4, 742)}$
&
\begin{tabular}{cccc}
1 & 2 & 2 & 2 \\
2 & 2 & 2 & 2 \\
3 & 3 & 3 & 3 \\
3 & 3 & 3 & 4 \\
\end{tabular}
&
$S_{(4, 743)}$
&
\begin{tabular}{cccc}
1 & 2 & 3 & 3 \\
2 & 2 & 3 & 3 \\
3 & 3 & 3 & 3 \\
3 & 3 & 3 & 3 \\
\end{tabular}
&
$S_{(4, 744)}$
&
\begin{tabular}{cccc}
1 & 2 & 3 & 3 \\
2 & 2 & 3 & 3 \\
3 & 3 & 3 & 3 \\
3 & 3 & 3 & 4 \\
\end{tabular}\\
\hline
$S_{(4, 745)}$
&
\begin{tabular}{cccc}
1 & 2 & 3 & 4 \\
2 & 2 & 3 & 4 \\
3 & 3 & 3 & 4 \\
3 & 3 & 3 & 4 \\
\end{tabular}
&
$S_{(4, 746)}$
&
\begin{tabular}{cccc}
1 & 2 & 3 & 3 \\
2 & 3 & 3 & 3 \\
3 & 3 & 3 & 3 \\
3 & 3 & 3 & 3 \\
\end{tabular}
&
$S_{(4, 747)}$
&
\begin{tabular}{cccc}
1 & 2 & 3 & 3 \\
2 & 3 & 3 & 3 \\
3 & 3 & 3 & 3 \\
3 & 3 & 3 & 4 \\
\end{tabular}\\
\hline
$S_{(4, 748)}$
&
\begin{tabular}{cccc}
1 & 2 & 3 & 4 \\
2 & 3 & 3 & 4 \\
3 & 3 & 3 & 4 \\
3 & 3 & 3 & 4 \\
\end{tabular}
&
$S_{(4, 749)}$
&
\begin{tabular}{cccc}
1 & 3 & 3 & 3 \\
2 & 3 & 3 & 3 \\
3 & 3 & 3 & 3 \\
3 & 3 & 3 & 3 \\
\end{tabular}
&
$S_{(4, 750)}$
&
\begin{tabular}{cccc}
1 & 3 & 3 & 3 \\
2 & 3 & 3 & 3 \\
3 & 3 & 3 & 3 \\
3 & 3 & 3 & 4 \\
\end{tabular}\\
\hline
\end{tabular}
\end{table}

\begin{table}[htbp]
\label{tb8}
\begin{tabular}{cccccc}
\hline
Semiring & $\cdot$ & Semiring & $\cdot$ & Semiring & $\cdot$\\
\hline
$S_{(4, 751)}$
&
\begin{tabular}{cccc}
1 & 3 & 3 & 4 \\
2 & 3 & 3 & 4 \\
3 & 3 & 3 & 4 \\
3 & 3 & 3 & 4 \\
\end{tabular}
&
$S_{(4, 752)}$
&
\begin{tabular}{cccc}
1 & 1 & 1 & 1 \\
2 & 2 & 2 & 2 \\
3 & 3 & 3 & 3 \\
4 & 4 & 4 & 4 \\
\end{tabular}
&
$S_{(4, 753)}$
&
\begin{tabular}{cccc}
1 & 1 & 1 & 4 \\
2 & 2 & 2 & 4 \\
3 & 3 & 3 & 4 \\
4 & 4 & 4 & 4 \\
\end{tabular}\\
\hline
$S_{(4, 754)}$
&
\begin{tabular}{cccc}
1 & 1 & 3 & 3 \\
2 & 2 & 3 & 3 \\
3 & 3 & 3 & 3 \\
4 & 4 & 4 & 4 \\
\end{tabular}
&
$S_{(4, 755)}$
&
\begin{tabular}{cccc}
1 & 1 & 3 & 4 \\
2 & 2 & 3 & 4 \\
3 & 3 & 3 & 4 \\
4 & 4 & 4 & 4 \\
\end{tabular}
&
$S_{(4, 756)}$
&
\begin{tabular}{cccc}
1 & 1 & 3 & 4 \\
2 & 2 & 3 & 4 \\
3 & 3 & 4 & 4 \\
4 & 4 & 4 & 4 \\
\end{tabular}\\
\hline
$S_{(4, 757)}$
&
\begin{tabular}{cccc}
1 & 1 & 4 & 4 \\
2 & 2 & 4 & 4 \\
3 & 3 & 4 & 4 \\
4 & 4 & 4 & 4 \\
\end{tabular}
&
$S_{(4, 758)}$
&
\begin{tabular}{cccc}
1 & 2 & 2 & 2 \\
2 & 2 & 2 & 2 \\
3 & 3 & 3 & 3 \\
4 & 4 & 4 & 4 \\
\end{tabular}
&
$S_{(4, 759)}$
&
\begin{tabular}{cccc}
1 & 2 & 2 & 4 \\
2 & 2 & 2 & 4 \\
3 & 3 & 3 & 4 \\
4 & 4 & 4 & 4 \\
\end{tabular}\\
\hline
$S_{(4, 760)}$
&
\begin{tabular}{cccc}
1 & 2 & 3 & 3 \\
2 & 2 & 3 & 3 \\
3 & 3 & 3 & 3 \\
4 & 4 & 4 & 4 \\
\end{tabular}
&
$S_{(4, 761)}$
&
\begin{tabular}{cccc}
1 & 2 & 3 & 4 \\
2 & 2 & 3 & 4 \\
3 & 3 & 3 & 4 \\
4 & 4 & 4 & 4 \\
\end{tabular}
&
$S_{(4, 762)}$
&
\begin{tabular}{cccc}
1 & 2 & 3 & 4 \\
2 & 2 & 3 & 4 \\
3 & 3 & 4 & 4 \\
4 & 4 & 4 & 4 \\
\end{tabular}\\
\hline
$S_{(4, 763)}$
&
\begin{tabular}{cccc}
1 & 2 & 3 & 4 \\
2 & 2 & 4 & 4 \\
3 & 3 & 4 & 4 \\
4 & 4 & 4 & 4 \\
\end{tabular}
&
$S_{(4, 764)}$
&
\begin{tabular}{cccc}
1 & 2 & 4 & 4 \\
2 & 2 & 4 & 4 \\
3 & 3 & 4 & 4 \\
4 & 4 & 4 & 4 \\
\end{tabular}
&
$S_{(4, 765)}$
&
\begin{tabular}{cccc}
1 & 2 & 2 & 2 \\
2 & 2 & 2 & 2 \\
3 & 4 & 4 & 4 \\
4 & 4 & 4 & 4 \\
\end{tabular}\\
\hline
$S_{(4, 766)}$
&
\begin{tabular}{cccc}
1 & 2 & 3 & 4 \\
2 & 2 & 3 & 4 \\
3 & 4 & 4 & 4 \\
4 & 4 & 4 & 4 \\
\end{tabular}
&
$S_{(4, 767)}$
&
\begin{tabular}{cccc}
1 & 2 & 3 & 4 \\
2 & 2 & 4 & 4 \\
3 & 4 & 4 & 4 \\
4 & 4 & 4 & 4 \\
\end{tabular}
&
$S_{(4, 768)}$
&
\begin{tabular}{cccc}
1 & 2 & 4 & 4 \\
2 & 2 & 4 & 4 \\
3 & 4 & 4 & 4 \\
4 & 4 & 4 & 4 \\
\end{tabular}\\
\hline
$S_{(4, 769)}$
&
\begin{tabular}{cccc}
1 & 2 & 3 & 3 \\
2 & 3 & 3 & 3 \\
3 & 3 & 3 & 3 \\
4 & 4 & 4 & 4 \\
\end{tabular}
&
$S_{(4, 770)}$
&
\begin{tabular}{cccc}
1 & 2 & 3 & 4 \\
2 & 3 & 3 & 4 \\
3 & 3 & 3 & 4 \\
4 & 4 & 4 & 4 \\
\end{tabular}
&
$S_{(4, 771)}$
&
\begin{tabular}{cccc}
1 & 2 & 3 & 4 \\
2 & 3 & 4 & 4 \\
3 & 4 & 4 & 4 \\
4 & 4 & 4 & 4 \\
\end{tabular}\\
\hline
$S_{(4, 772)}$
&
\begin{tabular}{cccc}
1 & 2 & 2 & 4 \\
2 & 4 & 4 & 4 \\
3 & 4 & 4 & 4 \\
4 & 4 & 4 & 4 \\
\end{tabular}
&
$S_{(4, 773)}$
&
\begin{tabular}{cccc}
1 & 2 & 3 & 4 \\
2 & 4 & 4 & 4 \\
3 & 4 & 4 & 4 \\
4 & 4 & 4 & 4 \\
\end{tabular}
&
$S_{(4, 774)}$
&
\begin{tabular}{cccc}
1 & 2 & 4 & 4 \\
2 & 4 & 4 & 4 \\
3 & 4 & 4 & 4 \\
4 & 4 & 4 & 4 \\
\end{tabular}\\
\hline
$S_{(4, 775)}$
&
\begin{tabular}{cccc}
1 & 3 & 3 & 3 \\
2 & 3 & 3 & 3 \\
3 & 3 & 3 & 3 \\
4 & 4 & 4 & 4 \\
\end{tabular}
&
$S_{(4, 776)}$
&
\begin{tabular}{cccc}
1 & 3 & 3 & 4 \\
2 & 3 & 3 & 4 \\
3 & 3 & 3 & 4 \\
4 & 4 & 4 & 4 \\
\end{tabular}
&
$S_{(4, 777)}$
&
\begin{tabular}{cccc}
1 & 3 & 3 & 4 \\
2 & 4 & 4 & 4 \\
3 & 4 & 4 & 4 \\
4 & 4 & 4 & 4 \\
\end{tabular}\\
\hline
$S_{(4, 778)}$
&
\begin{tabular}{cccc}
1 & 4 & 4 & 4 \\
2 & 4 & 4 & 4 \\
3 & 4 & 4 & 4 \\
4 & 4 & 4 & 4 \\
\end{tabular}
&
$S_{(4, 779)}$
&
\begin{tabular}{cccc}
1 & 1 & 1 & 1 \\
2 & 2 & 2 & 2 \\
4 & 4 & 4 & 4 \\
4 & 4 & 4 & 4 \\
\end{tabular}
&
$S_{(4, 780)}$
&
\begin{tabular}{cccc}
1 & 1 & 3 & 4 \\
2 & 2 & 3 & 4 \\
4 & 4 & 4 & 4 \\
4 & 4 & 4 & 4 \\
\end{tabular}\\
\hline
$S_{(4, 781)}$
&
\begin{tabular}{cccc}
1 & 1 & 4 & 4 \\
2 & 2 & 4 & 4 \\
4 & 4 & 4 & 4 \\
4 & 4 & 4 & 4 \\
\end{tabular}
&
$S_{(4, 782)}$
&
\begin{tabular}{cccc}
1 & 2 & 2 & 2 \\
2 & 2 & 2 & 2 \\
4 & 4 & 4 & 4 \\
4 & 4 & 4 & 4 \\
\end{tabular}
&
$S_{(4, 783)}$
&
\begin{tabular}{cccc}
1 & 2 & 3 & 4 \\
2 & 2 & 3 & 4 \\
4 & 4 & 4 & 4 \\
4 & 4 & 4 & 4 \\
\end{tabular}\\
\hline
$S_{(4, 784)}$
&
\begin{tabular}{cccc}
1 & 2 & 3 & 4 \\
2 & 2 & 4 & 4 \\
4 & 4 & 4 & 4 \\
4 & 4 & 4 & 4 \\
\end{tabular}
&
$S_{(4, 785)}$
&
\begin{tabular}{cccc}
1 & 2 & 4 & 4 \\
2 & 2 & 4 & 4 \\
4 & 4 & 4 & 4 \\
4 & 4 & 4 & 4 \\
\end{tabular}
&
$S_{(4, 786)}$
&
\begin{tabular}{cccc}
1 & 2 & 3 & 4 \\
2 & 4 & 4 & 4 \\
4 & 4 & 4 & 4 \\
4 & 4 & 4 & 4 \\
\end{tabular}\\
\hline
$S_{(4, 787)}$
&
\begin{tabular}{cccc}
1 & 2 & 4 & 4 \\
2 & 4 & 4 & 4 \\
4 & 4 & 4 & 4 \\
4 & 4 & 4 & 4 \\
\end{tabular}
&
$S_{(4, 788)}$
&
\begin{tabular}{cccc}
1 & 4 & 4 & 4 \\
2 & 4 & 4 & 4 \\
4 & 4 & 4 & 4 \\
4 & 4 & 4 & 4 \\
\end{tabular}
&
$S_{(4, 789)}$
&
\begin{tabular}{cccc}
1 & 1 & 1 & 1 \\
3 & 3 & 3 & 3 \\
3 & 3 & 3 & 3 \\
3 & 3 & 3 & 3 \\
\end{tabular}\\
\hline
\end{tabular}
\end{table}

\begin{table}[htbp]
\label{tb9}
\begin{tabular}{cccccc}
\hline
Semiring & $\cdot$ & Semiring & $\cdot$ & Semiring & $\cdot$\\
\hline
$S_{(4, 790)}$
&
\begin{tabular}{cccc}
1 & 1 & 1 & 1 \\
3 & 3 & 3 & 3 \\
3 & 3 & 3 & 3 \\
3 & 3 & 3 & 4 \\
\end{tabular}
&
$S_{(4, 791)}$
&
\begin{tabular}{cccc}
1 & 2 & 3 & 3 \\
3 & 3 & 3 & 3 \\
3 & 3 & 3 & 3 \\
3 & 3 & 3 & 3 \\
\end{tabular}
&
$S_{(4, 792)}$
&
\begin{tabular}{cccc}
1 & 2 & 3 & 3 \\
3 & 3 & 3 & 3 \\
3 & 3 & 3 & 3 \\
3 & 3 & 3 & 4 \\
\end{tabular}\\
\hline
$S_{(4, 793)}$
&
\begin{tabular}{cccc}
1 & 2 & 3 & 4 \\
3 & 3 & 3 & 4 \\
3 & 3 & 3 & 4 \\
3 & 3 & 3 & 4 \\
\end{tabular}
&
$S_{(4, 794)}$
&
\begin{tabular}{cccc}
1 & 3 & 3 & 3 \\
3 & 3 & 3 & 3 \\
3 & 3 & 3 & 3 \\
3 & 3 & 3 & 3 \\
\end{tabular}
&
$S_{(4, 795)}$
&
\begin{tabular}{cccc}
1 & 3 & 3 & 3 \\
3 & 3 & 3 & 3 \\
3 & 3 & 3 & 3 \\
3 & 3 & 3 & 4 \\
\end{tabular}\\
\hline
$S_{(4, 796)}$
&
\begin{tabular}{cccc}
1 & 3 & 3 & 4 \\
3 & 3 & 3 & 4 \\
3 & 3 & 3 & 4 \\
3 & 3 & 3 & 4 \\
\end{tabular}
&
$S_{(4, 797)}$
&
\begin{tabular}{cccc}
1 & 1 & 1 & 1 \\
3 & 3 & 3 & 3 \\
3 & 3 & 3 & 3 \\
4 & 4 & 4 & 4 \\
\end{tabular}
&
$S_{(4, 798)}$
&
\begin{tabular}{cccc}
1 & 1 & 1 & 4 \\
3 & 3 & 3 & 4 \\
3 & 3 & 3 & 4 \\
4 & 4 & 4 & 4 \\
\end{tabular}\\
\hline
$S_{(4, 799)}$
&
\begin{tabular}{cccc}
1 & 2 & 3 & 3 \\
3 & 3 & 3 & 3 \\
3 & 3 & 3 & 3 \\
4 & 4 & 4 & 4 \\
\end{tabular}
&
$S_{(4, 800)}$
&
\begin{tabular}{cccc}
1 & 2 & 3 & 4 \\
3 & 3 & 3 & 4 \\
3 & 3 & 3 & 4 \\
4 & 4 & 4 & 4 \\
\end{tabular}
&
$S_{(4, 801)}$
&
\begin{tabular}{cccc}
1 & 2 & 3 & 4 \\
3 & 4 & 4 & 4 \\
3 & 4 & 4 & 4 \\
4 & 4 & 4 & 4 \\
\end{tabular}\\
\hline
$S_{(4, 802)}$
&
\begin{tabular}{cccc}
1 & 3 & 3 & 3 \\
3 & 3 & 3 & 3 \\
3 & 3 & 3 & 3 \\
4 & 4 & 4 & 4 \\
\end{tabular}
&
$S_{(4, 803)}$
&
\begin{tabular}{cccc}
1 & 3 & 3 & 4 \\
3 & 3 & 3 & 4 \\
3 & 3 & 3 & 4 \\
4 & 4 & 4 & 4 \\
\end{tabular}
&
$S_{(4, 804)}$
&
\begin{tabular}{cccc}
1 & 3 & 3 & 4 \\
3 & 4 & 4 & 4 \\
3 & 4 & 4 & 4 \\
4 & 4 & 4 & 4 \\
\end{tabular}\\
\hline
$S_{(4, 805)}$
&
\begin{tabular}{cccc}
1 & 4 & 4 & 4 \\
3 & 4 & 4 & 4 \\
3 & 4 & 4 & 4 \\
4 & 4 & 4 & 4 \\
\end{tabular}
&
$S_{(4, 806)}$
&
\begin{tabular}{cccc}
1 & 1 & 1 & 1 \\
4 & 4 & 4 & 4 \\
4 & 4 & 4 & 4 \\
4 & 4 & 4 & 4 \\
\end{tabular}
&
$S_{(4, 807)}$
&
\begin{tabular}{cccc}
1 & 2 & 2 & 4 \\
4 & 4 & 4 & 4 \\
4 & 4 & 4 & 4 \\
4 & 4 & 4 & 4 \\
\end{tabular}\\
\hline
$S_{(4, 808)}$
&
\begin{tabular}{cccc}
1 & 2 & 3 & 4 \\
4 & 4 & 4 & 4 \\
4 & 4 & 4 & 4 \\
4 & 4 & 4 & 4 \\
\end{tabular}
&
$S_{(4, 809)}$
&
\begin{tabular}{cccc}
1 & 2 & 4 & 4 \\
4 & 4 & 4 & 4 \\
4 & 4 & 4 & 4 \\
4 & 4 & 4 & 4 \\
\end{tabular}
&
$S_{(4, 810)}$
&
\begin{tabular}{cccc}
1 & 3 & 3 & 4 \\
4 & 4 & 4 & 4 \\
4 & 4 & 4 & 4 \\
4 & 4 & 4 & 4 \\
\end{tabular}\\
\hline
$S_{(4, 811)}$
&
\begin{tabular}{cccc}
1 & 4 & 4 & 4 \\
4 & 4 & 4 & 4 \\
4 & 4 & 4 & 4 \\
4 & 4 & 4 & 4 \\
\end{tabular}
&
$S_{(4, 812)}$
&
\begin{tabular}{cccc}
2 & 2 & 2 & 2 \\
2 & 2 & 2 & 2 \\
2 & 2 & 2 & 2 \\
2 & 2 & 2 & 2 \\
\end{tabular}
&
$S_{(4, 813)}$
&
\begin{tabular}{cccc}
2 & 2 & 2 & 2 \\
2 & 2 & 2 & 2 \\
2 & 2 & 2 & 2 \\
2 & 2 & 2 & 3 \\
\end{tabular}\\
\hline
$S_{(4, 814)}$
&
\begin{tabular}{cccc}
2 & 2 & 2 & 2 \\
2 & 2 & 2 & 2 \\
2 & 2 & 2 & 2 \\
2 & 2 & 2 & 4 \\
\end{tabular}
&
$S_{(4, 815)}$
&
\begin{tabular}{cccc}
2 & 2 & 2 & 2 \\
2 & 2 & 2 & 2 \\
2 & 2 & 2 & 3 \\
2 & 2 & 2 & 4 \\
\end{tabular}
&
$S_{(4, 816)}$
&
\begin{tabular}{cccc}
2 & 2 & 2 & 4 \\
2 & 2 & 2 & 4 \\
2 & 2 & 2 & 4 \\
2 & 2 & 2 & 4 \\
\end{tabular}\\
\hline
$S_{(4, 817)}$
&
\begin{tabular}{cccc}
2 & 2 & 2 & 2 \\
2 & 2 & 2 & 2 \\
2 & 2 & 2 & 2 \\
2 & 2 & 3 & 4 \\
\end{tabular}
&
$S_{(4, 818)}$
&
\begin{tabular}{cccc}
2 & 2 & 2 & 2 \\
2 & 2 & 2 & 2 \\
2 & 2 & 2 & 3 \\
2 & 2 & 3 & 4 \\
\end{tabular}
&
$S_{(4, 819)}$
&
\begin{tabular}{cccc}
2 & 2 & 2 & 2 \\
2 & 2 & 2 & 2 \\
2 & 2 & 3 & 3 \\
2 & 2 & 3 & 3 \\
\end{tabular}\\
\hline
$S_{(4, 820)}$
&
\begin{tabular}{cccc}
2 & 2 & 2 & 2 \\
2 & 2 & 2 & 2 \\
2 & 2 & 3 & 3 \\
2 & 2 & 3 & 4 \\
\end{tabular}
&
$S_{(4, 821)}$
&
\begin{tabular}{cccc}
2 & 2 & 2 & 2 \\
2 & 2 & 2 & 2 \\
2 & 2 & 3 & 4 \\
2 & 2 & 3 & 4 \\
\end{tabular}
&
$S_{(4, 822)}$
&
\begin{tabular}{cccc}
2 & 2 & 2 & 2 \\
2 & 2 & 2 & 2 \\
2 & 2 & 3 & 3 \\
2 & 2 & 4 & 4 \\
\end{tabular}\\
\hline
$S_{(4, 823)}$
&
\begin{tabular}{cccc}
2 & 2 & 2 & 2 \\
2 & 2 & 2 & 2 \\
2 & 2 & 3 & 4 \\
2 & 2 & 4 & 4 \\
\end{tabular}
&
$S_{(4, 824)}$
&
\begin{tabular}{cccc}
2 & 2 & 2 & 4 \\
2 & 2 & 2 & 4 \\
2 & 2 & 3 & 4 \\
2 & 2 & 4 & 4 \\
\end{tabular}
&
$S_{(4, 825)}$
&
\begin{tabular}{cccc}
2 & 2 & 2 & 2 \\
2 & 2 & 2 & 2 \\
2 & 2 & 4 & 4 \\
2 & 2 & 4 & 4 \\
\end{tabular}\\
\hline
$S_{(4, 826)}$
&
\begin{tabular}{cccc}
2 & 2 & 3 & 3 \\
2 & 2 & 3 & 3 \\
2 & 2 & 3 & 3 \\
2 & 2 & 3 & 3 \\
\end{tabular}
&
$S_{(4, 827)}$
&
\begin{tabular}{cccc}
2 & 2 & 3 & 3 \\
2 & 2 & 3 & 3 \\
2 & 2 & 3 & 3 \\
2 & 2 & 3 & 4 \\
\end{tabular}
&
$S_{(4, 828)}$
&
\begin{tabular}{cccc}
2 & 2 & 3 & 4 \\
2 & 2 & 3 & 4 \\
2 & 2 & 3 & 4 \\
2 & 2 & 3 & 4 \\
\end{tabular}\\
\hline
\end{tabular}
\end{table}

\begin{table}[htbp]
\label{tb10}
\begin{tabular}{cccccc}
\hline
Semiring & $\cdot$ & Semiring & $\cdot$ & Semiring & $\cdot$\\
\hline
$S_{(4, 829)}$
&
\begin{tabular}{cccc}
2 & 2 & 4 & 4 \\
2 & 2 & 4 & 4 \\
2 & 2 & 4 & 4 \\
2 & 2 & 4 & 4 \\
\end{tabular}
&
$S_{(4, 830)}$
&
\begin{tabular}{cccc}
2 & 2 & 2 & 2 \\
2 & 2 & 2 & 2 \\
2 & 2 & 2 & 2 \\
4 & 4 & 4 & 4 \\
\end{tabular}
&
$S_{(4, 831)}$
&
\begin{tabular}{cccc}
2 & 2 & 2 & 4 \\
2 & 2 & 2 & 4 \\
2 & 2 & 2 & 4 \\
4 & 4 & 4 & 4 \\
\end{tabular}\\
\hline
$S_{(4, 832)}$
&
\begin{tabular}{cccc}
2 & 2 & 2 & 2 \\
2 & 2 & 2 & 2 \\
2 & 2 & 3 & 4 \\
4 & 4 & 4 & 4 \\
\end{tabular}
&
$S_{(4, 833)}$
&
\begin{tabular}{cccc}
2 & 2 & 2 & 4 \\
2 & 2 & 2 & 4 \\
2 & 2 & 3 & 4 \\
4 & 4 & 4 & 4 \\
\end{tabular}
&
$S_{(4, 834)}$
&
\begin{tabular}{cccc}
2 & 2 & 3 & 4 \\
2 & 2 & 3 & 4 \\
2 & 2 & 3 & 4 \\
4 & 4 & 4 & 4 \\
\end{tabular}\\
\hline
$S_{(4, 835)}$
&
\begin{tabular}{cccc}
2 & 2 & 2 & 2 \\
2 & 2 & 2 & 2 \\
3 & 3 & 3 & 3 \\
3 & 3 & 3 & 3 \\
\end{tabular}
&
$S_{(4, 836)}$
&
\begin{tabular}{cccc}
2 & 2 & 2 & 2 \\
2 & 2 & 2 & 2 \\
3 & 3 & 3 & 3 \\
3 & 3 & 3 & 4 \\
\end{tabular}
&
$S_{(4, 837)}$
&
\begin{tabular}{cccc}
2 & 2 & 3 & 3 \\
2 & 2 & 3 & 3 \\
3 & 3 & 3 & 3 \\
3 & 3 & 3 & 3 \\
\end{tabular}\\
\hline
$S_{(4, 838)}$
&
\begin{tabular}{cccc}
2 & 2 & 3 & 3 \\
2 & 2 & 3 & 3 \\
3 & 3 & 3 & 3 \\
3 & 3 & 3 & 4 \\
\end{tabular}
&
$S_{(4, 839)}$
&
\begin{tabular}{cccc}
2 & 2 & 3 & 4 \\
2 & 2 & 3 & 4 \\
3 & 3 & 3 & 4 \\
3 & 3 & 3 & 4 \\
\end{tabular}
&
$S_{(4, 840)}$
&
\begin{tabular}{cccc}
2 & 2 & 2 & 2 \\
2 & 2 & 2 & 2 \\
3 & 3 & 3 & 3 \\
4 & 4 & 4 & 4 \\
\end{tabular}\\
\hline
$S_{(4, 841)}$
&
\begin{tabular}{cccc}
2 & 2 & 2 & 4 \\
2 & 2 & 2 & 4 \\
3 & 3 & 3 & 4 \\
4 & 4 & 4 & 4 \\
\end{tabular}
&
$S_{(4, 842)}$
&
\begin{tabular}{cccc}
2 & 2 & 3 & 3 \\
2 & 2 & 3 & 3 \\
3 & 3 & 3 & 3 \\
4 & 4 & 4 & 4 \\
\end{tabular}
&
$S_{(4, 843)}$
&
\begin{tabular}{cccc}
2 & 2 & 3 & 4 \\
2 & 2 & 3 & 4 \\
3 & 3 & 3 & 4 \\
4 & 4 & 4 & 4 \\
\end{tabular}\\
\hline
$S_{(4, 844)}$
&
\begin{tabular}{cccc}
2 & 2 & 3 & 4 \\
2 & 2 & 3 & 4 \\
3 & 3 & 4 & 4 \\
4 & 4 & 4 & 4 \\
\end{tabular}
&
$S_{(4, 845)}$
&
\begin{tabular}{cccc}
2 & 2 & 4 & 4 \\
2 & 2 & 4 & 4 \\
3 & 3 & 4 & 4 \\
4 & 4 & 4 & 4 \\
\end{tabular}
&
$S_{(4, 846)}$
&
\begin{tabular}{cccc}
2 & 2 & 2 & 2 \\
2 & 2 & 2 & 2 \\
4 & 4 & 4 & 4 \\
4 & 4 & 4 & 4 \\
\end{tabular}\\
\hline
$S_{(4, 847)}$
&
\begin{tabular}{cccc}
2 & 2 & 3 & 4 \\
2 & 2 & 3 & 4 \\
4 & 4 & 4 & 4 \\
4 & 4 & 4 & 4 \\
\end{tabular}
&
$S_{(4, 848)}$
&
\begin{tabular}{cccc}
2 & 2 & 4 & 4 \\
2 & 2 & 4 & 4 \\
4 & 4 & 4 & 4 \\
4 & 4 & 4 & 4 \\
\end{tabular}
&
$S_{(4, 849)}$
&
\begin{tabular}{cccc}
2 & 3 & 3 & 3 \\
3 & 3 & 3 & 3 \\
3 & 3 & 3 & 3 \\
3 & 3 & 3 & 3 \\
\end{tabular}\\
\hline
$S_{(4, 850)}$
&
\begin{tabular}{cccc}
2 & 3 & 3 & 3 \\
3 & 3 & 3 & 3 \\
3 & 3 & 3 & 3 \\
3 & 3 & 3 & 4 \\
\end{tabular}
&
$S_{(4, 851)}$
&
\begin{tabular}{cccc}
2 & 3 & 3 & 4 \\
3 & 3 & 3 & 4 \\
3 & 3 & 3 & 4 \\
3 & 3 & 3 & 4 \\
\end{tabular}
&
$S_{(4, 852)}$
&
\begin{tabular}{cccc}
2 & 3 & 3 & 3 \\
3 & 3 & 3 & 3 \\
3 & 3 & 3 & 3 \\
4 & 4 & 4 & 4 \\
\end{tabular}\\
\hline
$S_{(4, 853)}$
&
\begin{tabular}{cccc}
2 & 3 & 3 & 4 \\
3 & 3 & 3 & 4 \\
3 & 3 & 3 & 4 \\
4 & 4 & 4 & 4 \\
\end{tabular}
&
$S_{(4, 854)}$
&
\begin{tabular}{cccc}
2 & 3 & 4 & 4 \\
3 & 4 & 4 & 4 \\
4 & 4 & 4 & 4 \\
4 & 4 & 4 & 4 \\
\end{tabular}
&
$S_{(4, 855)}$
&
\begin{tabular}{cccc}
2 & 4 & 4 & 4 \\
4 & 4 & 4 & 4 \\
4 & 4 & 4 & 4 \\
4 & 4 & 4 & 4 \\
\end{tabular}\\
\hline
$S_{(4, 856)}$
&
\begin{tabular}{cccc}
3 & 3 & 3 & 3 \\
3 & 3 & 3 & 3 \\
3 & 3 & 3 & 3 \\
3 & 3 & 3 & 3 \\
\end{tabular}
&
$S_{(4, 857)}$
&
\begin{tabular}{cccc}
3 & 3 & 3 & 3 \\
3 & 3 & 3 & 3 \\
3 & 3 & 3 & 3 \\
3 & 3 & 3 & 4 \\
\end{tabular}
&
$S_{(4, 858)}$
&
\begin{tabular}{cccc}
3 & 3 & 3 & 4 \\
3 & 3 & 3 & 4 \\
3 & 3 & 3 & 4 \\
3 & 3 & 3 & 4 \\
\end{tabular}\\
\hline
$S_{(4, 859)}$
&
\begin{tabular}{cccc}
3 & 3 & 3 & 3 \\
3 & 3 & 3 & 3 \\
3 & 3 & 3 & 3 \\
4 & 4 & 4 & 4 \\
\end{tabular}
&
$S_{(4, 860)}$
&
\begin{tabular}{cccc}
3 & 3 & 3 & 4 \\
3 & 3 & 3 & 4 \\
3 & 3 & 3 & 4 \\
4 & 4 & 4 & 4 \\
\end{tabular}
&
$S_{(4, 861)}$
&
\begin{tabular}{cccc}
3 & 3 & 4 & 4 \\
3 & 3 & 4 & 4 \\
4 & 4 & 4 & 4 \\
4 & 4 & 4 & 4 \\
\end{tabular}\\
\hline
$S_{(4, 862)}$
&
\begin{tabular}{cccc}
3 & 3 & 4 & 4 \\
3 & 4 & 4 & 4 \\
4 & 4 & 4 & 4 \\
4 & 4 & 4 & 4 \\
\end{tabular}
&
$S_{(4, 863)}$
&
\begin{tabular}{cccc}
3 & 4 & 4 & 4 \\
3 & 4 & 4 & 4 \\
4 & 4 & 4 & 4 \\
4 & 4 & 4 & 4 \\
\end{tabular}
&
$S_{(4, 864)}$
&
\begin{tabular}{cccc}
3 & 3 & 4 & 4 \\
4 & 4 & 4 & 4 \\
4 & 4 & 4 & 4 \\
4 & 4 & 4 & 4 \\
\end{tabular}\\
\hline
$S_{(4, 865)}$
&
\begin{tabular}{cccc}
3 & 4 & 4 & 4 \\
4 & 4 & 4 & 4 \\
4 & 4 & 4 & 4 \\
4 & 4 & 4 & 4 \\
\end{tabular}
&
$S_{(4, 866)}$
&
\begin{tabular}{cccc}
4 & 4 & 4 & 4 \\
4 & 4 & 4 & 4 \\
4 & 4 & 4 & 4 \\
4 & 4 & 4 & 4 \\
\end{tabular}
&
\\
\hline
\end{tabular}
\end{table}

\end{document}